\documentclass[utf-8]{amsbook} 

\usepackage{amsthm, amssymb, amsmath, amscd}
\usepackage{eurosym}
\usepackage{amsfonts}
\usepackage{amsmath}
\usepackage{setspace}
\usepackage[bookmarks=true]{hyperref} 
\usepackage{bbm}
\usepackage{graphicx,amscd}

\usepackage{slashed}

\usepackage{color}

\usepackage[all,cmtip]{xy}

\usepackage{indentfirst}

\newtheorem{theorem}{Theorem}[section]
\newtheorem{theorem*}{Theorem}[]
\newtheorem{metatheorem*}[theorem*]{Meta theorem}

\newtheorem{conjecture}[theorem]{Conjecture}
\newtheorem{cor}[theorem]{Corollary}

\newtheorem{lemma}[theorem]{Lemma}

\newtheorem{prop}[theorem]{Proposition}

\theoremstyle{definition}
\newtheorem{define}[theorem]{Definition}
\newtheorem{define*}[theorem*]{Definition}
\newtheorem{ex}[theorem]{Example}
\newtheorem{remark}[theorem]{Remark}

\DeclareFontFamily{OT1}{pzc}{}
\DeclareFontShape{OT1}{pzc}{m}{it}{<-> s * [1.1] pzcmi7t}{}
\DeclareMathAlphabet{\mathpzc}{OT1}{pzc}{m}{it}

\newcommand{\Hom}{\textnormal{Hom}}
\newcommand{\End}{\textnormal{End}}
\newcommand{\Lip}{\textnormal{Lip}}
\newcommand{\Dom}{\textnormal{Dom}\,}
\newcommand{\K}{\mathbb{K}}

\newcommand{\A}{\mathcal{A}}
\newcommand{\N}{\field{N}}
\newcommand{\id}{\mathrm{id}}
\newcommand{\ch}{\mathrm{ch}}
\newcommand{\ind}{\mathrm{ind}}

\newcommand{\field}[1]{\mathbb{#1}}

\newcommand{\zkz}{\field{Z}/k\field{Z}}
\newcommand{\C}{\field{C}}

\newcommand{\R}{\field{R}}
\newcommand{\Z}{\mathbbm{Z}}

\begin{document}
\title[Bordisms and $KK$]{Bordisms and unbounded $KK$-theory}

\date{\today}

\author{Robin J. Deeley, Magnus Goffeng, Bram Mesland}
\date{\today}

\begin{abstract}
This monograph studies $KK$-theory in its unbounded model. The central object is the $KK$-bordism group obtained by imposing the $KK$-bordism relation on unbounded $KK$-cycles. In the paradigm of noncommutative geometry, an unbounded $KK$-cycle is a noncommutative geometry in its own right and our approach allows for the study of mildly noncommutative geometries (orbifolds, foliations et cetera) as if they were closed manifolds. The techniques we introduce enable us to directly import manifold techniques and arguments into the important yet technical field of unbounded $KK$-theory. 

Recent decades have seen a tremendous progress in the study of the unbounded model for $KK$ as well as secondary invariants, the first motivated by refining computational tools in Kasparov's $KK$-theory and the second by applications to geometry and topology. The aim of this work is to provide a common framework for these two areas: equipping unbounded $KK$-cycles with a geometrically motivated relation recovering Kasparov's $KK$-theory that is computationally tractable for working with secondary invariants. 
\end{abstract}

\maketitle

\tableofcontents

\section{Introduction}

The seminal Atiyah-Singer index theorem \cite{AS1} dates back to the 1960's and exhibits the deep connection between algebra, topology and geometry on the one hand, and analysis on the other hand that elliptic operators and their index theory provide. To facilitate the connection further, Atiyah \cite{ASell} envisioned a (generalized) homology theory where abstract elliptic operators played the role of cycles and the index theorem was a functoriality statement. Atiyah's vision led Kasparov \cite{Kas1, Kas2} to form his $KK$-theory and Baum-Douglas \cite{BD,BDbor} to form their geometric $K$-homology. The purpose of this monograph is to provide a new approach that marries the analytic techniques of Kasparov's $KK$-theory with the geometric ideas underlying Baum-Douglas' geometric $K$-homology and lies closer to current trends in geometry, topology and mathematical physics. Furthermore, this work aims to provide a coherent set of techniques for unbounded $KK$-theory that encompasses existing applications to index theory, $K$-homology and noncommutative geometry. 

Kasparov's bivariant $K$-theory ($KK$-theory) is a fundamental tool in the study of $C^*$-algebras and applications of $C^*$-algebras to topology, geometry and mathematical physics. Being rooted in $C^*$-algebra theory, $KK$-theory is analytic in nature and its origins lie in the analytic side of the Atiyah-Singer index theorem. Even so the bounded cycles used to define elements in the $KK$-group are at first glance far from the operators typically associated with index theory problems. Such operators are typically elliptic first order differential operators and in particular unbounded. Explicit examples are the de Rham operator, the signature operator, the spin or ${\rm spin^c}$ Dirac operator, among others. From an opposing point of view, one has Baum-Douglas' geometric $K$-homology that takes ${\rm spin^c}$-manifolds as its cycles and closely models the geometric applications. Baum-Douglas' geometric $K$-homology is a $K$-homological incarnation of classical bordism groups where the Thom isomorphism has been implemented by means of having coefficients in $K$-theory. Kasparov's analytic theory and Baum-Douglas geometric theory both come with their own virtues that are of importance in applications. Despite the existence of natural isomorphisms of the theories the discrepancy of the two creates technical problems and created pockets of ad hoc results that have been reinvented several times.

Based on the abundance of the naturally occurring unbounded examples and the conflict between geometric and analytic, we look to the work of Baaj-Julg \cite{baajjulg} who defined the notion of unbounded $KK$-cycle. The bounded transform takes these unbounded operators to Kasparov's notion of bounded $KK$-cycles and Baaj and Julg proved that it is a surjective map onto the $KK$-group \cite{baajjulg} in favourable circumstances. The standard examples from index theory listed above as well as Baum-Douglas' geometric cycles are easily shown to give Baaj-Julg cycles. Perhaps more to the point in our context, these examples give canonical Baaj-Julg cycles, but do not give canonical bounded $KK$-cycles. The associated bounded $KK$-cycles depend on the choice of a cut off function used in the bounded transform. To be clear, the $KK$-class associated to such an operator is canonical as different choices of cut off function lead to homotopic $KK$-cycles. Through this viewpoint, one can see the passage from defining $KK$-theory from bounded cycles to unbounded cycles as a movement to reduce the choices required to define cycles from index theoretical inputs and provide quantitative control. 

Baaj-Julg's unbounded approach to $KK$-theory plays an integral part in Alain Connes' program for spectral noncommutative geometry \cite{connesgravity,Connesbook,Connestrace,connesreconstr}, see also \cite{LRV,renniereconst}. The unbounded aspects of spectral noncommutative geometry have found numerous applications in particle physics \cite{cacicetal,chamconnes,chamconnes2,marcolliearly,nestcc,walterbook} and topological states of matter \cite{bellissardqhe,bellikthe,bournekellen,bourneherman,bournemesland,kellenschulz}. Unbounded $KK$-theory has in the mentioned applications to theoretical physics found utility in its explicit nature connecting analysis to geometry. In recent years the focus has been on the constructive approach to the Kasparov product \cite{BMS,leschkaad1, leschkaad2,kaaddiffabs,kucerovskyprod,LeschMesland,thebeastofmesland,mesren}, a method we shall only occasionally use in the monograph but it will not be center stage. When using Kasparov products, the unbounded approach by no means has a computational monopoly at the level of homotopy classes. In the same way, we by no means claim that the unbounded approach we take in this monograph is a unique method for solving index theoretical problems. {\bf However, the unbounded approach does carry a level of computational ubiquity that we want to highlight.} In this monograph we take unbounded $KK$-theory as a framework in which to perform computations in index theory and $KK$-theory pivoted towards geometric applications. With the viewpoint on noncommutative geometry as geometry, the techniques of this monograph provide geometric proofs for analytically technical results in index theory and bivariant $KK$-theory.

An equivalence relation is needed so that one can define $KK$-theory using Baaj-Julg's unbounded cycles. There are various results towards defining $KK$-groups via the unbounded picture, \cite{Kaadunbdd,meslandvandung} define unbounded $KK$-groups using a homotopy relation as Kasparov did and \cite{DGM} defines $KK$-bordism groups. In $KK$-theory there is an informal mantra that any reasonable equivalence relation on bounded $KK$-cycles will lead to the $KK$-group. {\bf The main analytic result of the present work is a proof that a similar situation occurs in the context of unbounded cycles.} Nevertheless we have concentrated our study on a relation defined using Hilsum's notion of bordism. Why have we chosen to do so? The answer again lies in our desire to make constructions in index theory, if not canonical, at least as explicit as possible. {\bf Hilsum \cite{hilsumcmodbun,hilsumfol,hilsumbordism} has an extensive body of work rich in ideas that aligns with the desire for canonical, and geometric constructions in the realm of noncommutative geometry.} The motivation  to search for a geometric as well as explicit equivalence relation is based on our goal to explore the geometry and index theory arising from the equivalence relation, much in analogy with Atiyah-Patodi-Singer index theory \cite{APS1,APS2,APS3}, in order to place the above listed works on applications of noncommutative geometry in a coherent foundation. Some terminology is required to proceed further with the justification of our choice of Hilsum bordism as the main relation of study. 

\subsection{Analytic versus geometric models for $K$-homology}

Let us discuss the different models and how the associated homology theories conceptualize the Atiyah-Singer index theorem. For the moment we will work with $K$-homology as it illustrates the general case of $KK$-theory. Let $X$ be a finite CW-complex. The Atiyah-Singer index theorem can be encoded into $K$-homology by considering two models. On the topological side one has the Baum-Douglas model for $K$-homology \cite{BD, BHS}, which uses cycles of the form $(M,E,f)$ where
\begin{enumerate}
\item $M$ is a closed, smooth, Riemannian, ${\rm spin^c}$-manifold;
\item $E$ is a smooth Hermitian vector bundle over $M$;
\item $f: M \rightarrow X$ is a continuous map.
\end{enumerate}
On the analytic side one has Kasparov's model, as described in Higson-Roe's classical book \cite{HRbook}, which uses cycles of the form $(H, \rho, F)$ where
\begin{enumerate}
\item $H$ is a complex separable Hilbert space;
\item $\rho: C(X) \rightarrow \mathbb{B}(H)$ is a representation;
\item $F\in \mathcal{B}(H)$ satisfying the following for any $a\in C(X)$, 
\[ \rho(a)(F-F^*)\in \mathcal{K}(H), \rho(a)(F^2-I)\in \mathcal{K}(H), [\rho(a), F] \in \mathcal{K}(H) \]
\end{enumerate}
Kasparov's analytic cycles form what Atiyah \cite{ASell} called abstract elliptic operators.

There is a map from geometric $K$-homology to analytic $K$-homology that is defined via
\begin{equation}
\label{bladlaad}
 (M, E, f) \mapsto  [(L^2(M; S\otimes E), \rho , \beta(D_E))] 
 \end{equation}
where $S$ is the spinor bundle associated to the ${\rm spin^c}$-structure, $L^2(M; S\otimes E)$ is the space of $L^2$-sections of $S\otimes E$, $\rho: C(X) \rightarrow \mathbb{B}(L^2(M; S\otimes E))$ is defined via multiplication of functions, $\beta(x):=x(1+x^2)^{-1/2}$ is the bounded transform, and $D$ is a Dirac operator associated to $S\otimes E$. We note that this map involves a choice of connection to obtain the Dirac operator, a choice of function in the bounded transform, et cetera and hence is not quite defined at the level of cycles but rather maps a geometric cycle to a class in analytic $K$-homology. The map \eqref{bladlaad} induces an isomorphism from geometric $K$-homology to analytic $K$-homology \cite{BHS, Jak, Rav} and in particular, when $X$ is a point we have the Atiyah-Singer's index theorem as stated in \cite{BD,BDbor} and discussed in more detail in \cite{baumvanerpI,baumvanerpII}.

Intermediary between the geometric Baum-Douglas model and the analytic Kasparov model sits Baaj-Julg's unbounded picture. Its role in relation to the geometric and analytic models was not clarified until the works \cite{DGM,Kaadunbdd,meslandvandung}, and we aim to provide further context in this work. The cycles for Baaj-Julg's unbounded picture can be motivated by the appearance of Dirac operators and the bounded transform in the transformation \eqref{bladlaad}. Baaj-Julg's unbounded cycles have the form $(H, \rho, D)$ where $D$ is an unbounded self adjoint operator that satisfies similar conditions to $F$, for details see Definition \ref{UnbKKcycDef} below. These unbounded cycles are a variation on Connes' notion of spectral triples, the key object of study in Connes' program for spectral noncommutative geometry \cite{ConnesIHES,Connesbook,varillyetal}. The reader can find a number of applications and examples in the literature \cite{Connestrace,conneslott,connesgravity,connesreconstr,GMCK,GMRtwist,GRU,leschkaad1,Hoch3,LRV,lottlimit,neshtuset,renniereconst,rieffelquant}.

The unbounded cycles sit between the geometric and analytic realization of $K$-homology as we can factor the map \eqref{bladlaad} as
\begin{equation}
\label{bladlaadtwo}
 (M, E, f) \mapsto (L^2(M; S\otimes E), \rho , D_E)\mapsto [(L^2(M; S\otimes E), \rho , \beta(D_E))] 
 \end{equation}
We note that we can by augmenting the $(M, E, f)$ cycles with further data realize the first map in \eqref{bladlaadtwo} as a map at the level of cycles. Indeed, following \cite{HReta, Kescont}, one can consider cycles of the form $(M, E, f, D_E)$ where $M$ and $E$ are as before, $f$ is assumed to be Lipschitz and $D_E$ is a specific Dirac operator associated to the Clifford bundle $S\otimes E$.

Based on this discussion at the cycles level, a desirable feature of an equivalence relation on unbounded cycles is that an equivalence between geometric cycles explicitly produces an equivalence of the associated unbounded cycles. Hilsum's notion of bordism has this property, but first we will take a moment to unpack the previous sentence. The relation on geometric cycles is defined by taking the relation generated by direct sum/disjoint union, bordism and vector bundle modification. For example, suppose that $(M, E, f, D)$ is a cycle that is a boundary and $(W, F, g, D_W)$ is an explicit choice of bordism with boundary $(M, E, f, D)$. The associated analytic $K$-homology element is the zero element but we can ask for more; namely, is there an explicit Hilsum bordism associated to $(W, F, g, D_W)$? Indeed there is such a Hilsum bordism; this should not be too surprising since Hilsum defined his notion of bordism to naturally include examples from differential geometry. Somewhat more surprisingly, the two other parts of relation (direct sum/disjoint and vector bundle modification) from geometric $K$-homology also lead to explicit $KK$-bordisms. In summary, using Hilsum bordism as our relation strengthens the relationship between the geometric cycles of the form $(M, E, f, D)$ and the analytic cycles of the form $(H, \rho, D)$.

Strengthening the relationship between geometric and analytic models is not only satisfying from a philosophical point of view but also has important implications in geometry and topology. The relative constructions considered by the first and second listed authors \cite{DeeZkz, DeeZkz2, DeeRZ, DeeMappingCone, DGsurI, DGrelI} are a prototypical example. Each of these relative constructions are based on the bordism relation in geometric $K$-homology (and bivariant generalizations). Once the geometric cycles and associated group have been defined one needs a map to the relevant analytic group. At the $KK$-theory level the relative group can be obtained from a mapping cone construction and the associated long exact sequence. The Hilsum bordism relation allows us to apply the relative process to unbounded analytic cycles in a way completely analogous to the one used on geometric cycles. In doing so, the methods of this monograph will serve to unify the various isomorphisms between the geometric and analytic relative groups considered in \cite{DeeZkz2, DeeMappingCone}. We complete this work in the sequel to this monograph \cite{monographtwo}. The motivation for relative constructions is to better understand secondary invariants; a methodology that has been studied in the context of positive scalar curvature metrics \cite{HReta, PSrhoInd, PSZmap, SchPSCsur}, homotopy structure on manifolds and Lipschitz manifolds \cite{HRSur1, HRSur2, HRSur3, HReta, PS, PSsignInd, xieyu, zenobi} and for foliations \cite{andro,antoninifol,connesskandalis,gorolott,hilsumfol,kordy,tu99,tu99a}. All these examples have one thing in common, they concern mildly noncommutative geometries constructed from a group action on a classical geometric object.

\subsection{Main results and overview of the monograph}

In this monograph we overview the existing sphere of examples that in the literature has motivated the usage of $KK$-theory and NCG, as well as study how Hilsum's $KK$-bordisms fit into this picture. The viewpoint we take is that of unbounded $KK$-theory and the reader can find a summary of the basic elements underlying this theory in Section \ref{backgroundsec}. There is to the authors' knowledge no current, complete overview of modern examples in NCG of geometric origins. Part \ref{part:prelim} aims at providing a partial remedy to this omission. The field is too big to aim for completeness, but Section \ref{lknaldknadljnbla} contains an extensive overview of the topical examples central to this work and to unbounded $KK$-theory at large. We claim no novelty but we provide the examples in order to showcase how various NCG-structures fit into unbounded $KK$-theory. Let us summarize the examples of Section \ref{lknaldknadljnbla}. We refer the reader to Section \ref{backgroundsec} for notations and terminology.
There are canonical equivariant unbounded cycles and chains associated with several geometric situations, such as:
\begin{enumerate}
\item For Riemannian manifolds equipped with a Dirac type operator twisted by a $C^*$-bundle, possibly only Lipschitz manifolds, there is an associated unbounded chain. If the manifold is complete we arrive at a closed cycle for Lipschitz functions vanishing at infinity, and there is a rich theory of Dirac-Schrödinger operators.
\item For any $K$-oriented map $f:X\to Y$ of smooth manifolds, the wrong way map comes from an unbounded cycle.
\item The NCG-philosophy extends to truly noncommutative cases producing cycles and chains also for foliations, Cuntz-Pimsner algebras, groupoid $C^*$-algebras with cocycles, continuous trace algebras as well as Fell algebras.
\item Moreover, in the above cases when there is a group action present we arrive at equivariant cycles and chains that allow for descent constructions and  Baum-Connes assembly, producing mildly noncommutative geometries already from commutative geometries.
\end{enumerate}

We proceed in Section \ref{sec:analchain} by studying analytic techniques for manipulating continuity and ``smooth'' subalgebras. The results of Subsections \ref{subsec:autocont} and \ref{subsec:funccalc} allow us to deduce automatic continuity of chains and expand smooth subalgebras via functional calculus, these results are on the whole well known. In Subsection \ref{orderreducsubsec} we introduce methods of complex interpolation \cite{bergloef} into unbounded $KK$-theory; this tool is well studied in functional analysis but in noncommutative geometry it has at best been hinted at \cite{GMRtwist}.

We study Hilsum's notion of $KK$-bordisms in unbounded $KK$-theory in more detail in Part \ref{partII}, where we connect $KK$-bordisms to the examples of Part \ref{part:prelim}.  A significant feature is that $KK$-bordisms form an equivalence relation on the unbounded $KK$-cycles for a pair $(\mathcal{A},B)$ of a $*$-algebra $\mathcal{A}$ and a $C^*$-algebra, and the set of equivalence classes $\Omega_*(\mathcal{A},B)$ forms an abelian group under direct sum of cycles. This fact goes back to \cite{DGM}, and by a result of Hilsum the bounded transform induces a well defined mapping
\begin{equation}
\label{boundednada}
\beta:\Omega_*(\mathcal{A},B)\to KK_*(A,B),
\end{equation}
where $A$ is the $C^*$-closure of $\mathcal{A}$. {\bf When $\mathcal{A}$ satisfies that the bounded transform \eqref{boundednada} is an isomorphism for all $B$, we say that $\mathcal{A}$ is $\Omega_*$-admissible.} As we discuss soon in Theorem \ref{mainthmadm}, large classes of $*$-algebras are $\Omega_*$-admissible.

We shall see how new, purely noncommutative, phenomena arise from $KK$-bordisms. \emph{An important feature is that in all key examples we can construct canonical $KK$-bordisms from pure analytic geometric reasons; a situation that differs from homotopies.} The reader can find the precise notions described in Section \ref{subsecglying}, Section \ref{exampleofboridmsmssec} focuses on general structure, while the main examples in Section \ref{sec:ncgexbord} relate back to the examples of Section \ref{lknaldknadljnbla}. The results in Sections \ref{exampleofboridmsmssec} and \ref{sec:ncgexbord} can be summarized as follows. 

\begin{theorem*}
There are canonical equivariant $KK$-bordisms associated with
\begin{enumerate}
\item Complete Riemannian manifolds with boundary equipped with a Dirac type operator twisted by a $C^*$-bundle, possibly only Lipschitz, see Subsections \ref{diraconcstarbundleexbordism}, \ref{diracschexbord}.
\item A positive scalar curvature metric on a compact manifold (see Subsection \ref{pscdiracexII}) or two homotopic topological manifolds in dimension $\neq 4$ (see Subsection \ref{bordimadinklm}).
\item The vector bundle modification relation in geometric $K$-homology implementing the Thom isomorphism \cite{BD,BDbor}, or more generally there is a null-bordism for weakly degenerate unbounded cycles, see Subsection \ref{subsecweakdegff}.
\item A homotopy of unbounded operators $(D_t)_{t\in [0,1]}$ induces a bordism with boundary $(\mathpzc{E},D_0)+(-(\mathpzc{E},D_1))$, see Subsection \ref{subsec:homotopyop}.
\item A reasonable perturbation $V$ of an operator $D$ in an unbounded cycle induces a bordism with boundary $(\mathpzc{E},D)+(-(\mathpzc{E},D+V))$, see Subsection \ref{subsec:perturbyop}.
\item A homotopy $(\pi_t)_{t\in [0,1]}$ of the left action in an unbounded cycle induces a bordism with boundary $({}_{\pi_0}\mathpzc{E},D)+(-({}_{\pi_1}\mathpzc{E},D))$, see Subsection \ref{subsec:homotoperep}.
\item Performing order reduction induces a bordism with boundary $(\mathpzc{E},D)+(-(\mathpzc{E},D|D|^{-\alpha}))$ for $\alpha\in [0,1)$ and also $KK$-bordisms can be order reduced in a similar fashion. See more in Subsection \ref{orderredusubsec}.
\end{enumerate}
There are also well-defined operations for descent of equivariant cycles, that for any second countable locally compact group $G$ produces a map 
$$\rtimes G:\Omega^G_*(\mathcal{A},B)\to \Omega_*(\mathcal{A}\rtimes_c G,B\rtimes G).$$
Here $\mathcal{A}\rtimes_{ c} G:=C_c(G,\mathcal{A})$ with the twisted convolution product. And finally, for a proper, cocompact action of $G$ on a metric space $X$ we have an assembly map 
$$\mu_X:\Omega^G_*(\mathrm{Lip}_0(X),B)\to \Omega_*(\C,B\rtimes G).$$
\end{theorem*}

The conceptual benefit of $KK$-bordisms is not just the canonical $KK$-bordisms arising in key examples, but in addition that it allows for the importing of geometric notions and ideas into the technical area of $KK$-theory. A concrete example is how wrong way maps give rise to unbounded cycles that up to explicit $KK$-bordisms are independent of geometric choices, see Section \ref{kkwrongbor}. A more lofty, yet conceptually interesting, feature is that rigorous proofs by picture are a possibility with $KK$-bordisms. We develop this line of thought in Part \ref{part:geocons}. The primary techniques of straightening the angle, the doubling construction, gluing on infinite cylinders (à la Melrose's $b$-calculus \cite{melroseaps}) are developed in Section \ref{sec:prelbord}. We also extend the ideas of complex interpolation from Subsection \ref{orderreducsubsec} to $KK$-bordisms in Subsection \ref{subsec:orderred}. These techniques are applied in Section \ref{sec:propoered} to prove the following rigidity results for the $KK$-bordism groups.

\begin{theorem*}
\label{mainthm3}
For a locally compact group $G$, a $G-C^*$-algebra $B$ the functor $\Omega_*^G(-,B)$, on the category of $G-*$-algebras $\mathcal{A}$ dense in a $C^*$-algebra $A$, is an additive functor invariant up to isomorphism under the following operations
\begin{enumerate}
\item Complex interpolation: if $\mathcal{A}_1$ is a Banach$^*$-algebra and $\mathcal{A}_1\subseteq \mathcal{A}\subseteq A$ the restriction mapping induces an isomorphism
$$\Omega_*^G(\mathcal{A},B)\cong \Omega_*^G(\mathcal{A}_1,B),$$
as soon as there is an $0<\alpha \leq 1$ such that $\mathcal{A}\subseteq [A,\mathcal{A}_1]_\alpha$.
\item Functional calculus: if $\mathcal{A}_1\subseteq A$ is a $^*$-algebra and $ \mathcal{A}\subseteq A$ is the smallest $*$-algebra containing $\mathcal{A}_1$ closed under smooth functional calculus, the restriction mapping induces an isomorphism
$$\Omega_*^G(\mathcal{A},B)\cong \Omega_*^G(\mathcal{A}_1,B).$$
\item Morita equivalence by FGPs: if $\mathcal{E}$ is a finitely generated projective $\mathcal{A}$-module the Kasparov product by $\mathcal{E}$ induces an isomorphism 
$$\Omega_*^G(\mathcal{A},B)\cong \Omega_*^G(\End_{\mathcal{A}}(\mathcal{E}),B).$$
\end{enumerate}
For a fixed $G-*$-algebra $\mathcal{A}$, the functor  $\Omega_*^G(\mathcal{A},-)$ is an additive functor invariant under Morita equivalence. 
\end{theorem*}

For the trivial group $G$ and under assumptions similar to $\Omega_*$-admissibility, we show in Section \ref{secexact} how $KK$-bordism groups additionally to the rigidity properties of Theorem \ref{mainthm3} enjoy exactness and homotopy invariance.

Index theory has found numerous applications in geometry and topology \cite{andro,antoninifol,connesskandalis,DeeZkz, DeeZkz2, DeeRZ, DeeMappingCone, DGrelI,DGsurI, HRSur1, HRSur2, HRSur3, HReta, PS, PSrhoInd, PSsignInd,PSZmap, SchPSCsur, tu99,tu99a,xieyu,zenobi}, and for manifolds with boundary the relevant index theory comes from the work of Atiyah-Patodi-Singer \cite{APS1,APS2,APS3} -- APS index theory. In particular higher APS index theory is of importance in topological applications \cite{guohochs,hilsumlip,hochswang,LP}, where the index theory arises from mildly noncommutative geometries with boundary. Following the general spirit of this work, we show in Section \ref{highihereinda} that $KK$-bordisms behave much like a noncommutative manifold with boundary admitting a fully developed higher APS index theory. Such methods will play an important role in the follow-up work \cite{monographtwo} where the above mentioned canonical geometric $KK$-bordisms will be utilized for secondary invariants. We summarize the results of Section \ref{highihereinda} in the following theorem.

\begin{theorem*}
Let $B$ be a unital $C^*$-algebra and $\mathfrak{X}=(\mathfrak{X}^\circ,\Theta,\mathfrak{X}^\partial)$ a $(\C,B)$-bordism. Then $\mathfrak{X}$, or at least its very full modification (for details see Subsection \ref{subsec:specsec}), admits a trivializing operator $A$. For $\mathfrak{X}^\circ=(\mathpzc{N},T)$ and $\mathfrak{X}^\partial=(\mathpzc{E},D)$, we define the APS-realization of $\mathfrak{X}$ with trivializing operator $A$ to be 
\begin{align*}
T_{\rm APS}(A):=&T^*, \quad \mbox{with domain}\\
 \Dom T_{\rm APS}(A)&:=\{f\in \Dom(T^*): \; f|_{\partial\mathfrak{X}}\in \ker \chi_{[0,\infty)}(D+A)\}.
 \end{align*}
For more details on the domain, see Subsection \ref{subsec:apsforborddad}. The APS-realization of $\mathfrak{X}$ with trivializing operator $A$ is Fredholm and has a well defined index 
$$\ind_{\rm APS}(\mathfrak{X},A)\in K_*(B).$$
The APS-index of $KK$-bordisms satisfies the ordinary gluing formulas and spectral flow formulas (for details see Subsection \ref{subsec:glue}). 
\end{theorem*}

The results discussed above are somewhat independent of the connection to Kasparov's $KK$-theory implemented by the bounded transform. To fully understand the utility of $KK$-bordisms, we study the mapping properties of the bounded transform in Part \ref{partoniso}. The main technical result of this part is found in Section \ref{bddtransformisosec}, where we show that for $\mathcal{A}$ countably generated the bounded transform is an isomorphism. Our proof is inspired by work of Kaad \cite{Kaadunbdd}. Building on this result, and the rigidity properties summarized in Theorem \ref{mainthm3} above we conclude the following classes of examples where  the bounded transform is an isomorphism.

\begin{theorem*}
\label{mainthmadm}
The bounded transform
$$\beta:\Omega_*(\mathcal{A},B)\to KK_*(A,B),$$
is an isomorphism for all $C^*$-algebras $B$, that is $\mathcal{A}$ is $\Omega_*$-admissible, if it is of the following form:
\begin{enumerate}
\item $\mathcal{A}$ is any countably generated $*$-algebra dense in a $C^*$-algebra.
\item $\mathcal{A}=\mathrm{Lip}_0(X)$ and $A=C_0(X)$ for a Lipschitz manifold $X$.
\item $\mathcal{A}=C^\infty_c(\mathcal{G})$  and $A=C^*(\mathcal{G})$ for an étale groupoid $\mathcal{G}$ over an open subset $X$ of a compact, separable, metrizable space (and smooth functions are defined relative to an embedding into $\pmb{\Box}\times \pmb{C}$ -- the Hilbert cube product a Cantor set). In particular, it holds if $\mathcal{G}$ is of one of the following forms:
\begin{itemize}
\item $\mathcal{G}:=\Gamma$ is a discrete countable group, and so $C^\infty_c(\mathcal{G})=\C[\Gamma]$ is the group algebra and $C^*(\mathcal{G})=C^*(\Gamma)$. In particular, we have the natural isomorphism induced by the bounded transform
$$\Omega_*^\Gamma(\C,B)\cong KK_*^\Gamma(\C,B).$$
\item $\mathcal{G}:=X\rtimes \Gamma$, for a discrete countable group $\Gamma$ acting on $X$, and so $C^\infty_c(\mathcal{G})=C^\infty_c(X)\rtimes_{\rm alg} \Gamma$ and $C^*(\mathcal{G})=C_0(X)\rtimes \Gamma$.
\item $\mathcal{G}$ is the Deaconu groupoid associated with a surjective local homeomorphism $g:X\to X$, and so $C^\infty_c(\mathcal{G})=C^\infty_c(X)\rtimes_{\rm alg} \N$ and $C^*(\mathcal{G})=C_0(X)\rtimes_g \N=O_{E_g}$ is a Cuntz-Pimsner algebra.
\item $\mathcal{G}$ is the groupoid defining an AF-algebra $A$, i.e. $\mathcal{G}$ is a groupoid over the totally disconnected spectrum of a MASA in $A$, and $A=C^*(\mathcal{G})$.
\end{itemize}
\end{enumerate}
\end{theorem*}

We do note that there are examples where the bounded transform $\beta:\Omega_*(\mathcal{A},B)\to KK_*(A,B)$ is neither injective nor surjective, e.g. when $\mathcal{A}=L^1(G)$ for an amenable group $G$. We explore such examples in Section \ref{bddtransformnotisosec} and utilize these results in Section \ref{sec:failexc} to deduce that the $KK$-bordism groups in general fails to satisfy excision.

\subsection{Acknowledgements}

The authors wish to thank Michel Hilsum for an encouraging discussion leading up to this work. We are also grateful to our colleagues Jens Kaad, Matthias Lesch, and Adam Rennie for several rewarding discussions and collaborations running in parallel to writing this monograph. We thank Magnus Fries for proof reading an earlier draft and Heather Dudock for help with the figures. RJD was partially supported by NSF Grants DMS 2000057 and DMS 2247424, Simons Foundation Gift MP-TSM-00002896, and Simons Foundation Collaboration Grant for Mathematicians number 638449. MG was supported by the Swedish Research Council Grant VR 2018-0350 and VR 2025-03923.

\subsection{AI usage statement}

The authors acknowledge using ChatGPT for proofreading and editorial suggestions on the manuscript.

\part{Preliminaries} 
\label{part:prelim}

\section{Elements of unbounded $KK$-theory}
\label{backgroundsec}

The main object of study in this work is the class of cycles appearing in the unbounded model of Kasparov's $KK$-theory. The unbounded model was first introduced by Baaj-Julg \cite{baajjulg}, and has since played a prominent role in noncommutative geometry and index theory \cite{ConnesIHES,Connesbook,Connestrace,conneslott,connesgravity,GMCK,GMRtwist,GRU,varillyetal,leschkaad1,kaaddiffabs,Kaadunbdd,Hoch3,lottlimit,neshtuset,rieffelquant}. We start by introducing notation and terminology, mainly following \cite{DGM}. These notions will be exemplified in the next section.

We let $A$ denote a $C^*$-algebra and $\mathcal{A}\subseteq A$ a dense $*$-subalgebra. The structures we study depend on the choice of $\mathcal{A}\subseteq A$, and it can be thought of as a choice of smooth structure (see \cite{Connesbook}). The prototypical example is obtained from, $M$, a closed smooth manifold by taking $A$ to be the continuous complex valued functions on $M$ denoted by $C(M)$, and $\mathcal{A}$ to be the smooth complex-valued functions on $M$ denoted by $C^{\infty}(M)$. However, as we will see, one is also interested in the case where $A$ is as in the previous sentence, but $\mathcal{A}$ is the complex valued Lipschitz functions on $M$ denoted by ${\rm Lip}(M)$.

For a $C^*$-algebra $B$, we use letters such as $\mathpzc{E}$ and $\mathpzc{F}$ to denote $B$-Hilbert $C^*$-modules and should we want to emphasize the dependence on $B$ we write $\mathpzc{E}_B$ and $\mathpzc{F}_B$ et cetera. The $C^*$-algebra of $B$-adjointable operators on $\mathpzc{E}$ is denoted by $\End_B^*(\mathpzc{E})$ and the ideal of $B$-compact operators by $\mathbb{K}_B(\mathpzc{E})$. For more details, see \cite{lancesbook}. When speaking of a left action of $\mathcal{A}$ or $A$ on $\mathpzc{E}$, we mean a $*$-homomorphism $A\to \End_B^*(\mathpzc{E})$.

A densely defined operator  $T$ from $\mathpzc{E}$ to $\mathpzc{F}$ will be written as $T:\mathpzc{E}\dashrightarrow \mathpzc{F}$. The graph of $T$ is written as $\mathrm{Graph}(T):=\{(\xi,T\xi)\in \mathpzc{E}\oplus\mathpzc{F}: \xi\in \Dom(T)\}$. We say that 
\begin{itemize}
\item $T$ is semi-regular if $T^*$ is densely defined;
\item $T$ is  regular if $1+T^*T$ has dense range;
\item $T$ is symmetric if $\mathpzc{E}=\mathpzc{F}$ and $T\subseteq T^*$;
\item $T$ is self-adjoint if $\mathpzc{E}=\mathpzc{F}$ and $T=T^*$.
\end{itemize}
For more on operators on Hilbert $C^*$-modules, see \cite{leschkaad2,lancesbook}. Note that for $B=\C$, $\mathpzc{E}$ is a Hilbert space and all semi-regular operators are regular. The domain $\Dom(T)$ carries a $B$-valued inner product
$$\langle\xi_1,\xi_2\rangle_T:=\langle\xi_1,\xi_2\rangle+\langle T\xi_1,T\xi_2\rangle,$$
called the graph inner product. It is the inner product that makes the graph inclusion 
$$\Dom(T)\hookrightarrow \mathrm{Graph}(T)\subseteq \mathpzc{E}\oplus\mathpzc{F},\quad  \xi\mapsto \xi\oplus T\xi,$$ 
an isometry where the inner product on the codomain is obtained from the one on $\mathpzc{E}\oplus\mathpzc{F}$. We say that $T$ is closed if $\Dom(T)$ is closed with respect to the norm defined from this inner product; in this case $\Dom(T)$ is a Hilbert $C^*$-module in its own right (it is unitarily isomorphic to $\mathrm{Graph}(T)$). 

An operator $T:\mathpzc{E}\dashrightarrow \mathpzc{F}$ is said to be regular if $T$ is closed, $T^*$ is densely defined and $1+T^*T$ has dense range.  If $T$ is a closed operator, $T$ is regular if and only if the inclusion $\iota: \Dom(T)\hookrightarrow \mathpzc{E}$ is adjointable. If $T$ is regular then $1+T^*T$ is bijective and the adjoint of $\iota$ is given by
\[\iota^{*}:\mathpzc{E}\to \Dom T,\quad \xi\mapsto (1+T^*T)^{-1}\xi.\]
Regularity of $T$ turns out to hold if and only if the sub-module $\mathrm{Graph}(T)\subseteq \mathpzc{E}\oplus\mathpzc{F}$ is complemented. Indeed, if $T$ is regular, then $\mathrm{Graph}(T)$ is the image of the graph projection 
$$p_T=\begin{pmatrix} 
(1+T^*T)^{-1}& T^*(1+TT^*)^{-1}\\
T(1+T^*T)^{-1}& TT^*(1+TT^*)^{-1}
\end{pmatrix}.$$
Moreover, letting $\pi_1:\mathpzc{E}\oplus\mathpzc{F}\to \mathpzc{E}$ denote the projection onto the first factor, there is a bijective correspondence 
\begin{align}
\label{graphdbibidkd}
\left\{p\in \mathrm{End}_B^*(\mathpzc{E}\oplus\mathpzc{F}): p^2=p^*=p, \; \pi_1\mathrm{im}(p)\subseteq \mathpzc{E} \mbox{ dense}, \; \pi_1|_{\mathrm{im}(p)} \mbox{ injective}\right\}\\
\nonumber
\leftrightarrow \left\{T:\mathpzc{E}\dashrightarrow \mathpzc{F}: \; T \mbox{ is a regular operator}\right\},
\end{align}
set up by mapping a projection $p$ to the operator $\xi\mapsto \pi_2(\pi_1|_{\mathrm{im}(p)}^{-1}\xi)$ with domain $\pi_1\mathrm{im}(p)$ and the inverse mapping being $T\mapsto p_T$.

If $T$ is a closed, regular operator, then $T^*T$ is self-adjoint. We define 
$$|T|:=\sqrt{T^*T},$$ 
from functional calculus of self-adjoint regular operators (see for example \cite{lancesbook}).

For a densely defined operator $T:\mathpzc{E}\dashrightarrow \mathpzc{E}$, we define the Lipschitz algebra 
$$\mathrm{Lip}(T):=\{a\in \End^*_B(\mathpzc{E}): a\Dom(T)\subseteq \Dom(T)\;\mbox{and}\; [T,a], \,[T,a^*] \;\mbox{are bounded}\}.$$
The terminology is motivated by the fact that the Lipschitz algebra of a first order elliptic differential operator on a compact manifold coincides with the algebra of Lipschitz continuous functions. The algebra $\mathrm{Lip}(T)$ is an operator algebra when embedding it is a norm-closed subalgebra 
$$\mathrm{Lip}(T)\hookrightarrow \End_B^*(\mathpzc{E}\oplus \mathpzc{E}), \quad a\mapsto \begin{pmatrix} a& 0\\ [T,a]& a\end{pmatrix}.$$
See more in \cite{BKM}. If $T$ is regular, we also write
$$\mathrm{Lip}_0(T):=\{a\in \mathrm{Lip}(T): a(1+T^*T)^{-1}\in \K_B(\mathpzc{E})\}.$$
If $T$ is symmetric and regular, we write
\begin{align*}
\mathrm{Lip}(T,T^*)&:=\{a\in\mathrm{Lip}(T): a\Dom(T^*)\subseteq \Dom(T)\}\quad\mbox{and}\\
\mathrm{Lip}_0(T,T^*)&:=\mathrm{Lip}(T,T^*)\cap \mathrm{Lip}_0(T).
\end{align*}

In the next definition, we give the necessary definitions of unbounded $KK$-chains, cycles, and the relations and operations thereon. The definition of an {\it even} cycle (or chain) is obtained by including the bracketed data; the definition of {\it odd} cycle (or chain) is obtained by removing them.

\begin{define} 
\label{UnbKKcycDef}
An unbounded $KK$-chain over $(\mathcal{A}, B)$ is a pair, $(\mathpzc{E}, D)$ where 
\begin{enumerate}
\item $\mathpzc{E}$ is a (graded) Hilbert $(A, B)$-bimodule via a $*$-homomorphism $\pi:A\to \End_B^*(\mathpzc{E})$;
\item $D$ is an (odd) regular operator acting as an unbounded operator on $\mathpzc{E}$ with 
$$\pi(\mathcal{A})\subseteq \mathrm{Lip}(D).$$
\end{enumerate}
We write $\mathfrak{X}=(\mathpzc{E}, D)$. Moreover, we say the following.
\begin{itemize}
\item $\mathfrak{X}$ is symmetric if $D$ is symmetric.
\item $\mathfrak{X}$ is half-closed if $D$ is symmetric with 
$$\pi(\mathcal{A})\subseteq \mathrm{Lip}_0(D,D^*).$$
\item $\mathfrak{X}$ is closed if it is half-closed and $D$ is self-adjoint. We typically refer to closed $KK$-chains as cycles, or if we want to be overly clear a closed cycle.
\item $\mathfrak{X}$ is called a unital cycle if it is a cycle, and $A$ is unital and acts unitally on $\mathpzc{E}$. 
\item $\mathfrak{X}$ is Lipschitz if $\pi(\mathcal{A})\subseteq \mathrm{Lip}(|D|)$, where $|D|:=\sqrt{D^*D}$.
\end{itemize}
An isomorphism $\alpha:\mathfrak{X}_1\xrightarrow{\sim} \mathfrak{X}_2$ of two $(\mathcal{A},B)$-chains $\mathfrak{X}_i=(\mathpzc{E}_i, D_i)$ is a unitary isomorphism $\alpha:\mathpzc{E}_1\to \mathpzc{E}_2$ of $B$-modules, intertwining the $\mathcal{A}$-action and satisfying that 
$$D_2\alpha=\alpha D_1.$$
\end{define}

We emphasize that the identity $D_2\alpha=\alpha D_1$ is as unbounded operators, and the condition $\alpha\mathrm{Dom}(D_1)=\mathrm{Dom}(D_2)$ is implicit.

\begin{remark}
The way the notions of chains and cycles will be used, is that a cycle forms a cycle for the $KK$-bordism group thought of as a bivariant homology theory.  Cycles form a structure analogous to a compact manifold in unbounded $KK$, compare to Example \ref{diraconcstarbundleex} below. The relations imposed on this will be via $KK$-bordisms -- a special type of chain introduced by Hilsum \cite{hilsumcmodbun,hilsumfol,hilsumbordism} that can be thought of as an ``unbounded $KK$-cycle with boundary which is of product type at the boundary''. Similar ideas have previously appeared in \cite{DGM} and using more operator theoretic homotopy relations in \cite{Kaadunbdd,meslandvandung}.

A prototypical example of a cycle is a spectral triple $(\mathcal{A},\mathpzc{H},D)$, a notion introduced by Connes \cite{Connesbook}. A spectral triple is an $(\mathcal{A},\C)$-cycle. In the paradigm of noncommutative geometry, spectral triples are to be thought of as a spectral representation of a Riemannian manifold with self-adjointness of $D$ encoding completeness of the Riemannian metric \cite{connesreconstr}. As we shall see below in Example \ref{diraconcstarbundleex}, chains relax these geometric conditions and can be thought of as a generalization of non-compact manifolds, or more general singular manifolds for that matter. 

We nevertheless stress that in general the viewpoint taken in this work is not on cycles and chains as giving us a noncommutative geometry, i.e. as a generalization of manifolds, but rather that they form cycles and chains in a noncommutative generalization of bordism theory. The reader should keep Baum-Douglas' geometric model for $K$-homology in mind, \cite{BD,BDbor}, where a spin$^c$-manifold $M$ is thought of as a cycle for the $K$-homology of $M$ rather than describing its own geometry. This conceptual difference will not play any role at the level of technicalities but might be worthwhile to keep in mind.
\end{remark}

\begin{define} 
\label{oppKKCyc}
The \emph{opposite} of an $(\mathcal{A},B)$-chain $\mathfrak{X}=(\mathpzc{E}, D)$ is the chain
$$-\mathfrak{X}=
\begin{cases}
(-\mathpzc{E}, D), \quad\mbox{for even $(\mathpzc{E}, D)$},\\
(\mathpzc{E}, -D), \quad\mbox{for odd $(\mathpzc{E}, D)$.}
\end{cases}$$
Note that $-\mathpzc{E}$ is the module $\mathpzc{E}$ with the opposite grading. If $\mathfrak{X}$ is symmetric/half-closed/closed then so is $-\mathfrak{X}$.

The \emph{sum} of two $(\mathcal{A},B)$-chains $\mathfrak{X}_1=(\mathpzc{E}_1, D_1)$ and $\mathfrak{X}_2=(\mathpzc{E}_2, D_2)$ is given by their direct sum:
\begin{equation}
\label{dirsum}
\mathfrak{X}_1+\mathfrak{X}_2:=(\mathpzc{E}_1\oplus \mathpzc{E}_2, D_1\oplus D_2).
\end{equation}
Note that if $\mathfrak{X}_1$ and $\mathfrak{X}_2$ are symmetric/half-closed/closed then so is $\mathfrak{X}_1+\mathfrak{X}_2$.
\end{define}

\begin{define}
\label{pushpull}
If $\phi:B_1\to B_2$ is a $*$-homomorphism and $\mathfrak{X}=(\mathpzc{E},D)$ is an $(\mathcal{A},B_1)$-chain, we define the $(\mathcal{A},B_2)$-chain $\phi_*(\mathfrak{X})$ by 
$$\phi_*(\mathfrak{X}):=(\mathpzc{E}\otimes_\phi B_2,D\otimes_\phi 1_{B_2}).$$
If $\mathfrak{X}$ is symmetric/half-closed/closed then so is $\phi_*(\mathfrak{X})$.

If $\phi:\mathcal{A}_1\to \mathcal{A}_2$ is a $*$-homomorphism continuous in the topologies of the ambient $C^*$-algebras and $\mathfrak{X}=(\mathpzc{E},D)$ is an $(\mathcal{A}_2,B)$-chain, we define the $(\mathcal{A}_1,B)$-chain $\phi^*(\mathfrak{X})$ by 
$$\phi^*(\mathfrak{X}):=(_\phi\mathpzc{E},D),$$
where ${}_\phi \mathpzc{E}$ denotes $\mathpzc{E}$ equipped with the left action pulled back along $\phi$. If $\mathfrak{X}$ is symmetric/half-closed/closed then so is $\phi^*(\mathfrak{X})$.
\end{define}

If $\mathpzc{E}$ and $\mathpzc{F}$ are Hilbert $C^*$-modules, we let $\mathpzc{E}\boxtimes\mathpzc{F}$ denote their exterior tensor product. If $\mathpzc{E}$ and $\mathpzc{F}$ are graded, we let $\mathpzc{E}\hat{\boxtimes}\mathpzc{F}$ denote their exterior graded tensor product. If exactly one of $\mathpzc{E}$ and $\mathpzc{F}$ is graded, we set $\mathpzc{E}\hat{\boxtimes}\mathpzc{F}:=\mathpzc{E}\boxtimes\mathpzc{F}$. If none of $\mathpzc{E}$ and $\mathpzc{F}$ is graded, we let $\mathpzc{E}\hat{\boxtimes}\mathpzc{F}:= \mathpzc{E}\boxtimes\mathpzc{F}\oplus \mathpzc{E}\boxtimes\mathpzc{F}$ with the summands being graded evenly and oddly, respectively. If $T_1$ and $T_2$ are operators on $\mathpzc{E}$ and $\mathpzc{F}$, respectively, $T_1\boxtimes T_2$ is defined as the exterior tensor product. We define the graded exterior product of $T_1$ and $T_2$ by
$$T_1\hat{\boxtimes}T_2:=
\begin{cases}
T_1\boxtimes\id_{\mathpzc{F}}+\gamma_{\mathpzc{E}}\boxtimes T_2,\quad& \mbox{if both $\mathpzc{E}$ and $\mathpzc{F}$ are graded};\\
i\gamma_{\mathpzc{E}}T_1\boxtimes\id_{\mathpzc{F}}+\gamma_{\mathpzc{E}}\boxtimes T_2, &\mbox{if $\mathpzc{E}$ is graded but $\mathpzc{F}$ is not};\\
T_1\boxtimes\gamma_{\mathpzc{F}}+\id_{\mathpzc{E}}\boxtimes i\gamma_{\mathpzc{F}}T_2, &\mbox{if $\mathpzc{F}$ is graded but $\mathpzc{E}$ is not},
\end{cases}$$
and finally, if both $\mathpzc{E}$ and $\mathpzc{F}$ are ungraded, 
$$T_1\hat{\boxtimes}T_2:=\begin{pmatrix}0& iT_1\boxtimes\id_{\mathpzc{F}}+\id_{\mathpzc{E}}\boxtimes T_2\\
-iT_1\boxtimes\id_{\mathpzc{F}}+\id_{\mathpzc{E}}\boxtimes T_2&0 \end{pmatrix}.$$
If $T_1$ and $T_2$ are closed, the domain of $T_1\hat{\boxtimes}T_2$ is defined to be the Hilbert $C^*$-module $\Dom(T_1)\boxtimes\Dom(T_2)$ (viewing the domains as Hilbert $C^*$-modules in their graph inner product). We note that in this notation, 
$$T_1\hat{\boxtimes}T_2=T_1\hat{\boxtimes}0+0\hat{\boxtimes}T_2.$$
It is in general not clear if $T_1\hat{\boxtimes}T_2$ is closed or regular when $T_1$ and $T_2$ are. However, if one of them is self-adjoint and the other is symmetric, then $T_1\hat{\boxtimes}T_2$ is closed, symmetric and regular with $(T_1\hat{\boxtimes}T_2)^*=\overline{T_1^*\hat{\boxtimes}T_2^*}$, see \cite[Proposition 1.14]{DGM}. Another special case is considered below in Example \ref{suspensioofncada}. The reader should bear in mind that the construction of $T_1\hat{\boxtimes}T_2$ is closely related to the exterior product in $KK$-theory, cf. \cite{baajjulg}, and throughout this work more complicated construction relating to the interior product need also be considered.

Throughout the paper, we shall also use a variety of different techniques. These include
\begin{itemize}
\item Kasparov's $KK$-theory in its bounded incarnation, see more in \cite{Bla, jeto, Kas1,Kas2}.
\item Results for regularity and adjoints of sums $S+T$, see more in \cite{DGM, leschkaad2,LeschMesland}.
\item Functional calculus of self-adjoint regular operators on Hilbert $C^*$-modules, see more in \cite{lancesbook}.
\item Smooth and holomorphic calculus on local $C^*$-algebras, see more in \cite{BlaCun,varillyetal}.
\item Complex interpolation of Banach algebras, see more in \cite{bergloef}.
\end{itemize}

\section{Examples of cycles and chains}
\label{lknaldknadljnbla}

In order to provide further intuition to the cycles, chains and noncommutative geometries that are of relevance for this work let us collect some salient examples from the literature before entering the bulk of the text. We also provide examples of how to construct new cycles and chains from old ones by means of auxiliary structures. The reader well acquainted with the literature can skip this chapter on their first read. For more details and sources of possible $KK$-bordisms, see \cite{DGM, hilsumfol, hilsumlip, hilsumcmodbun}. In this work we treat these noncommutative objects as interesting in their own right, for further justifications and applications, see \cite{Connesbook}. The reader need not be an expert in all of the examples discussed in this section but can benefit from revisiting \cite{DGM} and become familiar with the material in Subsection \ref{suspensioofncada} and \ref{diraconcstarbundleex}. 

\subsection{Suspension operations on chains}
\label{suspensioofncada}
We start by describing a general procedure for suspending a chain with the Dirac operator on an interval. Suspension by intervals is a standard technique for manifolds (with boundary) and is not novel in the realm of unbounded $KK$-theory, yet somewhat unexplored in the literature. Most of the technical conclusions in this subsection are derived from  \cite[Theorem 1.18]{DGM}. The usage of suspension by intervals is central in the work of Hilsum \cite{hilsumbordism}: both in defining $KK$-bordisms and in its relation to $KK$-theory. Suspensions will be an important tool throughout this monograph and we now give a number of variations on this theme.

\subsubsection{Suspension by $[0,1]$} 
Let $\partial$ denote the operator $i\frac{\mathrm{d}}{\mathrm{d}x}$ on $L^2[0,1]$ with a superscript indicating its domain: $\partial^{\rm min}$ denotes the graph closure of the domain $C^\infty_c(0,1)$ and $\partial^{\rm max}$ its adjoint. For an operator $T:\mathpzc{N}\dashrightarrow\mathpzc{N}$, we define $\Psi(\mathpzc{N}):=L^2[0,1]\hat{\boxtimes}\mathpzc{N}$ and
\begin{equation}
\label{psiofd}
\Psi(T):=\overline{\partial^{\rm min}\hat{\boxtimes}T}:\Psi(\mathpzc{N})\dashrightarrow \Psi(\mathpzc{N}).
\end{equation}
It follows from \cite[Theorem 1.18]{DGM} that $\Psi(T)=\partial^{\rm min}\hat{\boxtimes}T$ (i.e. $\partial^{\rm min}\hat{\boxtimes}T$ is closed) is regular and symmetric if $T$ is regular and symmetric. Moreover, $\Psi(T)^*=\overline{\partial^{\rm max}\hat{\boxtimes}T^*}$ if $T$ is regular and symmetric.

\begin{define}
\label{suspbyinte}
For an $(\mathcal{A},B)$-chain $\mathfrak{Z}=(\mathpzc{N},T)$ we define its suspension by the unit interval as the $(C^\infty_c((0,1),\mathcal{A}),B)$-chain 
$$\Psi(\mathfrak{Z}):=(\Psi(\mathpzc{N}),\Psi(T))=(L^2[0,1]\hat{\boxtimes}\mathpzc{N},\partial^{\rm min}\hat{\boxtimes}T).$$
\end{define}

\begin{prop}
\label{suspandprop}
Let $\mathfrak{Z}=(\mathpzc{N},T)$ be an $(\mathcal{A},B)$-chain.
\begin{itemize}
\item If $\mathfrak{Z}$ is symmetric, so is $\Psi(\mathfrak{Z})$. 
\item If $\mathfrak{Z}$ is half-closed, the inclusion $\Dom(\Psi(T))\hookrightarrow \Psi(\mathpzc{N})$ is locally compact for the $\mathcal{A}$-action. In particular, $\Psi(\mathfrak{Z})$ is a half-closed $(C^\infty_c((0,1),\mathcal{A}),B)$-chain when $\mathfrak{Z}$ is half-closed.
\end{itemize}
\end{prop}

Compare Proposition \ref{suspandprop} to \cite[Example 1.17]{DGM}.

\begin{proof}
If $T$ is symmetric and regular, $\Psi(T)$ is symmetric and regular by \cite[Theorem 1.18]{DGM}. This proves that $\Psi(\mathfrak{Z})$ is a symmetric chain when $\mathfrak{Z}$ is. Indeed, if $\Dom(T)\hookrightarrow \mathpzc{N}$ is locally compact for the $\mathcal{A}$-action then $\Dom(\Psi(T))\hookrightarrow \Psi(\mathpzc{N})$ is locally compact for the $\mathcal{A}$-action by \cite[Theorem 1.18]{DGM}.
\end{proof}

\subsubsection{Suspension by $[0,\infty)$} 
We let $\partial_{(0,\infty)}$ denote the operator $i\frac{\mathrm{d}}{\mathrm{d}x}$ on $L^2[0,\infty)$ equipped with its minimal domain, that is, $\partial_{(0,\infty)}$ denotes the graph closure of the domain $C^\infty_c(0,\infty)$. For an operator $T:\mathpzc{N}\dashrightarrow\mathpzc{N}$, we define $\Psi_{(0,\infty)}(\mathpzc{N}):=L^2[0,\infty)\hat{\boxtimes}\mathpzc{N}$ and
\begin{equation}
\label{psiofdinfty}
\Psi_{(0,\infty)}(T):=\overline{\partial_{(0,\infty)}\hat{\boxtimes}T}:\Psi_{(0,\infty)}(\mathpzc{N})\dashrightarrow \Psi_{(0,\infty)}(\mathpzc{N}).
\end{equation}
An approximation argument combined with an argument as in \cite[Theorem 1.18]{DGM} show that $\Psi_{(0,\infty)}(T)=\partial_{(0,\infty)}\hat{\boxtimes}T$ (i.e. $\partial_{(0,\infty)}\hat{\boxtimes}T$ is closed) is regular and symmetric if $T$ is regular and symmetric. Moreover, $\Psi_{(0,\infty)}(T)^*=\overline{\partial^*_{(0,\infty)}\hat{\boxtimes}T^*}$ if $T$ is regular and symmetric.

\begin{define}
\label{suspbyintehalf}
For an $(\mathcal{A},B)$-chain $\mathfrak{Z}=(\mathpzc{N},T)$ we define its suspension by the half interval as the $(C^\infty_c((0,\infty),\mathcal{A}),B)$-chain 
$$\Psi_{(0,\infty)}(\mathfrak{Z}):=(\Psi_{(0,\infty)}(\mathpzc{N}),\Psi_{(0,\infty)}(T))=(L^2[0,\infty)\hat{\boxtimes}\mathpzc{N},\partial_{(0,\infty)}\hat{\boxtimes}T).$$
\end{define}

\begin{prop}
\label{suspandprophalf}
Let $\mathfrak{Z}=(\mathpzc{N},T)$ be an $(\mathcal{A},B)$-chain.
\begin{itemize}
\item If $\mathfrak{Z}$ is symmetric, so is $\Psi_{(0,\infty)}(\mathfrak{Z})$. 
\item If $\mathfrak{Z}$ is half-closed, then $\Psi_{(0,\infty)}(\mathfrak{Z})$ is a half-closed $(C^\infty_c((0,\infty),\mathcal{A}),B)$-chain.
\end{itemize}
\end{prop}

\begin{proof}
If $T$ is symmetric and regular, $\Psi_{(0,\infty)}(T)$ is symmetric and regular by an approximation argument combined with an argument as in \cite[Theorem 1.18]{DGM}. This proves that $\Psi_{(0,\infty)}(\mathfrak{Z})$ is a symmetric chain when $\mathfrak{Z}$ is. Indeed, if $\Dom(T)\hookrightarrow \mathpzc{N}$ is locally compact for the $\mathcal{A}$-action then $\Dom(\Psi_{(0,\infty)}(T))\hookrightarrow \Psi_{(0,\infty)}(\mathpzc{N})$ is locally compact for the $C^\infty_c((0,\infty),\mathcal{A})$-action.
\end{proof}

\subsubsection{Suspension by $\R$} 
We let $\partial_{\infty}$ denote the operator $i\frac{\mathrm{d}}{\mathrm{d}x}$ on $L^2(\R)$ equipped with its minimal domain, that is, $\partial_{\infty}$ denotes the graph closure of the domain $C^\infty_c(\R)$. The operator $\partial_{\infty}$ is self-adjoint and its domain coincides with its maximal domain $H^1(\R)$. For an operator $T:\mathpzc{N}\dashrightarrow\mathpzc{N}$, we define $\Psi_{\infty}(\mathpzc{N}):=L^2(\R)\hat{\boxtimes}\mathpzc{N}$ and
\begin{equation}
\label{psiofdrinfty}
\Psi_{\infty}(T):=\partial_{\infty}\hat{\boxtimes}T:\Psi_{\infty}(\mathpzc{N})\dashrightarrow \Psi_{\infty}(\mathpzc{N}).
\end{equation}
Since $\partial_\infty$ is self-adjoint, \cite[Theorem 1.3]{DGM} implies that the operator $\Psi_{\infty}(T)$ is closed, regular and symmetric if $T$ is regular and symmetric. Moreover, $\Psi_{\infty}(T)^*=\overline{\partial_{\infty}\hat{\boxtimes}T^*}$ if $T$ is regular and symmetric and if additionally $T$ is self-adjoint, then so is $\Psi_{\infty}(T)$.

\begin{define}
\label{suspbyinteR}
For an $(\mathcal{A},B)$-chain $\mathfrak{Z}=(\mathpzc{N},T)$ we define its suspension by the real line as the $(C^\infty_b(\R,\mathcal{A}),B)$-chain 
$$\Psi_{\infty}(\mathfrak{Z}):=(\Psi_{\infty}(\mathpzc{N}),\Psi_{\infty}(T))=(L^2(\R)\hat{\boxtimes}\mathpzc{N},\partial_{\infty}\hat{\boxtimes}T).$$
\end{define}

As in the subsubsections above, we conclude the following using \cite[Theorem 1.3]{DGM}.

\begin{prop}
\label{suspandpropR}
Let $\mathfrak{Z}=(\mathpzc{N},T)$ be an $(\mathcal{A},B)$-chain.
\begin{itemize}
\item If $\mathfrak{Z}$ is symmetric, then $\Psi_{\infty}(\mathfrak{Z})$ is a symmetric $(C^\infty_b(\R,\mathcal{A}),B)$-chain.
\item If $\mathfrak{Z}$ is half-closed, then $\Psi_{\infty}(\mathfrak{Z})$ is a half-closed $(C^\infty_c(\R,\mathcal{A}),B)$-chain.
\item If $\mathfrak{Z}$ is closed, then $\Psi_{\infty}(\mathfrak{Z})$ is a closed $(C^\infty_c(\R,\mathcal{A}),B)$-chain.
\end{itemize}
\end{prop}

\subsubsection{Shubin-suspension by $\R$} 
We let $\partial_{Sh}$ denote the odd operator 
$$\partial_{Sh}:=\begin{pmatrix}
0&i\partial_\infty+x\\
-i\partial_\infty+x& 0
\end{pmatrix}$$
on the graded Hilbert space $L^2(\R,\C^2)$ equipped with its domain as a sum. The operator $\partial_{Sh}$ is self-adjoint and its domain coincides with $\{f\in H^1(\R,\C^2): \hat{f}\in H^1(\R,\C^2)\}$. The operator $\partial_{Sh}$ has compact resolvent since $\partial_{Sh}^2$ up to a constant matrix is the harmonic oscillator. For an operator $T:\mathpzc{N}\dashrightarrow\mathpzc{N}$, we define $\Psi_{Sh,\infty}(\mathpzc{N}):=L^2(\R,\C^2)\hat{\boxtimes}\mathpzc{N}$ and
\begin{equation}
\label{psiofdinftysh}
\Psi_{Sh,\infty}(T):=\partial_{Sh}\hat{\boxtimes}T:\Psi_{Sh,\infty}(\mathpzc{N})\dashrightarrow \Psi_{Sh,\infty}(\mathpzc{N}).
\end{equation}
Since $\partial_{Sh}$ is self-adjoint,  \cite[Theorem 1.3]{DGM} implies that the operator $\Psi_{Sh,\infty}(T)$ is closed, regular and symmetric if $T$ is regular and symmetric. Moreover, $\Psi_{Sh,\infty}(T)^*=\overline{\partial_{Sh}\hat{\boxtimes}T^*}$ if $T$ is regular and symmetric and if additionally $T$ is self-adjoint, then so is $\Psi_{Sh,\infty}(T)$.

\begin{define}
\label{suspbyinteShR}
For an $(\mathcal{A},B)$-chain $\mathfrak{Z}=(\mathpzc{N},T)$ we define its Shubin suspension by the real line as the $(C^\infty_b(\R,\mathcal{A}),B)$-chain 
$$\Psi_{Sh,\infty}(\mathfrak{Z}):=(\Psi_{Sh,\infty}(\mathpzc{N}),\Psi_{Sh,\infty}(T))=(L^2(\R,\C^2)\hat{\boxtimes}\mathpzc{N},\partial_{Sh}\hat{\boxtimes}T).$$
\end{define}

As in the subsubsections above, we conclude the following using \cite[Theorem 1.3]{DGM} and the fact that $\partial_{Sh}$ has compact resolvent.

\begin{prop}
\label{suspandpropShR}
Let $\mathfrak{Z}=(\mathpzc{N},T)$ be an $(\mathcal{A},B)$-chain.
\begin{itemize}
\item If $\mathfrak{Z}$ is symmetric, so is $\Psi_{Sh,\infty}(\mathfrak{Z})$. 
\item If $\mathfrak{Z}$ is half-closed, then $\Psi_{Sh,\infty}(\mathfrak{Z})$ is a half-closed $(C^\infty_b(\R,\mathcal{A}),B)$-chain.
\item If $\mathfrak{Z}$ is closed, then $\Psi_{Sh,\infty}(\mathfrak{Z})$ is a closed $(C^\infty_b(\R,\mathcal{A}),B)$-chain.
\end{itemize}
\end{prop}

\subsubsection{Shubin-suspension by $[0,\infty)$} 
\label{shubinbyhalf}
We let $\partial_{Sh,[0,\infty)}$ denote the odd operator 
\begin{equation}
\label{shuboinddn}
\partial_{Sh,[0,\infty)}:=\begin{pmatrix}
0&i\partial_{[0,\infty)}+x\\
-i\partial_{[0,\infty)}+x& 0
\end{pmatrix}
\end{equation}
on the graded Hilbert space $L^2([0,\infty),\C^2)$ equipped with its domain as a sum. The operator $\partial_{Sh,[0,\infty)}$ is symmetric and its domain coincides with $\Dom(\partial_{Sh})\cap H^1_0([0,\infty))$, where we use that $H^1_0([0,\infty))\subseteq H^1(\R)$ is a closed subspace. The operator $\partial_{Sh,[0,\infty)}$ has a compact domain inclusion because $\partial_{Sh}$ has a compact domain inclusion and  $\Dom(\partial_{Sh.[0,\infty)}\hookrightarrow  L^2([0,\infty))$ factors over $\Dom(\partial_{Sh})\hookrightarrow L^2(\R)\to L^2([0,\infty))$ where the last map is the projection. The operator $\partial_{Sh,[0,\infty)}^*$ is described by the same differential expression as in Equation \eqref{shuboinddn} but has domain 
$$\Dom(\partial_{Sh,[0,\infty)}^*)=\{f|_{[0,\infty)}: f\in \Dom(\partial_{Sh})\}=\{f\in H^1([0,\infty)): xf \in L^2([0,\infty))\}.$$
For an operator $T:\mathpzc{N}\dashrightarrow\mathpzc{N}$, we define $\Psi_{Sh,[0,\infty)}(\mathpzc{N}):=L^2([0,\infty),\C^2)\hat{\boxtimes}\mathpzc{N}$ and
\begin{equation}
\label{psiofdinftysh0inf}
\Psi_{Sh,[0,\infty)}(T):=\partial_{Sh,[0,\infty)}\hat{\boxtimes}T:\Psi_{Sh,[0,\infty)}(\mathpzc{N})\dashrightarrow \Psi_{Sh,[0,\infty)}(\mathpzc{N}).
\end{equation}
An approximation argument combining \cite[Theorem 1.3]{DGM} with \cite[Theorem 1.18]{DGM} show that $\Psi_{Sh,[0,\infty)}(T)$ is closed, regular and symmetric if $T$ is closed, regular and symmetric. Moreover, 
$$\Psi_{Sh,[0,\infty)}(T)^*=\overline{\partial^*_{Sh,[0,\infty)}\hat{\boxtimes}T^*},$$ 
if $T$ is regular and symmetric.

\begin{define}
\label{suspbyinteShhalf}
For an $(\mathcal{A},B)$-chain $\mathfrak{Z}=(\mathpzc{N},T)$ we define its Shubin suspension by the half line $[0,\infty)$ as the $(C^\infty_b([0,\infty),\mathcal{A}),B)$-chain 
$$\Psi_{Sh,[0,\infty)}(\mathfrak{Z}):=(\Psi_{Sh,[0,\infty)}(\mathpzc{N}),\Psi_{Sh,[0,\infty)}(T))=(L^2([0,\infty),\C^2)\hat{\boxtimes}\mathpzc{N},\partial_{Sh,[0,\infty)}\hat{\boxtimes}T).$$
\end{define}

As above, we conclude the following using \cite[Theorem 1.3]{DGM} and compact domain inclusion of $\partial_{Sh,[0,\infty)}$.

\begin{prop}
\label{suspandpropShhalf}
Let $\mathfrak{Z}=(\mathpzc{N},T)$ be an $(\mathcal{A},B)$-chain.
\begin{itemize}
\item If $\mathfrak{Z}$ is symmetric, so is $\Psi_{Sh,[0,\infty)}(\mathfrak{Z})$. 
\item If $\mathfrak{Z}$ is half-closed, then $\Psi_{Sh,[0,\infty)}(\mathfrak{Z})$ is a half-closed $(C^\infty_b([0,\infty),\mathcal{A}),B)$-chain.
\end{itemize}
\end{prop}

\subsection{Dirac operators on $C^*$-bundles}
\label{diraconcstarbundleex}
An important construction of $KK$-chains comes from geometric $K$-homology \cite{BD,BDbor}. The construction is that of unbounded $KK$-cycles from Dirac operators twisted by $C^*$-bundles and was explicitly considered in this context in \cite[Section 1.5 and 4]{DGM}, yet is present in a substantial amount of literature, e.g. \cite{hankpapsch,hilsumcmodbun,hilsumfol,hilsumbordism,LP, LPGAFA,PS, PSrhoInd,PSsignInd,WahlProductAPS}. In this setting, the adjectives closed, half-closed and symmetric all have geometric meaning. Let $W$ denote a complete Riemannian manifold with boundary. Whenever referring to a manifold, we tacitly mean a smooth manifold. We will use the convention that $C^\infty_c(W^\circ)$ denotes the space of smooth functions compactly supported in $W^\circ$ and $C^\infty_c(\overline{W})$ denotes the space of smooth functions  compactly supported in $\overline{W}$. In other words, elements of both spaces are compactly supported but elements of $C^\infty_c(W^\circ)$ are supported away from the boundary $\partial W$.

Consider a unital $C^*$-algebra $B$ and a Clifford-$B$-bundle $\mathcal{E}_B\to W$ with Clifford connection $\nabla_{\mathcal{E}}$. From this data, we can define the $B$-linear Dirac operator $\slashed{D}_\mathcal{E}:C^\infty_c(W^\circ, \mathcal{E}_B)\to C^\infty_c(W^\circ, \mathcal{E}_B)$. Such a situation arises for instance when $B=C^*(\Gamma)$, for a discrete group $\Gamma$, $W$ is a complete Riemannian spin$^c$-manifold with boundary and 
$$\mathcal{E}_{C^*(\Gamma)}=p^*S_{W}\times_{\Gamma}C^*(\Gamma)=S_W\otimes \mathcal{L}_{\tilde{W}},$$ 
for a $\Gamma$-Galois covering $p:\tilde{W}\to W$ and the complex spinor bundle $S_W\to W$. Here $\mathcal{L}_{\tilde{W}}:=\tilde{W}\times_{\Gamma}C^*(\Gamma)$ is the Mishchenko bundle. The Clifford connection $\nabla_{\mathcal{E}}$ is defined from the Clifford connection on $S_W$ and the flat Hermitian connection on the Mishchenko bundle. This particular example plays an important role in the study of the assembly mapping for free actions of discrete groups, see for instance \cite{land}. It is studied in more detail below in Subsection \ref{diraconcstarbundleexbordism}. 

We say that a Dirac operator $\slashed{D}_\mathcal{E}$ on $\mathcal{E}_B$ is complete if any $f\in L^2(W,\mathcal{E}_B)$ such that $\slashed{D}_\mathcal{E}f\in L^2(W,\mathcal{E}_B)$ (in distributional sense) can be approximated by a sequence $(f_j)_{j\in \N}\subseteq C^\infty_c(\overline{W},\mathcal{E}_B)$ in the sense that $f_j\to f$ and $\slashed{D}_\mathcal{E}f_j\to \slashed{D}_\mathcal{E}f$ in $L^2$. Unless otherwise stated, we tacitly assume that all Dirac operators are complete. This is automatic if $\overline{W}$ is compact or if the Clifford connection is of bounded geometry.

By \cite[Lemma 1.24]{DGM}, the closure $D_\mathcal{E}^{\rm min}$ of $\slashed{D}_\mathcal{E}$ with domain $C^\infty_c(W^\circ, \mathcal{E}_B)$ as an operator on the $B$-Hilbert $C^*$-module $L^2(W,\mathcal{E}_B)$ fits into a symmetric chain $(L^2(W,\mathcal{E}_B),D_{\mathcal{E}}^{\rm min})$ for $(C^\infty_c(\overline{W}),B)$. It is half-closed when restricting to a $(C^\infty_c(W^\circ),B)$-chain and closed if and only if $\partial W=\emptyset$.

\begin{figure}
   \centering
 \includegraphics[height=5.35cm]{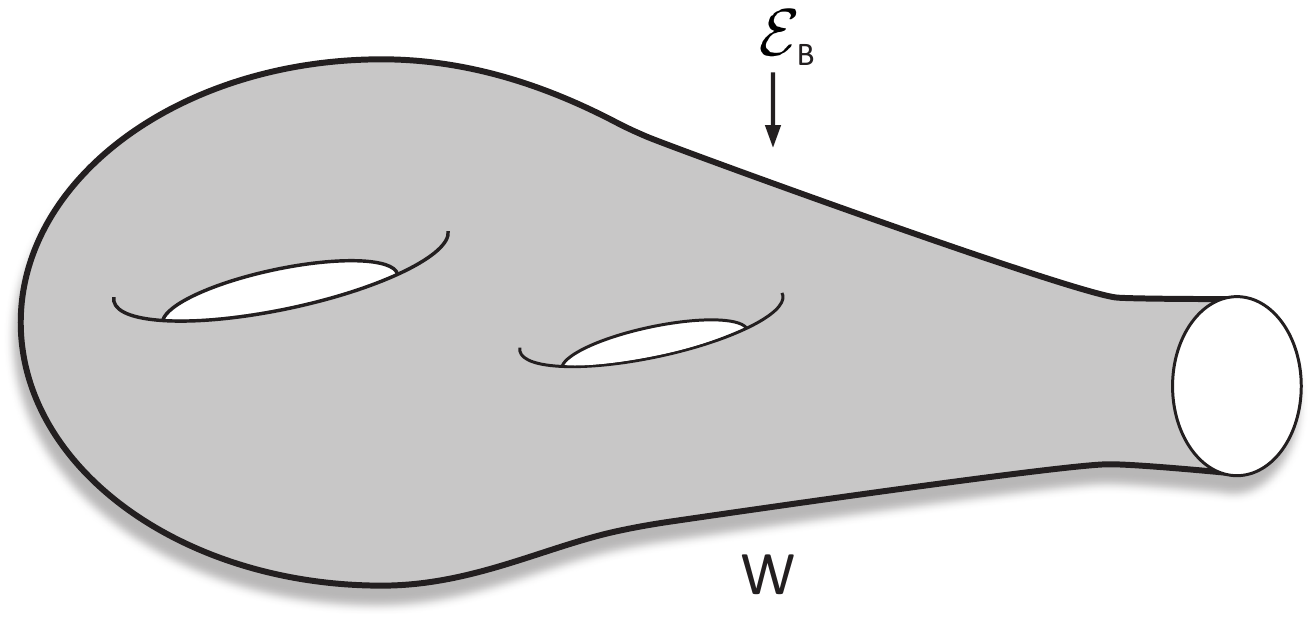}
  \caption{A $C^*$-bundle $\mathcal{E}_B$ over a compact manifold with boundary as in Example \ref{diraconcstarbundleex}.}
\end{figure}

For any metric space $X$, closed subspace $Y\subseteq X$, and Lipschitz continuous mapping $f:(W,\partial W)\to (X,Y)$ (i.e. $f(\partial W)\subseteq Y$) we can pull back the action along $f$ and consider $(L^2(W,\mathcal{E}_B),D_{\mathcal{E}}^{\rm min})$ as a symmetric $(\mathrm{Lip}(X),B)$-chain and as a half-closed $(\mathrm{Lip}_0(X\setminus Y),B)$-chain. Here $\mathrm{Lip}(X)$ denotes the algebra of uniformly bounded and uniformly Lipschitz functions on $X$ and $\mathrm{Lip}_0(X\setminus Y):=\mathrm{Lip}(X)\cap C_0(X\setminus Y)$.

The reader should note that in this class of examples there are few isomorphisms in the following sense. If $(L^2(W,\mathcal{E}_B),D_{\mathcal{E}}^{\rm min})$ and $(L^2(W,\mathcal{E}'_B),D_{\mathcal{E}'}^{\rm min})$ are two $(C^\infty_c(W^\circ),B)$-chains constructed as above, then they are isomorphic if and only if $\mathcal{E}_B\to W$ and $\mathcal{E}_B'\to W$ are isomorphic as Clifford $B$-bundles via a (unitary Clifford linear) isomorphism intertwining the Clifford connections.

\subsection{Group actions on cycles and chains}
\label{subsecgequicalalad}

Let us now consider a fairly general construction relating to the descent of a certain class of equivariant unbounded cycles. There is no generally accepted notion of equivariant unbounded cycles, and a standard reference is \cite{kucerovskythesis}. We shall fix a notion for equivariant unbounded cycles that is sufficient to study proper group actions on manifolds (cf. Subsection \ref{groupactionsex} below). 

\begin{define}
\label{gequicalalad}
Consider a second countable locally compact group $G$ acting on $C^*$-algebras $A$ and $B$, and the action of $G$ preserves the dense $*$-subalgebra $\mathcal{A}\subseteq A$. We shall say that $(\mathpzc{E},D)$ is a $G$-equivariant $(\mathcal{A},B)$-chain if 
\begin{enumerate}
\item $(\mathpzc{E},D)$ is an $(\mathcal{A},B)$-chain, $\mathpzc{E}$ is a $G$-equivariant $B$-Hilbert $C^*$-module and the left action $A\to \End_B^*(\mathpzc{E})$ is $G$-equvariant;
\item any $g\in G$ preserves $\Dom(D)$;
\item for any $g\in G$ and $a\in \mathcal{A}$, $a(D-gDg^{-1})$ and $(D-gDg^{-1})a$ extend to $B$-linear adjointable operators on $\mathpzc{E}$;
\item for any $a\in \mathcal{A}$, the functions
$$g\mapsto a(D-gDg^{-1})\quad \mbox{and}\quad g\mapsto(D-gDg^{-1})a,$$
are $*$-strongly continuous functions $G\to \End_B^*(\mathpzc{E})$.
\end{enumerate}
\end{define}

The reader can find that our definition of a $G$-equivariant $(\mathcal{A},B)$-chain in Definition \ref{gequicalalad} coincides with that of uniformly equivariant unbounded cycle in the sense of \cite[Definition I.2.7]{adathesis}, see also \cite[Definition 8.7]{danthesis}.

The descent construction will be a construction of a $(C_c(G,\mathcal{A}),B\rtimes G)$-chain from a $G$-equivariant $(\mathcal{A},B)$-chain. The construction follows \cite[Chapter 3]{Kas2} with suitable modifications to the unbounded setting, see also \cite[Chapter I.2.2]{adathesis}. Equip $C_c(G,\mathpzc{E})$ with the left $C_c(G,A)$-action 
$$a.\xi(g)=\int_Ga(h)h.\xi(h^{-1}g)\mathrm{d}g,$$
and the right $C_c(G,B)$-action 
$$\xi.b(g)=\int_G\xi(h)h.b(h^{-1}g)\mathrm{d}g.$$
We write $A\rtimes G$ for the $C^*$-closure of $C_c(G,A)$ with the product 
$$a_1a_2(g):=\int_Ga_1(h)h.a_2(h^{-1}g)\mathrm{d}g$$
and $a^*(g)=\Delta(g)^{-1}g.a(g^{-1})^*$ where $\Delta:G\to \R_+$ denotes the modular function. We define $\mathpzc{E}\rtimes G$ as the $B\rtimes G$-Hilbert $C^*$-module closure of $C_c(G,\mathpzc{E})$ in the inner product
$$\langle \xi_1,\xi_2\rangle_{\mathpzc{E}\rtimes G}(g):=\int_Gh^{-1}\langle \xi_1(h),\xi_2(hg)\rangle_{\mathpzc{E}}\mathrm{d}g,$$
The left $C_c(G,A)$-action defines a left action $A\rtimes G\to \End_{B\rtimes G}^*(\mathpzc{E}\rtimes G)$. By \cite[Chapter 3]{Kas2}, any $T\in \End_B^*(\mathpzc{E})$ induces a $B\rtimes G$-linear adjointable operator $T\rtimes G\in \End_{B\rtimes G}^*(\mathpzc{E}\rtimes G)$ by defining 
$$(T\rtimes G)\xi(g):=T\xi(g), \quad \xi\in C_c(G,\mathpzc{E}),$$
and extending by continuity. In the unbounded setting, the authors have not been able to find a general definition of descent of regular operators. We use the following construction. 

\begin{define}
\label{descentofunbbbddd}
Let $T:\mathpzc{E}\dashrightarrow\mathpzc{F}$ be a regular operator between two $G$-equivariant $B$-Hilbert $C^*$-modules. We define the descent of $T$ as the regular operator $T\rtimes G:\mathpzc{E}\rtimes G\dashrightarrow\mathpzc{F}\rtimes G$ characterized from the bijective correspondence in Equation \eqref{graphdbibidkd} by declaring 
$$p_{T\rtimes G}:=p_T\rtimes G.$$ 
\end{define}

The operator $T\rtimes G$ is well defined as it is easily checked that $\pi_1|_{\mathrm{im}(p_T\rtimes G)}$ is injective and we have that $\pi_1(\mathrm{im}(p_T\rtimes G))\subseteq \mathpzc{E}\rtimes G$ is dense because $C_c(G,\Dom(T))\subseteq \pi_1(\mathrm{im}(p_T\rtimes G))$. A different but equivalent definition is possible. Namely, we can define $T\rtimes G$ from $\Dom(T\rtimes G):=\Dom(T)\rtimes G$ on which $T$ is defined by setting 
$$(T\rtimes G)\xi(g):=T\xi(g), \quad\mbox{for}\; \xi\in C_c(G,\Dom(T)),$$
and extending by continuity in the graph norm. The latter definition is more straightforward to work with, but Definition \ref{descentofunbbbddd} makes it easier to see that the process produces a regular operator. It is clear that $T\rtimes G$ is symmetric/self-adjoint if $T$ is and that $T\rtimes G$ has compact domain inclusion if $T$ does. For a $G$-equivariant $(\mathcal{A},B)$-chain $(\mathpzc{E},D)$ we write 
$$(\mathpzc{E},D)\rtimes G:=(\mathpzc{E}\rtimes G,D\rtimes G).$$
We call $(\mathpzc{E},D)\rtimes G$ the \emph{descended cycle} of the $G$-equivariant $(\mathcal{A},B)$-chain $(\mathpzc{E},D)$.

We write 
$$\mathcal{A}\rtimes_cG:=C_c(G,\mathcal{A}).$$

\begin{prop}
\label{descenfoscles}
For a $G$-equivariant $(\mathcal{A},B)$-chain $(\mathpzc{E},D)$ (in the sense of Definition \ref{gequicalalad}), $(\mathpzc{E},D)\rtimes G$ is a well defined $(\mathcal{A}\rtimes_c G,B\rtimes G)$-chain. If $(\mathpzc{E},D)$ is a symmetric/half-closed/closed $(\mathcal{A},B)$-chain, then $(\mathpzc{E},D)\rtimes G$ is a symmetric/half-closed/closed $(\mathcal{A}\rtimes_cG,B\rtimes G)$-chain. 

Moreover, if $(\mathpzc{E},D)$ is a $G$-equivariant symmetric $(\mathcal{A},B)$-chain with $a(1+D^*D)^{-1}\in \K_B(\mathpzc{E})$ for all $a\in \mathcal{A}$, then
\begin{equation}
\label{cpakakadldalad}
a(1+(D\rtimes G)^*D\rtimes G)^{-1}\in \K_{B\rtimes G}(\mathpzc{E}\rtimes G),\quad \forall a\in A\rtimes G.
\end{equation}
\end{prop}

\begin{proof}
The first part of the proposition follows from the discussion above combined with the argument in \cite[Chapter 3]{Kas2}. To prove Equation \eqref{cpakakadldalad}, we note that the equality $p_{D\rtimes G}=p_D\rtimes G$ implies the equality
$$(1+(D\rtimes G)^*D\rtimes G)^{-1}=(1+D^*D)^{-1}\rtimes G.$$
As such, Equation \eqref{cpakakadldalad} follows from the fact that $a(K\rtimes G)\in \K_{B\rtimes G}(\mathpzc{E}\rtimes G)$ for $a\in A\rtimes G$ whenever $K\in \End_B^*(\mathpzc{E})$ is $A$-locally compact.
\end{proof}

\subsection{Group actions on manifolds}
\label{groupactionsex}
The constructions from Subsection \ref{diraconcstarbundleex} extend to isometric group actions, a situation well studied in the literature at least in bounded $KK$-theory \cite{BCHig,Kas2}. We here follow the terminology and notations of Subsection \ref{diraconcstarbundleex}. Consider a (second countable) locally compact group $G$ acting by smooth isometries on a complete Riemannian manifold with boundary $W$. If, for instance, $G$ acts properly and smoothly on a manifold with boundary $W$ then there is a complete Riemannian metric on $W$ in which $G$ acts by isometries. For a unital $C^*$-algebra $B$, assume that $S_B\to \overline{W}$ is a $G$-equivariant $B$-Clifford bundle. Pick a complete Dirac operator $\slashed{D}_S$ and let $D_S^{\rm min}$ denote its minimal closure, see details in Subsection \ref{diraconcstarbundleex}.

\begin{prop}
\label{descenofbundlesoddnam}
In the context of the preceding paragraph, $(L^2(W,S_B),D_S^{\rm min})$ is a symmetric $G$-equivariant $(C^\infty_c(\overline{W}),B)$-chain restricting to a half-closed $G$-equivariant $(C^\infty_c(W^\circ),B)$-chain. In particular, $(L^2(W,S_B)\rtimes G,D_S^{\rm min}\rtimes G)$ is a symmetric $(C^\infty_c(\overline{W})\rtimes_cG,B\rtimes G)$-chain with 
$$a(1+(D_S^{\rm min}\rtimes G)^*D_S^{\rm min}\rtimes G)^{-1}\in \K_{B\rtimes G}(L^2(W,S_B)\rtimes G),\quad \forall a\in C_0(\overline{W})\rtimes G.$$
The chain $(L^2(W,S_B)\rtimes G,D_S^{\rm min}\rtimes G)$ restricts to a half-closed $(C^\infty_c(W^\circ)\rtimes_cG,B\rtimes G)$-chain.
\end{prop}

\begin{proof}
By the discussion in Subsection \ref{diraconcstarbundleex} and the fact that $G$ acts smoothly and isometrically (thus preserving the domain), it only remains to prove that  for any $a\in C^\infty_c(\overline{W})$, the functions
$$g\mapsto a(D_S^{\rm min}-gD_S^{\rm min}g^{-1})\quad \mbox{and}\quad g\mapsto(D_S^{\rm min}-gD_S^{\rm min}g^{-1})a,$$
are well defined norm continuous functions $G\to \End_B^*(L^2(W,S_B)$. Since $\slashed{D}_S$ is a Dirac operator, and  $G$ acts smoothly and isometrically, there is a smooth section $V\in C(G,C^\infty(\overline{W};\End_B^*(S_B)))$ such that 
$$\slashed{D}_S-g\slashed{D}_Sg^{-1}=V(g).$$
In particular, $(D_S^{\rm min}-gD_S^{\rm min}g^{-1})a=V(g)a$ and $a(D_S^{\rm min}-gD_S^{\rm min}g^{-1})=aV(g)$ act as an element of $C^\infty_c(\overline{W};\End_B^*(S_B))$ and depend continuously on $g$. The final conclusion of the proposition is immediate from Proposition \ref{descenfoscles}. The proof is complete.
\end{proof}

 Let us consider a special case and a consequence of Proposition \ref{descenofbundlesoddnam}.

\subsubsection{$B=\C$} 
Assume that $B=\C$, so $S\to W$ is a Clifford bundle and $\slashed{D}_S$ is an ordinary Dirac operator. 
The multiplicative unitary $U_W\in U\mathcal{M}(\K(L^2(W,S))\otimes C_0(G))$ defining the $G$-action on $L^2(W,S)$ defines an isomorphism of $C^*(G)$-Hilbert $C^*$-modules 
\begin{equation}
\label{spaeickalad}
L^2(W,S)\rtimes G\cong L^2(W,S)\otimes C^*(G).
\end{equation}
Equipping $L^2(W,S)\rtimes G$ with the left $G$-action $g\cdot(f\otimes \xi):=g\cdot f\otimes u_g\xi$ for $g\in G$, where $u_g\in \mathcal{M}(C^*(G))$ is the unitary associated with a group element, and the left $C_0(\overline{W})$-action $a\cdot(f\otimes \xi):=af\otimes \xi$ there is an induced left action $C_0(\overline{W})\rtimes G\to \End_{C^*(G)}^*(L^2(W,S)\otimes C^*(G))$ compatible with the isomorphism \eqref{spaeickalad}. The operator $D_S^{\rm min}\otimes 1$ is a regular and symmetric operator in $L^2(W,S)\otimes C^*(G)$ that under the isomorphism \eqref{spaeickalad} differs from $D_S^{\rm min}\rtimes G$ by a locally bounded operator (cf. the proof of Proposition \ref{descenofbundlesoddnam} and Definition \ref{locbbddd} below). The same argument as above shows that $(L^2(W,S)\otimes C^*(G),D_{S}^{\rm min}\otimes 1)$ is a symmetric $(C^\infty_c(\overline{W})\rtimes_cG,C^*(G))$-chain restricting to a half-closed $(C^\infty_c(W^\circ)\rtimes_cG,C^*(G))$-chain. If $\partial W=\emptyset$, the construction produces a closed cycle. By our discussion, there is a locally bounded $V$ such that 
$$(L^2(W,S)\otimes C^*(G),D_{S}^{\rm min}\otimes 1+V)\cong (L^2(W,S_B),D_S^{\rm min})\rtimes G.$$

Localising in the trivial $G$-representation $\epsilon$ recovers $(L^2(W,S),D_{S}^{\rm min})$ viewed as a $(C^\infty_c(\overline{W})\rtimes_cG,\C)$-chain, indeed we have that
$$(L^2(W,S)\otimes C^*(G),D_{S}^{\rm min}\otimes 1)\otimes_\epsilon \C=(L^2(W,S),D_{S}^{\rm min}).$$
The equality holds both as symmetric $(C^\infty_c(\overline{W})\rtimes_cG,\C)$-chains and as half-closed $(C^\infty_c(W^\circ)\rtimes_cG,\C)$-chains. The $(C^\infty_c(\overline{W})\rtimes_cG,\C)$-chain $(L^2(W,S),D_{S}^{\rm min})$ can be studied in a more direct fashion. A computation (see for instance \cite[Section 5.2]{boundarypap}) shows that for any $a\in C^\infty_c(\overline{W})\rtimes_cG$, $a$ preserves $\Dom(D_S^{\rm min})$ with $[D_S^{\rm min},a]$ bounded and any $j\in C^\infty_c(W^\circ)\rtimes_cG$ satisfies $$j\Dom((D_S^{\rm min})^*)\subseteq \Dom(D_S^{\rm min}).$$ 
The same argument as in Example \ref{diraconcstarbundleex} shows that $(L^2(W,S),D_{S}^{\rm min})$ is a symmetric $(C^\infty_c(\overline{W})\rtimes_cG,\C)$-chain restricting to a half-closed $(C^\infty_c(W^\circ)\rtimes_cG,\C)$-chain. The construction produces a closed cycle if and only if $\partial W=\emptyset$, indeed, by for instance \cite{BaerBall}, $\partial W=\emptyset$ if and only if $D_S^{\rm min}= (D_S^{\rm min})^*$.

\subsubsection{Assembly of a proper action} Returning to the case of a general (unital) $G-C^*$-algebra $B$, we shall define the assembled chain of the symmetric $G$-equivariant $(C^\infty_c(\overline{W}),B)$-chain $(L^2(W,S_B),D_S^{\rm min})$ when the $G$-action on $W$ is proper and cocompact. The assembly of the $G$-equivariant $(C^\infty_c(\overline{W}),B)$-chain $(L^2(W,S_S),D_{S}^{\rm min})$ will be a symmetric $(\C,B\rtimes G)$-chain which is closed if $\partial W=\emptyset$. This construction of assembly is related to the well known Baum-Connes assembly map, which will be discussed more in the sequel \cite{monographtwo}. 

We assume that the $G$-action on $W$ is proper and cocompact (i.e. $\overline{W}/G$ is compact).  Fix a nonnegative function $c\in C^\infty_c(\overline{W})$ such that $\int_G c(g\cdot x)\mathrm{d}g=1$ for all $x\in \overline{W}$, we call such a $c$ a \emph{support function}. We define the element $p_W\in C(G,C^\infty_c(\overline{W}))$ by 
\begin{equation}
\label{aksakakadldala}
p_W(g,x):=\sqrt{c(x)c(g\cdot x)}.
\end{equation}
Since the action is proper, $p_W\in C_c(G,C^\infty_c(\overline{W}))$ and so $p_W\in C_0(\overline{W})\rtimes G$. A short computation shows that $p_W$ is a projection. In particular, $p_W$ acts as a bounded $B\rtimes G$-linear adjointable operator on $L^2(W,S_B)\rtimes G$. Again, a short computation shows that $p_W$ preserves $\Dom(D_{S}^{\rm min}\rtimes G)$ and $[D_{S}^{\rm min}\rtimes G,p_W]$ has a bounded extension to a $B\rtimes G$-linear adjointable operator on  $L^2(W,S_B)\rtimes G$. We note that $C(\overline{W}/G)$ defines a central subalgebra of $\mathcal{M}(C_0(\overline{W})\rtimes G)$ and as such acts from the left on $L^2(W,S_B)\rtimes G$. We let $C^\infty(\overline{W}/G)\subseteq C(\overline{W}/G)$ denote the unital dense $*$-subalgebra defined from 
$$ C^\infty(\overline{W}/G):=\left\{\int_G [g^*a]\mathrm{d}g: a\in C^\infty_c(\overline{W})\right\}.$$
We let $C^\infty_c(W^\circ /G)\subseteq C_0(W^\circ/G)$ denote the dense $*$-subalgebra defined from 
$$ C^\infty_c(W^\circ/G):=\left\{\int_G [g^*a]\mathrm{d}g: a\in C^\infty_c(W^\circ)\right\}.$$

\begin{prop}
\label{assemblfldoado}
Let $W$, $S_B\to W$ and $D_S^{\rm min}$ be as above. Fix a support function $c$ and define $p_W$ as in Equation \eqref{aksakakadldala}. Write 
$$D_{S,\mu}^{\rm min}:=p_W(D_{S}^{\rm min}\rtimes G)p_W,$$
as a densely defined operator on $p_W(L^2(W,S_B)\rtimes G)$. Then the following holds:
\begin{enumerate}
\item $D_{S,\mu}^{\rm min}$ is regular and symmetric. Moreover, $(1+(D_{S,\mu}^{\rm min})^*D_{S,\mu}^{\rm min})^{-1}$ defines a $B\rtimes G$-compact operator on $L^2(W,S_B)\rtimes G$. The operator $D_{S,\mu}^{\rm min}$ is self-adjoint if $\partial W=\emptyset$.
\item Any $a\in C^\infty(\overline{W}/G)$ preserves $\Dom(D_{S,\mu}^{\rm min})$ and $[D_{S,\mu}^{\rm min},a]$ extends to a  $B\rtimes G$-linear adjointable operator on $L^2(W,S_B)\rtimes G$. 
\item Any $j\in C^\infty_c(W^\circ /G)$ satisfies $j\Dom((D_{S,\mu}^{\rm min})^*)\subseteq \Dom(D_{S,\mu}^{\rm min})$. 
\end{enumerate}
In particular, $(p_W(L^2(W,S_B)\rtimes G),p_W(D_{S}^{\rm min}\rtimes G)p_W)$ defines a symmetric $(C^\infty(\overline{W}/G),B\rtimes G)$-chain that restricts to a half-closed $(C^\infty_c(W^\circ/G),B\rtimes G)$-chain. The chain $(p_W(L^2(W,S_B)\rtimes G),p_W(D_{S}^{\rm min}\rtimes G)p_W)$ is closed if $\partial W=\emptyset$.
\end{prop}

\begin{proof}
Using that $p_W$ preserves $\Dom(D_{S}^{\rm min}\rtimes G)$ and that $[D_{S}^{\rm min}\rtimes G,p_W]$ has a bounded extension to a $B\rtimes G$-linear adjointable operator, it follows from \cite{thebeastofmesland} that $D_{S,\mu}^{\rm min}$ is regular and symmetric, and is self-adjoint if $\partial W=\emptyset$. To prove that $(1+(D_{S,\mu}^{\rm min})^*D_{S,\mu}^{\rm min})^{-1}$ is $B\rtimes G$-compact, it suffices to prove that $\Dom(D_{S,\mu}^{\rm min})=p_W\Dom(D_S^{\rm min}\rtimes G)\hookrightarrow p_W(L^2(W,S_B)\rtimes G)$ is $B\rtimes G$-compact. This is immediate as it factors over the $\Dom(D_S^{\rm min}\rtimes G)\xrightarrow{p_W} L^2(W,S_B)\rtimes G$ which is $B\rtimes G$-compact from Proposition \ref{descenofbundlesoddnam} because $p_W\in C_0(\overline{W})\rtimes G$. 

Item 2) follows from that any $a\in C^\infty(\overline{W}/G)$ lifts to a $G$-invariant function $\tilde{a}\in C^\infty_b(\overline{W})$ thus preserving $\Dom(D_S^{\rm min})$, with $[D_S^{\rm min},\tilde{a}]$ bounded, and $\tilde{a}$ commutes with $p_W$. Item 3) follows in a similar way noting that any $j\in C^\infty_c(W^\circ /G)$ lifts to a $G$-invariant function $\tilde{j}\in C^\infty_b(\overline{W})$ supported in a $G$-compact subset of $W^\circ$ and so $j(\Dom(D_S^{\rm min})^*)\subseteq \Dom(D_S^{\rm min})$. 
\end{proof}

We remark that the structure on $(p_W(L^2(W,S_B)\rtimes G),p_W(D_{S}^{\rm min}\rtimes G)p_W)$ as a symmetric $(C^\infty(\overline{W}/G),B\rtimes G)$-chain will later be used to construct $KK$-bordisms between different assembled $G$-equivariant cycles. In light of Proposition \ref{assemblfldoado}, we make the following definition.

\begin{define}
When the $G$-action on $W$ is proper and cocompact, and $S_B\to W$ is as above, we define the symmetric $(C^\infty(\overline{W}/G),B\rtimes G)$-chain
$$\mu_0(L^2(W,S),D_{S}^{\rm min}):=(p_W(L^2(W,S_B)\rtimes G),p_W(D_{S}^{\rm min}\rtimes G)p_W).$$
The \emph{assembled chain} of $(L^2(W,S_B),D_S^{\rm min})$ is defined as the symmetric $(\C,B\rtimes G)$-chain
$$\mu(L^2(W,S),D_{S}^{\rm min}):=\iota^*\mu_0(L^2(W,S),D_{S}^{\rm min}),$$
where $\iota:\C\to C^\infty(\overline{W}/G)$ denotes the inclusion of the unit.
\end{define}

\begin{remark}
The assembled chain is a $(\C,B\rtimes G)$-chain obtained from compressing the descended cycle with $p_W$, i.e. an unbounded Kasparov product. See more in \cite{BCHig,Kas2}.
\end{remark}

\begin{remark}
If $W$ is a $G$-compact manifold (i.e. $\partial W=\emptyset$) then the assembled chain $\mu(L^2(W,S_B)\rtimes G,D_{S}^{\rm min}\rtimes G)$ is a $(\C,B\rtimes G)$-cycle and $\mu_0(L^2(W,S_B)\rtimes G,D_{S}^{\rm min}\rtimes G)$ is a $(C^\infty(W/G),B\rtimes G)$-cycle. Up to $KK$-bordism, the main object of study in this work, the cycles $\mu(L^2(W,S_B)\rtimes G,D_{S}^{\rm min}\rtimes G)$ and $\mu_0(L^2(W,S_B)\rtimes G,D_{S}^{\rm min}\rtimes G)$ are independent of the choices involved and up to $KK$-bordism depend only on the pair $(W,S_B)$ of a manifold $W$ with a proper, cocompact $G$-action and a $G$-equivariant bundle $S_B\to W$ of finitely generated projective $B$-bundles. These results can be found in Subsection \ref{diraconcstarbundleexbordism} below.
\end{remark}

\begin{remark}
We remark that if $G=\Gamma$ is discrete and acts properly and freely on $W$ and $B=\C$, then by \cite{land} we have an isomorphism of $(\C,C^*(\Gamma))$-chains
$$\mu(L^2(W,S),D_{S}^{\rm min})\cong(L^2(W/\Gamma,S\otimes \mathcal{L}_W),D_{S\otimes \mathcal{L}}^{\rm min}+A),$$
where $\mathcal{L}_W\to W/\Gamma$ is the Mishchenko bundle, the right hand side is as in Example \ref{diraconcstarbundleex} and $A$ is an operator defined by a bundle endomorphism.
\end{remark}

\subsubsection{Pulled back chains}
Similarly to Example \ref{diraconcstarbundleex}, for any metric space $X$ with an isometric $G$-action and a closed $G$-invariant subspace $Y\subseteq X$, and an equivariant Lipschitz continuous mapping $f:(W,\partial W)\to (X,Y)$ (i.e. $f(\partial W)\subseteq Y$) we can pull back the action along $f$ and consider the constructions above as symmetric chains over $\mathrm{Lip}_0(X)\rtimes_cG$ and as half-closed chains over $\mathrm{Lip}_0(X\setminus Y)\rtimes_cG$.

\subsection{Foliations}
\label{foliationsex}
Let us take our examples one step further by considering a foliated manifold. Let $M$ be a manifold which we for simplicity assume to be compact. A foliation of $M$ is an integrable subbundle $\mathcal{F}\subseteq TM$. We let $\mathcal{G}_\mathcal{F}\rightrightarrows M$ denote the associated holonomy groupoid. More details can be found in the literature \cite{connesskandalis,hilsumfol,kordy}.

Let $\slashed{D}_\mathcal{F}$ be a formally self-adjoint first order longitudinal elliptic differential operator acting on sections of some vector bundle $S\to M$. We can complete $C^\infty_c(\mathcal{G}_\mathcal{F},s^*S)$ to a $C^*(\mathcal{G}_\mathcal{F})$-Hilbert $C^*$-module $\mathpzc{E}_\mathcal{F}$ and $\slashed{D}_\mathcal{F}$ lifts to a densely defined operator on $\mathpzc{E}_\mathcal{F}$ whose closure we denote by $D_\mathcal{F}$. It is well known that $D_\mathcal{F}$ is a self-adjoint regular operator with $(i\pm D_\mathcal{F})^{-1}\in \mathbb{K}_{C^*(\mathcal{G}_\mathcal{F})}(\mathpzc{E}_\mathcal{F})$. Therefore $(\mathpzc{E}_\mathcal{F},D_\mathcal{F})$ forms a $(C^\infty(M),C^*(\mathcal{G}_\mathcal{F}))$-cycle. 

On the other hand, if $\slashed{T}_\mathcal{F}$ is a formally self-adjoint first order transverally elliptic differential operator acting on sections of some vector bundle $H\to M$ we denote its closure by $T_\mathcal{F}$ on $L^2(M,H)$. It is well known that $T_\mathcal{F}$ is a self-adjoint operator with $a(i\pm T_\mathcal{F})^{-1}\in \mathbb{K}(L^2(M,H))$ for any $a\in C^*(\mathcal{G}_\mathcal{F})$. Therefore $(L^2(M,H),T_\mathcal{F})$ forms a $(C^\infty_c(\mathcal{G}_\mathcal{F}),\C)$-cycle. 

Assume now that $W\subseteq M$ is a smooth $\mathcal{G}_\mathcal{F}$-invariant domain. We write $\mathpzc{E}_{W,\mathcal{F}}$ for the closure of 
$C^\infty_c(W^\circ)\mathpzc{E}_\mathcal{F}$ as a $C^*(\mathcal{G}_\mathcal{F})$-module. Let $D_{W,\mathcal{F}}^{\rm min}$ denote the closure of $$\slashed{D}_\mathcal{F}:C^\infty_c(r^{-1}(W^\circ),s^*S)\to C^\infty_c(r^{-1}(W^\circ),s^*S),$$ as an operator on $\mathpzc{E}_{W,\mathcal{F}}$. There is a continuous inclusion $\Dom(D_{W,\mathcal{F}}^{\rm min})\hookrightarrow \Dom(D_\mathcal{F})$. In particular, if $D_{W,\mathcal{F}}^{\rm min}$ is regular then $(1+(D_{W,\mathcal{F}}^{\rm min})^*D_{W,\mathcal{F}}^{\rm min})^{-1}\in\mathbb{K}_{C^*(\mathcal{G}_\mathcal{F})}(\mathpzc{E}_{W,\mathcal{F}})$ and the pair $(\mathpzc{E}_{W,\mathcal{F}},D_{W,\mathcal{F}}^{\rm min})$ forms a symmetric $(C^\infty(\overline{W}),C^*(\mathcal{G}_\mathcal{F}))$-chain and a half-closed $(C^\infty_c(W^\circ),C^*(\mathcal{G}_\mathcal{F}))$-chain.

In the same spirit, we define $T_{W,\mathcal{F}}^{\rm min}$ as the closure of $$\slashed{T}_\mathcal{F}:C^\infty_c(W^\circ,H)\to C^\infty_c(W^\circ,H),$$ in $L^2(W,H).$ We write $C^\infty_c(\mathcal{G}_\mathcal{F},W):=C^\infty_c(W^\circ)C^\infty_c(\mathcal{G}_\mathcal{F})C^\infty_c(W)$ which is a $*$-subalgebra of $C^*(\mathcal{G}_\mathcal{F})$. Since there is a continuous inclusion $\Dom(T_{W,\mathcal{F}}^{\rm min})\hookrightarrow \Dom(T_\mathcal{F})$, we have that $a(1+(T_{W,\mathcal{F}}^{\rm min})^*T_{W,\mathcal{F}}^{\rm min})^{-1}\in\mathbb{K}(L^2(W,H))$ for any $a\in C^\infty_c(\mathcal{G}_\mathcal{F},W)$. The pair $(L^2(W,H),T_{W,\mathcal{F}}^{\rm min})$ forms a half-closed $(C^\infty_c(\mathcal{G}_\mathcal{F},W),\C)$-chain.

\subsection{Wrong way maps in unbounded KK}
\label{kkwrong}

Connes-Skandalis \cite{connesskandalis} studied wrong way maps in $KK$-theory, later tied closer together with $KK$-theory in Emerson-Meyer's work on correspondences \cite{emersonmeyer}. Related questions in unbounded $KK$-theory have been studied in detail in the last years \cite{kaadfactor,kaadhalf,kaadfactor2}. The wrong way map associated with a $K$-oriented smooth map $f:X\to Y$ of smooth manifolds is a class $[f_!]\in KK_*(C_0(X),C_0(Y))$. Recall that a $K$-orientation of a smooth map $f$ is a choice of spin$^c$-structure on $T_f=T^*X\oplus f^*T^*Y$. A direct way to construct $[f_!]$ is to use the spin$^c$-structure on $T_f$ to construct a Bott like element in $K^*_X(T^*X\times Y)$ and apply Poincaré duality $K^*_X(T^*X\times Y)\to KK_*(C_0(X),C_0(Y))$. Wrong way maps were proven in \cite{connesskandalis} to be functorial in the sense that if $f:X\to Y$ and $g:Y\to Z$ are $K$-oriented, then 
\begin{equation}
\label{alkdnadln}
[(g\circ f)_!]=[f_!]\otimes_{C_0(Y)}[g_!].
\end{equation}
We can describe the class  $[f_!]$ via an unbounded cycle: when returning to this construction below in Subsection \ref{kkwrongbor}, we shall see the unbounded cycle is canonical up to $KK$-bordism and for an appropriate unbounded Kasparov product fulfills \eqref{alkdnadln} up to bordism, 

Consider a $K$-oriented smooth map $f:X\to Y$. We can always factor $f$ as 
\begin{equation}
\label{adlanladnfactor}
f=\pi\circ \iota\circ 0,
\end{equation}
where 
\begin{enumerate}
\item[i)] $0:X\to N$ is the zero section of a spin$^c$-vector bundle $N\to X$ of even rank;
\item[ii)] $\iota:N\to Z$ is an inclusion as an open subset into a smooth manifold $Z$;
\item[iii)] $\pi:Z\to Y$ is a fibre bundle with compact fibres and a fibrewise spin$^c$-structure.
\end{enumerate}
Indeed, the maps are constructed as follows. We can choose an embedding $j:X\to Z_0$ into a compact spin$^c$-manifold $Z_0$. Set $Z:=Z_0\times Y$ and take $N$ as the normal bundle of the embedding $j\times f:X\to Z$. Without loss of generality, we can assume that $N$ has even rank. The inclusion $\iota:N\to Z$ is as a tubular neighborhood of $\mathrm{im}(j\times f)$. Finally, $\pi:Z\to Y$ is defined as the projection onto $Y$. By functoriality of wrong way maps, we have that 
$$[f_!]=[0_!]\otimes_{C_0(N)}[\iota_!]\otimes_{C_0(Z)}[\pi_!].$$
We approach the construction of an unbounded representative of $[f_!]$ via the factorization \eqref{adlanladnfactor}, reducing the construction to maps from the cases i-iii and to forming an unbounded Kasparov product.\\

\subsubsection{Case i: wrong way maps from zero sections}

Consider an even rank spin$^c$-vector bundle $p:N\to X$, and as above $0:X\to N$ will denote the zero section viewed as a smooth $K$-oriented map of manifolds. The class $[0_!]\in KK_0(C_0(X),C_0(N))$ is the well studied Thom isomorphism associated with the spin$^c$-structure on $N$ \cite{careywang,connesskandalis,emersonmeyer}. We fix a metric on the vector bundle $N\to X$. Let $S_N\to X$ denote the complex spinor bundle for defining the spin$^c$-structure of $N$, i.e. $S$ is a choice of Clifford bundle for $N$ defining a graded irreducible bundle of modules for the Clifford algebra bundle $\C\ell(N)\to X$. In particular, by direct comparison to the constructions in \cite{careywang,connesskandalis}, the class $[0_!]\in KK_0(C_0(X),C_0(N))$ has an unbounded representative given by the  $(C_0(X),C_0(N))$-cycle $(C_0(N,p^*S_N),c)$ where $c$ is the self-adjoint and regular unbounded multiplier of $C_0(N,p^*\C\ell(N))$ defined from Clifford multiplication by the fibrewise coordinate. In other words, $c$ is defined from its action on the core $C_c(N,p^*\C\ell(N))$ by $[c\xi](x,v):=v.\xi(x,v)$ for $\xi\in C_c(N,p^*\C\ell(N))$ and where the $.$ denotes Clifford multiplication. 

The $(C_0(X),C_0(N))$-cycle $(C_0(N,p^*S_N),c)$ depends only on the choice of metric on $N$ and the choice of complex spinor bundle $S_N$. But in fact, a change of metric on $N$ or of complex spinor bundle $S_N$ (as long as it defines the same spin$^c$-structure) does not change the unitary equivalence class of $(C_0(N,p^*S_N),c)$. However, the choices of spin$^c$-structures on $N$ form an $H^2(X,\Z)$-torsor and so does the set of unitary equivalence classes of possible unbounded representatives $(C_0(N,p^*S_N),c)$ as $S_N$ varies over possible spin$^c$-structures.\\

\subsubsection{Case ii: wrong way maps from open inclusions}

If $\iota:N\to Z$ is an inclusion as an open subset into a smooth manifold $Z$, the corresponding wrong way map $[\iota_!]$ is well known to be represented by the unbounded $(C_0(N),C_0(Z))$-cycle $(C_0(N)_{C_0(Z)},0)$ where $C_0(N)_{C_0(Z)}$ denotes the space $C_0(N)$ is viewed as a right $C_0(Z)$-Hilbert $C^*$-module via the ideal inclusion $C_0(N)\hookrightarrow C_0(Z)$ induced by $\iota$. For future reference, we remark that despite the unbounded cycle representing $[\iota_!]$ is easy to describe it is precisely this factor in $f=\pi\circ \iota\circ 0$ that creates most analytic problems since $\iota(N)\subseteq Z$ is in general not a complete submanifold and as such is delicate to treat in unbounded $KK$-theory \cite{kaadhalf,mesren}. \\

\subsubsection{Case iii: wrong way maps from spin$^c$-fibre bundles}

Let us turn to considering a fibre bundle $\pi:Z\to Y$ with compact fibres and a fibrewise spin$^c$-structure. This case has been studied in unbounded $KK$-theory in \cite{kaadfactor}. We choose a fibrewise Riemannian structure and write $S_Z\to Z$ for the fibrewise spin$^c$-structure. That is, $S_Z$ is an irreducible Clifford bundle for the bundle of vertical vectors $T_v:=\ker(D\pi:TZ\to TY)$ which we assume to be graded if $T_v$ has even rank. We can complete $C^\infty_c(Z;S_Z)$ into a $C_0(Y)$-Hilbert module $L^2(Z/Y;S_Z)$ through the $C_0(Y)$-valued inner product 
\begin{equation}
\label{innerprodforz}
\langle f_1,f_2\rangle_{L^2(Z/Y;S_Z)}(y):=\int_{\pi^{-1}(y)} \langle f_1,f_2\rangle_{S_Z}\mathrm{d} \rho,
\end{equation}
where $\langle\cdot,\cdot\rangle_{S_Z}$ denotes the Hermitian structure on $S_Z$ and $\mathrm{d}  \rho$ the fibrewise volume density induced from the fibrewise Riemannian structure. We can choose a fibrewise Dirac operator $\slashed{D}_Z:C^\infty_c(Z;S_Z)\to C^\infty_c(Z;S_Z)$, which is odd if the fibres have even dimension. The fibrewise Dirac operator is $C_c^\infty(Y)$-linear and extends to a self-adjoint, regular unbounded $C_0(Y)$-linear operator $D_Z$ on $L^2(Z/Y;S_Z)$, with core $C^\infty_c(Z;S_Z)$. In fact, the case at hand is a special case of Subsection \ref{foliationsex} since $D_Z$ is a longitudinal Dirac operator on the manifold $Z$ foliated by the fibres $\{\pi^{-1}(y):\; y\in Y\}$. We see that the collection $(L^2(Z/Y;S_Z),D_Z)$ is a $(C^\infty_c(Z),C_0(Y))$-cycle for representing $[\pi_!]$, see \cite{connesskandalis}.\\

\subsubsection{Composing together to $f$}

We shall now construct an unbounded representative for the Kasparov product of the cycles  $(C_0(N,p^*S_N),c)$, $(C_0(N)_{C_0(Z)},0)$ and $(L^2(Z/Y;S_Z),D_Z)$ arising from the factorization $f=\pi\circ\iota\circ 0$ in Equation \eqref{adlanladnfactor}. Even if it is theoretically possible to do the unbounded Kasparov product constructively via choices of connections, we will build the product by hand for two reasons: firstly, we want to avoid the technical burdens involved in the formalism for constructive unbounded Kasparov products, and secondly, we are in fact only interested in finding a good unbounded representative for $[f_!]$. 

We first note that the Kasparov product of $(C_0(N,p^*S_N),c)$, $(C_0(N)_{C_0(Z)},0)$ can be computed directly in the constructive unbounded Kasparov product as the unbounded $(C_0(X),C_0(Z))$-cycle $(C_0(N,p^*S_N)_{C_0(Z)},c)$ where $C_0(N,p^*S_N)_{C_0(Z)}$ denotes $C_0(N,p^*S_N)$ viewed as a $C_0(Z)$-Hilbert $C^*$-module via the ideal inclusion $C_0(N)\subseteq C_0(Z)$ defined from the open embedding $\iota$.

In light of this observation, we define the $C_0(Y)$-Hilbert $C^*$-module $\mathpzc{E}_f$ as the closure of $C_c^\infty(N,p^*S_N\otimes S_Z|_{\iota(N)})$ in the $C_0(N)$-valued inner product defined analogously to \eqref{innerprodforz}. The $C^\infty_c(X)$-action on $C_c^\infty(N,p^*S_N\otimes S_Z|_{\iota(N)})$ extends by continuity to a left $C_0(Y)$-linear adjointable action of $C_0(X)$ on $\mathpzc{E}_f$. Here we note that $\mathpzc{E}_f$ depends not only on $f$, but on the choice of factorization \eqref{adlanladnfactor}. The natural surjection from the projective tensor product
\begin{equation}
\label{prodmap}
C^\infty_c(N,p^*S_N)\tilde{\otimes} C^\infty_c(Z,S_Z)\to C_c^\infty(N,p^*S_N\otimes S_Z|_{\iota(N)}),
\end{equation}
defined from interior tensor product of vector bundles on $N$ extends by continuity to an isomorphism of $(C_0(X),C_0(Y))$-Hilbert bimodule 
$$C_0(N,p^*S_N)\otimes_{C_0(N)}L^2(Z/Y,S_Z)\to \mathpzc{E}_f.$$

\begin{lemma}
\label{adlknalknad}
Given a Clifford connection $\nabla$ on $S_N$ and the data prescribed above, there is a uniquely determined, self-adjoint, regular, $C_0(Y)$-linear unbounded operator $D_f$ on $\mathpzc{E}_f$ with core $C_c^\infty(N,p^*S_N\otimes S_Z|_{\iota(N)})$ such that 
$$D_f\xi=(1_{p^*S_N}\otimes_\nabla D_Z)\xi+(c\otimes 1_{S_Z})\xi,$$
for $\xi\in C_c^\infty(N,p^*S_N\otimes S_Z|_{\iota(N)})$. Here we are identifying $(1_{p^*S_N}\otimes_\nabla D_Z)$ via the map \eqref{prodmap} with the differential operator on $C_c^\infty(N,p^*S_N\otimes S_Z|_{\iota(N)})$ obtained from restricting $D$ to an operator on $C^\infty_c(N,S_Z|_{\iota(N)})$ and twisting by the Clifford connection on $S_N$.
\end{lemma}

\begin{proof}
We write $\slashed{D}_f=1_{p^*S_N}\otimes_\nabla D_Z+c\otimes 1_{S_Z}$. We here interpret the Clifford connection $\nabla$ on $S_N$ as lifted to $p^*S_N$. We write $D_{\iota(N),0}$ for the closure of $1_{p^*S_N}\otimes_\nabla D_Z$ acting on $C_c^\infty(N,p^*S_N\otimes S_Z|_{\iota(N)})$. Arguing as in  \cite{kucerovskycallias}, we see that $\slashed{D}_f^2-(1_{p^*S_N}\otimes_\nabla D_Z)^2-c^2\otimes 1_{S_Z}$ is bounded from below, so $D_f$ is the sum 
$$D_f=D_{\iota(N),0}+\overline{c\otimes 1_{S_Z}}.$$
Similarly, after carrying out a mollifier argument, we see that 
$$D_f^*=D_{\iota(N),0}^*+\overline{c\otimes 1_{S_Z}}=D_{\iota(N),0}+\overline{c\otimes 1_{S_Z}}.$$
\end{proof}

Given a Clifford connection $\nabla$ on $S_N$ and a vertical connection $\nabla'$ on $S_Z$, we can construct a vertical Clifford connection $\tilde{\nabla}$ on $p^*S_N\otimes S_Z|_{\iota(N)}\to \iota(N)$. We can define the submodule $\mathpzc{E}_f^1\subseteq \mathpzc{E}_f$ as the closure of $C_c^\infty(N,p^*S_N\otimes S_Z|_{\iota(N)})$ in the inner product 
$$\langle f_1,f_2\rangle_{\mathpzc{E}_f^1}:=\langle f_1,f_2\rangle_{\mathpzc{E}_f}+\langle \tilde{\nabla}f_1, \tilde{\nabla}f_2\rangle_{L^2(Z/Y;p^*S_N\otimes S_Z\otimes T_v)}.$$
Standard arguments show that $\mathpzc{E}_f^1$ does not depend on the choices of metrics and connections. The proof of Lemma \ref{adlknalknad} and elliptic regularity in the longitudinal pseudodifferential calculus on $\pi:Z\to Y$ implies the following characterization of the domain of $D_f$.

\begin{prop}
The domain of $D_f$ is given by 
$$\mathrm{Dom}(D_f)=\mathpzc{E}_f^1\cap (i+c\otimes 1_{S_Z})^{-1}\mathpzc{E}_f.$$
\end{prop}

We now come to the main descriptive result for our unbounded representative of the wrong way map.

\begin{theorem}
\label{wrongwayconst}
The pair $(\mathpzc{E}_f,D_f)$ forms a $(C^\infty_c(X),C_0(Y))$-cycle representing $[f_!]\in KK_*(C_0(X),C_0(Y))$.
\end{theorem}

\begin{proof}
We have that $[f_!]=[0_!]\otimes_{C_0(N)}[\iota_!]\otimes_{C_0(Z)}[\pi_!]$, and by the discussion above $[0_!]\otimes_{C_0(N)}[\iota_!]$ is represented by the $(C_0(X),C_0(Z))$-cycle $(C_0(N,p^*S_N)_{C_0(Z)},c)$. Therefore, it suffices to show that $[(\mathpzc{E}_f,D_f)]$ is the Kasparov product of the class $[(C_0(N,p^*S_N)_{C_0(Z)},c)]$ with the class $[(L^2(Z/Y;S_Z),D_Z)]$. This fact follows using Kucerovsky's theorem \cite[Theorem 13]{kucerovskyprod} applied in the same way as in \cite{kucerovskycallias}.

\end{proof}

\subsection{Groupoid quasi-cocycles}
\label{groupoidsex}
We have already seen a variety of cycles on groupoids above, such as transformation groupoids in Subsection \ref{groupactionsex} or holonomy groupoids of foliations in Subsection \ref{foliationsex}, and we shall see more of them below in Subsection \ref{fellalgebraex} on Fell algebras. Let us consider a slightly more exotic construction, extending constructions in \cite{bournemesland,GRU,meslandcocycle}. Let $\mathcal{G}$ denote a topological groupoid with object space $X$. 

\begin{define}
Consider a continuous field of real Hilbert spaces $\mathfrak{H}$ over the unit space of $\mathcal{G}$, with a strongly continuous unitary $\mathcal{G}$-action. Let $c:\mathcal{G}\to \mathfrak{H}$ be a continuous mapping such that $c(g)\in \mathfrak{H}_{r(g)}$. 
\begin{itemize}
\item We say that $c$ is a cocycle if for any composable $g,h\in \mathcal{G}$ we have that 
$$c(gh)=c(g)+gc(h).$$
If $c$ is a cocycle, we define its kernel $\ker(c)\subseteq \mathcal{G}$ as the closed subgroupoid $\{g\in \mathcal{G}:c(g)=0\in \mathfrak{H}_{r(g)}\}$.
\item We say that $c$ is a quasi-cocycle if for any $g\in \mathcal{G}$, the continuous function 
$$r^{-1}(\mathcal{G})\ni h\mapsto c(gh)-gc(h)\in \mathfrak{H}_{r(g)},$$
is norm bounded. 
\item We say that $c$ is proper if for any compact $K\subseteq X$ and any bounded set $B\subseteq \mathfrak{H}|_K$, the set $c^{-1}(B)\subseteq \mathcal{G}$ is pre-compact. 
\end{itemize}
\end{define}

\begin{remark}
We note that there is a one-to-one correspondence between cocycles and affine isometric actions. To see this, suppose that $\mathfrak{H}$ is a continuous field of Hilbert spaces with a strongly continuous unitary $\mathcal{G}$-action. Then a cocycle $c$ and an affine isometric action $\pi_c$ mutually determine each other from the rule that $c(g)=\pi_c(g).0$. 
\end{remark}

\begin{remark}
A groupoid is said to be a-$T$-menable if it admits a proper cocycle on a continuous field of Hilbert spaces $\mathfrak{H}$. This notion was used in \cite{tu99a}. In fact, by \cite{tu99a} any amenable groupoid admits a proper cocycle on a continuous field of Hilbert spaces $\mathfrak{H}$. A priori, $\mathfrak{H}$ is only a field of \emph{real} Hilbert spaces.
\end{remark}

Let us mention a few examples: 
\begin{enumerate}
\item If $\mathfrak{H}=X\times \mathcal{H}$, for a Hilbert space $\mathcal{H}$ with the trivial action, a cocycle is just a groupoid cocycle $\mathcal{G}\to \mathcal{H}$. Unbounded Kasparov cycles for this context was considered for $\mathcal{H}=\R$ in \cite{meslandcocycle} and for a finite-dimensional $\mathcal{H}$ in \cite{bournemesland}.
\item If $\mathcal{G}=\Gamma$ is a discrete group, $\mathfrak{H}$ is a $\Gamma$-Hilbert space and a (quasi-)cocycle is simply a group (quasi-)cocycle $c:\Gamma\to \mathfrak{H}$. Unbounded Kasparov cycles were studied in this context in \cite{GRU}. By definition, a group admits a proper cocycle if and only if it is a-T-menable (e.g. if it is amenable) and a proper quasi-cocycle if and only if it is a-TT-menable (e.g. if it is hyperbolic). Only virtually abelian groups admit proper quasi-cocycles with values in a finite-dimensional Hilbert space, so in general it is required that $\mathcal{H}$ is infinite-dimensional. 
\end{enumerate}

Assume that $S_\mathfrak{H}$ is a continuous field of complex Hilbert spaces over the unit space of $\mathcal{G}$, with a strongly continuous unitary $\mathcal{G}$-action, and that there is a $\mathcal{G}$-equivariant linear mapping $\mathfrak{c}:\mathfrak{H}\to \mathbb{B}(S_{\mathfrak{H}})$ such that 
\begin{itemize}
\item $\mathfrak{c}(\xi)\in \mathbb{B}(S_{\mathfrak{H},x})$ for $\xi\in \mathfrak{H}_x$;
\item $\mathfrak{c}(\xi)^*=\mathfrak{c}(\xi)$ and $\mathfrak{c}(\xi)^2=\|\xi\|^2$;
\item for any continuous section $\xi$ of $\mathfrak{H}$, $\mathfrak{c}$ is a strictly continuous section.
\end{itemize}
Let $X$ denote the unit space of $\mathcal{G}$. Let $B$ denote the $C^*$-subalgebra of the multiplier algebra of the norm continuous sections of $\mathbb{K}(S_{\mathfrak{H}})$ generated by the set 
$$\{\mathfrak{c}(\xi):\; \xi\;\mbox{a continuous global section of $\mathfrak{H}$ vanishing at $\infty$}\}.$$
The $C^*$-algebra $B$ carries a strongly continuous action of $\mathcal{G}$. As a $C_0(X)$-algebra, the fibre of $B$ in each point is the CAR-algebra. 

We define $L^2(\mathcal{G};B)$ as the $B$-Hilbert $C^*$-module completion of the space of continuous compact supported functions $f:\mathcal{G}\to B$ such that $f(g)\in B_{r(g)}$ in the $B$-valued inner product 
$$\langle f_1,f_2\rangle_{L^2(\mathcal{G};B)}(x)=\int_{g\in r^{-1}(x)} f_1(g)^*f_2(g).$$
The groupoid action on $B$ defines a $\mathcal{G}$-action on $L^2(\mathcal{G};B)$ by $(g.f)(h)=g.(f(g^{-1}h))$. There is a left action of $C^*(\mathcal{G})$ on $L^2(\mathcal{G};B)$ obtained by integrating the action $\rho(g)f(h)=f(hg)$. 

\begin{prop}
Let $c$ be a proper quasi-cocycle on a groupoid $\mathcal{G}$, and $B$ and $L^2(\mathcal{G},B)$ constructed as above. Define the operator $D_c$ densely on $L^2(\mathcal{G},B)$ as 
$$D_cf(g)=\mathfrak{c}(c(g))f(g),$$
on 
$$\Dom(D_c):=\{f\in L^2(\mathcal{G},B): \int_{g\in s^{-1}(x)} \|c(g)\|_{r(g)}^2f(g^{-1})^*f(g)\; \mbox{converges in $B$}\}.$$
Then $D_c$ is a regular and self-adjoint operator fitting into an unbounded $(C_c(\mathcal{G}),B)$-Kasparov module $(L^2(\mathcal{G},B),D_c)$.
\end{prop}

\begin{proof}
It is clear that $D_c$ is a regular and self-adjoint operator from the fact that $(i\pm D_c)^{-1}f(g)=(i\pm \mathfrak{c}(c(g)))^{-1}f(g)$. Since $c$ is proper, $(i\pm D_c)^{-1}$ is defined from pointwise multiplication by a function in $C_0(\mathcal{G},B)$ and so $a(i\pm D_c)^{-1}$ is $B$-compact for any $a\in C_c(\mathcal{G})$ (acting as a convolution operator). The proof that $C_c(\mathcal{G})$ preserves $\Dom(D_c)$ and commutes with $D_c$ up to a bounded operator goes as in \cite{GRU}.
\end{proof}

\begin{remark}
It is in general not clear what the class of this cycle is in $KK$. When there is a complex structure on $\mathfrak{H}$, the constructions above relate to the constructions of $\gamma$-elements \cite{HigKasp,tu99a} for proving the Baum-Connes conjecture for a-T-menable group(oid)s.
\end{remark}

\subsection{Continuous trace algebras}
\label{conttracealgexex}

Continuous trace algebras are mildly noncommutative $C^*$-algebras. A continuous trace algebra is a $C^*$-algebra $A$ with Hausdorff spectrum such that for any representation $\pi$ there is a neighborhood $U$ of $[\pi]\in \hat{A}$ and an element $x\in A$ such that $\pi'(x)$ is a rank one projection for all $[\pi']\in U$. A continuous trace algebra always arise as the space of global sections of a bundle of compact operators on some Hilbert space over the spectrum and is up to stable isomorphism characterized by its spectrum and its Dixmier-Douady invariant, an element in the third integral $\mathrm{\check{C}}$ech cohomology of the spectrum. There are several techniques to treat bivariant $K$-theory of continuous trace algebras. We shall focus on a simple example arising from $T$-duality. In general $T$-duality produces a dual continuous trace algebra to a continuous trace algebra whose spectrum is a principal circle bundle, for a detailed overview see \cite{rosenbergconttrabook}. 

Consider a principal circle bundle $\pi:Z\to X$ where we assume $Z$ to be a closed manifold. The $U(1)$-action on $Z$ induces an $\R$-action and we consider the algebra $A:=C(Z)\rtimes \R$. We write $\mathcal{A}:=\mathcal{S}(\R,C^\infty(Z))\subseteq A$ for the dense $*$-subalgebra of smooth Schwarz functions. We can identify $C^\infty(Z)\cong \oplus_{k\in \Z}^\mathcal{S} C^\infty(X,L_k)$, where $L_k:=Z\times_{\chi_k}\C$ for the character $\chi_k(z):=z^k$ of $U(1)$ and $\oplus_{k\in \Z}^\mathcal{S}$ denotes that we allow for all sequences $(a_k)_{k\in \Z}$ with $\sum_k k^N\|a_k\|_{C^l(X,L_k)}<\infty$ for all $N$ and $l$. The $*$-algebra $\mathcal{A}$ can similarly be identified with $\mathcal{A}\cong \oplus_{k\in \Z}^\mathcal{S} \mathcal{S}(X\times \R,L_k)$ where $\oplus_{k\in \Z}^\mathcal{S}$ denotes that we allow for all sequences $(a_k)_{k\in \Z}$ with $\sum_k k^N\|(1+t^2)^Ma_k\|_{C^l(X\times \R,L_k)}<\infty$ for all $N$, $M$ and $l$.

We have that $A$ is a continuous trace algebra with spectrum $X\times S^1$ and its Dixmier-Douady invariant $\delta_{DD}(A)\in \check{H}^3(X\times S^1,\Z)$ is uniquely determined by the property that its push down to $X$ coincides with the first Chern class $c_1(Z)\in \check{H}^2(X,\Z)$ of $Z$. Indeed, given an $x\in X$ and a $z=\mathrm{e}^{i\theta}\in S^1$, there is a representation $\lambda_{(x,z)}:A\to \mathbb{K}(L^2[0,2\pi])$ (unique up to unitary isomorphism) defined on a sequence $(a_k)_{k\in \Z}\in \mathcal{A}\cong \oplus_{k\in \Z}^\mathcal{S} \mathcal{S}(X\times \R,L_k)$ and $f\in L^2[0,2\pi]$ as 
$$\lambda_{(x,z)}f(u):=\frac{1}{2\pi}\sum_{k,l\in \Z}\int_0^{2\pi} a_l(x,u-k-t)\mathrm{e}^{ik\theta+ilt}f(t)\mathrm{d}t$$

Let $Z$ be equipped with a $U(1)$-equivariant Riemannian metric, and consider a $U(1)$-equivariant Clifford bundle $S\to Z$. We pick a $U(1)$-equivariant Dirac operator $D$ on $S$ equipped with its maximal domain $H^1(Z,S)$ making $D$ self-adjoint. Then $(L^2(Z,S),D)$ forms a cycle for $(C^\infty(Z),\C)$. We can also consider $(L^2(Z,S)\hat{\boxtimes} L^2(\R),D\hat{\boxtimes} \partial)$ which is clearly an $(\mathcal{A},\C)$-cycle when $A=C(Z)\rtimes \R$ acts on $L^2(Z,S)\hat{\boxtimes} L^2(\R)$ via the ordinary $C(Z)$-action and $\R$ acts diagonally with the dual action on $L^2(\R)$. In fact, the two are related in $K$-homology by that 
$$[(L^2(Z,S),D)]\mapsto [(L^2(Z,S)\hat{\boxtimes} L^2(\R),D\hat{\boxtimes} \partial)],$$
under the Connes-Thom isomorphism $K^*(C(Z))\cong K^{*+1}(C(Z)\rtimes \R)=K^{*+1}(A)$.

\subsection{Fell algebras}
\label{fellalgebraex}
Fell algebras are one step further than continuous trace algebras away from commutative spaces. A Fell algebra is a $C^*$-algebra $A$ such that for any representation $\pi$ there is a neighborhood $U$ of $[\pi]\in \hat{A}$ and an element $x\in A$ such that $\pi'(x)$ is a rank one projection for all $[\pi']\in U$. The definition of a Fell algebra is formally the same as that of a continuous trace algebra with the Hausdorff condition on $\hat{A}$ dropped. Being the spectrum of a $C^*$-algebra, the spectrum is only mildly non-Hausdorff, it is always locally locally compact and locally Hausdorff (see \cite{CHR, HKS}). As in the case of continuous trace algebras, Fell algebras are characterized up to stable isomorphism by their spectrum and a Dixmier-Douady type invariant (see \cite{CHR, HKS}). The seemingly innocent transition from Hausdorff spectrum to non-Hausdorff spectrum opens up a quite much larger class of examples, see for instance \cite{DGY}. An instance where Fell algebras have played an important role in index theory is in \cite{KasparovSkandalis} where they were used to describe the $\gamma$-element for a group acting on a building, e.g. a discrete subgroup of a linear adelic group.

Let us consider a simple example of a Fell algebra. Let $X$ be a compact locally Hausdorff space admitting a local homeomorphism $\psi:M\to X$ from a smooth manifold $M$. Existence of $\psi$ ensures that $X$ is a smooth compact non-Hausdorff manifold. The equivalence relation $R(\psi):=\{(x,y)\in M\times M: \,\psi(x)=\psi(y)\}$ is an étale groupoid over $M$, and $A_\psi:=C^*(R(\psi))$ is a Fell algebra with spectrum $X$ (see \cite{CHR}). We also set $\mathcal{A}_\psi:=C^\infty_c(R(\psi))$ viewed as a dense $*$-subalgbra of $A_\psi$.

We assume that $M$ admits an $R(\psi)$-equivariant Riemannian metric and that $S\to M$ is an $R(\psi)$-equivariant Clifford bundle. The reader should note that these choices correspond to picking a Riemannian metric and a Clifford bundle on $X$. Riemannian metrics need not exist on non-Hausdorff manifolds (see for example \cite{KasparovSkandalis}). However, if a Riemannian metric exists then Clifford bundles can always be constructed. From this data, and a choice of connection, we can construct a Dirac operator $\slashed{D}_S$ whose principal symbol is $R(\psi)$-equivariant that extends to a self-adjoint operator $D_S$ in $L^2(M,S)$. We arrive at a cycle $(L^2(M,S),D_S)$ for $(\mathcal{A}_\psi,\C)$. 

We can make the construction slightly more explicit. Assume that $X$ is a smooth compact non-Hausdorff manifold. Pick a covering $(U_j)_{j=1}^N$ of $X$ consisting of Hausdorff charts on $X$. We can define $M:=\dot{\cup}_{j=1}^N U_j$ and $\psi:M\to X$ by $\psi|_{U_j}=\id_{U_j}$. If $X$ admits a Riemannian metric it induces an $R(\psi)$-equivariant metric by simply restricting to each $U_j$. We can pick an $R(\psi)$-equivariant Clifford bundle $S\to M$ and a choice of Clifford connection on $X$ induces a Dirac operator $\slashed{D}_S$ whose principal symbol is $R(\psi)$-equivariant. By equipping $\slashed{D}_S$ with suitable self-adjoint boundary conditions on $\partial M$ (e.g. APS-conditions), we arrive at a self-adjoint operator $D_S$ on $L^2(M,S)$. We note that in this context, we have an embedding of $C^*$-algebras
$$A_\psi\hookrightarrow M_N(C_0(M)), $$
that we denote by $\psi^*$ and is given by $\psi^*(f):=fe_{ij}$ for $f\in C_0(R(\psi)\cap (U_i\times U_j))=C_0(U_i\cap U_j)$, where $e_{ij}$ is the matrix unit with a $1$ in the $(i,j)$-th place. Since $C^*(R(\psi))=\sum_{i,j}C_0(R(\psi)\cap (U_i\times U_j))$ the map $\psi^*$ is uniquely determined from this property. In particular, for any Clifford bundle $S_0\to M$ we can ensure that $S:=\C^N\otimes S_0$ is $R(\psi)$-equivariant. If we take a Clifford bundle $S_0$ on $M$, and construct $D_S$ as $D_{S_0}\otimes 1_{\C^N}$ we have the following equality of $(\mathcal{A}_\psi,\C)$-cycles
$$(L^2(M,S),D_S)=\psi^*\left[((L^2(M,S),D_{S_0})\otimes \C^N\right].$$

We can also construct examples of symmetric and half-closed chains on a smooth compact non-Hausdorff manifold with boundary $X$. Assume that there is a local homeomorphism $\psi:\overline{W}\to \overline{X}$ from a smooth compact manifold with boundary $W$. There is an induced mapping $\psi_\partial:\partial W\to \partial X$, and since $\psi$ is a local homeomorphism $\psi^\circ:W^\circ \to X^\circ$ is proper. We arrive at a short exact sequence of Fell algebras
$$0\to A_{\psi^\circ}\to A_\psi\to A_{\psi_\partial}\to 0.$$
As above, we assume that $W$ admits an $R(\psi)$-equivariant Riemannian metric and that $S\to M$ is an $R(\psi)$-equivariant Clifford bundle. We construct a Dirac operator $\slashed{D}_S$ whose principal symbol is $R(\psi)$-equivariant. The closure of $\slashed{D}_S$ is a symmetric operator $D_S$ in $L^2(W,S)$ with $\Dom(D_S)=H^1_0(W^\circ,S)$. The pair $(L^2(W,S),D_S)$ forms a symmetric $(\mathcal{A}_\psi,\C)$-chain and a half-closed $(\mathcal{A}_{\psi^\circ},\C)$-chain.

\subsection{Cuntz-Krieger algebras}
\label{ckex}
Cuntz-Krieger algebras form a well studied class of $C^*$-algebras. They are often purely infinite making their noncommutative geometry challenging. In \cite{GMCK}, a class of spectral triples on Cuntz-Krieger algebras was constructed exhausting their $K$-homology groups. The starting point for this construction was an unbounded Kasparov cycle over a maximal abelian subalgebras, and the spectral triples came out from localizing in points of the spectrum of the maximal abelian subalgebra. This idea was later extended to Cuntz-Pimsner algebras in \cite{GMR} and to the stable Ruelle algebra of a Wieler solenoid in \cite{DGMW}. We consider a special case of the construction for Cuntz-Pimsner algebras in Example \ref{cpalgebronmanfdodoex}. For the sake of exemplifying cycles we explain the construction of spectral triples from \cite{GMCK,GMR} in the special case of the Cuntz algebra $O_N$, with a particular focus on homotopy relations. The spectral triples considered here were in a mildly modified form studied in a more geometric context in \cite{gmon} for $O_N$ and later in general in \cite{gergofmes}. 

We now consider a noncommutative example of cycles on the $C^*$-algebra $O_N$. This is the universal $C^*$-algebra generated by $N$ isometries with orthogonal ranges. In other words, $O_N$ is generated by elements $S_1,\ldots,S_N$ satisfying $1=S_i^*S_i=\sum_{j=1}^N S_jS_j^*$ for any $i$. We let $O_N^\infty\subseteq O_N$ denote the $*$-algebra generated by the elements $S_1,\ldots,S_N$. There is a unique KMS-state $\Psi\in O_N^*$ of inverse temperature $\log(N)$ and we let $L^2(O_N)$ denote the associated GNS-representation of $O_N$. By \cite[Lemma 2.13]{GMR}, there is an orthogonal decomposition 
$$L^2(O_N)=\bigoplus_{n\in \Z}\bigoplus_{k=\max(0,-n)} \mathcal{H}_{n,k},$$ 
where $\mathcal{H}_{n,k}$ are spaces of dimension at most $N^{n+2k}$. 

\begin{remark}
\label{functorforhnk}
The construction of $\mathcal{H}_{n,k}$ is independent of basis. In fact $\mathcal{H}_{n,k}$ defines a covariant functor on the category of finite-dimensional vector spaces with morphisms being unitary isomorphisms. In particular, given a preferred basis of $\C^N$ such as the standard basis, there is a preferred ON-basis for $\mathcal{H}_{n,k}$. 
\end{remark}

Consider the function 
\begin{equation}
\label{definedodepsi} 
\psi:\{(n,k)\in \Z\times \N: k+n\geq 0\}\to \Z, \quad 
\psi(n,k):=
\begin{cases} 
n, \; &k=0\\
-|n|-k,\; &k>0. 
\end{cases}
\end{equation}
There is an odd $(O_N^\infty,\C)$-cycle $(L^2(O_N),D)$ defined from declaring $D$ diagonal in the decomposition $L^2(O_N)=\bigoplus_{n\in \Z}\bigoplus_{k=\max(0,-n)} \mathcal{H}_{n,k}$ and setting 
$$D|_{\mathcal{H}_{n,k}}=\psi(n,k).$$ 
The associated $K$-homology class $[(L^2(O_N),D)]\in KK^1(O_N,\C)\cong \Z/(N-1)\Z$ is a generator.

Let $\lambda:O_N\to \mathbb{B}(L^2(O_N))$ be the representation $a\mapsto \sum_{j=1}^N S_j aS_j^*$ and let $_\lambda L^2(O_N)$ denote $L^2(O_N)$ viewed as an $(O_N,\C)$-Hilbert $C^*$-module for the left action by $\lambda$. We claim that as $(O_N^\infty,\C)$-cycles, there is an isomorphism 
$$N(L^2(O_N),D)\cong (_\lambda L^2(O_N),D+B),$$ 
for a bounded operator $B\in \mathbb{B}(L^2(O_N))=\End_\C^*(L^2(O_N))$. It should be noted that this observation does not depend on the particular spectral triple at hand, but rather on the relations in the algebra $O_N$. Indeed, define $\alpha:L^2(O_N)\otimes \C^N\to L^2(O_N)$ by $\alpha(x_1,\ldots, x_N):=\sum_{i=1}^N S_ix_i$. This is a unitary due to the relations $1=S_i^*S_i=\sum_{j=1}^N S_jS_j^*$ (which implies $S_i^*S_j=\delta_{ij}$). Finally, 
$$\alpha^*D\alpha=(S_i^*DS_j)_{i,j=1}^N=D\otimes \id_{\C^N}+(S_i^*[D,S_j])_{i,j=1}^N.$$
Therefore, the operator $B:=\alpha(S_i^*[D,S_j])_{i,j=1}^N\alpha^*=\sum [D,S_j]S_j^*\in \mathbb{B}(L^2(O_N))$ will ensure that $\alpha:N(L^2(O_N),D)\xrightarrow{\sim} (_\lambda L^2(O_N),D+B)$. 

In fact, for any unitary $u\in O_N$ we can define a $*$-endomorphism $\varphi_u$ on $O_N$ by declaring $\varphi_u(S_i):=uS_i$. The unitary $u$ is determined by $u=\sum_{j=1}^N \varphi_u(S_i)S_i^*$.  We have that $\lambda=\varphi_u$ where $u=\sum_{i,j=1}^N S_iS_jS_i^*S_j^*$ is a symmetry. Therefore $u$ is norm-smoothly homotopic to $1$ in the unitary group of $O_N$. We can conclude that there exists a $*$-homomorphism $\pi:O_N\to C([0,1],O_N)$ such that $\mathrm{ev}_0\circ \pi=\id_{O_N}$, $\mathrm{ev}_1\circ \pi=\lambda$ and $\pi$ restricts to a $*$-homomorphism $O_N^\infty\to C^\infty([0,1],\tilde{O}_N^\infty)$ where $\tilde{O}_N^\infty$ is the closure of $O_N^\infty$ under holomorphic functional calculus in $O_N$. We shall return to this example below in Subsection \ref{ckexbordism}.

\subsection{Cuntz-Pimsner algebras on manifolds}
\label{cpalgebronmanfdodoex}
Above in Example \ref{ckex}, we saw a spectral triple on the Cuntz algebra. Let us consider a geometrically more interesting example, previously considered in \cite{boundarypap}, of a Cuntz-Pimsner algebra defined from a vector bundle on a manifold. Let $E\to W$ be a Hermitian vector bundle on a smooth Riemannian manifold $W$, potentially with boundary. For simplicity we assume that $E$ has constant rank $N$. The sections $\mathpzc{V}:=C_0(W,E)$ (vanishing at infinity but not necessarily on the boundary) forms a $C_0(W)$-Hilbert $C^*$-module. Since $C_0(W)$ is commutative, we view $\mathpzc{V}$ as a bimodule over $C_0(W)$. We can form the Fock module
$$\mathpzc{F}_\mathpzc{V}:=\bigoplus_{n\geq 0} \mathpzc{V}^{\otimes n}=\bigoplus_{n\geq 0} C_0(W,E^{\otimes n}),$$
where we interpret $\mathpzc{V}^{\otimes 0}=C_0(W)$ and the direct sum is of Hilbert $C^*$-modules. Any section $\xi\in \mathpzc{V}=C_0(W,E)$ defines a creation operator 
$$T_\xi\eta:=\xi\otimes \eta, \quad \eta\in \mathpzc{F}_\mathpzc{V}.$$
The creation operator $T_\xi$ is bounded and adjointable, with $\|T_\xi\|\leq \|\xi\|$ and $T_\xi^*$ annihilates $\mathpzc{V}^{\otimes 0}$ and acts on a simple tensor $\eta=\eta_1\otimes \cdots\otimes \eta_n$ as 
$$T_\xi^*\eta=\langle \xi,\eta_1\rangle \eta_2 \otimes \cdots\otimes \eta_n.$$
The $C^*$-algebra generated by all creation operators is denoted by $\mathcal{T}_{\mathpzc{V}}\subseteq \End_{C_0(W)}^*(\mathpzc{F}_\mathpzc{V})$ and is called the Cuntz-Pimsner-Toeplitz algebra of $\mathpzc{V}$. Using some arguments with partitions of unity, it is readily seen that $\K_{C_0(W)}(\mathpzc{F}_\mathpzc{V})$ forms an ideal in $\mathcal{T}_{\mathpzc{V}}$ and the Cuntz-Pimsner algebra of $\mathpzc{V}$ is defined as the quotient $\mathcal{O}_{\mathpzc{V}}:=\mathcal{T}_{\mathpzc{V}}/\K_{C_0(W)}(\mathpzc{F}_\mathpzc{V})$. By definition, it fits into an extension 
$$0\to \K_{C_0(W)}(\mathpzc{F}_\mathpzc{V})\to \mathcal{T}_{\mathpzc{V}}\to \mathcal{O}_{\mathpzc{V}}\to 0.$$
This extension was, in much larger generality, proven by Pimsner \cite{Pimsner} to be semisplit and so defines a bivariant $KK$-class $\partial_{\mathpzc{V}}\in KK_1( \mathcal{O}_{\mathpzc{V}},C_0(W))$. 

We consider the operator valued weight $\psi:\mathcal{O}_{\mathpzc{V}}\dashrightarrow C_0(W)$ constructed in \cite{renniesim}, and we write $\Xi_{\mathpzc{V}}$ for the completion of its domain as a $C_0(W)$-Hilbert $C^*$-module. Left multiplication in $\mathcal{O}_{\mathpzc{V}}$ defines a representation 
$$\mathcal{O}_{\mathpzc{V}}\to \End_{C_0(W)}^*(\Xi_{\mathpzc{V}}).$$
By \cite{GMR}, we can decompose 
$$\Xi_{\mathpzc{V}}=\bigoplus_{n\in \Z,\, k\in \N} \Xi_{n,k}^{\mathpzc{V}},$$
where $\Xi_{n,k}^{\mathpzc{V}}=0$ unless $n+k\geq 0$. By the arguments in \cite[Section 5.4]{boundarypap}, we have that 
$$\Xi_{n,k}^{\mathpzc{V}}=C_0(W;\mathcal{H}_{n,k}(E)),$$
where $\mathcal{H}_{n,k}$ is the functor discussed in Remark \ref{functorforhnk}. In particular, we can consider $\Xi_{\mathpzc{V}}$ as the Hilbert $C^*$-module of sections of the (strongly continuous) bundle of Hilbert spaces $\bigoplus_{n,k} \mathcal{H}_{n,k}(E)$. Defining the function $\psi$ as in Equation \ref{definedodepsi}, and let $D$ be the diagonal operator in the decomposition $\Xi_{\mathpzc{V}}=\bigoplus_{n\in \Z,\, k\in \N} \Xi_{n,k}^{\mathpzc{V}}$ with
$$D|_{\mathcal{H}_{n,k}}=\psi(n,k).$$ 
Write $\mathcal{O}_{\mathpzc{V}}^{\rm Lip}$ for the $*$-algebra generated by the image of $\{T_\xi: \xi\in \mathrm{Lip}_c(W,E)\}$ in $\mathcal{O}_{\mathpzc{V}}$. The inclusion $\mathcal{O}_{\mathpzc{V}}^{\rm Lip}\subseteq \mathcal{O}_{\mathpzc{V}}$ is dense. The data $(\Xi_{\mathpzc{V}},D)$ defines an $(\mathcal{O}_{\mathpzc{V}}^{\rm Lip}, C_0(W))$-cycle by \cite{GMR}. 

Following \cite{GMR}, let us construct a $(\mathcal{O}_{\mathpzc{V}}^{\rm Lip},\C)$-chain starting from the $(\mathcal{O}_{\mathpzc{V}}^{\rm Lip}, C_0(W))$-cycle $(\Xi_{\mathpzc{V}},D)$ and a Dirac operator on $W$. For simplicity, we assume that $W$ is a compact manifold with boundary and that $E$ is a vector bundle on $\overline{W}$. We write $\mathcal{O}_{\mathpzc{V}}$ for the Cuntz-Pimsner algebra defined from the module of sections of $E\to \overline{W}$ and $\mathcal{O}_{\mathpzc{V}}^\circ$ for the Cuntz-Pimsner algebra defined from the module of sections of $E|_{W^\circ}\to W^\circ$. We note that $\mathcal{O}_{\mathpzc{V}}^\circ=\mathcal{O}_{\mathpzc{V}}C_0(W^\circ)=C_0(W^\circ)\mathcal{O}_{\mathpzc{V}}$ is an ideal in $\mathcal{O}_{\mathpzc{V}}$ and $\mathcal{O}_{\mathpzc{V}}/\mathcal{O}_{\mathpzc{V}}^\circ$ is the Cuntz-Pimsner algebra defined from the module of sections of $E|_{\partial W}\to \partial W$.

Choose a smooth frame for $E$ and let $p_E\in M_{N_p}(C^\infty(\overline{W}))$ denote the associated projection. We have that $\Xi_{n,k}^{\mathpzc{V}}$ is the finitely generated projective $C(\overline{W})$-module defined from the projection $\mathcal{H}_{n,k}(p_E)$. Let $\nabla_{n,k}$ denote the Grassmann connection on $\mathcal{H}_{n,k}(E)$ defined from $\mathcal{H}_{n,k}(p_E)$. 

Assume that $\slashed{D}$ is a Dirac operator on some Clifford bundle $S\to W$. Write $\slashed{D}_{n,k}$ for the twisted Dirac operator $1\otimes_{\nabla_{n,k}}\slashed{D}$ on $S\otimes \mathcal{H}_{n,k}(E)$.  We write 
$$L^2(W;S\otimes \Xi_{\mathpzc{V}}):=\Xi_{\mathpzc{V}}\otimes_{C(\overline{W})}L^2(W;S)=\bigoplus_{n,k} L^2(W;S\otimes \mathcal{H}_{n,k}(E)).$$
Define the operator $D_E$ densely on $L^2(W;S\otimes \Xi_{\mathpzc{V}})$ as the closure of the operator defined on $\bigoplus_{n,k}^{\rm alg} C^\infty_c(W;S\otimes \mathcal{H}_{n,k}(E))$ as $\bigoplus_{n,k}(\slashed{D}_{n,k}+\psi(n,k)\gamma_S\otimes 1_{\mathcal{H}_{n,k}(E)})$. 

\begin{prop}
The closed operator $D_E$ on the Hilbert space $L^2(W;S\otimes \Xi_{\mathpzc{V}})$ is symmetric and its domain consists of all 
$$(f_{n,k})_{n,k}\in \bigoplus_{n,k} H^1_0(W;S\otimes \mathcal{H}_{n,k}(E))$$
such that 
$$\sum_{n,k}\left(\|f\|_{H^1_0(W;S\otimes \mathcal{H}_{n,k}(E))}^2+|\psi(n,k)|^2\|f\|_{L^2(W;S\otimes \mathcal{H}_{n,k}(E))}^2\right)<\infty.$$
Moreover, $(L^2(W;S\otimes \Xi_{\mathpzc{V}}),D_E)$ forms a symmetric $(\mathcal{O}_{\mathpzc{V}}^{\rm Lip},\C)$-chain restricting to a half-closed $(\mathcal{O}_{\mathpzc{V}}^{\rm Lip}\cap \mathcal{O}_{\mathpzc{V}}^\circ,\C)$-chain. The cycle is closed if and only if $\partial W=\emptyset$.
\end{prop}

\subsection{Dirac-Schr\"{o}dinger operators}
\label{diracschex}
A well studied class of operators are Dirac-Schr\"{o}dinger operators. See for instance \cite{bunkerelative, guohochs,HRbook,leschkaad1,kucerovskycallias,vandungencallias}. We introduce this notion in a fair amount of generality, similarly to that in Example \ref{diraconcstarbundleex}. We consider a complete Riemannian manifold with boundary $W$ and a Clifford-$B$-bundle $\mathcal{E}_B\to W$ equipped with a complete $B$-linear Dirac operator $\slashed{D}_\mathcal{E}:C^\infty_c(W^\circ, \mathcal{E}_B)\to C^\infty_c(W^\circ, \mathcal{E}_B)$. We also fix a self-adjoint element $V\in C^\infty(W^\circ, \End_B^*(\mathcal{E}_B))$ commuting with Clifford multiplication. Consider the differential expression
\begin{equation}
\label{dirscalakdn}
\slashed{D}_{\mathcal{E},V}:=
\begin{cases}
\begin{pmatrix}0&\slashed{D}_\mathcal{E}+iV\\
\slashed{D}_\mathcal{E}-iV& 0\end{pmatrix}, \; &\mbox{if $\dim(W)$ is odd}\\
\,&\\
\slashed{D}_\mathcal{E}+\gamma V, \; &\mbox{if $\dim(W)$ is even}.
\end{cases}
\end{equation}
At this stage, we consider $\slashed{D}_{\mathcal{E},V}$ as a $B$-linear operator on $C^\infty_c(W^\circ, \mathcal{E}_B)$, or in the odd-dimensional case $C^\infty_c(W^\circ, \mathcal{E}_B\oplus \mathcal{E}_B)$. We shall be interested in realizations of $\slashed{D}_{\mathcal{E},V}$ in the Hilbert $C^*$-module $L^2(W, \mathcal{E}_B)$ and to study its index theory. These operators are called Dirac-Schrödinger operators, or sometimes Callias type operators. Already at this stage, we note that if $D_{\mathcal{E},V}$ is a closed, regular, symmetric extension of $\slashed{D}_{\mathcal{E},V}$ then $(L^2(W, \mathcal{E}_B),D_{\mathcal{E},V})$ is a symmetric $(C^\infty_c(\overline{W}),B)$-chain. Further properties of $V$ at the boundary of $W$ and at infinity ensure that this chain restricts to a half-closed or even a closed chain. The analysis of Dirac-Schrödinger operators requires some further tools that we review in the coming chapters, and as such we return to study further details below in Example \ref{diracschexbord}.

\subsection{Spin-Dirac operators and positive scalar curvature}
\label{pscdiracex}
Turning to a more specific example, let $W$ be an $n$-dimensional spin manifold with boundary equipped with a Riemannian metric $g$ of positive scalar curvature. The scalar curvature $s_g$ is a scalar function formally defined as the full contraction of the metric tensor, and a direct relation to the geometry can be seen from the Taylor expansion at $r=0$ of the Riemannian volume of a ball centered at $x\in W^\circ$ with small enough radius $r>0$:
$$\mathrm{vol}(B(x,r))=\omega_nr^n\left(1-\frac{s_g(x)}{6(n+2)}r^2+O(r^4)\right).$$
Here $\omega_n$ denotes the volume of the unit ball in $\R^n$. 

We consider the spin bundle $S_W\to W$ and let $\slashed{D}_W$ denote the spin-Dirac operator defined from the Riemannian metric. For more details, see \cite{lawsonmich,PSrhoInd,SchPSCsur,PSZmap,stolzicm,xieyu}. The differential operator $\slashed{D}_W$ acts on $C^\infty(W,S_W)$. By Example \ref{diraconcstarbundleex} the minimal extension of $\slashed{D}_W$ fits into a symmetric $(C^\infty_c(\overline{W}),\C)$-chain and into a half-closed $(C^\infty_c(W^\circ),\C)$-chain; these chains are closed if and only if $\partial W=\emptyset$. In the presence of an isometric group action, Subsection \ref{groupactionsex} provides the context for constructing descent and crossed products of these chains. The Schrödinger-Lichnerowicz formula gives that as differential operators we have the equality
$$\slashed{D}_W^2=\nabla^\dagger\nabla+\frac{s_g}{4},$$
where $\nabla$ is the spin connection and $\nabla^\dagger$ its formal adjoint. 

If $s_g>0$ we say that $g$ has positive scalar curvature and if $\inf_{x\in W}s_g(x)>0$ we say that $g$ has uniformly positive scalar curvature. We tacitly assume $s_g>0$ for the remainder of this example. It is here of interest to note some special cases. 
\begin{enumerate}
\item If $W$ is a complete Riemannian manifold, then $\slashed{D}_W$ is essentially self-adjoint and extends to a self-adjoint operator $D_W$ on $L^2(W,S_W)$. If $g$ has uniformly positive scalar curvature the Schr\"{o}dinger-Lichnerowicz formula shows that $D_W$ has a spectral gap at $0$ of size $\inf_{x\in W}s_g(x)$. This occurs for instance if $W$ is the total space of a spin-Galois covering, and $g$ is the equivariant lift of a metric of positive scalar curvature.
\item If $W$ is a complete Riemannian manifold with has uniformly positive scalar curvature, and $\mathcal{E}_B\to W$ is a Hermitian flat $B$-bundle (of bounded geometry) the twisted spin-Dirac operator $\slashed{D}_{W,\mathcal{E}}$ is essentially self-adjoint and its closure $D_{W,\mathcal{E}}$ to a regular operator on $L^2(W,S_W\otimes \mathcal{E}_B)$ (see \cite{hankpapsch}). Moreover, the Schrödinger-Lichnerowicz formula gives the equality
$$D_W^2=\nabla^*_\mathcal{E}\nabla_\mathcal{E}+\frac{s_g}{4},$$
where $\nabla_\mathcal{E}$ is the spin connection twisted by $\mathcal{E}$. In particular,  $D_{W,\mathcal{E}}$ as a self-adjoint operator on $L^2(W,S_W\otimes \mathcal{E}_B)$ has a spectral gap at $0$ of size $\sqrt{\inf_{x\in W}s_g(x)}/2$. 
\end{enumerate}

We shall revisit this class of examples below in Section \ref{pscdiracexII} and the sequel \cite{monographtwo}. The fact that spin-Dirac operators have spectral gaps is an analytic fact allowing us to construct explicit null-bordisms. We note that in the special cases considered above we have not mentioned the case of when $W$ has a boundary; in this case suitable boundary conditions are required and we shall take more care with such realizations below in Section \ref{pscdiracexII} .

\subsection{Lipschitz manifolds}
\label{topolmfsex}

An interesting class of geometric examples comes from Lipschitz manifolds. What makes this class particularly interesting is a theorem of Sullivan \cite{SulLip} stating that any topological manifold in dimension $\neq 4$ admits a Lipschitz structure. As such, there are subtle rigidity questions at hand that we will study using $KK$-bordisms. These questions have previously been studied in for instance \cite{hilsum83, hilsum89,sulltele, teleacta, teleihes, zenobi}.

We shall recall some basic notions to set suitable notations and terminology. Our convention throughout this work is that a manifold is tacitly assumed to be smooth, and weaker notions of manifolds are singled out with a further adjective. Recall that a \emph{topological manifold} $M$ is a Hausdorff space equipped with an atlas $(U_j,\psi_j)_{j\in I}$ consisting of an open cover $(U_j)_{j\in I}$ and homeomorphisms $\psi_j:U_j\to V_j\subseteq \R^{n_j}$, where $V_j$ is open. A topological manifold with boundary is defined from an atlas with boundary: a collection $(U_j,\psi_j)_{j\in I}$ as before but the codomain of the homeomorphism $\psi_j:U_j\to V_j\subseteq \R^{n_j}_+$ is an open subset $V_j\subseteq \R^{n_j}_+:=\{(x_1,\ldots, x_{n_j})\in \R^{n_j}: x_{n_j}\geq 0\}$. For a topological manifold with boundary, the boundary is denoted by $\partial M$ and has an atlas defined from $(U_j\cap \psi_j^{-1}(V_j\cap (\R^{n_j-1}\times \{0\})),\psi_j|)_{j\in I}$.

If $M$ is connected, then $n_j$ is independent of $j$. In general, the dimension is locally constant (with value $n_j$ on $U_j$). We note that it implicitly follows that for each $i,j\in I$, the map
\begin{equation}
\label{psiipsijbetdldald}
\psi_i\circ \psi_j^{-1}:\psi_j(U_i\cap U_j)\to  \psi_i(U_i\cap U_j),
\end{equation}
is a continuous map of open subsets of $\R^{n_j}$ (or $\R^{n_j}_+$ if the boundary is non-empty). Of course, if each of the maps in Equation \eqref{psiipsijbetdldald} is smooth we call $(U_j,\psi_j)_{j\in I}$ a smooth atlas and we call $M$ a manifold. 

If $(U_j,\psi_j)_{j\in I}$ is an atlas such that the maps in Equation \eqref{psiipsijbetdldald} are Lipschitz continuous, we call $(U_j,\psi_j)_{j\in I}$ a Lipschitz atlas and we say that $M$ has a Lipschitz structure. A topological manifold equipped with a Lipschitz structure is called a \emph{Lipschitz manifold}. A Lipschitz manifold $M$ with Lipschitz atlas $(U_j,\psi_j)_{j\in I}$ satisfying that the Jacobian determinants of the Lipschitz mappings appearing in \eqref{psiipsijbetdldald} are a.e. positive, we say that $M$ is an \emph{oriented Lipschitz manifold}. For more details, see for instance \cite{hilsum83,teleihes}. 

Similarly, a Lipschitz manifold with boundary is defined and the boundary of a Lipschitz manifold with boundary is again a Lipschitz manifold. We note that if $W$ is a Lipschitz manifold with boundary, then its double $2W:=W\cup_{\partial W}W$ is a Lipschitz manifold. If $W$ is an oriented Lipschitz manifold with boundary, we orient $2W$ as $W\cup_{\partial W}(-W)$. 

Due to its central importance, we state Sullivan's theorem on Lipschitz structures on topological manifolds from \cite{SulLip}. The statement in \cite{SulLip} is for topological manifolds but it follows from the doubling construction that it extends to topological manifolds with boundary.

\begin{theorem}
\label{lnknad}
Let $M$ be a topological manifold with boundary and no connected component of dimension $4$. Then $M$ can be equipped with the structure of a Lipschitz manifold with boundary and it is unique up to topological isotopy.
\end{theorem}

Let us define the relevant geometric structures on Lipschitz manifolds. Since the transition functions in a Lipschitz atlas are only required to be Lipschitz, it is in general not the case that the tangent bundle exists as a vector bundle. However, the tangent bundle and its exterior powers exist as measurable vector bundles so their spaces of $L^2$-sections are well defined. Therefore some more care is needed in construction. We shall only recall the salient points of this construction, and refer the reader to \cite{hilsum83,teleihes} for further details.

Fix a Lipschitz manifold with boundary $W$ and an atlas $(U_j,\psi_j)_{j\in I}$. We tacitly assume each $U_j$ to be pre-compact in $W$. We write $U_j^\circ:=U_j\cap \psi_j^{-1}(V_j\cap (\R^{n_j-1}\times (0,\infty)))$ so that $(U_j^\circ,\psi_j|)_{j\in I}$ is an atlas for the Lipschitz manifold $W^\circ:=W\setminus \partial W$.

An $L^2_{\rm loc}$-form of degree $k$ on $W$ is a collection $\omega=(\omega_j)_{j\in I}$ where $\omega_j\in L^2(V_j,\wedge^k\R^{n_j})$ satisfies the compatibility condition that 
$$\omega_j=(\psi_i\circ \psi_j^{-1})^*\omega_i, \quad\mbox{on $\psi_j(U_i\cap U_j)$ for any $i,j\in I$}.$$ 
Write $L^2_{\rm loc}(W;\wedge^k)$ for the vector space of $L^2_{\rm loc}$-forms of degree $k$. We say that an $L^2_{\rm loc}$-form $\omega=(\omega_j)_{j\in I}$ of degree $k$ is in the domain of the exterior derivative if $\mathrm{d}\omega_j\in L^2(V_j,\wedge^{k+1}\R^{n_j})$, where the exterior derivative is interpreted in a distributional sense. If $\omega=(\omega_j)_{j\in I}$ is in the domain of the exterior derivative, we define 
$$\mathrm{d}\omega:=(\mathrm{d}\omega_j)_{j\in I},$$
which is readily seen to be an $L^2_{\rm loc}$-form of degree $k+1$. Since $\mathrm{d}^2=0$ for forms on $\R^n$, for any $L^2_{\rm loc}$-form $\omega$ in the domain of the exterior derivative, $\mathrm{d}\omega$ is also in the domain of the exterior derivative and $\mathrm{d}^2\omega=0$. 

Analogously to the case on a manifold, we can define the support of an $L^2_{\rm loc}$-form. If $\omega$ is a compactly supported $L^2_{\rm loc}$-form we can define $\int_W \omega:=\sum_j \int_{V_j} \chi_j\omega_j$, where $(\chi_j)_j$ is a collection of functions in $\Lip_c(V_j)$ such that $\sum_{j} \psi_j^*\chi_j=1$ on $W$. For a compactly supported $\omega$ in the domain of the exterior derivative, it follows from Stokes' theorem that $\int_W \mathrm{d}\omega=\int_{\partial W}\omega$. 

A Riemannian metric $g$ on a Lipschitz manifold is a collection $g=(g_j)_{j\in I}$ where $g_j$ is a metric on $V_j$ and the collection satisfies the compatibility condition that 
$$g_j=(\psi_i\circ \psi_j^{-1})^*g_i, \quad\mbox{on $\psi_j(U_i\cap U_j)$ for any $i,j\in I$}.$$ 
Given a Riemannian metric, we can define the Hodge star of an $L^2_{\rm loc}$-form. For simplicity, we assume that our Lipschitz manifold has constant dimension $n$. If $\omega$ is of degree $k$, its Hodge star is the $L^2_{\rm loc}$-form $*_g\omega$ of degree $n-k$ defined by
$$*_g\omega=(*_{g_j}\omega_j)_{j\in I},$$
where $*_{g_j}$ is the Hodge star in $V_j$. The compatibility condition on $\omega$ and $g$ ensures that  $*_g\omega$ is an $L^2_{\rm loc}$-form of degree $n-k$. We define 
$$L^2(W;\wedge^k)=\left\{\omega\in L^2_{\rm loc}(W;\wedge^k): \int_W\omega\wedge *_g\omega:=\sum_j \int_{V_j} \chi_j\omega_j\wedge *_{g_j}\omega_j<\infty\right\}.$$
We also write $L^2(W;\wedge^*):=\bigoplus_k L^2(W;\wedge^k)$. The operator $*_g$ defines a unitary isomorphism $L^2(W;\wedge^k)\to L^2(W;\wedge^{n-k})$ and $*_g^2=(-1)^{k(n-k)}$. We define the symmetry $\tau_g$ on $L^2(W;\wedge^*)$ by setting 
$$\tau_g\omega:=i^{(n-k)(k-1)+n/2}*_g\omega, \quad \omega\in L^2(W;\wedge^k).$$
We write 
\begin{equation}
\label{decomwithhodegrestar}
L^2(W;\wedge^*)=L^2(W;\wedge_+)\oplus L^2(W;\wedge_-),
\end{equation}
where $L^2(W;\wedge_\pm):=\ker(\tau_g\mp1)$.

We have the closed densely defined operator $\mathrm{d}:L^2(W;\wedge^k)\dashrightarrow L^2(W;\wedge^{k+1})$ defined by restricting the exterior derivative to 
$$\Dom(\mathrm{d}):=\{\omega\in L^2(W;\wedge^*): \mathrm{d}\omega\in L^2(W;\wedge^{*+1})\}.$$
The formal adjoint of the exterior derivative can be defined from the densely defined $\delta:L^2(W;\wedge^k)\dashrightarrow L^2(W;\wedge^{k-1})$ which acts as 
$$\delta\omega:=(-1)^{n(k+1)+1}*_g\mathrm{d}*_g\omega, \quad\mbox{for  $\omega\in \Dom(\delta):=*_g\Dom(\mathrm{d})$,}$$
of degree $k$. We define the Hodge-de Rham operator 
$$\slashed{D}_{\rm HdR}:=\mathrm{d}+\delta,$$
as a formal differential expression acting on $\Dom(\delta)\cap\Dom(\mathrm{d})$. We note that $\slashed{D}_{\rm HdR}$ and its domain depends on the Riemannian metric. Note that $\tau_g$ preserves the domain of $\slashed{D}_{\rm HdR}$ and 
$$\tau_g\slashed{D}_{\rm HdR}=(-1)^{n+1}\slashed{D}_{\rm HdR}\tau_g \quad \mbox{thereon}.$$
We further note that the compactly supported Lipschitz functions $\Lip_c(W^\circ)$ preserve $\mathrm{Dom}(\slashed{D}_{\rm HdR})$ and elements of $\Lip_c(W^\circ)$ have bounded commutators with $\slashed{D}_{\rm HdR}$. Moreover, $\Lip_c(W^\circ)\mathrm{Dom}(\slashed{D}_{\rm HdR})$ is invariant under $\tau_g$.

\begin{define}
Let $W$ be an oriented Riemannian Lipschitz manifold of dimension $n$ with Riemannian metric $g$. 

If $n$ is even, we define $D_{\rm sign}$ as the closure in $L^2(W,\wedge^*)$ of the restriction of $\slashed{D}_{\rm HdR}$ to $\Lip_c(W^\circ)\mathrm{Dom}(\slashed{D}_{\rm HdR})$. We consider $D_{\rm sign}$ as an odd operator in the grading defined from $\tau_g$.

If $n$ is odd, we define $D_{\rm sign}$ as the closure in $L^2(W,\wedge_+)$ of the restriction of $\slashed{D}_{\rm HdR}$ to $\Lip_c(W^\circ)\mathrm{Dom}(\slashed{D}_{\rm HdR})\cap L^2(W,\wedge_+)$, where $L^2(W,\wedge_+)$ is as in Equation \eqref{decomwithhodegrestar}.
\end{define}

We shall need some further terminology. We say that a Riemannian metric on a Lipschitz manifold $W$ with boundary is complete if the distance function associated with the Riemannian metric makes $W$ into a complete metric space. 

\begin{prop}
\label{signatureforlip}
Let $W$ be an oriented Lipschitz manifold with boundary equipped with a complete Riemannian metric. 
\begin{enumerate}
\item $D_{\rm sign}$ is symmetric. Moreover, $(1+D_{\rm sign}^*D_{\rm sign})^{-1}$ is a $C_0(\overline{W})$-locally compact operator on $L^2(W,\wedge^*W)$ (or $L^2(W,\wedge_+)$ if the dimension is odd). The operator $D_{\rm sign}$ is self-adjoint if $\partial W=\emptyset$.
\item Any $a\in \Lip_0(\overline{W})$ preserves $\Dom(D_{\rm sign})$ and $[D_{\rm sign},a]$ extends to a bounded operator on $L^2(W,\wedge^*W)$ (or $L^2(W,\wedge_+)$ if the dimension is odd). 
\item Any $j\in \Lip_0(W^\circ)$ satisfies $j\Dom((D_{\rm sign})^*)\subseteq\Dom(D_{\rm sign})$. 
\end{enumerate}
In particular, $(L^2(W,\wedge^*W),D_{\rm sign})$ (or $(L^2(W,\wedge_+),,D_{\rm sign})$ if the dimension is odd) defines a symmetric $(\Lip_0(\overline{W}),\C)$-chain that restricts to a half-closed $(\Lip_0(W^\circ),\C)$-chain. The chain $(L^2(W,\wedge^*W),D_{\rm sign})$ (or $(L^2(W,\wedge_+),,D_{\rm sign})$ if the dimension is odd) is closed if $\partial W=\emptyset$.
\end{prop}

\begin{proof}
It follows from the construction that $\Lip_c(W^\circ)\mathrm{Dom}(D_{\rm sign})$ is a core for $D_{\rm sign}$. Stokes theorem now implies that $D_{\rm sign}$ is symmetric on a core, and so it is symmetric. It was proven in \cite{hilsum89} that if a Riemannian metric is complete and $\partial W=\emptyset$, then $D_{\rm sign}$ is self-adjoint. To complete the proof of item 1), it suffices to prove that $\Lip_c(W)\mathrm{Dom}(D_{\rm sign})\hookrightarrow L^2(W,\wedge^*)$ (or $L^2(W,\wedge_+)$ if the dimension is odd) is compact. By going to the double, we can assume that $W$ has empty boundary in which case compactness of $\Lip_c(W)\mathrm{Dom}(D_{\rm sign})\hookrightarrow L^2(W,\wedge^*)$ follows from the Rellich type result of \cite[Theorem 7.1]{teleihes}.

Item 2) follows from that $\Lip_0(\overline{W})$ preserves the core $\Lip_c(W^\circ)\mathrm{Dom}(D_{\rm sign})$ and on there has bounded commutators. Item 3) follows from a short computation showing that if $\xi\in \Dom((D_{\rm sign})^*)$ is compactly supported in $W^\circ$, then $\xi\in \mathrm{Dom}(\slashed{D}_{\rm HdR})$ (or $\mathrm{Dom}(\slashed{D}_{\rm HdR})\cap L^2(W,\wedge_+)$ so it is in the core $\Lip_c(W^\circ)\mathrm{Dom}(D_{\rm sign})$ for $D_{\rm sign}$.
\end{proof}

Now we turn to twists of the signature operator on Lipschitz manifolds by Hermitian $B$-bundles, where $B$ is a unital $C^*$-algebra. As a starting point, we are using the constructions of Hilsum \cite{hilsumlip}. A  $B$-bundle $\mathcal{E}_B\to W$ is a locally trivial bundle of finitely generated projective $B$-modules. By standard results, $\mathcal{E}_B$ has a Lipschitz structure unique up to isomorphism. A Hermitian $B$-bundle is a $B$-bundle equipped with a fibrewise $B$-valued inner product such that if $U\subseteq W$ is any open subset then 
$$\langle \xi_1,\xi_2\rangle_{\mathcal{E}}\in \Lip_0(U),\quad \mbox{for all $\xi_1,\xi_2\in \Lip_0(U,\mathcal{E}_B)$}.$$
A Lipschitz connection on $\mathcal{E}_B$ is a $B$-linear map $\nabla:\Lip(W,\mathcal{E}_B)\to L^\infty(W; T^*W\otimes \mathcal{E}_B)$ such that in any local trivialization $\mathcal{E}_B|_U\cong U\times pB^N$ (for some $N$ and some $p\in \Lip(U,M_N(B))$) we have that $\nabla|_U=\mathrm{d}+\theta_U$ and its curvature 
$$\nabla|_U^2=\mathrm{d}\theta|_U+\theta|_U\wedge \theta|_U$$ 
is a well defined bounded $2$-form. If $\nabla$ is a Lipschitz connection, then its curvature $\nabla^2\in L^\infty(W,\wedge^2 T^*W\otimes \End_B^*(\mathcal{E}_B))$ is a well defined bounded $2$-form valued endomorphism of $\mathcal{E}_B$. We say that a Lipschitz connection on $\mathcal{E}_B$ is Hermitian if 
$$\mathrm{d}\langle \xi_1,\xi_2\rangle_{\mathcal{E}}=\langle \nabla\xi_1,\xi_2\rangle_{\mathcal{E}}+\langle \xi_1,\nabla\xi_2\rangle_{\mathcal{E}},$$
for any two Lipschitz sections $\xi_1$ and $\xi_2$.

The next proposition is an immediate consequence of Proposition \ref{signatureforlip} and the results of \cite[Section 1.4]{DGM}.

\begin{prop}
\label{technicalpropforlipschitz}
Let $W$ be an oriented Lipschitz manifold with boundary which is either compact or equipped with a complete metric. Assume that $\mathcal{E}_B\to W$ is a Hermitian $B$-bundle with a Hermitian Lipschitz connection $\nabla_\mathcal{E}$. Write 
$$D_{{\rm sign},\mathcal{E}}:=D_{\rm sign}\otimes_{\nabla_\mathcal{E}}1.$$
\begin{enumerate}
\item $D_{{\rm sign},\mathcal{E}}$ is symmetric. Moreover, $(1+D_{{\rm sign},\mathcal{E}}^*D_{{\rm sign},\mathcal{E}})^{-1}$ is a $C_0(\overline{W})$-locally $B$-compact operator on $L^2(W,\wedge^*W\otimes \mathcal{E}_B)$. The operator $D_{{\rm sign},\mathcal{E}}$ is self-adjoint if $\partial W=\emptyset$.
\item Any $a\in \Lip_0(\overline{W})$ preserves $\Dom(D_{{\rm sign},\mathcal{E}})$ and $[D_{{\rm sign},\mathcal{E}},a]$ extends to a $B$-linear adjointable operator on $L^2(W,\wedge^*W\otimes \mathcal{E}_B)$. 
\item Any $j\in \Lip_0(W^\circ)$ satisfies $j\Dom(D_{{\rm sign},\mathcal{E}}^*)\subseteq\Dom(D_{{\rm sign},\mathcal{E}})$. 
\end{enumerate}
In particular, $(L^2(W,\wedge^*W\otimes \mathcal{E}_B),D_{{\rm sign},\mathcal{E}})$ defines a symmetric $(\Lip_0(\overline{W}),B)$-chain that restricts to a half-closed $(\Lip_0(W^\circ),B)$-chain. The chain $(L^2(W,\wedge^*W\otimes \mathcal{E}_B),D_{{\rm sign},\mathcal{E}})$ is closed if $\partial W=\emptyset$.
\end{prop}

The following result of Hilsum-Skandalis \cite{hilsumskand} will later be of importance for constructing relative cycles from homotopy structures on topological manifolds (in dimension $\neq 4$). The results of \cite{hilsumskand} are generally formulated for manifolds, but as is described on \cite[page 95]{hilsumskand} the results extend to Lipschitz manifolds.

\begin{theorem}
\label{hilsumskandforbla}
Let $M$ and $N$ be two compact oriented Lipschitz manifolds of dimension $n$ and set $Z:=M\dot{\cup}-N$. Assume that $\mathcal{E}_B\to N$ is a flat Hermitian $B$-bundle, and write $\mathcal{F}_B:=f^*\mathcal{E}_B\dot{\cup}\mathcal{E}_B\to Z$ and pick with a flat Hermitian Lipschitz connection $\nabla_\mathcal{F}$. 

If $f:M\to N$ is a homotopy equivalence, then there is an explicitly constructed $T_f\in \End_B^*(L^2(Z,\wedge^*Z\otimes \mathcal{F}_B))$ such that the operator $D_{{\rm sign},\mathcal{F}}+T_f$ is invertible. 
\end{theorem}

\subsection{Cutting down cycles to ideals}

Let $(\mathpzc{E},D)$ be an $(\mathcal{A},B)$-cycle. Assume that $\mathpzc{E}_0\subseteq \mathpzc{E}$ is an $A$-invariant submodule. We write $J:=\{a\in A: a\mathpzc{E}\subseteq \mathpzc{E}_0\}$ which by definition is a closed twosided ideal in $A$. We can define $D_0:=D|_{\mathpzc{E}_0}$ with $\Dom(D_0):=\Dom(D)\cap \mathpzc{E}_0$. Let us assume that $D_0$ is densely defined on $\mathpzc{E}_0$. The operator $D_0$ is symmetric and a quick argument shows that $D_0$ is closed. Since there is a continuous inclusion $\Dom(D_0)\hookrightarrow \Dom(D)$, we conclude that if $D_0$ is regular then $(\mathpzc{E}_0,D_0)$ is a symmetric $(\mathcal{A},B)$-chain that restricts to a half-closed $(\mathcal{A}\cap J,B)$-chain. Based on the previous statement, there are a number of natural questions:
\begin{enumerate}
\item Under what conditions is $D_0$ regular? 
\item Can the adjoint domain $\Dom(D_0^*)$ and the Lipschitz algebra $\mathrm{Lip}(D_0^*,D_0)$ be described? 
\end{enumerate}
Related results can be found in \cite{magnus2} restricted to the case $B=\C$, where $(\mathpzc{E}_0,D_0)$ was proven to form a relative spectral triple and $\Dom(D_0^*)$ was explicitly described. The questions above are to the authors' knowledge wide open in the case of a general $B$.

\section{Analysis of chains}
\label{sec:analchain}

The main object of study in this monograph are the cycles and chains of a pair $(\mathcal{A},B)$. In this subsection we collect some more subtle points concerning the analysis of these objects. These include automatic continuity results when $\mathcal{A}$ is a Banach algebra, various versions of functional calculus preserving classes of chains as well as results introducing methods of complex interpolation in the study of unbounded chains.

\subsection{Automatic continuity}
\label{subsec:autocont}

For chains where $\mathcal{A}$ is a Banach algebra, we can from the closed graph theorem conclude an automatic continuity result ensuring ``differentiability'' of elements from $\mathcal{A}$.

\begin{prop}[Automatic continuity of chains]
\label{autoodkkd}
Let $B$ be a $C^*$-algebra and $\mathcal{A}$ be equipped with a Banach algebra norm $\|\cdot\|_{\mathcal{A}}$ making the inclusion $\mathcal{A}\subseteq A$ continuous. If $(\mathpzc{E},D)$ is an $(\mathcal{A},B)$-chain, then there is a constant $C>0$ such that 
$$\|[D,a]\|_{\End_B(\mathpzc{E})}\leq C\|a\|_{\mathcal{A}}, \quad\forall a\in \mathcal{A}.$$
\end{prop}

\begin{proof}
Consider the linear mapping $[D,-]:\mathcal{A}\to \End_B(\mathpzc{E})$, $a\mapsto [D,a]$. This linear mapping is closed since if $a_j\to a$ in $\mathcal{A}$ and $[D,a_j]\to T$ in $\End_B^*(\mathpzc{E})$, then $a_j\to a$ in $\End_B^*(\mathpzc{E})$ (since $\mathcal{A}\subseteq A$ continuous), so $T\xi=[D,a]\xi$ for all $\xi\in \Dom(D)$, and it follows that $T=[D,a]$. The closed graph theorem implies that $[D,-]$ is bounded and the proposition follows.
\end{proof}

\subsection{Functional calculus}
\label{subsec:funccalc}

We shall now study the effect of functional calculus on the algebra $\mathcal{A}$ to the sets of cycles and chains. The material in this section is not new, see \cite{BlaCun,varillyetal}, however we present it in a way adapted for our purposes. For a tempered distribution $f$ on $\R^n$, we let $\hat{f}$ denote its Fourier transform. Note that $\hat{f}$ is a smooth function as soon as $f$ is compactly supported. We also note that $\hat{f}$ is integrable whenever $f\in C^\alpha_c(\R^n)$ for $\alpha>n$, where for $\alpha=k+\alpha_0$ with $\alpha_0\in [0,1)$, 
$$C^\alpha_c(\R^n):=\{f\in C_c(\R^n): \partial_x^\beta f\in C^{\alpha_0}(\R^n), \; \forall |\beta|\leq k\}.$$
Here $C^{\alpha_0}(\R^n)$ denotes the space of functions H\"{o}lder continuous with exponent $\alpha_0$. We will occasionally also consider the spaces 
$$C^{k,1}_c(\R^n):=\{f\in C_c(\R^n): \partial_x^\beta f\in \mathrm{Lip}(\R^n), \; \forall |\beta|\leq k\}.$$
Note that there is a strict isometric inclusion $C^{k+1}_c(\R^n)\subseteq C^{k,1}_c(\R^n)$ and a continuous, non-dense, inclusion $C^{k,1}_c(\R^n)\subseteq C^{\alpha}_c(\R^n)$ for $\alpha<k+1$.

\begin{define}
Assume $\alpha>n$. Let $\mathcal{A}\subseteq A$ be a $*$-subalgebra. For a self-adjoint $n$-tuple $(a_1,\ldots, a_n)\in \mathcal{A}^n$ and a function $f\in C^\alpha_c(\R^n)$, we define 
$$f(a_1,\ldots,a_n):=\frac{1}{(2\pi)^n}\int_{\R^n}\hat{f}(t_1,\ldots, t_n)\mathrm{e}^{it_1a_1}\cdots \mathrm{e}^{it_na_n}\mathrm{d}t_1\cdots \mathrm{d}t_n\in \tilde{A}.$$
Here $\tilde{A}$ denotes the unitalization of $A$. 

Assume that $\mathcal{A}$ is unital. For $\alpha>1$, let $\mathcal{A}^{(\alpha)}$ denote the $*$-algebra generated by elements $f(a_1,\ldots,a_n)$ where $(a_1,\ldots, a_n)\in \mathcal{A}^n$ is a self-adjoint $n$-tuple and $f\in C^{\alpha n+1}_c(\R^n)$. We define the $*$-algebra
$$\overline{\mathcal{A}}^{(\alpha)}:=\cup_{k\in \N} \mathcal{A}_k^{(\alpha)},$$
where $\mathcal{A}_k^{(\alpha)}$ is defined inductively by $\mathcal{A}_0^{(\alpha)}:=\mathcal{A}$ and $\mathcal{A}_{k+1}^{(\alpha)}:=(\mathcal{A}_k^{(\alpha)})^{(\alpha)}$.

Assume that $\mathcal{A}$ is non-unital. For $\alpha>1$, let $\mathcal{A}^{(\alpha)}$ denote the $*$-algebra generated by elements $f(a_1,\ldots,a_n)$ where $(a_1,\ldots, a_n)\in \mathcal{A}^n$ is a self-adjoint $n$-tuple and $f\in C^{\alpha n+1}_c(\R^n)$ satisfying the condition that $f(t_1,\ldots, t_n)=0$ whenever $t_j=0$ for some $j$. We define the $*$-algebra
$$\overline{\mathcal{A}}^{(\alpha)}:=\cup_{k\in \N} \mathcal{A}_k^{(\alpha)},$$
where $\mathcal{A}_k^{(\alpha)}$ is defined inductively by $\mathcal{A}_0^{(\alpha)}:=\mathcal{A}$ and $\mathcal{A}_{k+1}^{(\alpha)}:=(\mathcal{A}_k^{(\alpha)})^{(\alpha)}$.

\end{define}

\begin{remark}
We remark that for any self-adjoint $a$ in a unital $C^*$-algebra $A$ and $t\in \R$, $\mathrm{e}^{ita}\in A$ is a well defined unitary and depends norm continuously on $t$. As such, for any self-adjoint $n$-tuple $(a_1,\ldots, a_n)\in \mathcal{A}^n$, the function $(t_1,\ldots, t_n)\mapsto \mathrm{e}^{it_1a_1}\cdots \mathrm{e}^{it_na_n}$ is a well defined norm continuous function $\mathbb{R}^n\to \mathcal{U}(A)$. Therefore the integral defining $f(a_1,\ldots,a_n)$ converges in norm topology for $f\in C^\alpha_c(\R^n)$ and $\alpha>n$. 

We also note that it is only for technical reasons we assume our functions $f$ to have compact support. In the ambient $C^*$-algebra $A$, the association $f\mapsto f(a_1,\ldots,a_n)$ coincides with successive functional calculus and as such extends by continuity to any $f\in C(\R^n)$ and depends only on the compact subset $\prod_{j=1}^n \mathrm{Spec}(a_j)\subseteq \R^n$.
\end{remark}

\begin{theorem}
\label{funccalccompthm}
Assume that $\mathpzc{E}$ is a Hilbert $C^*$-module and $D:\mathpzc{E}\dashrightarrow \mathpzc{E}$ is a regular operator. Let $\mathcal{A}\subseteq \End_B^*(\mathpzc{E})$ be a $*$-algebra. Then for any $\alpha>1$, the following holds:
\begin{enumerate}
\item If $\mathcal{A}\subseteq \mathrm{Lip}(D)$, then $\overline{\mathcal{A}}^{(\alpha)}\subseteq \mathrm{Lip}(D)$
\item If $\mathcal{A}\subseteq \mathrm{Lip}_0(D)$, then $\overline{\mathcal{A}}^{(\alpha)}\subseteq \mathrm{Lip}_0(D)$
\item If $D$ is symmetric and $\mathcal{A}\subseteq \mathrm{Lip}(D,D^*)$, then $\overline{\mathcal{A}}^{(\alpha)}\subseteq \mathrm{Lip}(D,D^*)$
\end{enumerate}
In particular, if $(\mathpzc{E},D)$ is a symmetric/half-closed/closed chain for $(\mathcal{A},B)$ then $(\mathpzc{E},D)$ also defines a symmetric/half-closed/closed chain for $(\overline{\mathcal{A}}^{(\alpha)},B)$ for all $\alpha>1$.
\end{theorem}

Before proving the theorem we will need a lemma. For a function $f$ on $\R^n$, we define the seminorm
$$\|f\|_W:=\frac{1}{(2\pi)^n}\sum_{j=1}^n\int_{\R^n} |t_j \hat{f}(t_1,\ldots, t_n)|\mathrm{d}t_1\cdots \mathrm{d}t_n.$$
We note that if $f\in C^{\alpha n+1}_c(\R^n)$, then $\hat{f}$ is smooth and satisfies $\hat{f}(t)=O(|t|^{-\alpha n-1+\epsilon})$ as $|t|\to \infty$, for any $\epsilon>0$. Thus $\|f\|_W<\infty$ for any $f\in C^{\alpha n+1}_c(\R^n)$ with $\alpha>1$.

\begin{lemma}
\label{funccalclemmavari}
Assume that $\mathpzc{E}$ is a Hilbert $C^*$-module and $D:\mathpzc{E}\dashrightarrow \mathpzc{E}$ is a regular operator. If $f$ satisfies $\|f\|_W<\infty$, then for any self-adjoint $n$-tuple $(a_1,\ldots, a_n)\in \mathrm{Lip}(D)^n$, it holds that $f(a_1,\ldots, a_n)\in \mathrm{Lip}(D)$ and
$$\|[D,f(a_1,\ldots, a_n)]\|\leq  \|f\|_W \max_j \|[D,a_j]\|.$$
\end{lemma}

\begin{proof}
We compute that 
\begin{align*}
[D,f(a_1,\ldots, a_n)]&=\frac{1}{(2\pi)^n}\int_{\R^n}\hat{f}(t_1,\ldots, t_n)[D,\mathrm{e}^{it_1a_1}\cdots \mathrm{e}^{it_na_n}]\mathrm{d}t_1\cdots \mathrm{d}t_n=\\
&=\frac{1}{(2\pi)^n}\sum_{j=1}^n\int_{\R^n}\hat{f}(t_1,\ldots, t_n)\mathrm{e}^{it_1a_1}\cdots [D,\mathrm{e}^{it_ja_j}]\cdots \mathrm{e}^{it_na_n}\mathrm{d}t_1\cdots \mathrm{d}t_n.
\end{align*}
The proof now proceeds as in \cite[Lemma 10.15]{varillyetal}.
\end{proof}

\begin{proof}[Proof of Theorem \ref{funccalccompthm}]
Since $\mathrm{Lip}(D)$, $\mathrm{Lip}_0(D)$ and $ \mathrm{Lip}(D,D^*)$ are $*$-algebras, it suffices to prove that the properties listed in items 1)-3) are preserved under the procedure $\mathcal{A}\rightsquigarrow\mathcal{A}^{(\alpha)}$. 

It follows directly from Lemma \ref{funccalclemmavari} that $\mathcal{A}^{(\alpha)}\subseteq \mathrm{Lip}(D)$ if $\mathcal{A}\subseteq \mathrm{Lip}(D)$. This proves item 1). Item 2) follows from item 1), since $\mathcal{A}^{(\alpha)}\subseteq A$ and a continuity argument shows that $a(1+D^*D)^{-1}\in \K_B(\mathpzc{E})$ for all $a \in A$ if $\mathcal{A}\subseteq \mathrm{Lip}_0(D)$. 

To prove item 3), we note that the statement is trivial unless $\mathcal{A}$ is non-unital. Indeed, if the identity operator belongs to $\mathrm{Lip}(D,D^*)$ then $D$ is self-adjoint in which case item 3) follows from item 1). We must therefore show that $$f(a_1,\ldots, a_n)\Dom(D^*)\subseteq \Dom(D)$$ for all $a_1,\ldots, a_n\in \mathrm{Lip}(D,D^*)$ and $f\in C^{\alpha n+1}_c(\R^n)$ satisfying $f(t_1,\ldots, t_n)=0$ whenever $t_j=0$ for some $j$. We note that the condition on $f$ guarantees that  
$$f(a_1,\ldots,a_n):=\frac{1}{(2\pi)^n}\int_{\R^n}\hat{f}(t_1,\ldots, t_n)(\mathrm{e}^{it_1a_1}-1)\cdots (\mathrm{e}^{it_na_n}-1)\mathrm{d}t_1\cdots \mathrm{d}t_n.$$
Note that $\R\to t_j\mapsto \mathrm{e}^{it_ja_j}-1\in \mathrm{Lip}(D,D^*)$ is continuous since $a_j\in \mathrm{Lip}(D,D^*)$ and $\mathrm{e}^{it_ja_j}-1=\sum_{k=1}^\infty \frac{(it)^k}{k!}a_j^k$ as series in $\mathrm{Lip}(D,D^*)$ which is uniformly absolutely convergent on compacts in $t_j\in \R$. A direct computation as in the proof of Lemma \ref{funccalclemmavari} shows that 
$$\|(\mathrm{e}^{it_1a_1}-1)\cdots (\mathrm{e}^{it_na_n}-1)\|_{\mathrm{Lip}(D,D^*)}\leq C(|t_1|+\cdots |t_n|),$$
for a constant $C$ depending only on the values of $\|a_j\|_{\mathrm{Lip}(D,D^*)}$. We conclude that $f(a_1,\ldots, a_n)\Dom(D^*)\subseteq \Dom(D)$.
\end{proof}

We now proceed to holomorphic functional calculus. If $\mathcal{A}\subseteq A$ is a unital $*$-algebra, we let $\mathcal{A}^{\mathcal{O}}$ denote the $*$-algebra generated by $\{f(a)\in A: a\in \mathcal{A}, \; f\in \mathcal{O}(\mathrm{Spec}_A(a))\}$. Here $\mathcal{O}(K)$ denotes the space of functions holomorphic on a neighbourhood of a compact subset $K\subseteq \C$. We define 
$$\overline{\mathcal{A}}^{\mathcal{O}}:=\cup_{k\in \N} \mathcal{A}^{\mathcal{O}}_k,$$ 
where $\mathcal{A}^{\mathcal{O}}_0:=\mathcal{A}$ and $\mathcal{A}^{\mathcal{O}}_k$ are defined inductively by $\mathcal{A}^{\mathcal{O}}_{k+1}:=(\mathcal{A}^{\mathcal{O}}_k)^{\mathcal{O}}$. If $\mathcal{A}$ is non-unital, we define $\mathcal{A}^{\mathcal{O}}$ denote the $*$-algebra generated by $$\left\{f(a)\in A: a\in \mathcal{A}, \; f\in \mathcal{O}(\mathrm{Spec}_A(a)), \; f(0)=0\right\}$$ and $\overline{\mathcal{A}}^{\mathcal{O}}$ is defined in the analogous way. The following theorem holds by standard techniques from contour integration.

\begin{theorem}
\label{funccalcholothm}
Assume that $\mathpzc{E}$ is a Hilbert $C^*$-module and $D:\mathpzc{E}\dashrightarrow \mathpzc{E}$ is a regular operator. Let $\mathcal{A}\subseteq \End_B^*(\mathpzc{E})$ be a $*$-algebra. Then the following holds:
\begin{enumerate}
\item If $\mathcal{A}\subseteq \mathrm{Lip}(D)$, then $\overline{\mathcal{A}}^{\mathcal{O}}\subseteq \mathrm{Lip}(D)$
\item If $\mathcal{A}\subseteq \mathrm{Lip}_0(D)$, then $\overline{\mathcal{A}}^{\mathcal{O}}\subseteq \mathrm{Lip}_0(D)$
\item If $D$ is symmetric and $\mathcal{A}\subseteq \mathrm{Lip}(D,D^*)$, then $\overline{\mathcal{A}}^{\mathcal{O}}\subseteq \mathrm{Lip}(D,D^*)$
\end{enumerate}
In particular, if $(\mathpzc{E},D)$ is a symmetric/half-closed/closed chain for $(\mathcal{A},B)$ then $(\mathpzc{E},D)$ also defines a symmetric/half-closed/closed chain for $(\overline{\mathcal{A}}^{\mathcal{O}},B)$ for all $\alpha>1$.
\end{theorem}

\subsection{Order reduction and complex interpolation} 
\label{orderreducsubsec}

Consider a Banach $*$-algebra $\mathcal{A}$ densely and continuously embedded in a $C^*$-algebra $A$. For $\alpha\in [0,1]$, we write 
$$\mathcal{A}_\alpha:=[A,\mathcal{A}]_\alpha,$$
for the complex interpolation space, which again is a Banach $*$-algebra. For details, see \cite{bergloef}. Note that $\mathcal{A}_0=A$ and $\mathcal{A}_1=\mathcal{A}$. We have that $\mathcal{A}\subseteq \mathcal{A}_\alpha$ is dense for all $\alpha\in [0,1]$, hence also $\mathcal{A}_\alpha\subseteq A$ is dense for all $\alpha\in [0,1]$. The next theorem gives a procedure for ``reducing the order'' of a cycle in such a way that it extends to a complex interpolation space. This result will be used later in this monograph to study $KK$-bordism groups of low regularity functions on manifolds, e.g. the algebra of Lipschitz functions. For a closed and regular operator $D$ we use the notation
$$F_D:= D(1+D^*D)^{-1/2},$$
which is an adjointable operator with $F_D^*=F_{D^*}$. If $D$ is self-adjoint, then $F_D$ is self-adjoint.

\begin{theorem}
\label{interpolationthm}
Let $\alpha\in (0,1]$ and assume that $(\mathpzc{E},D)$ is a symmetric chain for $(\mathcal{A},B)$ with $D$ self-adjoint. Define the operator 
\begin{align*}
D_\alpha&:=D(1+D^2)^{(\alpha-1)/2}=F_D(1+D^2)^{\alpha/2},\\
\Dom(D_\alpha)&:=\Dom(1+D^2)^{\alpha/2}=(1+D^2)^{-\alpha/2}\mathpzc{E}.
\end{align*}
Then the following holds:
\begin{enumerate}
\item $D_\alpha$ is closed, self-adjoint and regular. 
\item $(\mathpzc{E},D_\alpha)$ is a symmetric chain for $(\mathcal{A}_\beta,B)$ with $\mathcal{A}_\beta\subseteq \mathrm{Lip}(|D_\alpha|)$ for any $\beta\in (\alpha,1]$ 
\item If $(\mathpzc{E},D)$ is a closed chain for $(\mathcal{A},B)$, then $(\mathpzc{E},D_\alpha)$ is a closed chain for $(\mathcal{A}_\beta,B)$ for any $\beta\in (\alpha,1]$.
\end{enumerate}
\end{theorem}

The reader should note that if $(\mathpzc{E},D)$ is a symmetric chain for $(\mathcal{A},B)$ with $D$ self-adjoint, then $(\mathpzc{E},D)$ is half-closed if and only if it is closed.

\begin{remark}
It seems unlikely that Theorem \ref{interpolationthm} can be made into a sharp result, i.e. that $(\mathpzc{E},D_\alpha)$ would be a symmetric chain for $(\mathcal{A}_\alpha,B)$. At best, one could hope for the theorem being sharp for Lipschitz cycles. Compare to \cite[Chapter 1, Proposition 7.4]{taylortoolspde}.
\end{remark}

In order to prove the theorem, we shall need a series of lemmas. The first one is similar to results studied in \cite[Appendix A.3]{adathesis}.

\begin{lemma}
\label{cpxpowerlemma}
Let $\Delta$ be a positive, invertible, self-adjoint, and regular operator on a $B$-Hilbert module $\mathpzc{E}$. For any $\epsilon>0$, there is a $C>0$ such that 
$$\|[\Delta^{it},a]\Delta^{1-\epsilon}\|\leq C|t|\|[\Delta,a]\|,$$
for all $t\in \R$ and $a\in \mathrm{Lip}(\Delta)$.
\end{lemma}

\begin{proof}
We can write $\log\Delta$ as the strongly convergent integral 
$$\log\Delta=\int_0^\infty \left(\frac{1}{1+\lambda}-(\lambda+\Delta)^{-1}\right)\mathrm{d} \lambda,$$
on $\Dom(\Delta)$. For the commutator with $a\in \mathrm{Lip}(\Delta)$, we can write
$$[\log\Delta,a]=\int_0^\infty (\lambda+\Delta)^{-1}[\Delta,a](\lambda+\Delta)^{-1}\mathrm{d} \lambda.$$
Therefore, using $\|(\lambda+\Delta)^{-1}\|\lesssim (1+\lambda)^{-1}$, we have for any $\epsilon>0$ that 
$$\|[\log\Delta,a]\Delta^{1-\epsilon}\|\lesssim \|[\Delta,a]\|.$$
With $U_t:=\Delta^{it}\equiv \mathrm{e}^{it\log\Delta}$, the Duhamel formula implies that 
$$[\Delta^{it},a]\equiv (U_taU_t^*-a)U_t=i\int_0^t U_s[\log\Delta,a]U_{t-s}\mathrm{d}s.$$
The lemma now follows from an elementary estimate using that $\Delta$ commutes with $U_t$ for any $t$.
\end{proof}

\begin{lemma}
\label{lipforpower}
Let $D$ be a regular self-adjoint operator. Then 
$$\mathrm{Lip}(D)\subseteq \mathrm{Lip}((1+D^2)^{1/2-\epsilon}),$$ 
for any $\epsilon>0$.
\end{lemma}

\begin{proof}
By definition, $\mathrm{Lip}(D)$ is a $*$-algebra. It suffices to prove the lemma for $\epsilon\in (0,1/2)$. We write 
$$(1+D^2)^{\epsilon-1/2}=\frac{\cos(\pi\epsilon)}{\pi}\int_0^\infty \lambda^{\epsilon-1/2}(1+\lambda+D^2)^{-1}\mathrm{d}\lambda,$$
which is a norm convergent integral. Using the computations of \cite[Section 2.3 and Lemma B.1]{boundarypap}, we have for $a\in\mathrm{Lip}(D)$  that 
\begin{align*}
[(1+&D^2)^{\epsilon-1/2},a]=\frac{\cos(\pi\epsilon)}{\pi}\int_0^\infty \lambda^{\epsilon-1/2}[(1+\lambda+D^2)^{-1},a]\mathrm{d}\lambda=\\
=&\frac{\cos(\pi\epsilon)}{\pi}\int_0^\infty \lambda^{\epsilon-1/2}(1+\lambda+D^2)^{-1/2}D(1+\lambda+D^2)^{-1/2}[D,a](1+\lambda+D^2)^{-1}\mathrm{d}\lambda+\\
&\frac{\cos(\pi\epsilon)}{\pi}\int_0^\infty \lambda^{\epsilon-1/2}(1+\lambda+D^2)^{-1}[D,a]D(1+\lambda+D^2)^{-1}\mathrm{d}\lambda.
\end{align*}
By standard norm estimates, as in \cite[Section 2.3 and Lemma B.1]{boundarypap}, we have that 
\begin{align*}
[(1+&D^2)^{1/2-\epsilon},a]=-(1+D^2)^{1/2-\epsilon}[(1+D^2)^{\epsilon-1/2},a](1+D^2)^{1/2-\epsilon}=\\
=&-\frac{\cos(\pi\epsilon)}{\pi}\int_0^\infty \lambda^{\epsilon-1/2}(1+D^2)^{1/2-\epsilon}(1+\lambda+D^2)^{-1/2}\\
&\qquad\qquad\qquad\qquad D(1+\lambda+D^2)^{-1/2}[D,a](1+\lambda+D^2)^{-1}(1+D^2)^{1/2-\epsilon}\mathrm{d}\lambda-\\
&-\frac{\cos(\pi\epsilon)}{\pi}\int_0^\infty \lambda^{\epsilon-1/2}(1+D^2)^{1/2-\epsilon}(1+\lambda+D^2)^{-1}\\
&\qquad\qquad\qquad\qquad[D,a]D(1+\lambda+D^2)^{-1}(1+D^2)^{1/2-\epsilon}\mathrm{d}\lambda,
\end{align*}
applied to $\xi\in \Dom((1+D^2)^{1/2-\epsilon})$ produces a norm bounded integral bounded by $\|[D,a]\|\|\xi\|$, and so $[(1+D^2)^{1/2-\epsilon},a]$ is norm bounded on $\Dom((1+D^2)^{1/2-\epsilon})$. Indeed, the integrands have norms behaving like $\lambda^{\epsilon-1/2}(1+\lambda)^{-2\epsilon-1/2}$. Since $a$ preserves $\Dom(D)=\Dom((1+D^2)^{1/2})$ which is a core for $(1+D^2)^{1/2-\epsilon}$, we have that $a\in \mathrm{Lip}((1+D^2)^{1/2-\epsilon})$.
\end{proof}

\begin{proof}[Proof of Theorem \ref{interpolationthm}]
Item (1) follows directly from functional calculus of self-adjoint regular operator. The most involved part is the proof is item (2). Before entering into the proof of item (2), we study item (3). In fact, we prove that item (3) follows from item (2). Assume that $(\mathpzc{E},D)$ is a closed chain for $(\mathcal{A},B)$ with $(\mathpzc{E},D_\alpha)$ a symmetric chain for $(\mathcal{A}_\beta,B)$. For any $T\in \mathrm{End}_B^*(\mathpzc{E})$ it is equivalent that $T(1+D_\alpha^2)^{-1}\in \K_B(\mathpzc{E})$ and $T(1+D^2)^{-\alpha}\in \K_B(\mathpzc{E})$ because $D_\alpha^2=D^2(1+D^2)^{\alpha-1}$. Moreover, $T(1+D^2)^{-\alpha}\in \K_B(\mathpzc{E})$ holds if and only if $T(1+D^2)^{-1}\in \K_B(\mathpzc{E})$, so we conclude that $a(1+D_\alpha^2)^{-1}\in \K_B(\mathpzc{E})$ for all $a\in \mathcal{A}_\beta$ (and even for all $a\in A$). This completes the proof that item (3) follows from item (2).

We now consider item (2). Notice that for any $z\in \mathbb{C}$, we can define the regular operator 
$$D_z:=D(1+D^2)^{(z-1)/2}.$$ 
To prove item (2), the theory of complex interpolation in the form of the three lines theorem implies that it suffices to prove that for some $\delta>0$ and any $\epsilon\in (0,1)$, then for any $a\in \mathcal{A}$, there are constants $C_0, C_1\geq 0$ such that the operator valued function $F_{a,\delta}(z):=\mathrm{e}^{\delta z^2}[D_z,a]$ satisfies the bounds
\begin{equation}
\label{cpxest}
\|F(z)\|_{\mathrm{End}_B^*(\mathpzc{E})}\leq C_j,
\end{equation}
for all $z$ with $\mathrm{Re}(z)=j$ for $j\in \{0,1-\epsilon\}$. If this is indeed the case, there is for any $0\leq \alpha<\beta\leq 1$ a constant $C_{\alpha,\beta}=C_{\alpha,\beta}(\delta)>0$ such that $\|[D_\alpha,a]\|_{\mathrm{End}_B^*(\mathpzc{E})}\leq C_{\alpha,\beta} \|a\|_{\mathcal{A}_\beta}$ for all $a\in \mathcal{A}$. By density, we then conclude that $\|[D_\alpha,a]\|_{\mathrm{End}_B^*(\mathpzc{E})}\leq C_{\alpha,\beta} \|a\|_{\mathcal{A}_\beta}$ for all $a\in \mathcal{A}_\beta$.

Set $\Delta:=(1+D^2)^{1/2}$ which is positive, invertible, self-adjoint, and regular since $D$ is regular. Define the contraction $F:=D\Delta^{-1}$. We can write 
$$D_z=F\Delta^z.$$
For $\mathrm{Re}(z)=0$, $D_z$ is a contraction. Hence, $\|D_z\|_{\mathrm{End}_B^*(\mathpzc{E})}\leq 1$ for $\mathrm{Re}(z)=0$, and 
$$\|[D_z, a]\|_{\mathrm{End}_B^*(\mathpzc{E})}\leq 2\|a\|_A.$$
Using that $\mathrm{e}^{\delta (it)^2}=\mathrm{e}^{-\delta t^2}$ we can take $C_0=2\|a\|_A$. 

For $z=1-\epsilon+it$ and $a\in \mathcal{A}$, we write 
$$[D_z,a]=\Delta^{it}[D_{1-\epsilon},a]+[\Delta^{it},a]D_{1-\epsilon}=\Delta^{it}[D_{1-\epsilon},a]+[\Delta^{it},a]D_{1-\epsilon}.$$
The first term is uniformly bounded in $t$ by Lemma \ref{lipforpower} and the second term is bounded by an order of $|t|$ by  Lemma \ref{cpxpowerlemma}. Since $|\mathrm{e}^{\delta (1+it)^2}|=\mathrm{e}^{\delta(1-t^2)}$, the estimate in Equation \eqref{cpxest} follows for some $C_1$, and so does item 2).
\end{proof}

\part{Bordisms in $KK$}
\label{partII}

\section{Hilsum's notion of bordism and the $KK$-bordism group}
\label{subsecglying}

In this section we recall Hilsum's notion of $KK$-bordisms of unbounded $KK$-cycles \cite{hilsumbordism} and the associated bivariant $K$-theory group \cite{DGM}. The geometric underpinning and motivation for $KK$-bordisms go back to Hilsum \cite{hilsumcmodbun,hilsumfol,hilsumbordism} and we encourage the reader to read these papers for further context and ideas. While the material in this section can be found in \cite{DGM}, the results and notations set the stage for the examples of $KK$-bordisms we discuss below in Sections \ref{exampleofboridmsmssec} and \ref{sec:ncgexbord}. The $KK$-bordism group recalled in this section from \cite{DGM} is the central object of study in the present work.

As above, we fix $C^*$-algebras $A$ and $B$ and a dense $*$-subalgebra $\mathcal{A}\subseteq A$.

\begin{define} 
\label{HilBorDef} 
A symmetric chain with boundary is a collection $\mathfrak{X}=(\mathfrak{Z},\Theta,\mathfrak{Y})$ where
\begin{enumerate}
\item The boundary cycle $\mathfrak{Y}=(\mathpzc{E}, D)$ is an odd/even $(\mathcal{A}, B)$-cycle;
\item The interior chain $\mathfrak{Z}=(\mathpzc{N}, T)$ is a symmetric even/odd $(\mathcal{A}, B)$-chain;
\item The boundary data $\Theta=(\theta,p)$, describes boundary compatibility in the sense that
\begin{itemize}
\item $p$ is a projection in $\End_B^*(\mathpzc{N})$ that commutes with the $A$-action ($p$ is even in the graded situation);
\item $\theta: p\mathpzc{N} \rightarrow L^2[0,1]\hat{\boxtimes} \mathpzc{E}$ is an isomorphism, respecting the grading when there is one (i.e. $\mathfrak{Z}$ is even and $\mathfrak{Y}$ is odd).
\end{itemize}
\end{enumerate}
These objects are compatible in the following sense described using the representation 
$$b:C([0,1],A)\to \End_B^*(\mathpzc{N}), \quad b(\phi):=\phi(1)(1-p)+p\theta^{-1} \phi\theta.$$
We require the following conditions to hold:
\begin{enumerate}
\item[a)] For $\phi\in C^\infty_c(0,1]$,  
$$b(\phi)\Dom T^*\subseteq \Dom T\quad\mbox{and}\quad T^*b(\phi)=Tb(\phi)\quad\mbox{on}\quad \Dom T^*.$$
\item[b)] For $\phi\in C^\infty_c(0,1)$,  
$$\phi \Dom \Psi(D)=\theta b(\phi)\Dom T \quad \mbox{and} \quad T=\theta^{-1}\Psi(D)\theta\quad \mbox{on}\quad b(\phi)\Dom T.$$
\item[c)] For $\phi_1,\phi_2\in C^\infty[0,1]$ satisfying $\phi_1\phi_2=0$,  
$$b(\phi_1) T b(\phi_2)=0.$$
\end{enumerate}

If additionally $\mathfrak{Z}$ is a half-closed $(C^\infty_c((0,1],\mathcal{A}),B)$-chain under $b$ \footnote{In other words, $b(\phi)(I+T^*T)^{-1}$ is compact for any $\phi\in C^\infty_c((0,1],\mathcal{A})$.} we say that $\mathfrak{X}$ is a bordism. In this case we say that the $(\mathcal{A}, B)$-cycle $\partial \mathfrak{X}:=\mathfrak{Y}$ is the boundary of $\mathfrak{X}$ and that the half-closed $(C^\infty_c((0,1],\mathcal{A}),B)$-chain $\mathfrak{X}^\circ:=\mathfrak{Z}$ is the interior of $\mathfrak{X}$. We say that $\mathfrak{X}$ is a filling bordism for its boundary $\partial \mathfrak{X}$.
\end{define}

\begin{remark}
In a symmetric chain with boundary $\mathfrak{X}=(\mathfrak{Z},\Theta,\mathfrak{Y})$, we require the parity of $\mathfrak{Z}$ to be different from the parity of $\mathfrak{Y}$. Geometrically, it means the dimension of a boundary drops by one. In terms of cycles, the parity shift also arise from the boundary data $\Theta$ implementing a ``local'' equivalence of a suspension of $\mathfrak{Y}$ with $\mathfrak{Z}$ ``near its boundary''.
\end{remark}

\begin{figure}
   \centering
  \includegraphics[height=5.35cm]{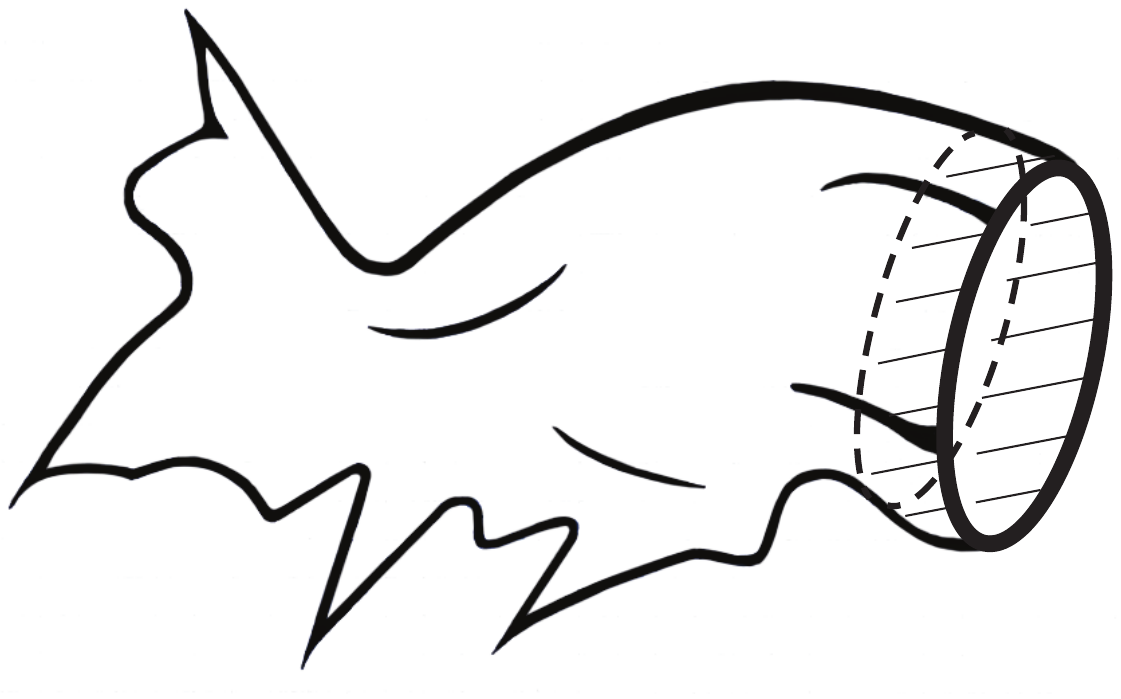}
  \caption{A graphical representation of a $KK$-bordism, with the shaded part on the right representing the collar neighbourhood of the boundary.}
\end{figure}

\begin{remark}
The reader should be made aware of the fact that for the primary relations we consider in this work, the only role of the boundary data $\Theta$ is to exist. In the upcoming work \cite{monographtwo}, we use bordisms to construct relative $KK$-groups to study secondary invariants in which case the precise boundary data will affect the class in the relative group. Therefore the boundary data is included in the notation for a bordism, even if it in all examples considered in this work is more or less clear from context. 
\end{remark}

\begin{remark}
\label{interohc}
By \cite[Lemma 5.4]{hilsumbordism}, the interior $\mathfrak{X}^\circ$ of a symmetric chain with boundary $\mathfrak{X}$ always forms a half-closed $(C^\infty_c((0,1),\mathcal{A}),B)$-chain under $b$. 
\end{remark}

\begin{remark}
It was proven in \cite[Theorem 6.2 and 7.2]{hilsumbordism} that if $\mathfrak{Y}=(\mathpzc{E},D)$ is the boundary of a bordism $\mathfrak{X}=(\mathfrak{Z},\Theta,\mathfrak{Y})$, then the bounded $KK$-cycle $(\mathpzc{E},D(1+D^2)^{-1/2})$ defines the $0$-class in $KK_*(A,B)$. The idea in the proof is that the half-closed $(C^\infty_c((0,1),\mathcal{A}),B)$-chain $\mathfrak{X}^\circ$ represents the class $(\mathpzc{E},D(1+D^2)^{-1/2})$ under $KK_*(A,B)\cong KK_{*+1}(C_0((0,1),A),B)$ and lifts to a half-closed $(C^\infty_c((0,1],\mathcal{A}),B)$-chain which vanishes as $KK_{*+1}(C_0((0,1],A),B)\cong 0$.
\end{remark}

\begin{ex}
\label{halflineex}
Let $(\mathpzc{E}, D)$ be an unbounded $(\mathcal{A},B)$-cycle. Consider the $(\mathcal{A},B)$-chain 
$$\mathfrak{Z}:=\Psi_{(0,\infty)}(\mathpzc{E},D)\equiv (L^2([0, \infty)\hat{\boxtimes}\mathpzc{E}, \partial_x^{\min}\hat{\boxtimes}D),$$ 
defined in Definition \ref{suspbyintehalf}. This chain is symmetric by Proposition \ref{suspandprophalf}, see also \cite[Proof of Lemma 2.5]{DGM}. For $\Theta=(\id_{(L^2[0,1]\hat{\boxtimes}\mathpzc{E}},\chi_{[0,1]}\otimes 1_{\mathpzc{E}})$, $(\mathfrak{Z},\Theta,\mathfrak{Y})$ is a symmetric chain with boundary. However, it is not a bordism as $b(\phi)(1+T^*T)^{-1/2}\notin \mathbb{K}_B(L^2([0, \infty)\hat{\boxtimes}\mathpzc{E})$ if $\phi(1)\neq 0$. See more in \cite[Lemma 2.5]{DGM}. One may conclude that any cycle is the boundary of a symmetric chain with boundary but a non-zero $KK$-class would obstruct the existence of a filling bordism.
\end{ex}

\begin{remark}
To save on the number of symbols used, we sometimes write $(\mathfrak{X}^\circ,\Theta,\mathfrak{X}^\partial)$ for a bordism. 
\end{remark}

\begin{define}
\label{isoofsymwbound}
Let $\mathfrak{X}_i=(\mathfrak{X}^\circ_i,\Theta_i,\mathfrak{X}^\partial_i)$, $i=1,2$ be symmetric $(\mathcal{A},B)$-chains with boundary. An isomorphism $\pmb{\alpha}:\mathfrak{X}_1\xrightarrow{\sim}\mathfrak{X}_2$ is a collection $\pmb{\alpha}=(\alpha^\circ,\alpha^\partial)$ where $\alpha^\circ:\mathfrak{X}^\circ_1\to \mathfrak{X}^\circ_2$ and $\alpha^\partial:\mathfrak{X}^\partial_1\to \mathfrak{X}^\partial_2$ are isomorphisms (see Definition \ref{UnbKKcycDef}) compatible in the sense that $\alpha^\circ p_1=p_2\alpha^\circ$ and the following diagram commutes:
$$\begin{CD}
p_1\mathpzc{N}_1 @>\alpha^\circ|>\cong > p_2\mathpzc{N} \\
@V\theta_1V\cong V @V\cong V\theta_2 V \\
L^2[0,1]\hat{\boxtimes}\mathpzc{E}_1 @>\id_{L^2[0,1]}\boxtimes\alpha^\partial>\cong > L^2[0,1]\hat{\boxtimes}\mathpzc{E}_2.
\end{CD}$$
Here $\mathfrak{X}^\circ_i=(\mathpzc{N}_i,T_i)$, $\Theta_i=(\theta_i,p_i)$ and $\mathfrak{X}^\partial_i=(\mathpzc{E}_i,D_i)$.
\end{define}

Recall the operations on chains from Definition \ref{oppKKCyc}.

\begin{define} 
\label{oppKKBor}
For a symmetric chain with boundary $\mathfrak{X}=(\mathfrak{X}^\circ,\Theta,\mathfrak{X}^\partial)$, we define the opposite symmetric chain with boundary
$$-\mathfrak{X}:=(-\mathfrak{X}^\circ,\Theta,-\mathfrak{X}^\partial).$$
For two symmetric $(\mathcal{A},B)$-chains with boundary $\mathfrak{X}_i=(\mathfrak{X}^\circ_i,\Theta_i,\mathfrak{X}^\partial_i)$, $i=1,2$, we define their direct sum to be 
$$\mathfrak{X}_1+ \mathfrak{X}_2:=(\mathfrak{X}^\circ_1+ \mathfrak{X}^\circ_2,\Theta_1\oplus \Theta_2,\mathfrak{X}^\partial_1+ \mathfrak{X}^\partial_2).$$
\end{define}

Hilsum \cite[Section 10]{hilsumbordism} has shown that one can glue two $KK$-bordisms with common boundary to produce an honest unbounded $KK$-cycle. A crucial point in proving that $KK$-bordism is an equivalence relation is to show that one can glue $KK$-bordisms along common summands in their boundary and obtain a new $KK$-bordism, see \cite[Section 2.4]{DGM}. We will need this construction, and review it shortly here.

We will use the following notation: given $KK$-bordisms $\mathfrak{X}=(\mathfrak{X}^\circ,\Theta,\mathfrak{X}^\partial)$ and $\mathfrak{X}'=(\mathfrak{X}'^\circ,\Theta',-\mathfrak{X}^\partial)$ with common boundary $\mathfrak{X}^\partial$ we let 
\begin{equation} \label{glueToGetClosed}
\mathfrak{X} \#_{\mathfrak{X}^\partial}\mathfrak{X}'
\end{equation}
denote the unbounded $KK$-cycle obtained from the process discussed in \cite[Section 10]{hilsumbordism}.

We will also use the generalization of this gluing construction developed in \cite[Section 2.4]{DGM}. More specifically, we need \cite[Theorem 2.20]{DGM}, which briefly put says the following: suppose $\mathfrak{X}=(\mathfrak{X}^\circ,\Theta,\mathfrak{X}^\partial)$ and $\mathfrak{X}'=(\mathfrak{X}'^\circ,\Theta',\mathfrak{X}'^\partial)$ are two $KK$-bordisms, and
$$\mathfrak{X}^\partial=\mathfrak{Y}_1+\mathfrak{Y}_2\quad\mbox{and}\quad \mathfrak{X}'^\partial=-\mathfrak{Y}_2+\mathfrak{Y}_3,$$
for some closed cycles $\mathfrak{Y}_1$, $\mathfrak{Y}_2$, and $\mathfrak{Y}_3$. Then there is an explicitly constructed bordism $\mathfrak{X}''$ such that 
$$\partial \mathfrak{X}''=\mathfrak{Y}_1+\mathfrak{Y}_3.$$
A picture describing this setup is found in Figure \ref{section25gluing}. Of course, the statement ``explicitly constructed" is informal; the reader can see the statement of \cite[Theorem 2.20]{DGM} for the precise details of the construction. The idea is to glue together $\mathfrak{X}^\circ$ with $\mathfrak{X}'^\circ$ along the tubular neighborhood (defined from the boundary data) of the component of the boundary corresponding to $\mathfrak{Y}_2$. As in Equation \eqref{glueToGetClosed}, we denote the bordism obtained from this construction by 
\begin{equation} 
\label{glueToGetBordism}
\mathfrak{X} \#_{\mathfrak{Y}_2}\mathfrak{X}'
\end{equation}
\begin{figure}
   \centering
  \includegraphics[height=5.35cm]{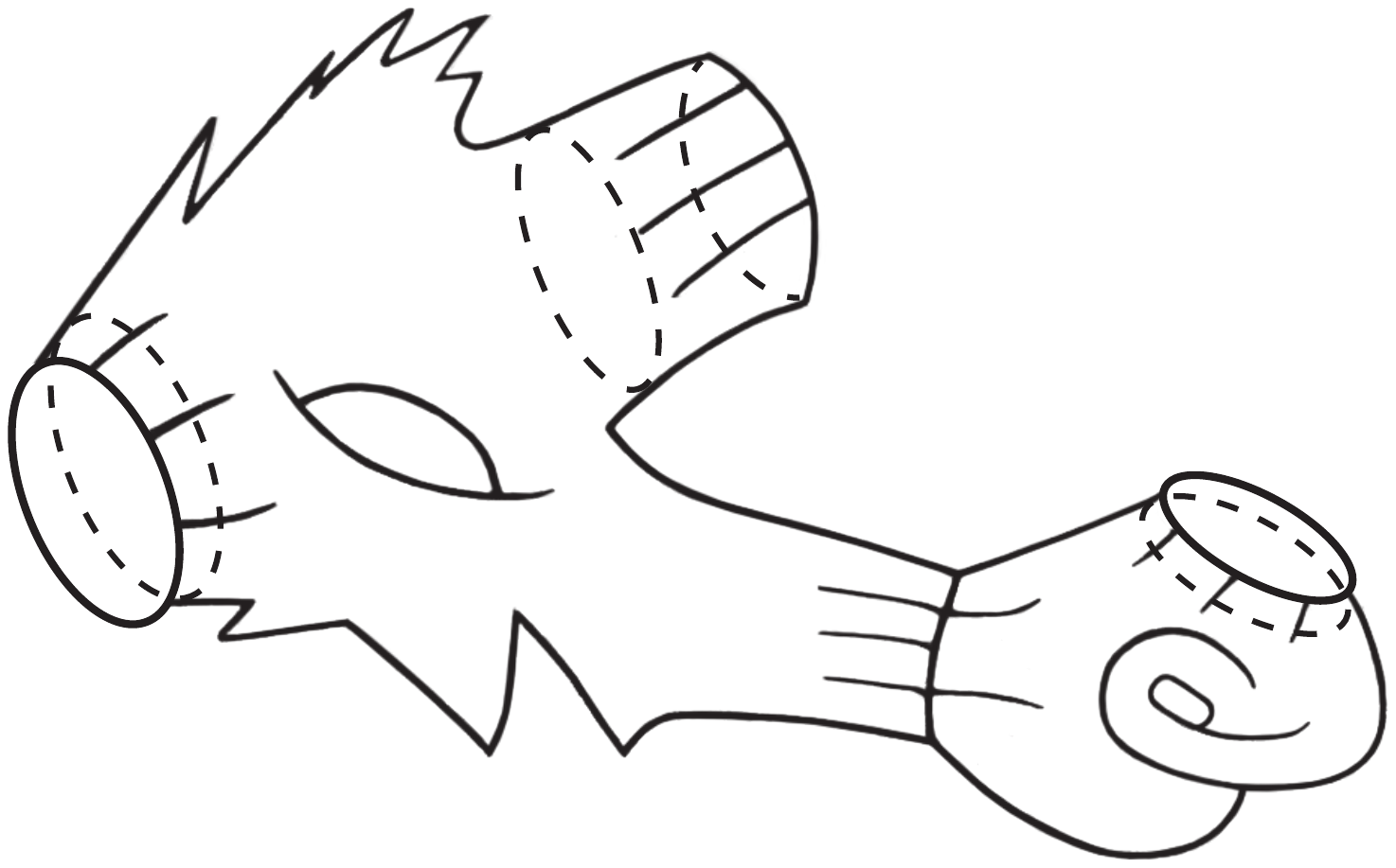}
 \caption{Gluing of two $KK$-bordisms $\mathfrak{X}$ and $\mathfrak{X}'$ along a common boundary $\mathfrak{Y}_2$.}
\label{section25gluing}
\end{figure}

We formalize this discussion in a proposition.

\begin{prop}
\label{subsecglyinggluing}
Let $\mathfrak{Y}_1$, $\mathfrak{Y}_2$ and $\mathfrak{Y}_3$ be $(\mathcal{A},B)$-cycles and $\mathfrak{X}$ and $\mathfrak{X}'$ be $(\mathcal{A},B)$-bordisms. If $\mathfrak{Y}_1\sim_{\rm bor} \mathfrak{Y}_2$ via $\mathfrak{X}$, i.e. $\partial \mathfrak{X}=\mathfrak{Y}_1+(-\mathfrak{Y}_2)$, and $\mathfrak{Y}_2\sim_{\rm bor} \mathfrak{Y}_3$ via $\mathfrak{X}'$, i.e. $\partial \mathfrak{X}'=\mathfrak{Y}_2+(-\mathfrak{Y}_3)$, then $\mathfrak{Y}_1\sim_{\rm bor} \mathfrak{Y}_3$ via $\mathfrak{X}\#_{\mathfrak{Y}_2}\mathfrak{X}'$.
\end{prop}

Let us now define the bordism relation on the set of isomorphism classes of $(\mathcal{A},B)$-cycles.

\begin{define}
\label{HilBorKKcycle}
Two $(\mathcal{A},B)$-cycles $\mathfrak{Y}_1$ and $\mathfrak{Y}_2$ are bordant if there exists a $KK$-bordism $\mathfrak{X}$ over $(\mathcal{A}, B)$, with $\partial \mathfrak{X}=\mathfrak{Y}_1+(-\mathfrak{Y}_2)$.
\end{define}

It was proven in \cite[Proposition 2.22]{DGM} that $KK$-bordism is an equivalence relation. Let us sketch why this is the case in terms of results stated below in this monograph:
\begin{itemize}
\item (Transitive) Follows from the gluing construction above in Proposition \ref{subsecglyinggluing}. 
\item (Reflexive) For any cycle $\mathfrak{Y}$, the cylinder construction $\Psi(\mathfrak{Y})$ (see Definition \ref{suspbyinte}) defines a bordism $\mathfrak{Y}\sim_{\rm bor}\mathfrak{Y}$.
\item (Symmetric) If $\mathfrak{Y}_1\sim_{\rm bor} \mathfrak{Y}_2$ via $\mathfrak{X}$, so $\partial \mathfrak{X}=\mathfrak{Y}_1+(- \mathfrak{Y}_2)$, then $\partial (-\mathfrak{X})= (-\mathfrak{Y}_1)+\mathfrak{Y}_2$ and $\mathfrak{Y}_2\sim_{\rm bor} \mathfrak{Y}_1$.
\end{itemize}

\begin{define} (See \cite[Definition 2.23]{DGM})
\label{defofbordgroup}
We define the $\Z/2\Z$-graded abelian group $\Omega_*( \mathcal{A}, B)$ as the set of bordism classes of $(\mathcal{A},B)$-cycles with addition defined from the direct sum $+$ of cycles. 
\end{define}

It is not immediate that $\Omega_*( \mathcal{A}, B)$ is a well-defined abelian group. This fact is proven in \cite[Theorem 2.24]{DGM}. To summarize this argument, it is clear that direct sum defines the structure of a $\Z/2\Z$-graded semigroup on $\Omega_*( \mathcal{A}, B)$ and since the bordism relation is reflexive, it follows that for any cycle $\mathfrak{Y}$ there is a bordism $\mathfrak{Y}+(-\mathfrak{Y})\sim_{\rm bor}0$ so any element admits an inverse. The explicit additive inverse is given at the level of cycles by $-[\mathfrak{Y}]=[-\mathfrak{Y}]$, where $[\cdot]$ denotes the $KK$-bordism class.

\begin{remark}
The $KK$-bordism group $\Omega_*( \mathcal{A}, B)$ relates to Kasparov's $KK$-group $KK_*(A,B)$ via the bounded transform $\beta:(\mathpzc{E},D)\mapsto (\mathpzc{E},D(1+D^2)^{-1/2})$. For more details, see below in Subsection \ref{subsecbddtrans}. In several instances, the bounded transform induces an isomorphism $\Omega_*( \mathcal{A}, B)\cong KK_*(A,B)$ as we will show in Part \ref{partoniso} below.
\end{remark}

\section{Examples of $KK$-bordisms from technical origins}
\label{exampleofboridmsmssec}

We consider several general classes of $KK$-bordisms that will play a role in the abstract study of the $KK$-bordism groups. We study classical equivalences coming from perturbations in Subsection \ref{subsec:perturbyop} as well as homotopies; in Subsection \ref{subsec:homotopyop} we study homotopies of the operator and in Subsection \ref{subsec:homotoperep} we study homotopies of the left action. We also study $KK$-bordisms whose origins are not classical; we show that weakly degenerate cycles produce a nullbordism in Subsection \ref{subsecweakdegff} and how complex interpolation allows us to reduce the order of operators in Subsection \ref{orderredusubsec}. In each of these cases, the situation at hand produces a canonical bordism.

\subsection{Homotopy of the operator}
\label{subsec:homotopyop}

A natural source of examples for $KK$-bordisms comes from paths $(D(t))_{t\in [0,1]}$ of operators in a fixed Hilbert $C^*$-module. To ensure that such paths indeed induce $KK$-bordisms, one can impose a technical condition similar to differentiability. In practice, such conditions are often utilized in commutator estimates for the derivative in the $t$-direction in the spirit of Kaad-Lesch \cite{leschkaad2} and later extended in \cite{DGM}. We first state the general technical result, after which we indicate how differentiability enters the game.

\begin{theorem}
\label{homotopywithsweakcond}
Let $\mathpzc{E}$ and $\mathpzc{W}$ be Hilbert $C^*$-modules and $i:\mathpzc{W}\to \mathpzc{E}$ an $\mathcal{A}$-locally compact dense adjointable inclusion. Assume that $\overline{D}=(D(t))_{t\in [0,1]}$ is a family of regular, self-adjointable operators satisfying 
\begin{enumerate}
\item $D(t)$ is constant on $[0,1/2]\cup [2/3,1]$
\item $\Dom(D(t))\subseteq \mathpzc{W}$ for all $t\in [0,1]$
\item $\mathcal{A}\subseteq \cap_{t\in [0,1]}\mathrm{Lip}(D(t))$
\item The operator $\partial^{\rm min}\hat{\boxtimes} \overline{D}$ defined on 
\[ \{\xi\in C^\infty_c((0,1),\mathpzc{W}): \xi(t)\in \Dom(D(t)) \hbox{ for all }t\in [0,1]\} \] is closable. Moreover, its closure $T$ is regular and 
\[ \{\xi\in C^\infty([0,1],\mathpzc{W}): \xi(t)\in \Dom(D(t))\hbox{ for all } t\in [0,1]\} 
\] is a core for its adjoint.
\end{enumerate}
Then $(\mathpzc{E},D(t))$ is a cycle for $(\mathcal{A},B)$ for all $t\in [0,1]$ and 
$$((L^2[0,1]\hat{\boxtimes}\mathpzc{E}, T),(\id, \chi_{[0,1/4]\cup[3/4,1]})),$$ 
defines a bordism 
$$(\mathpzc{E},D(0))\sim_{\rm bor}(\mathpzc{E},D(1)).$$
\end{theorem}

We remark that the conditions of Theorem \ref{homotopywithsweakcond} are a variation of the notion of unbounded operator homotopy from \cite{meslandvandung}. We shall below in Proposition \ref{homotopywithstrongcond} see that an unbounded operator homotopy satisfying an additional differentiability assumption induces a family as in Theorem \ref{homotopywithsweakcond}. Due to the technical nature of condition 4) it is unclear if all unbounded operator homotopies induce families as in Theorem \ref{homotopywithsweakcond}.

\begin{proof}
We first prove that $(\mathpzc{E},D(t))$ is a cycle for $(\mathcal{A},B)$ for all $t\in [0,1]$. For any $a\in \mathcal{A}$, item 2) ensures that the operator $(i\pm D(t))^{-1}$ on $\mathpzc{E}$ factors over $i:\mathpzc{W}\to \mathpzc{E}$ which by assumption ensures that $D(t)$ has $\mathcal{A}$-locally compact resolvent. In combination with item 3), we deduce that $\mathcal{A}\subseteq \mathrm{Lip}_0(D(t))$ for all $t\in [0,1]$. Since $D(t)$ is self-adjoint by assumption, $(\mathpzc{E},D(t))$ is a cycle.

To prove that $((L^2[0,1]\hat{\boxtimes}\mathpzc{E}, T),(\id, \chi_{[0,1/4]\cup[3/4,1]}))$ defines a bordism 
\[ (\mathpzc{E},D(0))\sim_{\rm bor}(\mathpzc{E},D(1)), \]
it suffices to prove that $(L^2[0,1]\hat{\boxtimes}\mathpzc{E}, T)$ is a half-closed $(C^\infty_c((0,1],\mathcal{A}),B)$-chain. The action of $C^\infty_c((0,1],\mathcal{A})$ that the boundary data induces on $L^2[0,1]\hat{\boxtimes}\mathpzc{E}$ factors over the natural action of $C_c((0,1),\mathcal{A})\cap \mathrm{Lip}([0,1],\mathcal{A})$ on $L^2[0,1]\hat{\boxtimes}\mathpzc{E}$. We note that the domain of $T$ is by continuity contained in $H^1_0[0,1]\hat{\boxtimes}\mathpzc{E}\cap L^2[0,1]\hat{\boxtimes}\mathpzc{W}$ which by local compactness of $i:\mathpzc{W}\to \mathpzc{E}$ is locally compactly included in $L^2[0,1]\hat{\boxtimes}\mathpzc{E}$. Combining these two facts with item 4), it is immediate that $(L^2[0,1]\hat{\boxtimes}\mathpzc{E}, T)$ is a half-closed $(C^\infty_c((0,1],\mathcal{A}),B)$-chain.
\end{proof}

Let us state a result that is useful in proving condition 4) of Theorem \ref{homotopywithsweakcond}. It follows from \cite[Theorem 1.3]{DGM}. 

\begin{lemma}
\label{lemmahomotopyresolven}
Let $\mathpzc{E}$ be a Hilbert $C^*$-module. Assume that $\overline{D}=(D(t))_{t\in [0,1]}$ is a family of regular, self-adjointable operators sharing a common core $\mathcal{D}_0$. We further assume that for any $\mu\in \mathbb{R}\setminus \{0\}$, the function $R_\mu:[0,1]\to \End_B^*(\mathpzc{E})$, $R_\mu(t):=(i\mu+ D(t))^{-1}$ is norm-differentiable and for each $t\in [0,1]$, $D(t)R_\mu'(t)$ extends to an adjointable operator and $t\mapsto D(t)R_\mu'(t)$ defines a norm-continuous function $[0,1]\to \End_B^*(\mathpzc{E})$.

Then the operator $\partial^{\rm min}\hat{\boxtimes} \overline{D}$ defined on $C^\infty_c((0,1),\mathcal{D}_0)$ is closable, and its closure $T$ is regular and $C^\infty([0,1],\mathcal{D}_0)$ is a core for its adjoint.
\end{lemma}

The following result follows from \cite[Proposition 2.13]{DGM}. 

\begin{prop}
\label{homotopywithstrongcond}
Let $\mathpzc{E}$ and $\mathpzc{W}$ be Hilbert $C^*$-modules and $i:\mathpzc{W}\to \mathpzc{E}$ an $\mathcal{A}$-locally compact dense adjointable inclusion. Assume that $D\in C^1([0,1],\Hom_B^*(\mathpzc{W},\mathpzc{E}))$ satisfies that $D(t)$ viewed as a densely defined operator on $\mathpzc{E}$ defines an $(\mathcal{A},B)$-cycle $(\mathpzc{E},D(t))$ for all $t\in [0,1]$. Then $\overline{D}:=(D(t))_{t\in [0,1]}$ satisfies the assumptions of Theorem \ref{homotopywithsweakcond}. In particular, 
$$(\mathpzc{E},D(0))\sim_{\rm bor}(\mathpzc{E},D(1)).$$
\end{prop}

\subsection{Perturbations}
\label{subsec:perturbyop}

Another natural source of examples of $KK$-bordisms arise from perturbations of the operator $D$ appearing in a cycle $(\mathpzc{E},D)$. As in the previous subsection, we give a rather general theorem depending on a technical assumption after which we specialize to cases where the technical assumptions are simplified into more conceptual assumptions. At the end we consider some corollaries that will play a role later in the monograph.

\begin{theorem}
\label{perturbationgenerelthm}
Let $(\mathpzc{E},D)$ and $(\mathpzc{E},D')$ be two $(\mathcal{A},B)$-cycles defined on the same $B$-Hilbert $C^*$-module $\mathpzc{E}$. Assume the following:
\begin{itemize}
\item $D$ and $D'$ share a common core $\mathcal{D}_0$.
\item For any $s\in [0,1]$, the operator $sD+(1-s)D'$ is regular and self-adjoint.
\end{itemize}
Then for any $\chi\in C^\infty_c([0,1),[0,1])$ with $\chi(0)=1$ and $\chi'\in C^\infty_c(0,1)$, the path 
$$D(t):=\chi(t)D+(1-\chi(t))D',$$
satisfies the assumptions of Lemma \ref{lemmahomotopyresolven} and Theorem \ref{homotopywithsweakcond}. In particular, the bordism constructed in Theorem \ref{homotopywithsweakcond} implements a bordism
$$(\mathpzc{E},D)\sim_{\rm bor}(\mathpzc{E},D').$$
\end{theorem}

\begin{proof}
Since $sD+(1-s)D'$ is a regular self-adjoint operator for any $s\in [0,1]$, $D(t)$ will for any $t\in [0,1]$ be regular and self-adjoint. On the common core, it is readily verified that $R_\mu$ satisfies that 
\begin{align*}
R_\mu'(t)=&(i\mu+D(t))^{-1}(\chi'(t)D'-\chi'(t)D)(i\mu+D(t))^{-1}=\\
=&\chi'(t)(i\mu+D(t))^{-1}(D'-D)(i\mu+D(t))^{-1}.
\end{align*}
We can conclude that $R_\mu$ is continuous in $t$ from the fact that $\Dom(D(t))=\Dom D\cap \Dom D'$ for $t\in \mathrm{supp}(\chi')\subseteq (0,1)$. Since 
$$(i\mu+D(t))R_\mu'(t)=\chi'(t)(D'-D)(i\mu+D(t))^{-1},$$
it follows that $(i\mu+D(t))R_\mu'(t)$ gives a well defined continuous mapping $[0,1]\to \End_B^*(\mathpzc{E})$. Combining Lemma \ref{lemmahomotopyresolven} with Theorem \ref{homotopywithsweakcond}, we see that $(\mathpzc{E},D)\sim_{\rm bor}(\mathpzc{E},D')$.
\end{proof}

We recall that if $V$ and $D$ are symmetric, regular operators with $\Dom(V)\supseteq \Dom(D)$, we say that $V$ is relatively bounded by $D$ with norm bound $a>0$ if for some $b>0$ and any $\xi\in \Dom(D)$, 
$$\langle V\xi,V\xi\rangle \leq a\langle D\xi,D\xi\rangle+b\langle \xi,\xi\rangle.$$

\begin{cor}
\label{locbddperturcor}
Let $(\mathpzc{E},D)$ be a cycle for $(\mathcal{A},B)$ and $V$ a symmetric, regular operator on $\mathpzc{E}$ which is relatively bounded with norm bound $<1$ such that $\mathcal{A}\subseteq \mathrm{Lip}(V)$. Then $(\mathpzc{E},D+V)$ is a cycle for $(\mathcal{A},B)$. Moreover, for any $\chi\in C^\infty_c([0,1),[0,1])$ with $\chi(0)=1$ and $\chi'\in C^\infty_c(0,1)$ the path 
$$D(t):=D+\chi(t)V,$$
satisfies the assumptions of Theorem \ref{homotopywithsweakcond} and induces a bordism 
$$(\mathpzc{E},D)\sim_{\rm bor}(\mathpzc{E},D+V).$$
\end{cor}

\begin{proof}
By the Kato-Rellich theorem (see \cite[Theorem 4.4]{leschkaad2}), $D+V$ is self-adjoint and regular with $\Dom(D+V)=\Dom(D)$, so $\Dom(D)$ is a common core for $D$ and $D+V$. In fact, the same holds for $D+sV$ any $s\in [0,1]$ so the corollary now follows from Theorem \ref{perturbationgenerelthm}.
\end{proof}

Specific $KK$-bordisms, such as that appearing in Corollary \ref{locbddperturcor}, will play a role in explicit relative constructions, so we make the following definition.

\begin{define}
\label{transgres}
Let $(\mathpzc{E},D)$ be a cycle for $(\mathcal{A},B)$ and $V$ a symmetric, regular operator on $\mathpzc{E}$ which is relatively bounded with norm bound $<1$. We denote the bordism from Corollary \ref{locbddperturcor} by $\mathcal{T}(\mathpzc{E},D,V)$; we call it the transgression bordism of $(\mathpzc{E},D)$ and $V$. 
\end{define}

\begin{define}
\label{locbbddd}
A symmetric regular operator $V$ is said to be $\mathcal{A}$-locally bounded if for any $a\in \mathcal{A}$, $a\Dom(V)\subseteq \Dom(V)$ and $aV$ and $Va$ extend to bounded adjointable operators. When $\mathcal{A}$ is clear from context, we just say that $V$ is locally bounded.
\end{define}

\begin{cor}[Invertible double]
\label{invertibleldodcor}
Let $(\mathpzc{E},D)$ be a cycle for $(\mathcal{A},B)$ and define the operator $\tilde{D}$ on $\tilde{\mathpzc{E}}:=\mathpzc{E}\oplus \mathpzc{E}$ as 
$$\tilde{D}:=
\begin{pmatrix}
D&(1+D^2)^{-1/2}\\
(1+D^2)^{-1/2}& -D
\end{pmatrix}$$
Equip $\tilde{\mathpzc{E}}$ with the left $A$-action defined by $a.(\xi_1\oplus \xi_2):=a\xi_1\oplus 0$. Then $(\tilde{\mathpzc{E}},\tilde{D})$ is an $(\mathcal{A},B)$-cycle satisfying 
\begin{enumerate}
\item $\tilde{D}$ is invertible.
\item $\tilde{D}$ is a locally bounded perturbation of $D\oplus 0$ on $\tilde{\mathpzc{E}}$. 
\item For any $\chi\in C^\infty_c([0,1),[0,1])$ with $\chi(0)=1$ and $\chi'\in C^\infty_c(0,1)$ the path
$$\tilde{D}(t):=
\begin{pmatrix}
D&\chi(t)(1+D^2)^{-1/2}\\
\chi(t)(1+D^2)^{-1/2}& -\chi(t)D
\end{pmatrix}$$
satisfies the assumptions of Theorem \ref{homotopywithsweakcond}.
\end{enumerate}
In particular, $D\oplus 0$ and $\tilde{D}$  satisfy the assumptions of Theorem \ref{perturbationgenerelthm}, and so Theorem \ref{homotopywithsweakcond}  induces a bordism
$$(\mathpzc{E},D)\sim_{\rm bor}(\tilde{\mathpzc{E}},\tilde{D}).$$
\end{cor}

We remark that below, in Corollary \ref{existenceofveryfoidleled} on page \pageref{existenceofveryfoidleled}, we show that doubling is not needed for there to exist a bounded perturbation of $D$ making it invertible as soon as $\A$ acts unitally and $D$ satisfies the technical assumption of being very full (see Definition \ref{veryfodldleed} on page \pageref{veryfodldleed}).

\begin{proof}
It is readily verified that $\tilde{D}$ is invertible. The operator $\tilde{D}$ is even a locally compact perturbation of $D\oplus 0$ because for any $a\in \mathcal{A}$, 
$$a(\tilde{D}-D\oplus 0)=\begin{pmatrix}
0&a(1+D^2)^{-1/2}\\
0& 0
\end{pmatrix}
\quad\mbox{and}\quad 
(\tilde{D}-D\oplus 0)a=
\begin{pmatrix}
0&0\\
(1+D^2)^{-1/2}a& 0
\end{pmatrix}.$$
Item 3) follows using Lemma \ref{lemmahomotopyresolven}, and the proof of the corollary is complete.
\end{proof}

\begin{cor}
\label{locbddanticommcor}
Let $(\mathpzc{E},D)$ be a cycle for $(\mathcal{A},B)$ and $V$ a self-adjoint, regular, $\mathcal{A}$-locally bounded operator on $\mathpzc{E}$. Assume that 
\begin{enumerate}
\item $(i\pm D)^{-1}\Dom V\subseteq \Dom V\cap \Dom(DV)$ 
\item $(DV+VD)(i\mu+D)^{-1}$ has an adjointable extension to $\mathpzc{E}$ for all $\mu\in \R\setminus\{0\}$
\end{enumerate}
Then $(\mathpzc{E},D+V)$ is a cycle for $(\mathcal{A},B)$. Moreover, for any $\chi\in C^\infty_c([0,1),[0,1])$ with $\chi(0)=1$ and $\chi'\in C^\infty_c(0,1)$ the path 
$$D(t):=D+\chi(t)V,$$
satisfies the assumptions of Theorem \ref{homotopywithsweakcond} and induces a bordism 
$$(\mathpzc{E},D)\sim_{\rm bor}(\mathpzc{E},D+V).$$
\end{cor}

\begin{proof}
The assumptions on $D$ and $V$ ensure that $D+sV$ is regular and self-adjoint on $\Dom(D)\cap \Dom(sV)$ for all $s\in [0,1]$ (see \cite[Theorem 4.4]{leschkaad2} or \cite[Theorem 1.3]{DGM}). Furthermore, the assumptions guarantee that the module $\Dom(D)\cap \Dom(V)$ is a common core for $D+sV$ for all $s\in [0,1]$. Since $V$ is locally bounded, it follows that $a\Dom(D+V)=a\Dom(D)$ for all $a\in \mathcal{A}$, and moreover that $(\mathpzc{E},D+V)$ is a cycle for $(\mathcal{A},B)$. As such, the assumptions of Theorem \ref{perturbationgenerelthm} hold and the assumptions of Theorem \ref{homotopywithsweakcond} follow. 
\end{proof}

\begin{remark}
\label{commentoncprreosls}
Assume that $(\mathpzc{E},D)$ and $V$ are as in Corollary \ref{locbddanticommcor}. We note that if $(i\pm V)^{-1}(i\pm D)^{-1}$ is a compact operator then so is $(i\pm(D+V))^{-1}$. This fact follows from that $(i\pm(D+V))^{-1}=\lim_{n\to \infty} (i+\frac{1}{n}V)^{-1}(i\pm(D+V))^{-1}$ in norm and $(i+\frac{1}{n}V)^{-1}(i\pm(D+V))^{-1}$ factors over 
$$\Dom(D)\hookrightarrow \mathpzc{E}\xrightarrow{(i+\frac{1}{n}V)^{-1}} \mathpzc{E},$$
which is compact if $(i\pm V)^{-1}(i\pm D)^{-1}$ is compact. In particular, if a cycle $(\mathpzc{E},D)$ admits a $V$ as in Corollary \ref{locbddanticommcor} with $(i\pm V)^{-1}(i\pm D)^{-1}$ being compact, then $(\mathpzc{E},D)\sim_{\rm bor} (\mathpzc{E},D+V)$ with $D+V$ being $B$-Fredholm. 

In \cite{vandungenodd}, a construction producing such a $V$ from an approximate identity (satisfying additional properties) was given.
\end{remark}

\subsection{Weakly degenerate cycles}
\label{subsecweakdegff}

The $KK$-bordisms in unbounded $KK$-theory encompasses a wider notion of bordism than that coming from algebraic topology. Indeed, we shall here give an example of a $KK$-bordism that we below in Corollary \ref{coronvecbunhvhvjhv} use to implement vector bundle modification in geometric $K$-homology. Vector bundle modification is the main difference between the bordism relation defining bordism groups and the Baum-Douglas relation defining $K$-homology groups. We make the following definition motivated by \cite[Definition 3.1]{DGM}.

\begin{define}
\label{lknalkadn}
Let $(\mathcal{E},D)$ be an $(\mathcal{A},B)$-chain. A weakly degenerate decomposition of $(\mathcal{E},D)$ is pair of regular operators $(D_0,S)$ such that $D=D_0+S$, implicitly it thus holds that $\Dom(D)=\Dom(D_0)\cap \Dom(S)$, and additionally that 
\begin{enumerate}
\item $S$ is self-adjoint and admits a bounded inverse, $\Dom(S)$ is preserved by $A$ and $A$ commutes with $S$ on $\Dom(S)$.
\item There is a common core $X\subseteq \Dom(S)\cap \Dom(D_0)$ for $S$ and $D_0$ with $SX_0\subseteq \Dom(D_0)$, $D_0X_0\subseteq \Dom(S)$ and 
$$SD_0+D_0S=0\quad\mbox{on}\quad X_0.$$
\item $\mathcal{A}\subseteq \mathrm{Lip}(D_0)$. 
\end{enumerate}
A chain admitting a weakly degenerate decomposition is called weakly degenerate.
\end{define}

Let $X$ denote the self-adjoint densely defined multiplication operator on $L^2[0,\infty)$ defined from a smooth increasing function $X:[0,\infty)\to [1,\infty)$ which is constant equal to $1$ on $[0,1]$ and coincides with $\sqrt{1+x^2}$ for $x$ large enough. If $(\mathpzc{E},D)$ is a cycle with a weakly degenerate decomposition $D=D_0+S$, we define the symmetric operator $D_{sh}$ on $\Psi_{[0,\infty)}(\mathpzc{E})$ as 
$$D_{sh}:=\Psi_{[0,\infty)}(D_0)+X(0\hat{\boxtimes}S).$$
The domain of $D_{sh}$ is given by the elements $\xi\in (H^1_0[0,\infty)\hat{\boxtimes}\mathpzc{E})\cap (L^2[0,\infty)\hat{\boxtimes}\Dom(D))\subseteq \Psi_{[0,\infty)}(\mathpzc{E})$ satisfying the condition $\xi(0)=0$ and 
$$\int_0^\infty\left(\|\xi'(t)\|_{\mathpzc{E}}^2+\|D_0\xi(t)\|_\mathpzc{E}^2+X(t)^2\|S\xi(t)\|_\mathpzc{E}^2\right)\mathrm{d}t<\infty.$$
It was proven in \cite[Theorem 3.4]{DGM} that $D_{sh}$ is a symmetric and regular operator. In fact, $((\Psi_{[0,\infty)}(\mathpzc{E}),D_{sh}),(\id,\chi_{[0,1]}))$ was proven to be a bordism with boundary $(\mathpzc{E},D)$ in \cite[Theorem 3.4]{DGM}. The chain $(\Psi_{[0,\infty)}(\mathpzc{E}),D_{sh})$ should be thought of as the Shubin suspension $\Psi_{Sh,[0,\infty)}(\mathpzc{E},D)$ from \ref{suspbyinteShhalf} but ``straightened out'' at the boundary. 

\begin{theorem}
\label{weakdegnullbordthm}
If $(\mathcal{E},D)$ is a weakly degenerate cycle then 
$$(\mathpzc{E},D)\sim_{\rm bor}0.$$ 
More precisely, this bordism is implemented by the Shubin bordism 
$$\mathrm{Sh}(\mathpzc{E},D=D_0+S):=((\Psi_{[0,\infty)}(\mathpzc{E}),D_{sh}),(\id,\chi_{[0,1]}),(\mathpzc{E},D))$$ 
constructed from a weakly degenerate decomposition.  
\end{theorem}

Weakly degenerate cycles play an important role in connecting unbounded $KK$-theory and geometric $K$-homology. Below in Subsection \ref{diraconcstarbundleexbordism}, we shall see how all the constituents of the Baum-Douglas relation defining geometric $K$-homology fit together with $KK$-bordisms and in particular in Corollary \ref{coronvecbunhvhvjhv} we see how Theorem \ref{weakdegnullbordthm} relates vector bundle modification to $KK$-bordisms. Let us turn to Kaad's notion of spectrally decomposable cycles and show that they are bounded perturbations of special type of weakly degenerate cycles. First, we recall what a spectrally decomposable cycle is.

\begin{define}[Definition 4.1 of \cite{Kaadunbdd}]
Let $(\mathpzc{E},D)$ be a cycle for $(\mathcal{A},B)$. We say that $(\mathpzc{E},D)$ is spectrally decomposed by $P$ if $P$ is an orthogonal projection on $\mathpzc{E}$ preserving $\Dom(D)$ and satisfying 
\begin{enumerate}
\item $P$ commutes with the $A$-action
\item $DP=PD$ on $\Dom(D)$
\item $DP$ and $D(P-1)$ are positive and regular operators
\item If $(\mathpzc{E},D)$ is an even cycle, the operator $2P-1$ is odd.
\end{enumerate}
If $(\mathpzc{E},D)$ admits a spectrally decomposing projection, we say that $(\mathpzc{E},D)$ is spectrally decomposable.
\end{define}

\begin{theorem}
\label{specdecompnulbord}
Let $(\mathpzc{E},D)$ be a cycle for $(\mathcal{A},B)$ which is spectrally decomposed by $P$. Then $(\mathpzc{E},D+i\gamma(2P-1))$ has a weakly degenerate decomposition given by $(D,i\gamma(2P-1))$. In particular, a spectrally decomposable cycle $(\mathpzc{E},D)$ for $(\mathcal{A},B)$ satisfies that 
$$(\mathpzc{E},D)\sim_{\rm bor} 0,$$
via the ``bounded-pertubation-bordism'' $(\mathpzc{E},D)\sim_{\rm bor}(\mathpzc{E},D+i\gamma(2P-1))$ of Corollary \ref{locbddperturcor} glued together with the ``weakly-degenerate-bordism'' $(\mathpzc{E},D+i\gamma(2P-1))\sim_{\rm bor} 0$ from Theorem \ref{weakdegnullbordthm}.
\end{theorem}

\begin{proof}
The conditions on a spectral decomposition ensures that $(D,i\gamma(2P-1))$ is a weakly degenerate decomposition of $D+i\gamma(2P-1)$ and the theorem follows from Corollary \ref{locbddperturcor} and Theorem \ref{weakdegnullbordthm}.
\end{proof}

\begin{cor}
\label{specedmopmad}
Let $(\mathpzc{E},D)$ be a cycle for $(\mathcal{A},B)$ with $D$ is invertible. Assume that there is a real valued odd function $\chi\in C^\infty(\R)$ with $\chi(x)=1$ for all $x>0$ belonging to the spectrum of $D$, such that $\chi(D)$ commutes with $A$. Then  $(\mathpzc{E},D)$ is spectrally decomposed by $P:=(\chi(D)+1)/2$ and $(\mathpzc{E},D)$ is nullbordant.
\end{cor}

We now present one more variation on the themes above.

\begin{define}
\label{cliffsymdef}
Suppose that $\mathfrak{Y}=(\mathpzc{E},D)$ is an $(\mathcal{A},B)$-cycle and that $J\in \End_B^*(\mathpzc{E})$ is a symmetry (i.e. $J=J^*$ and $J^2=1$) commuting with the $A$-action such that $JD=-DJ$ on $\Dom(D)$. If $\mathfrak{Y}$ is even, we assume that $J$ is odd. We then say that $J$ is a Clifford symmetry for $\mathfrak{Y}$.
\end{define}

\begin{prop}
\label{cliffsymlemprop}
Suppose that $J$ is a Clifford symmetry for a cycle $\mathfrak{Y}=(\mathpzc{E},D)$. Then the cycle $(\mathpzc{E},D+J)$ is weakly degenerate and $\mathfrak{Y}\sim_{\rm bor} 0$ via a canonical bordism $\mathfrak{X}_{\mathfrak{Y},J}$ constructed from gluing together the $KK$-bordism  $\mathfrak{Y}\sim_{\rm bor}(\mathpzc{E},D+J)$ (see Corollary \ref{locbddperturcor}) with the bordism $\mathrm{Sh}(\mathpzc{E},D+J)$ defined for a weakly degenerate cycle in Theorem \ref{weakdegnullbordthm}.
\end{prop}

\subsection{Homotopy of the representation}
\label{subsec:homotoperep}

The definition of bordism allows for encoding not just regular enough homotopies in the operator, but also in the action of the algebra $\mathcal{A}$. In this subsection, we shall use the notation ${}_\pi \mathpzc{E}$ for an $A-B$-Hilbert $C^*$-module for which the left action of $A$ is defined from a representation $\pi:A\to \End_B^*(\mathpzc{E})$. 

\begin{define}
Let $\mathpzc{E}$ be a $B$-Hilbert module equipped with a self-adjoint regular operator $D$ and $\pi_0,\pi_1:A\to \End_B^*(\mathpzc{E})$ two representations for which $({}_{\pi_0} \mathpzc{E},D)$ and $({}_{\pi_1} \mathpzc{E},D)$ are $(\mathcal{A},B)$-cycles. We say that $\pi_0$ is smoothly homotopic to $\pi_1$ if there is a $\Pi:A\to C([0,1],\End_B^*(\mathpzc{E}))$ such that 
\begin{enumerate}
\item $\Pi_t:=\mathrm{ev}_t\circ \pi$ satisfies that $\Pi_t=\pi_t$ for $t\in \{0,1\}$;
\item $\Pi(\mathcal{A})\subseteq C^1([0,1],\End_B^*(\mathpzc{E}))$;
\item For any $t\in [0,1]$, $({}_{\Pi_t} \mathpzc{E},D)$ is a cycle for $(\mathcal{A},B)$.
\end{enumerate}
For technical simplicity, we tacitly assume that $\Pi$ is constant on $[0,1/4]\cup[3/4,1]$.
\end{define}

\begin{theorem}
\label{homotopyofrep}
Let $({}_{\pi_0} \mathpzc{E},D)$ and $({}_{\pi_1} \mathpzc{E},D)$ be two $(\mathcal{A},B)$-cycles and $\Pi$ a smooth homotopy between $\pi_0$ and $\pi_1$. Then the collection 
$$(({}_{\Pi}(L^2[0,1]\hat{\boxtimes}\mathpzc{E}),\Psi(D)),(\id,\chi_{[0,1/4]\cup[3/4,1]})),$$ 
defines a bordism 
$$({}_{\pi_0} \mathpzc{E},D)\sim_{\rm bor}({}_{\pi_1} \mathpzc{E},D).$$
\end{theorem}

\begin{proof}
It is clear that if $(({}_{\Pi}(L^2[0,1]\hat{\boxtimes}\mathpzc{E}),\Psi(D)),(\id,\chi_{[0,1/4]\cup[3/4,1]}))$ is a bordism, then it defines a bordism $({}_{\pi_0} \mathpzc{E},D)\sim_{\rm bor}({}_{\pi_1} \mathpzc{E},D)$. By Proposition \ref{suspandprop}, $\Psi(D)$ is regular, symmetric and makes $(({}_{\Pi}(L^2[0,1]\hat{\boxtimes}\mathpzc{E}),\Psi(D)),(\id,\chi_{[0,1/4]\cup[3/4,1]}))$ into a symmetric chain with boundary and it is half-closed for the $C^\infty_c((0,1],\mathcal{A})$-action ensuring that we have a bordism.
\end{proof}

We will now state a consequence of Theorem \ref{homotopyofrep} that will play an important role later in the monograph when we want to prove injectivity statements for the bounded transform. The result was first proven by Kaad \cite{Kaadunbdd} in the context of unbounded homotopy, but we shall see that it readily extends to the context of $KK$-bordisms.

\begin{lemma}[Proposition 6.2 of \cite{Kaadunbdd}]
\label{prop62kaad}
Let $(\mathpzc{E},D)$ and $(\mathpzc{E},D')$ be two $(\mathcal{A},B)$-cycles on the same $B$-Hilbert $C^*$-module $\mathpzc{E}$ with left action $\pi:A\to \End_B^*(\mathpzc{E})$. Assume that $F\in \End_B^*(\mathpzc{E})$ is an odd symmetry such that 
\begin{enumerate}
\item $F\in \mathrm{Lip}(D)\cap \mathrm{Lip}(D')$
\item There exists positive regular operators $\Delta$ and $\Delta'$ defined on $\Dom(DF)$ and $\Dom(D'F)$, respectively, such that $\Delta-DF$ and $\Delta'-D'F$ have adjointable extensions to $\mathpzc{E}$.
\item For each $a\in \mathcal{A}$, $[F,a]\mathpzc{E}\subseteq \Dom(D)\cap \Dom(D')$.
\end{enumerate} 
Set $P:=(1+F)/2$. Then the operator 
$$\hat{D}_F:= 
\begin{pmatrix}
PDP-(1-P)D'(1-P)&0\\
0& (1-P)D(1-P)-PD'P
\end{pmatrix},$$
makes $(\mathpzc{E}\oplus (- \mathpzc{E}), \hat{D}_F)$ into a cycle for $(\mathcal{A},B)$ which is unitarily equivalent to the cycle $(\mathpzc{E}\oplus (-\mathpzc{E}), D\oplus (-D'))$ . Moreover, $F$ induces a representation 
$$\Pi:A\to C([0,1],\End_B^*(\mathpzc{E}\oplus (-\mathpzc{E}))),$$ 
such that 
\begin{itemize}
\item[a)] $\pi=\mathrm{ev}_0\circ \Pi$ is smoothly homotopic to $\pi':=\mathrm{ev}_1\circ\Pi$ relative to the cycle $(\mathpzc{E}\oplus (-\mathpzc{E}), \hat{D}_F)$.
\item[b)] There is an adjointable operator $V$ constructed from the data above such that $({}_{\pi'}(\mathpzc{E}\oplus (-\mathpzc{E})), \hat{D}_F+V)$ is spectrally decomposable. 
\item[c)] The symmetry $F$ and the operator $V$ induces a bordism 
$$(\mathpzc{E},D)\sim_{\rm bor} (\mathpzc{E},D').$$
\end{itemize}
\end{lemma}

\begin{proof}
The proofs of the fact that $(\mathpzc{E}\oplus (-\mathpzc{E}), \hat{D}_F)$ is a cycle for $(\mathcal{A},B)$ which is unitarily equivalent to $(\mathpzc{E}\oplus (-\mathpzc{E}), D\oplus (-D'))$ and of item a) and b) can be found in the proof of \cite[Proposition 6.2]{Kaadunbdd}. Item c) follows from item a) and b) using Corollary \ref{locbddperturcor}, Theorem \ref{specdecompnulbord} and Theorem \ref{homotopyofrep}.

\end{proof}

\begin{cor}
Let $(\mathpzc{E},D)$ and $(\mathpzc{E},D')$ be two $(\mathcal{A},B)$-cycles satisfying the following
\begin{itemize}
\item $D$ and $D'$ are invertible and 
$$D|D|^{-1}=D'|D'|^{-1}.$$
\item Both cycles $(\mathpzc{E},D)$ and $(\mathpzc{E},D')$ satisfy the Lipschitz condition (see Definition \ref{UnbKKcycDef} on page \pageref{UnbKKcycDef}).
\end{itemize}
Then there is a bordism $(\mathpzc{E},D)\sim_{\rm bor}(\mathpzc{E},D')$.
\end{cor}

\begin{proof}
The corollary follows by using $F:=D|D|^{-1}=D'|D'|^{-1}$ in Lemma \ref{prop62kaad}. For this choice of $F$, the operators $\Delta:=DF=|D|$ and $\Delta':=D'F=|D'|$ are positive and regular. Moreover, $[F,a]\mathpzc{E}\subseteq \Dom(D)$ by Lipschitz regularity of $D$ and $[F,a]\mathpzc{E}\subseteq \Dom(D')$ by Lipschitz regularity of $D'$.
\end{proof}

\subsection{Order reduction}
\label{orderredusubsec}

The order reduction procedure studied in Subsubsection \ref{orderreducsubsec} (see page \pageref{orderreducsubsec}) gives a way to deform cycles so that it can extend to a complex interpolation space. We shall now construct a bordism showing that order reduction does not change the bordism class. 

Consider a cycle $(\mathpzc{E},D)$ for $(\mathcal{A},B)$ with $D$ invertible. The results of this subsection holds for a general $*$-algebra $\mathcal{A}$, although they are most useful in the case when $\mathcal{A}$ is a Banach $*$-algebra. Choose a function $\chi\in C^\infty_c([0,1),[0,1])$ with $\chi(0)=1$ and $\chi'\in C^\infty_c(0,1)$ negative. Fix an $\alpha\in (0,1]$, and define the path of regular, self-adjoint operators 
\begin{equation}
\label{dalpahom}
D(t):=D(1+D^2)^{(1-\chi(t))(\alpha-1)/2}.
\end{equation}
We have that $D(0)=D$ and $D(1)=D_\alpha\equiv D(1+D^2)^{(\alpha-1)/2}$.

\begin{theorem}
\label{bordismforodrdthm}
Take $\alpha\in (0,1]$. Let $(\mathpzc{E},D)$ be an invertible $(\mathcal{A},B)$-cycle and let $(\mathpzc{E},D_\alpha)$ denote its order reduction. Then the path $(D(t))_{t\in [0,1]}$ from $D_\alpha$ to $D$ defined in Equation \eqref{dalpahom} satisfies the assumptions of Theorem \ref{homotopywithsweakcond} (see page \pageref{homotopywithsweakcond}). In particular, there is a bordism of $(\mathcal{A},B)$-cycles
$$(\mathpzc{E},D)\sim_{\rm bor} (\mathpzc{E},D_\alpha).$$ 
\end{theorem}

Before proceeding with the proof, we need a lemma.

\begin{lemma}
\label{relaraldknaldn}
Let $\mathfrak{s}$ be a self-adjoint and regular operator on a Hilbert $C^*$-module. Then $\mathfrak{s}\log(1+\mathfrak{s}^2)$ is relatively bounded by $\mathfrak{s}^2$ with arbitrary relative bound $a>0$.
\end{lemma}

\begin{proof}
This follows from the elementary statement that for any $a>0$ there is a $C>0$ such that 
$$-ax^2-C\leq x\log(1+x^2)\leq ax^2+C, \quad\forall x\in \R.$$
This last statement in turn follows from the fact that $x\log(1+x^2)=o(x^2)$ as $|x|\to \infty$.
\end{proof}

\begin{proof}[Proof of Theorem \ref{bordismforodrdthm}]
It is clear that the path $(D(t))_{t\in [0,1]}$ satisfies the assumptions 1) to 3) of Theorem \ref{homotopywithsweakcond} with $\mathpzc{W}:=\Dom(|D|^\alpha)=\Dom(D_\alpha)$. To prove assumption 4) of Theorem \ref{homotopywithsweakcond}, we will adapt an argument from the proof of \cite[Theorem 1.3]{DGM} (which in turn is a variation of the proof of \cite[Proposition 7.7]{leschkaad2}). For notational convenience, we set $\mathpzc{X}:=L^2[0,1]\hat{\boxtimes}\mathpzc{E}$

We write $T=\overline{\mathfrak{t}+\mathfrak{s}}$ where $\mathfrak{t}=\partial\hat{\boxtimes}0$ and $\mathfrak{s}$ is the densely defined operator on $\mathpzc{X}$ defined from the path $(D(t))_{t\in [0,1]}$. It is clear that $\mathfrak{t}$ and $\mathfrak{s}$ are symmetric and regular, indeed $\mathfrak{s}$ is self-adjoint. We write $\mathfrak{s}_1:=0\hat{\boxtimes}D$ which is a self-adjoint and regular operator on $\mathpzc{X}$. By construction, $\mathfrak{s}$ is relatively bounded by $\mathfrak{s}_1$ (with norm bound 1), and $\mathfrak{s}_1$ shares a common core with $\mathfrak{t}$ where $\mathfrak{s}_1\mathfrak{t}+\mathfrak{t}\mathfrak{s}_1=0$. By the local-global principle \cite[Theorem 1.1]{leschkaad2}, the theorem follows if $\mathfrak{t}+\mathfrak{s}$ is closed and $T^*=\overline{\mathfrak{t}^*+\mathfrak{s}}$. This follows upon proving that $\mathfrak{t}+\mathfrak{s}$ is closed and that $\Dom(\mathfrak{s}+\mathfrak{t})\subseteq \Dom((\mathfrak{t}^*+\mathfrak{s})^*)$ is dense in graph norm.

We first prove that $\mathfrak{s}+\mathfrak{t}$ is closed. We note that $\mathfrak{s}_1$ and $\mathfrak{t}$ anti-commute (on their common core $\Dom(\mathfrak{s}_1\mathfrak{t})=\Dom(\mathfrak{t}\mathfrak{s}_1)$) and that $\Dom(\mathfrak{t}\mathfrak{s}_1^2)=\Dom(\mathfrak{s}_1^2\mathfrak{t})$ is a core for $\mathfrak{s}+\mathfrak{t}$. For $\xi\in \Dom(\mathfrak{t}\mathfrak{s}_1^2)$, we compute that for any $a>0$ there is a $C>0$ such that
\begin{align*}
\langle (\mathfrak{s}+\mathfrak{t})\xi,(\mathfrak{s}+\mathfrak{t})\xi\rangle&=\langle \mathfrak{s}\xi,\mathfrak{s}\xi\rangle+\langle \mathfrak{t}\xi,\mathfrak{t}\xi\rangle+\langle \mathfrak{s}\xi,\mathfrak{t}\xi\rangle+\langle \mathfrak{t}\xi,\mathfrak{s}\xi\rangle\\
&=\langle \mathfrak{s}\xi,\mathfrak{s}\xi\rangle+\langle \mathfrak{t}\xi,\mathfrak{t}\xi\rangle-\frac{1-\alpha}{2}\langle \chi'\log(1+\mathfrak{s}_1^2)\mathfrak{s}\xi,\xi\rangle\geq\\
&\geq  (1-a)\langle \mathfrak{s}\xi,\mathfrak{s}\xi\rangle+\langle \mathfrak{t}\xi,\mathfrak{t}\xi\rangle-C\langle \xi,\xi\rangle.
\end{align*}
Indeed, the estimate at the end follows from Lemma \ref{relaraldknaldn} and the fact that $\log(1+\mathfrak{s}_1^2)$ is relatively bounded by $\log(1+\mathfrak{s}^2)$. We conclude that a Cauchy sequence in the graph norm of $\mathfrak{s}+\mathfrak{t}$ is also a Cauchy sequence in the graph norm of $\mathfrak{s}$ and the graph norm of $\mathfrak{t}$. Since $\Dom(\mathfrak{s})\cap \Dom(\mathfrak{t})$ is closed in the sum of graph norms from $\mathfrak{s}$ and $\mathfrak{t}$, and by the argument above also for the graph norm of $\mathfrak{s}+\mathfrak{t}$, the defining identity $\Dom(\mathfrak{s}+\mathfrak{t})=\Dom(\mathfrak{s})\cap \Dom(\mathfrak{t})$ allows us to conclude that $\mathfrak{s}+\mathfrak{t}$ is closed.

We now prove that $\Dom(\mathfrak{s}+\mathfrak{t})\subseteq \Dom((\mathfrak{t}^*+\mathfrak{s})^*)$ is dense in graph norm, which suffices for us to conclude $(\mathfrak{t}+\mathfrak{s})^*=\overline{\mathfrak{t}^*+\mathfrak{s}}$. Define a sequence of bounded operators by
$$u_{n,\pm}:=(i\pm \mathfrak{s}_1/n)^{-1}.$$
The sequence $(u_{n,\pm})_{n\in \N}$ is an approximate unit in the sense that for any $\xi\in \mathpzc{X}$, $u_{n,\pm}\xi\to \xi$ in norm on $\mathpzc{X}$. By construction, $u_n\mathpzc{X}\subseteq \Dom(\mathfrak{s}_1)\subseteq \Dom(\mathfrak{s})$ for any $n$. Since 
$$\Dom((\mathfrak{t}^*+\mathfrak{s})^*)\cap \Dom(\mathfrak{s})\subseteq \Dom(\mathfrak{t}),$$ 
we have that $u_{n,\pm}\Dom((\mathfrak{t}^*+\mathfrak{s})^*)\subseteq \Dom(\mathfrak{t}+\mathfrak{s})$. Take an element $\xi\in \Dom((\mathfrak{t}^*+\mathfrak{s})^*)$ and set $\xi_n:=u_{n,-} \xi$. We have that $\xi_n\in \Dom(\mathfrak{s}+\mathfrak{t})$. If $\xi_n\to \xi$ in the graph norm of $(\mathfrak{s}+\mathfrak{t}^*)^*$, our proof is complete. For $\eta\in \Dom(\mathfrak{t}^*)$, we compute (using $\mathfrak{s}_1\mathfrak{t}^*+\mathfrak{t}^*\mathfrak{s}_1=0$ on $\Dom(\mathfrak{t}^*\mathfrak{s}_1)=\Dom(\mathfrak{s}_1\mathfrak{t}^*)$) that 
\begin{align*}
\langle \xi_n,\mathfrak{t}^*\eta\rangle&=\left\langle \left(-\frac{i}{n}\mathfrak{s}_1+1\right)^{-1}\xi, \mathfrak{t}^*\eta\right\rangle=\\
&=\left\langle \left(\frac{i}{n}\mathfrak{s}_1+1\right)^{-1}(\mathfrak{s}+\mathfrak{t}^*)^*\xi,\eta\right\rangle-\left\langle \mathfrak{s}c\left(\frac{\mathfrak{s}_1}{n}\right)\left(-\frac{i}{n}\mathfrak{s}_1+1\right)^{-1}\xi, \eta\right\rangle=\\
&=\left\langle u_{n,+}(\mathfrak{s}+\mathfrak{t}^*)^*\xi,\eta\right\rangle-\left\langle \mathfrak{s}c\left(\frac{\mathfrak{s}_1}{n}\right)\xi_n, \eta\right\rangle,
\end{align*}
where $c(x)=(1-ix)(1+ix)^{-1}$. Using that $c(\mathfrak{s}_1/n)\to 1$ strictly, we can conclude that $(\mathfrak{s}+\mathfrak{t}^*)^*\xi_n=(\mathfrak{s}+\mathfrak{t})\xi_n\to (\mathfrak{s}+\mathfrak{t}^*)^*\xi$ and that $\Dom(\mathfrak{s}+\mathfrak{t})\subseteq \Dom((\mathfrak{t}^*+\mathfrak{s})^*)$ is dense in graph norm.
\end{proof}

\begin{cor}
\label{surofrestcomplx}
Take $\alpha\in (0,1]$. Assume that $\mathcal{A}$ is a Banach $*$-algebra and that the inclusion $\mathcal{A}\hookrightarrow A$ is continuous. Let $\mathcal{A}_\alpha:=[A,\mathcal{A}]_\alpha$ denote the complex interpolation space. The pullback along the inclusion $\mathcal{A}\hookrightarrow \mathcal{A}_\alpha$ induces a surjection
$$\Omega_*(\mathcal{A}_\alpha,B)\to \Omega_*(\mathcal{A},B),$$
for any $C^*$-algebra $B$.
\end{cor}

\begin{proof}
Take a $\beta\in (0,\alpha)$. For any cycle $(\mathpzc{E},D)$ for $(\mathcal{A},B)$, Theorem \ref{bordismforodrdthm} shows that $(\mathpzc{E},D)\sim_{\rm bor} (\mathpzc{E},D_\beta)$ which by Theorem \ref{interpolationthm} (see page \pageref{interpolationthm}) defines a cycle for $(\mathcal{A}_\alpha,B)$. It follows that $[(\mathpzc{E},D)]$ is in the image of $\Omega_*(\mathcal{A}_\alpha,B)\to \Omega_*(\mathcal{A},B)$.
\end{proof}

\begin{cor}
\label{ononaod}
For any $*$-algebra $\mathcal{A}$ and $C^*$-algebra $B$, any element of $\Omega_*(\mathcal{A},B)$ can be represented by an invertible Lipschitz cycle.
\end{cor}

\begin{proof}
Consider a cycle $(\mathpzc{E},D)$ for $(\mathcal{A},B)$. By Corollary \ref{invertibleldodcor} (see page \pageref{invertibleldodcor}), we can assume that $D$ is invertible. Upon completing $\mathcal{A}$ in the norm $\|a\|_D:=\|a\|_A+\|[D,a]\|_{\End_B^*(\mathpzc{E})}$ we can assume that $\mathcal{A}$ is a Banach $*$-algebra. Take a $\beta\in (0,1)$. Theorem \ref{bordismforodrdthm} shows that $(\mathpzc{E},D)\sim_{\rm bor} (\mathpzc{E},D_\beta)$ which by Theorem \ref{interpolationthm} (see page \pageref{interpolationthm}) defines a Lipschitz cycle for $(\mathcal{A}_\alpha,B)$ for any $\alpha\in (\beta,1]$, in particular for $\alpha=1$. 
\end{proof}

\section[Examples of $KK$-bordisms from index theory and NCG]{Examples of $KK$-bordisms from index theory and noncommutative geometry}
\label{sec:ncgexbord}

In this section we will consider geometric examples of $KK$-bordisms, building on the constructions from Section \ref{exampleofboridmsmssec}. The purpose of these examples is to exemplify the notion of $KK$-bordism as well as providing explicit $KK$-bordisms that are of interest to applications in secondary invariants. Similarly to the previous section, in each of the cases the $KK$-bordisms arise canonically from the underlying geometry.

\subsection{Dirac operators on $C^*$-bundles}
\label{diraconcstarbundleexbordism}
A rather basic construction of $KK$-bordisms stem from the example considered in Subsection \ref{diraconcstarbundleex} of a Dirac operator twisted by a Hermitian $B$-bundle. We refer the reader to Subsection \ref{diraconcstarbundleex} for references, yet note that $KK$-bordisms was in this context only studied in \cite{DGM}. We follow the notation of Subsection \ref{diraconcstarbundleex}. If $W$ is a complete Riemannian manifold with boundary, we say that $W$ is of product type near the boundary if the metric $g_W$ on $W$ in a uniform collared neighborhood can be written as $g_W=\mathrm{d}t^2+g_{\partial W}$ where $t$ denotes the unit normal variable to $\partial W$ and by uniform collared neighborhood we mean the subset of $W$ defined from $0\leq t\leq \epsilon$ for some $\epsilon>0$. Similarly, we define the notion of a Clifford bundle and Clifford connection being product type near the boundary.

\begin{lemma}
\label{generiddakladlalald}
Let $W$ be an oriented complete Riemannian manifold with boundary, and $\mathcal{E}_B\to W$ a Clifford $B$-bundle with a Clifford connection with all structures being of product type near the boundary. Assume that $X$ is a metric space and that $f:W\to X$ is a proper Lipschitz continuous mapping. Then the $(\Lip_0(X),B)$-cycle $(f|_{\partial W})^*(L^2(\partial W,\mathcal{E}_B^\partial), D_{\mathcal{E}}^\partial)$ is nullbordant via a bordism whose interior chain is given by $f^*(L^2( W,\mathcal{E}_B), D_{\mathcal{E}}^{\rm min})$ and the boundary data $\Theta$ defined from the identification of a collar neighborhood of the boundary with the product of $\partial W$ with $0\leq t\leq \epsilon$.
\end{lemma} 

\begin{proof}
By Subsection \ref{diraconcstarbundleex}, the $(\Lip_0(\overline{W}),B)$-chain $(L^2( W,\mathcal{E}_B), D_{\mathcal{E}}^{\rm min})$ is symmetric and restricts to a half-closed $(\Lip_0(W^\circ),B)$-chain. We can assume that $\epsilon=1$ in the definition of the product structure at the boundary, otherwise we attach a cylinder. We let $p$ denote the characteristic function of the subset of $W$ where $0\leq t\leq 1$. It is immediate from the construction that 
$$(f^*(L^2( W,\mathcal{E}_B), D_{\mathcal{E}}^{\rm min}),(\theta,p), (f|_{\partial W})^*(L^2( \partial W,\mathcal{E}_B^\partial), D_{\mathcal{E}}^{\partial})),$$
is a symmetric chain with boundary. Here $\theta$ is the obvious isomorphism defined from the product structure at the boundary. The lemma now follows from the fact that the representation $b:C^\infty([0,1],\Lip_0(X))\to \End_B^*(L^2( W,\mathcal{E}_B))$ factors over $\Lip_0(\overline{W})$ (since $f$ is proper and Lipschitz) and $b(C^\infty_c((0,1],\Lip_0(X)))\subseteq \Lip_0(W^\circ)$.
\end{proof}

\begin{prop} 
\label{generiddakladlalaldbb}
Let $M$ be an oriented complete Riemannian manifold and $\mathcal{E}_B\to M$ a Clifford $B$-bundle with two Clifford connections from which two different twisted Dirac operators $D_{\mathcal{E}}$ and $D_{\mathcal{E}}'$ are constructed. Assume that $X$ is a metric space and that $f:M\to X$ is a proper Lipschitz continuous mapping. Then there is a $KK$-bordism 
$$f^*(L^2(M,\mathcal{E}_B), D_{\mathcal{E}})\sim_{\rm bor} f^*(L^2( M,\mathcal{E}_B), D_{\mathcal{E}}'),$$
of  $(\Lip_0(X),B)$-cycles.
\end{prop}

\begin{proof}
We can take $W=M\times [0,1]$ in Lemma \ref{generiddakladlalald} and define a Clifford connection on $\mathcal{E}_B\times [0,1]$ by interpolating between the two different Clifford connections, and the proposition follows from Lemma \ref{generiddakladlalald}.
\end{proof}

\begin{prop} 
\label{generiddakladlalaldcc}
Let $M$ be an oriented compact manifold with two Riemannian metrics $g$ and $g'$ and $\mathcal{E}_B\to M$ a $B$-bundle with the structures of a $B$-Clifford bundle with a Clifford connection from both metrics; inducing two different twisted Dirac operators $D_{\mathcal{E}}$ and $D_{\mathcal{E}}'$. Assume that $X$ is a metric space and that $f:M\to X$ is a Lipschitz continuous mapping. Then there is a $KK$-bordism 
$$f^*(L^2_g(M,\mathcal{E}_B), D_{\mathcal{E}})\sim_{\rm bor} f^*(L^2_{g'}( M,\mathcal{E}_B), D_{\mathcal{E}}'),$$
of  $(\Lip_0(X),B)$-cycles.
\end{prop}

\begin{proof}
We can take $W=M\times [0,1]$ in Lemma \ref{generiddakladlalald} and define the metric, and the Clifford bundle structure and Clifford connection on $\mathcal{E}_B\times [0,1]$ by interpolation, and the proposition follows from Lemma \ref{generiddakladlalald}.
\end{proof}

\begin{prop} 
Let $M$ and $M'$ be compact oriented Riemannian manifolds and $\mathcal{E}_B\to M$, $\mathcal{E}_B'\to M'$ are $B$-Clifford bundles with Clifford connections; inducing two twisted Dirac operators $D_{\mathcal{E}}$ and $D_{\mathcal{E}'}$. Assume that there exists an oriented compact Riemannian manifold with boundary $W$ such that $\partial W=M\dot{\cup}(-M')$ and a $B$-Clifford bundle $\mathcal{F}_B\to W$ such that $\mathcal{F}_B^\partial|_{\partial W}=\mathcal{E}_B\dot{\cup}(-\mathcal{E}_B')$ as Clifford $B$-bundles on $\partial W$. Assume that $X$ is a metric space and that $f:W\to X$ is a Lipschitz continuous mapping. Then there is a $KK$-bordism 
$$(f|_M)^*(L^2(M,\mathcal{E}_B), D_{\mathcal{E}})\sim_{\rm bor} (f|_{M'})^*(L^2( M',\mathcal{E}_B'), D_{\mathcal{E}'}),$$
of  $(\Lip_0(X),B)$-cycles with interior chain $f^*(L^2(W,\mathcal{F}_B), D_{\mathcal{F}})$ for a suitable $D_{\mathcal{F}}$ extending $D_{\mathcal{E}}\oplus (-D_{\mathcal{E}'})$.
\end{prop}

\begin{proof}
The proposition follows from Lemma \ref{generiddakladlalald}, Proposition \ref{generiddakladlalaldbb} and Proposition \ref{generiddakladlalaldcc}.
\end{proof}

Another instance of a $KK$-bordism with geometric origins arise from Baum-Douglas' notion of vector bundle modification \cite{BD,BDbor, BHS, HReta}. This was discussed in \cite[Subsection 4.2]{DGM}. In fact, the notion of weakly degenerate cycles in \ref{subsecweakdegff} is motivated by vector bundle modification and its appearance in the geometric structure group \cite{DGsurIII}. The analytic details originate in work of Higson-Roe on $\eta$-invariants \cite{HReta}. 

We first consider the case of vector bundle modification of oriented manifolds by oriented vector bundles. The spin$^c$-case is discussed below in Remark \ref{soniniun}. Oriented vector bundle modification was studied in detail in \cite{guentnerkhom}. Assume that $W$ is an oriented manifold with boundary and $V\to W$ is an oriented vector bundle of even rank equipped with a Riemannian structure. For simplicity, we assume that the rank of $V$ is constant. Let $\field{R}\to W$ denote the trivial real line bundle. At the level of manifolds, we define the vector bundle modification of $W$ as 
$$W^V:=S(V\oplus \field{R}),$$
i.e. the total space of the sphere bundle $\pi_V:S(V\oplus \field{R})\to W$. 

At the level of Clifford bundles, the vector bundle modification goes as follows. We start with the situation in the fibre. Let $2k$ denote the rank of $V$. Consider the complex exterior algebra $\Lambda^*:=\wedge^*T^*S^{2k}\otimes \C$ which is an $SO(2k)$-equivariant Clifford bundle on $S^{2k}$. Following  \cite{guentnerkhom}, we consider the two $SO(2k)$-equivariant gradings $\epsilon_1$ defined by form degree and $\epsilon_2$ defined from the Hodge $*$ as 
$$\epsilon_2\omega:=i^{p(p-1)+k}*\omega, \quad \omega\in \wedge^pT^*S^{2k}\otimes \C.$$
The gradings $\epsilon_1$ and $\epsilon_2$ commute and both grade $\Lambda^*$ as a Clifford bundle. We form the graded $SO(2k)$-equivariant Clifford bundle $S:=\ker(\epsilon_1\epsilon_2-1)$ graded by the form degree $\epsilon_1$. For later purposes, we note that the construction of $S$ extends verbatim when replacing $SO(2k)$ by $Spin^c(2k)$. This construction globalizes to a $B$-Clifford bundle $\mathcal{E}_B\to W$ using the oriented frame bundle $P_V\to W$ -- the principal $SO(2k)$-bundle of oriented ON-frames of $V$. Note that $S_V:=P_V\times_{SO(2k)}S\to P_V\times_{SO(2k)}S^{2k}=W^V$ is a well defined graded vector bundle with a graded Clifford action of vertical tangent vectors. Using that $T(W^V)$ splits into horizontal and vertical directions, we arrive at a $B$-Clifford bundle 
$$\mathcal{E}_B^V:=\pi_V^*\mathcal{E}_B\hat{\otimes} S_V\to W^V.$$
We note in particular that we can identify 
$$C^\infty(W^V,\mathcal{E}^V_B)= [C^\infty(P_V\times S^{2k}, \mathcal{E}_B\times S))]^{SO(2k)},$$ 
where we write $[\mathcal{V}]^{SO(2k)}$ for the $SO(2k)$-invariant part of a representation $\mathcal{V}$. Upon choosing $SO(2k)$-invariant Laplacians, the same identification holds for the $B$-Hilbert $C^*$-module analogues of Sobolev spaces:
\begin{equation}
\label{hsident}
H^s(W^V, \mathcal{E}^V_B)=\big[H^s(P_V\times S^{2k},\mathcal{E}_B\times S))\big]^{SO(2k)}.
\end{equation}

Let us turn to the vector bundle modification of Dirac operators. On the Clifford bundle $S\to S^{2k}$ there is an $SO(2k)$-equivariant Dirac operator $D_S$ such that the even part $D_S^+$ has a one-dimensional kernel, giving the trivial $SO(2k)$-representation, and $\ker D^-_S=0$. For later purposes, we note that $D_S$ is also $Spin^c(2k)$-equivariant and $\ker D_S^+$ is the trivial $Spin^c(2k)$-representation. Write $e_S$ for the projection onto $\ker D_S^+$. By elliptic regularity $e_S\in \Psi^{-\infty}(S^{2k},S)$ is a smoothing operator and since $\ker D_Q^+$ is the trivial representation, $e_S$ is invariant. More details can be found in \cite[Proposition 3.11]{BHS} or \cite{guentnerkhom}. For a $B$-Clifford bundle $\mathcal{E}_B\to W$ over an oriented manifold $W$ equipped with a Dirac operator $D_\mathcal{E}$, the associated vector bundle modified Dirac operator on $\mathcal{E}_B^V\to W^V$ is defined by
$$D^V_\mathcal{E}:=\left(D_\mathcal{E}\hat{\otimes} 1+1\hat{\otimes} D_S\right)|_{C^\infty(W^V, \mathcal{E}_B^V)}.$$
Computing in local coordinates one verifies that $D^V_\mathcal{E}$ is a well defined Dirac operator. By density, we often consider $D^V_\mathcal{E}$ as an operator on $H^1(W^V, \mathcal{E}^V_B)$, or the minimal realization $D^V_{\mathcal{E},{\rm min}}$ with domain $H^1_0(W^V, \mathcal{E}^V_B)$ or even the maximal realization $D^V_{\mathcal{E},{\rm max}}:=(D^V_{\mathcal{E},{\rm min}})^*$. Both  $D^V_{\mathcal{E},{\rm max}}$ and $D^V_{\mathcal{E},{\rm min}}$ are regular and closed by the discussion in Subsection \ref{diraconcstarbundleex}. It is readily verified in local coordinates that for any $f\in H^1(W^V, \mathcal{E}^V_B)$ we have
\begin{align*}
\langle (D_\mathcal{E}&\hat{\otimes} 1)f,(1\hat{\otimes}D_S)f\rangle_{ L^2(P_V\times S^{2k},\mathcal{E}_B\times S))}\\
&+\langle (1\hat{\otimes}D_S)f,(D_\mathcal{E}\hat{\otimes} 1)f\rangle_{ L^2(P_V\times S^{2k}, \mathcal{E}_B\times S))}=0.
\end{align*}
Therefore, we have on $H^2(W^V, \mathcal{E}^V_B)$ that 
$$(D^V_\mathcal{E})^2:=\left(D_\mathcal{E}^2\hat{\otimes} 1+1\hat{\otimes} D_S^2\right)|_{H^2(W^V, \mathcal{E}^V_B)}.$$

A key feature of this construction is how the projection $e_S$ decomposes the Sobolev spaces $H^s(W^V, \mathcal{E}^V_B)$. Let us recall some salient results from \cite{DGsurIII} summarized in the following theorem. We tacitly assume that $H^s(W^V, \mathcal{E}^V_B)$ is defined from an invariant Laplacian on $P_V\times S^{2k}$.

\begin{prop}
\label{eqprop}
Consider the operator 
$$e_S^V:=(1\boxtimes e_S)|_{C^\infty(W^V,\mathcal{E}^V_B)}.$$
The operator $e_S^V$ defines a projection in any of the Sobolev $B$-modules $H^s(W^V, \mathcal{E}^V_B)$ satisfying the following 
\begin{itemize}
\item $e_S^V$ commutes with the left action of $C^\infty(W)$ on $H^s(W^V,\mathcal{E}^V_B)$.
\item $[D^V_\mathcal{E},e_S^V]=0$ on $H^s(W^V,\mathcal{E}^V_B)$.
\item There is a graded $C^\infty(W)$-linear isomorphism of $B$-Hilbert modules
$$e^V_SH^s(W^V,\mathcal{E}^V_B)\cong H^s(W,\mathcal{E}_B),$$
intertwining $D^V_\mathcal{E}e_S^V$ with $D_\mathcal{E}$.
\item For any $f\in H^1(W^V, \mathcal{E}^V_B)$ we have that
\begin{align*}
\langle (D_\mathcal{E}&\hat{\otimes} 1)f,(1\hat{\otimes}D_S)f\rangle_{ L^2(P_V\times S^{2k},\mathcal{E}_B\times S))}\\
&+\langle (1\hat{\otimes}D_S)f,(D_\mathcal{E}\hat{\otimes} 1)f\rangle_{ L^2(P_V\times S^{2k}, \mathcal{E}_B\times S))}=0,\quad\mbox{and}\\
\langle (1\hat{\otimes} D_S)f,&(1\hat{\otimes} D_S)f\rangle_{ L^2(W^V, \mathcal{E}^V_B)}\geq c\langle (1-e_S^V)f, (1-e_S^V)f\rangle_{\mathpzc{E}_B}
\end{align*}
for a $c>0$ only depending on $k$.
\end{itemize}
\end{prop}

We summarize the features above in the following theorem.

\begin{theorem}
\label{decomposingmoddir}
Let $W$ be a complete Riemannian manifold with boundary, $V\to W$ an oriented vector bundle of even rank, $B$ a unital $C^*$-algebra and $D_\mathcal{E}$ a complete Dirac operator on a Clifford $B$-bundle $\mathcal{E}_B\to W$. Consider the $B$-Hilbert $C^*$-module $\mathpzc{E}_V:=(1-e_S^V)L^2(W^V,\mathcal{E}^V_B)$ and the operator thereon
$$D_\mathpzc{E}:=(1-e^V_S)D_{\mathcal{E},{\rm min}}^V,$$
with domain 
$$\mathrm{Dom}(D_{\mathpzc{E}}):=(1-e_S^V)H^1_0(W^V,\mathcal{E}^V_B).$$
Then the following holds:
\begin{enumerate}
\item The pair $(\mathpzc{E}_V,D_\mathpzc{E})$ is a symmetric $(\Lip_0(W),B)$-chain, and a half-closed  $(\Lip_0(W^\circ),B)$-chain. The half-closed $(\Lip_0(W^\circ),B)$-chain $(\mathpzc{E}_V,D_\mathpzc{E})$ is closed if and only if $\partial W=\emptyset$.
\item There is a unitary equivalence of symmetric $(\Lip_0(W),B)$-chains
$$(L^2(W^V,\mathcal{E}_B^V),D_{\mathcal{E},{\rm min}}^V)\cong (L^2(W,\mathcal{E}_B),D_{\mathcal{E},{\rm min}})+(\mathpzc{E}_V,D_\mathpzc{E}).$$
\item  The symmetric $(\Lip_0(W),B)$-chain $(\mathpzc{E}_V,D_\mathpzc{E})$ is weakly degenerate (see Definition \ref{lknalkadn}), where a weakly degenerate decomposition $D_\mathpzc{E}=D_{0,\mathpzc{E}}+S_\mathpzc{E}$ is given by 
\begin{align*}
D_{0,\mathpzc{E}}&:=\overline{D_\mathcal{E}\hat{\otimes} 1|_{(1-e_S^V)H^1_0(W^V, \mathcal{E}_B^V)}}\quad\mbox{and}\\ 
S_{\mathpzc{E}}&:=\overline{1\hat{\otimes} D_S|_{(1-e_S^V)H^1_0(W^V, \mathcal{E}_B^V)}}.
\end{align*}
\end{enumerate}
\end{theorem}

\begin{cor} 
\label{coronvecbunhvhvjhv}
Let $M$ be a complete Riemannian manifold, $V\to M$ an oriented vector bundle of even rank, $B$ a unital $C^*$-algebra and $D_\mathcal{E}$ a complete Dirac operator on a Clifford $B$-bundle $\mathcal{E}_B\to M$. Assume that $X$ is a metric space and that $f:M\to X$ is a proper Lipschitz continuous mapping. Then there is a $KK$-bordism 
$$(f\circ \pi_V)^*(L^2(M^V,\mathcal{E}_B^V), D_{\mathcal{E}}^V)\sim_{\rm bor} f^*(L^2( M,\mathcal{E}_B), D_{\mathcal{E}}),$$
of  $(\Lip_0(X),B)$-cycles.
\end{cor}

\begin{remark}
\label{soniniun}
The constructions above can be carried out also in the context of spin$^c$-manifolds. Let $W$ be a spin$^c$-manifold with boundary with spin$^c$-structure defined from the Clifford bundle $S_W\to W$. All Clifford $B$-bundles take the form $\mathcal{F}_B\otimes S_W$ where $\mathcal{F}_B\to W$ is a bundle of finitely generated projective $B$-modules. The procedure of vector bundle modification is carried out for $V\to W$ being even rank and spin$^c$, so $W^V$ inherits a spin$^c$-structure from the embedding $W^V=S(V\oplus \R)\subseteq V\oplus \R$. In this case, the vertical Clifford bundle $S_V$ decomposes as $S_V=S_{\rm vert}\otimes Q$ where $S_{\rm vert}\to W^V$ is a vertical Clifford module defined from the spin$^c$-structure of $V$ and $Q\to W^V$ is the Bott bundle (see for example \cite[Section 2.5]{Rav}). This implies that in the oriented setting of vector bundle modification,
$$(\mathcal{F}_B\otimes S_W)^V=[\pi_V^*\mathcal{F}_B\otimes Q]\otimes S_{W^V},$$
or in other words, $\mathcal{F}_B\to W$ is vector bundle modified to $\pi_V^*\mathcal{F}_B\otimes Q\to W^V$. The vector bundle modification at the level of Dirac operators is carried out analogously to above. 
\end{remark}

\subsection{Bordism invariance of the index of elliptic pseudodifferential operators}

Let us provide another application of $KK$-bordisms to index theory. The following result is in no way remarkable in its own right, but showcases the general idea for how to use $KK$-bordisms in index theory. The reader can contrast this result with the bordism proof of Atiyah-Singer's index theorem \cite{gilkeyas}.

\begin{prop}
\label{andindeinede}
Let $M$ be a compact manifold, $E_1,E_2\to M$ two Hermitian vector bundles and that $P:C^\infty(M;E_1)\to C^\infty(M;E_2)$ is an elliptic pseudodifferential operator of order $1$ with principal symbol $\sigma_1(P)\in C^\infty(S^*M,\pi^*\mathrm{Iso}(E_1, E_2))$ (also viewed as a homogeneous degree one section in $C^\infty(T^*M\setminus M,\pi^*\mathrm{Iso}(E_1, E_2))$). Assume the following:
\begin{enumerate}
\item $M=\partial W$ is the boundary of a compact manifold with boundary $W$; 
\item $E_1\oplus E_2$ extends to a vector bundle $F\to W$;
\item writing $\xi_n$ for the conormal variable to the boundary, the section 
$$\begin{pmatrix}
\xi_n& \sigma_1(P)\\
\sigma_1(P)^*& -\xi_n\end{pmatrix},$$
of $C^\infty(S^*W|_{\partial W},\pi^*\mathrm{Aut}(E_1\oplus E_2))$ extends to an element of $C^\infty(S^*W,\pi^*\mathrm{Aut}(F))$.
\end{enumerate}
Then $\ind(P)=0$.
\end{prop}

\begin{proof}
By assumption, we can take a formally self-adjoint, elliptic, first order pseudo-differential operator $T$ on $W$ such that 
$$Tf=\begin{pmatrix}
-i\partial_n& P\\
P^*& i\partial_n\end{pmatrix}f,$$
for any $f\in C^\infty_c(W^\circ,F)$ supported in a small enough collar neighborhood of the boundary. We view $T$ as a closed symmetric operator on $L^2(W,F)$ by viewing it as the closure of $T$ acting on $C^\infty_c(W^\circ,F)$. It follows from the Gårding inequality that we equivalently can define $T$ acting in a distributional sense on the domain $H^1_0(W,F)$. As such, it is clear that there is a bordism of $(\C,\C)$-cycles with interior term being the odd chain $(L^2(W,F),T)$ and boundary being the even chain
$$\left(L^2(M,E_1\oplus E_2), \begin{pmatrix}
0& P\\
P^*& 0\end{pmatrix}\right).$$
Therefore, $[P]=0$ in $KK_0(\C,\C)$, or in other words $\ind(P)=0$.
\end{proof}

\begin{remark}
There is a more direct way to prove Proposition \ref{andindeinede} from Atiyah-Singer's index theorem. The extension assumption on $\sigma_1(P)$ in Proposition \ref{andindeinede} implies that $[\sigma_1(P)]\in K^0(T^*M)$ is the image of some $x\in K^1(T^*\overline{W})$ under the restriction mapping $K^1(T^*\overline{W})\to K^1(T^*W|_M)$ composed with the Thom isomorphism $K^1(T^*W|_M)\cong K^0(T^*M)$. Therefore we can compute with Stokes' theorem that
\begin{align*}
\ind(P)=\int_{T^*M} \ch[\sigma_1(P)]\wedge \mathrm{Td}(M)=&\int_{\partial(T^*W)} \ch (x)\wedge \mathrm{Td}(M)=\\
=&\int_{T^*W} \mathrm{d}(\ch (x)\wedge \mathrm{Td}(W))=0.
\end{align*}
\end{remark}

\subsection{Descent and assembly of equivariant cycles}
\label{subsecgequicalaladbordism}

The theory of $KK$-bordisms from Section \ref{subsecglying} extend to the $G$-equivariant setting, providing an extension to the unbounded model of the results of \cite{BCHig,Kas2}. We follow the notations and conventions of Subsection \ref{subsecgequicalalad}, recall in particular Definition \ref{gequicalalad} of a $G$-equivariant chain. We fix a second countable locally compact group $G$ throughout the subsection.

\begin{define} 
\label{HilBorDefGequi} 
Assume that $A$ and $B$ are $G-C^*$-algebras and $\mathcal{A}\subseteq A$ is a $G$-invariant dense $*$-subalgebra. A $G$-equivariant symmetric $(\mathcal{A},B)$-chain with boundary is an $(\mathcal{A},B)$-symmetric chain with boundary $\mathfrak{X}=(\mathfrak{Z},\Theta,\mathfrak{Y})$ where
\begin{enumerate}
\item The boundary cycle $\mathfrak{Y}=(\mathpzc{E}, D)$ is a $G$-equivariant $(\mathcal{A}, B)$-cycle;
\item The interior chain $\mathfrak{Z}=(\mathpzc{N}, T)$ is a $G$-equivariant symmetric $(\mathcal{A}, B)$-chain;
\item The boundary data $\Theta=(\theta,p)$ satisfies that
\begin{itemize}
\item $p$ is a $G$-invariant projection in $\End_B^*(\mathpzc{N})$;
\item $\theta: p\mathpzc{N} \rightarrow L^2[0,1]\hat{\boxtimes} \mathpzc{E}$ is a $G$-equivariant isomorphism.
\end{itemize}
\end{enumerate}

If additionally $\mathfrak{X}$ is a bordism, we say that $\mathfrak{X}$ is a $G$-equivariant bordism. 
\end{define}

The results of Section \ref{subsecglying} extends verbatim to the $G$-equivariant setting. If two $G$-equivariant $(\mathcal{A},B)$-cycles $\mathfrak{Y}_1$ and $\mathfrak{Y}_2$ admit a $G$-equivariant $KK$-bordism $\mathfrak{X}$ over $(\mathcal{A}, B)$, with $\partial \mathfrak{X}=\mathfrak{Y}_1+(-\mathfrak{Y}_2)$ we say that $\mathfrak{Y}_1$ and $\mathfrak{Y}_2$ are $G$-equivariantly bordant and write
$$\mathfrak{Y}_1\sim_{{\rm bor},G} \mathfrak{Y}_2.$$
We define the $G$-equvariant $KK$-bordism group as 
$$\Omega_*^G(\mathcal{A},B):=\left\{\mbox{isomorphism classes of $G$-equivariant $(\mathcal{A},B)$-cycles}\right\}/\sim_{{\rm bor},G}.$$

Many of the constructions above, such as those in Section \ref{exampleofboridmsmssec}, extend to the equivariant setting. Let us discuss weak degeneracy in more detail. We follow the setup of Subsection \ref{subsecweakdegff}.

\begin{define}
\label{lknalkadn}
Let $(\mathcal{E},D)$ be a $G$-equivariant $(\mathcal{A},B)$-chain. An equivariant weakly degenerate decomposition of $(\mathcal{E},D)$ is pair of regular operators $(D_0,S)$ such that $D=D_0+S$, implicitly it thus holds that $\Dom(D)=\Dom(D_0)\cap \Dom(S)$, and additionally that 
\begin{enumerate}
\item $S$ is $G$-invariant, self-adjoint and admits a bounded inverse, $\Dom(S)$ is preserved by $A$ and $A$ commutes with $S$ on $\Dom(S)$.
\item There is a common $G$-invariant core for $X\subseteq \Dom(S)\cap \Dom(D_0)$ with $SX_0\subseteq \Dom(D_0)$, $D_0X_0\subseteq \Dom(S)$ and 
$$SD_0+D_0S=0\quad\mbox{on}\quad X_0.$$
\item $\mathcal{A}\subseteq \mathrm{Lip}(D_0)$. 
\end{enumerate}
A chain admitting an equivariant weakly degenerate decomposition is called equivariant weakly degenerate.
\end{define}

Following Subsection \ref{subsecweakdegff}, if $(\mathcal{E},D)$ is a cycle with an equivariant weakly degenerate decomposition $D=D_0+S$, we define the symmetric operator $D_{sh}$ on $\Psi_{[0,\infty)}(\mathpzc{E})$ as 
$$D_{sh}:=\Psi_{[0,\infty)}(D_0)+X(0\hat{\boxtimes}S).$$
The operator $D_{sh}$ is a symmetric and regular operator fitting into a bordism $((\Psi_{[0,\infty)}(\mathpzc{E}),D_{sh}),(\id,\chi_{[0,1]}))$ with boundary $(\mathpzc{E},D)$, see \cite[Theorem 3.4]{DGM}. Since $S$ is $G$-invariant, we have that $((\Psi_{[0,\infty)}(\mathpzc{E}),D_{sh}),(\id,\chi_{[0,1]}))$ is a $G$-equivariant bordism. We conclude the following generalization of Theorem \ref{weakdegnullbordthm}.

\begin{theorem}
\label{weakdegnullbordthmequi}
If $(\mathpzc{E},D)$ is a $G$-equivariant weakly degenerate cycle then there is a $G$-equivariant bordism
$$(\mathpzc{E},D)\sim_{{\rm bor},G}0.$$ 
\end{theorem}

We now turn to discussing descent of unbounded equivariant cycles, as in Subsection \ref{subsecgequicalalad}, at the level of $KK$-bordism groups.

\begin{prop}
The descent mapping 
$$(\mathpzc{E},D)\mapsto (\mathpzc{E},D)\rtimes G,$$
constructed on cycles in Proposition \ref{descenfoscles}, induces a group homomorphism
$$\rtimes G: \Omega_*^G(\mathcal{A},B)\to \Omega_*(\mathcal{A}\rtimes_c G,B\rtimes G).$$
\end{prop}

\begin{proof}
The proposition follows from Proposition \ref{descenfoscles} which implies that descent construction maps $G$-equivariant $(\mathcal{A},B)$-cycles to $(\mathcal{A}\rtimes_c G,B\rtimes G)$-cycles and $G$-equivariant $(\mathcal{A},B)$-bordisms to $(\mathcal{A}\rtimes_c G,B\rtimes G)$-bordisms.
\end{proof}

Recall the construction of assembly from Subsection \ref{groupactionsex}. 

\begin{prop}
\label{assprop}
Let $X$ be a metric space with a proper, cocompact and isometric action of $G$ and $B$ a $G-C^*$-algebra. The map 
\begin{align*}
\mu_0:&\Omega_*^G(\Lip_0(X),B)\to \Omega_*(\Lip(X/G),B\rtimes G),\\
&(\mathpzc{E},D)\mapsto (p_X(\mathpzc{E}\rtimes G), p_X(D\rtimes G)p_X),
\end{align*}
where $p_X\in C_c(G,\Lip_c(X))=\Lip_c(X)\rtimes_c G$ is defined as in Equation \eqref{aksakakadldala}, is well defined and independent of the choice function used to define $p_X$. 
In particular, the assembly mapping is well defined:
$$\mu:=\iota^*\circ \mu_0:\Omega_*^G(\Lip_0(X),B)\to \Omega_*(\C,B\rtimes G),$$
where $\iota:\C\to \Lip(X/G)$ is the unital inclusion.
\end{prop}

\begin{proof}
The arguments of Proposition \ref{assemblfldoado} are stated for manifolds, but extend verbatim to $G$-equivariant cycles and chains over $(\Lip_0(X),B)$. In particular, the mapping $\mu_0$ maps $G$-equivariant $(\Lip_0(X),B)$-cycles to $(\Lip(X/G),B\rtimes G)$-cycles and $G$-equivariant $(\Lip_0(X),B)$-bordisms to $(\Lip(X/G),B\rtimes G)$-bordisms. This shows that $\mu_0$ is well defined. It remains to prove that $\mu_0$ is independent of the choice function $c$ used to define the projection $p_X\in \Lip_c(X)\rtimes_c G$. 

Given a $G$-equivariant $(\Lip_0(X),B)$-cycle $(\mathpzc{E},D)$ we need to show that given two choice functions $c_0,c_1$ there is a $(\Lip(X/G),B\rtimes G)$-bordism 
\begin{equation}
\label{bodldllalaldadladal}
(p_{X,0}(\mathpzc{E}\rtimes G), p_{X,0}(D\rtimes G)p_{X,0})\sim_{\rm bor} (p_{X,1}(\mathpzc{E}\rtimes G), p_{X,1}(D\rtimes G)p_{X,1}),
\end{equation}
where $p_{X,j}$ is defined from $c_j$, $j=0,1$. The space of choice functions is path connected, so we can find a $c\in C^\infty([0,1],\Lip_c(X))$ such that $c(0)=c_0$, $c(1)=c_1$ and $c(t)$ is a choice function for all $t\in [0,1]$. We can assume that $c$ is constant near $t=0$ and $t=1$. We can therefore define a path $p_X\in C^\infty([0,1],\Lip_c(X)\rtimes_c G)$ of projections connecting the projection defined from $c_0$ with that defined from $c_1$. Consider the path $\overline{D}=(D(t))_{t\in [0,1]}$ defined by $D(t)=p_{X,t}(D\rtimes G)p_{X,t}$ of densely defined operators on $\mathpzc{E}\rtimes G$. By the argument of Lemma \ref{lemmahomotopyresolven}, the collection $(p_X(L^2[0,1]\hat{\boxtimes}(\mathpzc{E}\rtimes G)),\partial^{\rm min}\hat{\boxtimes} \overline{D})$ is a symmetric $(\Lip(X/G),B\rtimes G)$-chain being the interior chain in a bordism required in Equation \eqref{bodldllalaldadladal}.
\end{proof}

\subsection{Group actions}
\label{groupactionsexbord}
We now revisit the cycles and chains of Subsection \ref{groupactionsex} in light of $KK$-bordisms. Above, in Proposition \ref{assprop} we studied assembly of general $G$-equivariant cycles on manifolds. Now we turn to study more geometric origins of $KK$-bordisms.

\begin{lemma}
\label{generiddakladlalaldgaction}
Let $X$ be a metric space with a proper and isometric action of $G$ and $B$ a $G-C^*$-algebra. Assume that $W$ is an oriented complete Riemannian manifold with boundary and a proper, cocompact and isometric action of $G$, and $\mathcal{E}_B\to W$ a $G$-equivariant Clifford $B$-bundle with a Clifford connection with all structures being of product type near the boundary. Assume that $f:W\to X$ is a proper Lipschitz continuous mapping. Then the $G$-equivariant $(\Lip_0(X),B)$-cycle $(f|_{\partial W})^*(L^2(\partial W,\mathcal{E}_B^\partial), D_{\mathcal{E}}^\partial)$ is $G$-equivariantly nullbordant via a bordism whose interior chain is given by $f^*(L^2( W,\mathcal{E}_B), D_{\mathcal{E}}^{\rm min})$.

In particular, if the $G$-action on $X$ is cocompact then 
\begin{itemize}
\item $(f|_{\partial W})^*(L^2(\partial W,\mathcal{E}_B^\partial)\rtimes G, D_{\mathcal{E}}^\partial\rtimes G)\sim_{\rm bor}0$ in $\Omega_*(\Lip_0(X)\rtimes_c G, B\rtimes G)$;
\item $\mu_0\circ (f|_{\partial W})^*(L^2(\partial W,\mathcal{E}_B^\partial), D_{\mathcal{E}}^\partial))\sim_{\rm bor}0$ in $\Omega_*(\Lip(X/G), B\rtimes G)$;
\item $\mu(L^2(\partial W,\mathcal{E}_B^\partial), D_{\mathcal{E}}^\partial)\sim_{\rm bor}0$ in $\Omega_*(\C, B\rtimes G)$.
\end{itemize}
\end{lemma} 

The proof goes in the same way as in Subsection \ref{diraconcstarbundleexbordism}. We note the following corollary.

\begin{lemma}
\label{generiddakladlalaldgactioncor}
Let $M$ and $M'$ be two oriented manifolds with a proper and cocompact action of $G$, $B$ a $G-C^*$-algebra, and $\mathcal{E}_B\to M$ and $\mathcal{E}_B'\to M'$ be $G$-equivariant Clifford $B$-bundles. Assume that $W$ is an oriented complete Riemannian manifold with boundary and a proper, cocompact and isometric action of $G$, and $\mathcal{F}_B\to W$ a $G$-equivariant Clifford $B$-bundle with $\partial W=M\dot{\cup}(-M')$ and $\mathcal{F}_B^\partial|_{\partial W} =\mathcal{E}_B\dot{\cup}(-\mathcal{E}_B')$. Then 
$$\mu(L^2(M,\mathcal{E}_B), D_{\mathcal{E}})\sim_{\rm bor}\mu(L^2(M',\mathcal{E}_B'), D_{\mathcal{E}'}),$$
in $\Omega_*(\C, B\rtimes G)$, for any invariant Riemannian metrics and Clifford connections.
\end{lemma}

The argument of Theorem \ref{decomposingmoddir} extends also to the equvariant setting. The only non-triviality is that of verifying that when $V\to W$ is an oriented $G$-equivariant vector bundle of even rank with a $G$-invariant Riemannian structure, then $S_\mathpzc{E}$ defined from $1\otimes D_S$ in Theorem \ref{decomposingmoddir} is $G$-invariant. 

\begin{theorem}
\label{decomposingmoddirequi}
Let $W$ be a complete Riemannian manifold with boundary with a smooth isometric action of a group $G$. Assume that $V\to W$ is an oriented $G$-equivariant vector bundle of even rank with a $G$-invariant Riemannian structure, that $B$ a unital $G-C^*$-algebra and $D_\mathcal{E}$ a complete Dirac operator on a $G$-equivariant Clifford $B$-bundle $\mathcal{E}_B\to W$. Consider the $(\Lip_0(W),B)$-chain $(\mathpzc{E}_V,D_\mathpzc{E})$ constructed in Theorem \ref{decomposingmoddir}. Then the following holds:
\begin{enumerate}
\item The pair $(\mathpzc{E}_V,D_\mathpzc{E})$ is a $G$-equivariant symmetric $(\Lip_0(W),B)$-chain.
\item There is a $G$-equivariant unitary equivalence of symmetric $(\Lip_0(W),B)$-chains
$$(L^2(W^V,\mathcal{E}_B^V),D_{\mathcal{E},{\rm min}}^V)\cong (L^2(W,\mathcal{E}_B),D_{\mathcal{E},{\rm min}})+(\mathpzc{E}_V,D_\mathpzc{E}).$$
\item  The weakly degenerate decomposition 
$$D_\mathpzc{E}=D_{0,\mathpzc{E}}+S_\mathpzc{E},$$
constructed in Theorem \ref{decomposingmoddir} is an equivariant weakly degenerate decomposition.
\end{enumerate}
In particular, if $\partial W=\emptyset$, there is a $G$-equivariant bordism 
$$(L^2(W^V,\mathcal{E}_B^V),D_{\mathcal{E}}^V)\sim_{{\rm bor},G} (L^2(W,\mathcal{E}_B),D_{\mathcal{E}}).$$
\end{theorem}

\subsection{$KK$-bordism classes of wrong way cycles}
\label{kkwrongbor}

Let us return to the unbounded cycle constructed in Subsection \ref{kkwrong} to represent the wrong way map associated with a smooth $K$-oriented map $f:X\to Y$. We introduce the notation $\Omega_!(f)$ for the class in $\Omega_*(C^\infty_c(X),C_0(Y))$ defined from the cycle $(\mathpzc{E}_f,D_f)$ built in Theorem \ref{wrongwayconst}. From Theorem \ref{wrongwayconst}, we have that 
$$\beta(\Omega_!(f))=[f_!]\quad\mbox{in}\quad KK_*(C_0(X),C_0(Y)).$$
We start with our main result of this subsection.

\begin{theorem}
\label{adnadn}
The wrong way map in $KK$-bordisms $\Omega_!(f)\in \Omega_*(C^\infty_c(X),C_0(Y))$ depends only on $f$ and its $K$-orientation.
\end{theorem}

Before proving Theorem \ref{adnadn}, we state a lemma. The proof of the lemma is clear from the constructions of Section \ref{exampleofboridmsmssec}.

\begin{lemma}
Assume that $f:X\to Y$ is a $K$-oriented smooth map. Given a factorization 
$$f=\pi\circ \iota\circ 0,$$ 
as in Equation \eqref{adlanladnfactor}, the $KK$-bordism classes of $\Omega_!(f)$, $\Omega_!(\pi)$, $\Omega_!(\iota)$, and $\Omega_!(0)$ are independent of further choices.
\end{lemma}

\begin{proof}[Proof of Theorem \ref{adnadn}]
By the preceding lemma, it suffices to show that $\Omega_!(f)$ is independent of factorization $f=\pi\circ \iota\circ 0$ as in Equation \eqref{adlanladnfactor}. Assume that we have two such factorizations 
\[
\begin{CD}
N @>\iota>>  Z \\
@A 0 AA @ VV\pi V.\\
X @>f >> Y\\
@V0' VV @ AA \pi'A.\\
N' @>\iota' >> Z'
\end{CD} \]
From these two factorizations, we form the commuting diagram
\begin{equation}
\label{comdododo}
\xymatrix{
X\ar[rd]^{0}\ar[dd]^{\tilde{0}}   \ar[rrr]^{f} & && Y 
 \\ 
 &N\ar[r]^{\iota}\ar[dl]^{\hat{0}}& Z\ar[ur]^{\pi}&
 \\
N\oplus N'\ar[rrr]^{\tilde{\iota}}&&&Z_\pi\times_{\pi'}Z'\ar[ul]^{\hat{\pi}}\ar[uu]^{\tilde{\pi}}
}
\end{equation}
The idea in our proof will be to show that the associated wrong way elements in unbounded $KK$ produce a ``commuting diagram'' with respect to unbounded Kasparov products, up to canonical bordisms. The quotation marks are to highlight the fact that we at each step need to build the unbounded Kasparov product. 

It is an algebraic computation to show that for the left triangle in \eqref{comdododo}, one can form the constructive unbounded Kasparov product and that triangle commutes in the sense we are asking for, because we have an equality of unbounded cycles
\begin{align*}
(C_0(N\oplus N',p^*S_N&\hat{\otimes}(p')^*S_{N'}),c\otimes 1+1\otimes c')=\\
&(C_0(N,p^*S_N),c)\otimes_{C_0(N)} (C_0(N\oplus N',(p')^*S_{N'}),c').
\end{align*}
Similarly, combining the constructions of Section \ref{exampleofboridmsmssec} with well known computations of unbounded Kasparov products for Dirac operators (see details in \cite{kaadfactor}), the right triangle in \eqref{comdododo} commutes in the sense we are asking for up to locally bounded perturbations.

Therefore, it remains to prove commutativity of the bottom quadrangle in \eqref{comdododo}. It suffices to produce a canonical bordism of $(C_0(N),C_0(Z))$-cycles 
$$(C_0(N)_{C_0(Z)},0)\sim_{\rm bor} (\mathpzc{E}_\iota,D_\iota),$$
where $(\mathpzc{E}_\iota,D_\iota)$ is the cycle defined from the factorization of $\iota$ given by 
$$\begin{CD}
N @>\iota>>  Z \\
@V\hat{0} VV @ AA \hat{\pi}A.\\
N\oplus N' @>\tilde{\iota}>> Z_\pi\times_{\pi'}Z'
\end{CD} 
$$
To construct the bordism $(C_0(N)_{C_0(Z)},0)\sim_{\rm bor} (\mathpzc{E}_\iota,D_\iota)$ we can argue as for vector bundle modification in Subsection \ref{diraconcstarbundleexbordism}, producing $KK$-bordisms from an index computation via a splitting $(\mathpzc{E}_\iota,D_\iota)=(C_0(N)_{C_0(Z)},0)\oplus(\mathpzc{E}_{\iota,0},D_{\iota,0})$ where $(\mathpzc{E}_{\iota,0},D_{\iota,0})$ is weakly degenerate. In the case at hand, the result follows an index computation for a ball $B\subseteq \R^n$ just as the one in \cite[Lemma 5.1]{walterver}.
\end{proof}

\begin{conjecture}
\label{lknlknad}
The wrong way maps in $KK$-bordism theory are functorial in the sense that if $f:X\to Y$ and $g:Y\to Z$ are two $K$-oriented smooth maps, then $\Omega_!(g\circ f)\in \Omega_*(C^\infty_c(X),C_0(Z))$ will up to a \emph{canonical bordism} identify with the class of the constructive unbounded Kasparov product 
$$(\mathpzc{E}_f\otimes_{C_0(Y)}\mathpzc{E}_g,D_f\otimes 1_{\mathpzc{E}_g}+1\otimes_\nabla D_g),$$
for a geometrically defined connection $\nabla$. In other words, there is a canonical bordism 
$$(\mathpzc{E}_f\otimes_{C_0(Y)}\mathpzc{E}_g,D_f\otimes 1_{\mathpzc{E}_g}+1\otimes_\nabla D_g)\sim_{\rm bor} (\mathpzc{E}_{g\circ f},D_{g\circ f}).$$
\end{conjecture}

\begin{remark}
We remark that it follows from below, see Corollary \ref{closedinrcor}, that the bounded transform $\beta:\Omega_*(C^\infty_c(X),C_0(Y))\to KK_*(C_0(X),C_0(Y))$ is an isomorphism. So we can in the context of Conjecture \ref{lknlknad} deduce that as soon as there exists a constructive unbounded Kasparov product $(\mathpzc{E}_f\otimes_{C_0(Y)}\mathpzc{E}_g,D_f\otimes 1_{\mathpzc{E}_g}+1\otimes_\nabla D_g)$ it will be $KK$-bordant to $\Omega_!(g\circ f)$. The key point of Conjecture \ref{lknlknad} is that it conjectures the existence of a canonical bordism, i.e. one constructed from the geometric data. 
\end{remark}

We note that the proof of Theorem \ref{adnadn} implies the special case of Conjecture \ref{lknlknad} for the case that $f$ is an embedding and $Z$ is a point. Such a result extends the work \cite{walterver} to higher codimension but on the other hand provides a less detailed result.

\subsection{Proper quasicocycles on groupoids}
\label{onojknadprop}
Using $KK$-bordisms, we return to the cycle constructed in Subsection \ref{groupoidsex} from a proper quasicocycle on a groupoid to study its dependence modulo bordism. For a continuous field of Hilbert spaces $\mathfrak{H}\to \mathcal{G}^{(0)}$ as in Subsection \ref{groupoidsex}, we say that a cocycle $c:\mathcal{G}\to \mathfrak{H}$ is a quasi-coboundary if there is a global section $\xi$ of $\mathfrak{H}$ such that 
$$\mathcal{G}\to \mathfrak{H}, \quad g\mapsto c(g)-g.\xi_{s(g)}+\xi_{r(g)},$$
is a bounded function. We define the set
$$H^1_{pq}(\mathcal{G},\mathfrak{H}):=\{\mbox{proper quasi-cocycles }\mathcal{G}\to \mathfrak{H}\}/\sim $$
where $c_1\sim c_2$ if $c_1-c_2$ is a quasi-coboundary $\mathcal{G}\to \mathfrak{H}$. The set $H^1_{pq}(\mathcal{G},\mathfrak{H})$ has little algebraic structure since $0=c-c$ even if $c$ and $-c$ are proper, but $H^1_{pq}(\mathcal{G},\mathfrak{H})$ forms a subset of the quotient space of quasi-cocycles $\mathcal{G}\to \mathfrak{H}$ by the space of quasi-coboundaries. The following result follows from Corollary \ref{locbddperturcor}.

\begin{prop}
Notations as in Example \ref{groupoidsex}. Suppose that $c_0$ and $c_1$ are two proper quasicocycles $\mathcal{G}\to \mathfrak{H}$ and that $c_0-c_1$ is a quasi-coboundary, so for some global section $\xi$ of $\mathfrak{H}$, the continuous function  
$$\mathcal{G}\to \mathfrak{H}, \quad g\mapsto c_1(g)-c_0(g)-g.\xi_{s(g)}+\xi_{r(g)}$$
is uniformly bounded. Then $c_t(g):=c_0(g)-t(g.\xi_{s(g)}-\xi_{r(g)})$, $t\in [0,1]$, defines a bordism 
$$(L^2(\mathcal{G},B),D_{c_0})\sim_{\rm bor} (L^2(\mathcal{G},B),D_{c_1}+V)\sim_{\rm bor} (L^2(\mathcal{G},B),D_{c_1}),$$
for a bounded $V$. In particular the data assumed in Example \ref{groupoidsex} and the map $c\mapsto (L^2(\mathcal{G},B),D_{c})$ defines a mapping of sets
$$H^1_{pq}(\mathcal{G},\mathfrak{H})\to \Omega_1(C_c(\mathcal{G}),B).$$
\end{prop}

\subsection{Continuous trace algebras and Fell algebras}
In Subsection \ref{conttracealgexex} and \ref{fellalgebraex} we saw constructions of cycles on continuous trace algebras and more generally on Fell algebras. In the example above, we have seen ample examples of results of the form ``if a cycle defined on a manifold extends to a half-closed chain on a manifold with boundary, it is $KK$-nullbordant''. The same type of results can be formulated for continuous trace algebras and Fell algebras. But instead of beating a dead horse, we content ourselves to giving a lofty general principle and exemplify this in a special case. 

The general principle can in lofty terms be stated for a $*$-algebra $\mathcal{A}_X$ (say a dense $*$-subalgebra of a continuous trace algebra or a Fell algebra), and some $(\mathcal{A}_X,\C)$-cycle $(L^2(M,E),D_E)$ for a complete Riemannian manifold $M$ and a Dirac type operator on the Clifford bundle $E\to M$. As seen in the Subsections \ref{diraconcstarbundleexbordism} and \ref{groupactionsexbord}, if $M$ is bound by a complete Riemannian manifold with boundary $W$ to which $E$ extends to $E_W\to W$, and there is a $*$-algebra $\mathcal{A}_W$ on which the algebra $\Lip(W)$ of uniformly Lipschitz functions on $W$ acts as multipliers and there is an $(\mathcal{A}_W,\C)$-chain $(L^2(W,E_W),D_{E_W})$ which is i) half-closed for elements of $\mathcal{A}_W$ vanishing on the boundary, and ii) compatible with the $\mathcal{A}_X$-action near the boundary for a $*$-homomorphism $\mathcal{A}_X\to \mathcal{A}_W$ then  $(L^2(M,E),D_E)$ is $KK$-nullbordant as $(\mathcal{A}_X,\C)$-cycles via a $KK$-bordism with interior chain $(L^2(W,E_W),D_{E_W})$. We now turn to the concrete example.

\begin{ex}
As in Subsection \ref{conttracealgexex} we consider a principal circle bundle $\pi:Z\to X$ where we assume $Z$ to be a closed manifold. We are interested in the Schwartz algebra $\mathcal{A}_X:=\mathcal{S}(\R\times Z)=\mathcal{S}(\R,C^\infty(Z))$ which is a dense $*$-subalgebra of the continuous trace algebra $C(Z)\rtimes \R$ over $X\times S^1$. Assume that $f:W\to Z$ is a Lipschitz continuous $U(1)$-equivariant map from a manifold with boundary $W$ with a smooth $U(1)$-action, $E\to W$ is a $U(1)$-equivariant Clifford bundle and that $D_E$ is a $U(1)$–equivariant Dirac operator on $E$. By a gluing argument, we can assume that all structures on $W$ are of product type near the boundary. The pull-back $f$ defines $*$-homomorphisms 
$$f^*:C(Z)\rtimes \R\to C(\overline{W})\rtimes \R, \quad\mbox{and}\quad (f|_{\partial W})^*:C(Z)\rtimes \R\to C(\partial W)\rtimes \R.$$
Following the general principle above, or more precisely the arguments of the Subsections \ref{diraconcstarbundleexbordism} and \ref{groupactionsexbord}, we conclude that the $(\mathcal{S}(\R\times Z),\C)$-cycle 
$$(f|_{\partial W})^*(L^2(\partial W,E^\partial),D_E^\partial),$$
is $KK$-nullbordant via a $KK$-bordism with interior chain $f^*(L^2(W,E),D_E)$.

\end{ex}

\subsection{Cuntz-Krieger algebras}
\label{ckexbordism}

Let us proceed with the example on the Cuntz-Krieger algebras from Subsection \ref{ckex}. The spectral triples on Cuntz-Krieger algebras from \cite{GMCK,GMR}, and more generally on Smale space $C^*$-algebras from \cite{DGMW}, provide instances of truly noncommutative objects. As such, we believe there is mileage in a truly noncommutative geometric equivalence relation among cycles. One can for instance ask if the equivalence relations between shift groupoids studied in \cite{carletal}, or more generally \cite{brimunren}, carry any geometric content in the language of $KK$-bordisms. The homotopy argument for the Cuntz algebra $O_N$ of Subsection \ref{ckex} can be phrased in terms of $KK$-bordisms as follows, capturing the fact that $K^1(O_N)\cong \Z/(N-1)\Z$.

\begin{prop}
Let $(L^2(O_N),D)$ be the $(O_N^\infty,\C)$-cycle constructed in Subsection \ref{ckex} and $\lambda(a):=\sum_{j=1}^n S_j aS_j^*$. Then the $*$-homomorphism $\pi:O_N\to C([0,1],O_N)$ such that $\mathrm{ev}_0\circ \pi=\id_{O_N}$, $\mathrm{ev}_1\circ \pi=\lambda$ (preserving the closure of $O_N^\infty$ under holomorphic functional calculus) and the bounded perturbation $B$ (from Subsection \ref{ckex}) implements a $KK$-bordism of $(O_N^\infty,\C)$-cycles
$$N(L^2(O_N),D)\sim_{\rm bor} (L^2(O_N),D).$$
In particular, $(N-1)(L^2(O_N),D)\sim_{\rm bor} 0$.
\end{prop}

The proposition follows directly from Theorem \ref{homotopyofrep} and Corollary \ref{locbddanticommcor}.

\subsection{Cuntz-Pimsner algebras on manifolds with boundary}
We now come to full circle with the constructions of Example \ref{cpalgebronmanfdodoex}. Following the arguments of the Subsections \ref{diraconcstarbundleexbordism} and \ref{groupactionsexbord}, we conclude the following.

\begin{prop}
Let $W$ be a compact manifold with boundary and $X$ a compact metric space. Let $\slashed{D}_W$ be a Dirac operator of product type near the boundary acting on a Clifford bundle $S_W$. Suppose that $f:W\to X$ is a Lipschitz map and $E\to X$ is a vector bundle on $X$. Set $\mathpzc{V}:=C(X,E)$. Then, as $(\mathcal{O}^{\rm Lip}_{\mathpzc{V}},\C)$-cycles, we have that 
$$(f|_{\partial W})^*(L^2(\partial W,S_{\partial W}\otimes \Xi_{f^*\mathpzc{V}}), D_{f^*E})\sim_{\rm bor}0,$$
via the $(\mathcal{O}^{\rm Lip}_{\mathpzc{V}},\C)$-bordism with interior $f^*(L^2(W,S_{W}\otimes \Xi_{f^*\mathpzc{V}}), D_{f^*E})$.
\end{prop}

\subsection{Dirac-Schrödinger operators}
\label{diracschexbord}
In Subsection \ref{diracschex}, Dirac-Schrödinger operators of the form $\slashed{D}_{\mathcal{E},V}$ were discussed. Here $\slashed{D}_\mathcal{E}:C^\infty_c(W^\circ, \mathcal{E}_B)\to C^\infty_c(W^\circ, \mathcal{E}_B)$ is a complete $B$-linear Dirac operator and $V\in C^\infty(W^\circ, \End_B^*(\mathcal{E}_B))$ is self-adjoint and commutes with Clifford multiplication. In light of the techniques from Section \ref{exampleofboridmsmssec}, we can say more about such operators and their index theory. The reader can find references in Subsection \ref{diracschex}.

\begin{lemma}
\label{technicallemmafords}
We consider a complete Riemannian manifold with boundary $W$ and a Clifford-$B$-bundle $\mathcal{E}_B\to W$ equipped with a complete $B$-linear Dirac operator $\slashed{D}_\mathcal{E}:C^\infty_c(W^\circ, \mathcal{E}_B)\to C^\infty_c(W^\circ, \mathcal{E}_B)$ with minimal closure $D_{\mathcal{E}}^{\rm min}$ on $L^2(W,\mathcal{E}_B)$. Write $t\in C^\infty_b(W)\cap C_0(W^\circ)$ for a function nowhere zero in $W^\circ$ and defining the metric normal coordinate near the boundary. We also fix a Clifford linear self-adjoint element $V\in C^\infty(W^\circ, \End_B^*(\mathcal{E}_B))$ satisfying that for some $C>0$, and every $x\in W$, 
\begin{equation}
\label{estforvdirsch}
\|\nabla V(x)\|_{TW\otimes\End_B^*(\mathcal{E}_B)} \leq C(1+\|V(x)\|_{\End_B^*(\mathcal{E}_B)}).
\end{equation}
Then the operator $D_{\mathcal{E},V}^{\rm min}$ defined from $D_{\mathcal{E}}^{\rm min}$ and $V$ as in Equation \eqref{dirscalakdn} satisfies
\begin{enumerate}
\item $D_{\mathcal{E},V}^{\rm min}$ is regular and its adjoint is $D_{\mathcal{E},V}^{\rm max}$ -- defined from $D_{\mathcal{E}}^{\rm max}$ and $V$ as in Equation \eqref{dirscalakdn} . 
\item If $a\in C_b(W^\circ)$, satisfies that $a(i\pm V)^{-1}\in C_0(W;\End(\mathcal{E}_B))$ then the operator  
$$\Dom(D_{\mathcal{E},V}^{\rm min})\hookrightarrow L^2(W,\mathcal{E}_B)\xrightarrow{a}L^2(W,\mathcal{E}_B),$$
is $B$-compact.
\item If $(i\pm V)^{-1}\in tC_b(W^\circ;\End(\mathcal{E}_B))$ or $\partial W=\emptyset$ then 
$$(D_{\mathcal{E},V}^{\rm min})^*=D_{\mathcal{E},V}^{\rm min}.$$
\end{enumerate}
\end{lemma}

\begin{proof}
The estimate \eqref{estforvdirsch} and \cite[Theorem 1.3]{DGM} implies item 1), because $V$ is Clifford linear and as such commutes with the principal symbol of $D$. We focus the argument around the case that $W$ is odd-dimensional and study the upper right corner $\slashed{D}_\mathcal{E}+iV$ in $\slashed{D}_{\mathcal{E},V}$. To prove item 2), we follow the argument in Remark \ref{commentoncprreosls}. Indeed, the operator 
$$\Dom(D_{\mathcal{E}}^{\rm min}+iV)\hookrightarrow L^2(W,\mathcal{E}_B)\xrightarrow{a}L^2(W,\mathcal{E}_B),$$
is the norm limit (in the graph norm of the domain) of the sequence of operators 
$$\Dom(D_{\mathcal{E}}^{\rm min}+V)\hookrightarrow L^2(W,\mathcal{E}_B)\xrightarrow{a(i+V/n)^{-1}}L^2(W,\mathcal{E}_B).$$
The latter operator is compact since $a(i\pm V)^{-1}:\Dom(D_{\mathcal{E}}^{\rm min})\to L^2(W,\mathcal{E}_B)$ is compact as soon as $a(i\pm V)^{-1}\in C_0(W;\End(\mathcal{E}_B))$. To prove item 3) for $\partial W=\emptyset$, we note that $D_{\mathcal{E}}^{\rm max}=D_{\mathcal{E}}^{\rm min}$ is self-adjoint and item 3) follows from item 1). To prove item 3) for $(i\pm V)^{-1}\in tC_b(W^\circ;\End(\mathcal{E}_B))$, we note that this fact implies that
$$\Dom(D_{\mathcal{E}}^{\rm max})\cap\Dom(V)=\Dom(D_{\mathcal{E}}^{\rm min})\cap\Dom(V),$$
and it follows from item 1) that 
$$(D_{\mathcal{E}}^{\rm min}+iV)^*=\overline{D_{\mathcal{E}}^{\rm max}-iV}=\overline{D_{\mathcal{E}}^{\rm min}-iV}=D_{\mathcal{E}}^{\rm min}-iV.$$
\end{proof}

\begin{prop}
\label{lipcyclefords}
Consider a complete Riemannian manifold $M$, a metric space $X$ and a uniformly Lipschitz function $f:M\to X$. Assume that   $\mathcal{E}_B\to M$ is a Clifford-$B$-bundle equipped with a complete $B$-linear Dirac operator $\slashed{D}_\mathcal{E}$ with closure $D_{\mathcal{E}}$ on $L^2(M,\mathcal{E}_B)$. Assume that $V\in C^\infty(M, \End_B^*(\mathcal{E}_B))$ is a Clifford linear self-adjoint element satisfying 
\begin{itemize}
\item For some $C>0$, and every $x\in M$, the estimate \eqref{estforvdirsch} holds.
\item for any $a\in \Lip_0(X)$ it holds that 
$$f^*(a)(i\pm V)^{-1}\in C_0(M;\End(\mathcal{E}_B)).$$
\end{itemize}
Then $({}_{f^*}L^2(M;\mathcal{E}_B),D_{\mathcal{E},V})$ is a $(\Lip_0(X),B)$-cycle. 

Moreover, if there is a complete Riemannian manifold with boundary $W$ of product type near the boundary such that 
\begin{itemize}
\item $W$ bounds $M$ as a Riemannian manifold;
\item $\mathcal{E}_B\to M$ extends as a Clifford $B$-bundle to $\mathcal{F}_B\to W$;
\item there is a uniformly Lipschitz continuous map $f_W:W\to X$ such that $f_W|_{\partial W}=f$;
\item $V$ extends to a Clifford linear self-adjoint $V_W\in C^\infty(\overline{W}, \End_B^*(\mathcal{F}_B))$ of product type near the boundary satisfying the estimate \eqref{estforvdirsch} and for any $a\in \Lip_0(X)$ it holds that 
$$(f_W)^*(a)(i\pm V)^{-1}\in C_0(\overline{W};\End(\mathcal{F}_B)),$$
\end{itemize}
then $({}_{f^*}L^2(M;\mathcal{E}_B),D_{\mathcal{E},V})\sim_{\rm bor}0$ as $(\Lip_0(X),B)$-cycles.
\end{prop}

We omit the proof of Proposition \ref{lipcyclefords}, as the first part is a direct consequence of Lemma \ref{technicallemmafords} and the second part goes along the same lines as that of Lemma \ref{generiddakladlalald} (using Lemma \ref{technicallemmafords}). An immediate corollary of Proposition \ref{lipcyclefords}, using the index isomorphism $\Omega_*(\C,B)\cong K_*(B)$ from \cite[Theorem 3.7]{DGM}, provides a vanishing result for Dirac-Schrödinger in the presence of a bounding manifold.

\begin{cor}
Let $W$ be a complete Riemannian manifold with boundary of product type near the boundary, $\mathcal{F}_B\to W$ a Clifford $B$-bundle and $V\in C^\infty(\overline{W}, \End_B^*(\mathcal{F}_B))$ a Clifford linear self-adjoint endomorphism of product type near the boundary such that the estimate \eqref{estforvdirsch} holds and that
$$(i\pm V)^{-1}\in C_0(\overline{W};\End(\mathcal{F}_B)).$$
Then $D_{\mathcal{F},V^\partial}^{\partial W}$ is a $B$-Fredholm operator on $L^2(\partial W, \mathcal{F}^\partial)$ with  
$$\ind(D_{\mathcal{F},V^\partial}^{\partial W})=0\in K_*(B).$$
\end{cor}

Consider the continuously differentiable function 
$$\mathrm{sgnlog}:\R\to \R, \quad \mathrm{sgnlog}(x):=
\begin{cases}
x|x|^{-1}\log(1+|x|), \; &x\neq 0\\
0, \; &x=0.\end{cases}.$$
Note that $\mathrm{sgnlog}'(x)=(1+|x|)^{-1}$ and that $\mathrm{sgnlog}$ is a proper function. For a self-adjoint regular operator $T$, we write $T_{\log}:=\mathrm{sgnlog}(T)$.

\begin{lemma}
\label{logdampfpotential}
Let $M$, $X$, $f:M\to X$, $\mathcal{E}_B\to M$, $D_{\mathcal{E}}$, and $V\in C^\infty(M, \End_B^*(\mathcal{E}_B))$ be as in Proposition \ref{lipcyclefords}. Assume that $V$ satisfies 
\begin{itemize}
\item $(i\pm V)^{-1}\in C_0(M;\End(\mathcal{E}_B))$, and 
\item for some $\epsilon>0$, $V$ satisfies the following strengthening of \eqref{estforvdirsch}:
$$\|\nabla V(x)\|_{TW\otimes\End_B^*(\mathcal{E}_B)} \lesssim (1+\|V(x)\|_{\End_B^*(\mathcal{E}_B)})^{1-\epsilon}.$$
\end{itemize}
Then there is a constant $C>0$ such that $V_{\log}\in C^1(M, \End_B^*(\mathcal{E}_B))$ satisfies the estimate 
\begin{equation}
\label{logbound}
\|\nabla V_{\log}(x)\|_{TM\otimes\End_B^*(\mathcal{E}_B)} \leq C,
\end{equation}
for all $x\in M$. In particular, there is a $(\Lip_0(X),B)$-bordism
$$({}_{f^*}L^2(M;\mathcal{E}_B),D_{\mathcal{E},V})\sim_{\rm bor}({}_{f^*}L^2(M;\mathcal{E}_B),D_{\mathcal{E},V_{\log}}).$$
\end{lemma}

\begin{proof}
We first prove the bound \eqref{logbound}. It suffices to prove \eqref{logbound} outside the compact $K\subseteq M$ where $V$ fails to be invertible or $\|V(x)^{-1}\|_{\End_B^*(\mathcal{E}_B)}  \geq 1$. Set $v(x):=\|V(x)\|_{\End_B^*(\mathcal{E}_B)}$. The bound \eqref{estforvdirsch} ensures that there is a $C_1>0$ such that for $z\in \C$ and $x\in M\setminus K$,
\begin{equation}
\label{adflkandlkjnag}
\|\nabla(V |V|^z)\|_{\End_B^*(\mathcal{E}_B)} \leq C_1(1+|z|)v^{\mathrm{Re}(z)+1}.
\end{equation}
Consider the holomorphic function $F=F(z)$ with values in $C^\infty(M\setminus K,\End_B^*(\mathcal{E}_B))$ of the form 
$$F(z)=v^{-z-1}\nabla(V |V|^z).$$
The estimate \eqref{adflkandlkjnag} ensures that $\|F(z,x)\|_{\End_B^*(\mathcal{E}_B)} \leq C_1(1+|z|)$. We therefore have that $F$ is holomorphic with values in $C_b(M\setminus K,T^*M\otimes \End_B^*(\mathcal{E}_B))$. In particular, it holds that 
$$\partial_z|_{z=-1} F(z)=\log(v)\nabla(V |V|^{-1})+\nabla V_{\rm log}\in C_b(M\setminus K,T^*M\otimes \End_B^*(\mathcal{E}_B)).$$
Moreover, our assumptions guarantee that $\nabla(V |V|^{-1})=O(|V|^{-\epsilon})$ so  $\log(v)\nabla(V |V|^{-1})=O(\log(v)|V|^{-\epsilon})$ is bounded. Therefore $\nabla V_{\rm log}$ is bounded.

The proof of the final statement goes as that of Theorem \ref{bordismforodrdthm} but now using $s(t)=\chi(t)V+(1-\chi(t))V_{\log}=V|V|^{-1}(\chi(t)|V|+(1-\chi(t))\log(1+|V|))$ and $t=\partial+D_{\mathcal{E}}$. 
\end{proof}

\begin{theorem}
\label{calliasindxtheoremem}
Let $M$ be an odd-dimensional complete Riemannian manifold such that for some compact $K\subseteq M$ and a compact manifold $N$ such that 
$$M\setminus K\cong [1,\infty)\times N,$$
and that the metric takes the form $g|_{M\setminus K}=\mathrm{d} t^2+h(t)^2g_N$ for some smooth function $h$ and metric $g_N$ on $N$. Let $\mathcal{E}_B\to M$ be a $B$-bundle with connection $\nabla_\mathcal{E}$ and $V\in C^\infty(M, \End_B^*(\mathcal{E}_B))$ a self-adjoint endomorphism such that for some $\epsilon>0$, $V$ satisfies the following strengthening of \eqref{estforvdirsch}:
$$\|\nabla V(x)\|_{TW\otimes\End_B^*(\mathcal{E}_B)} \lesssim (1+\|V(x)\|_{\End_B^*(\mathcal{E}_B)})^{1-\epsilon}.$$
and that
$$(i\pm V)^{-1}\in C_0(M;\End(\mathcal{E}_B)).$$
For a Dirac operator $D_M$ on a Clifford bundle $S\to M$ of product type at infinity, with associated operator $D_N$ on the link, $D_M\otimes_{\nabla_\mathcal{E}}1+iV$ is a $B$-Fredholm operator on $L^2(M, S_M\otimes \mathcal{E}_B)$ with  
$$\ind(D_M\otimes_{\nabla_\mathcal{E}}1+iV)=\ind(\chi^+(V|_N)(D_N\otimes 1)\chi^+(V|_N))\in K_0(B),$$
where $\chi^+(V|_N)(D_N\otimes_{\nabla_\mathcal{E}}1)\chi^+(V|_N))$ denotes the compression of $D_N$ to  $\chi^+(V|_N)\mathcal{E}_B|_N\to N$.
\end{theorem}

\begin{proof}
We can reduce $|\nabla V|\lesssim 1$ with Lemma \ref{logdampfpotential} and the proof then proceeds as in \cite{kucerovskycallias}.
\end{proof}

From Theorem \ref{calliasindxtheoremem} and the Atiyah-Singer index theorem for families, we deduce the following. 

\begin{cor}
Let $M$ be an odd-dimensional complete Riemannian spin$^c$-manifold such that for some compact $K\subseteq M$ and a compact manifold $N$ such that 
$$M\setminus K\cong [1,\infty)\times N,$$
and that the metric takes the form $g|_{M\setminus K}=\mathrm{d} t^2+h(t)^2g_N$ for some smooth function $h$ and metric $g_N$ on $N$. Assume that $B=C(X)$ for a closed manifold $X$ and that $\mathcal{E}^\infty_X\to M$ is a locally trivial bundle of finitely generated $C^\infty(X)$-modules and $\mathcal{E}_B:=\mathcal{E}^\infty_X\otimes_{C^\infty(X)}\otimes C(X)$ with connection $\nabla_\mathcal{E}$ and total connection $\pmb{\nabla}_\mathcal{E}$. Assume that $V\in C^\infty(M, \End_{C^\infty(X)}^*(\mathcal{E}_X^\infty))$ is a self-adjoint endomorphism such that the estimate \eqref{estforvdirsch} holds and that
$$(i\pm V)^{-1}\in C_0(M;\End(\mathcal{E}_B)).$$
For the spin$^c$-Dirac operator $D_M$, with associated operator $D_N$ on the link, $D_M\otimes_{\nabla_\mathcal{E}}1+iV$ is a $B$-Fredholm operator on $L^2(M, S_M\otimes \mathcal{E}_B)$. The Chern character of the index $\ind(D_M\otimes_{\nabla_\mathcal{E}}1+iV)\in K_0(B)=K^0(X)$ is computed as  
$$\ch(\ind(D_M\otimes_{\nabla_\mathcal{E}}1+iV))=\int_N\ch(\chi^+(V|_N)\pmb{\nabla}_\mathcal{E}\chi^+(V|_N))\wedge \mathrm{Td}(N)\in H^{\rm ev}_{dR}(X),$$
\end{cor}

\subsection{Dirac operators on spin-manifolds with positive scalar curvature}
\label{pscdiracexII}

Returning to Dirac operators on spin-manifolds from Subsection \ref{pscdiracex}, we shall prove that positive scalar curvature metrics give rise to $KK$-nullbordisms. The following result shows that a uniformly positive scalar curvature gives rise to a canonical null-$KK$-bordism; the reader can compare this result to the work on secondary invariants in \cite{PSrhoInd,xieyu}.

\begin{theorem}
\label{somepsosslresult}
Let $W$ be a complete Riemannian manifold with uniformly positive scalar curvature with a smooth, isometric action of a second countable group $G$. Assume that $B$ is a $C^*$-algebra with trivial $G$-action and that $\mathcal{E}_B\to W$ is a Hermitian flat $B$-bundle (of bounded geometry). Denote the closure of the twisted spin-Dirac operator $\slashed{D}_{W,\mathcal{E}}$ by $D_{W,\mathcal{E}}$. Then $(L^2(W,S_W\otimes \mathcal{E}_B),D_{W,\mathcal{E}})$ is a $G$-equivariant symmetric $(\C,B)$-chain with $D_{W,\mathcal{E}}$ invertible and self-adjoint. 

Moreover, if $G$ acts properly and cocompactly on $W$, then $\mu(L^2(W,S_W \otimes \mathcal{E}_B),D_{W,\mathcal{E}})$ is a $(\C,B\rtimes G)$-cycle which is nullbordant via the bordism 
$$\mu\left(\mathrm{Sh}(L^2(W,S_W\otimes \mathcal{E}_B),D=0+D_{W,\mathcal{E}})\right),$$
which is the assembly of the bordism from Theorem \ref{weakdegnullbordthm} (with $D_0=0$ and $S=D_{M,\mathcal{E}}$).
\end{theorem}

\begin{proof}
The theorem can be deduced from the discussion in Section \ref{subsecgequicalaladbordism}, but we carry out the proof in more detail for the convenience of the reader. We start by noting that the bordism $\mu\left(\mathrm{Sh}(L^2(W,S_W\otimes \mathcal{E}_B),D=0+D_{W,\mathcal{E}})\right)$ is a bordism with interior 
\begin{equation}
\label{shubincyclforpscass}
(L^2[0,\infty)\hat{\boxtimes}p_WL^2(W,S_W\otimes \mathcal{E}_B)\rtimes G,p_W((D_{W,\mathcal{E}}\rtimes G)_{sh})p_W),
\end{equation}
where the bordism is constructed as in Subsection \ref{subsecweakdegff} (with $D_0=0$ and $S=p_W(D_{W,\mathcal{E}}\rtimes G)p_W$).

The operator $D_{W,\mathcal{E}}$ is self-adjoint by the discussion in Subsection \ref{diraconcstarbundleex}, because $W$ is complete. Since $\mathcal{E}$ is Hermitian flat, the Schrödinger-Lichnerowicz formula (cf. Subsection \ref{pscdiracex}) implies that 
\begin{equation}
\label{formest} 
\langle D_{W,\mathcal{E}}\xi,D_{W,\mathcal{E}}\xi\rangle \geq \langle s_g\xi,\xi\rangle,
\end{equation}
on the core $C^\infty_c(W,\mathcal{E}_B)$. By continuity, the form estimate \eqref{formest} extends to $\Dom(D_{W,\mathcal{E}})$. If $W$ has uniformly positive scalar curvature, $D_{W,\mathcal{E}}$ is invertible with 
$$\|D_{W,\mathcal{E}}^{-1}\|_{\End_B^*(L^2(W,\mathcal{E}_B))}\leq 2\sup_{x\in W}(s_g(x))^{-1/2}. $$

It remains to prove that the expression in Equation \eqref{shubincyclforpscass} is a symmetric chain that fits into the interior of a bordism with boundary $\mu(L^2(W,S_W \otimes \mathcal{E}_B),D_{W,\mathcal{E}})$. To shorten notation, we write 
$$S:=D_{W,\mathcal{E}}\rtimes G,$$
which is a self-adjoint and regular operator on the $B\rtimes G$-Hilbert $C^*$-module $\mathpzc{E}:=L^2(W,S_W \otimes \mathcal{E}_B)\rtimes G$. It follows from the discussion above that $S$ is invertible. Recall the construction from Subsection \ref{subsecweakdegff}: we have that
$$(D_{W,\mathcal{E}}\rtimes G)_{sh}=(S)_{sh}=\Psi_{[0,\infty)}(0)+X(0\hat{\boxtimes}S)=\partial_{[0,\infty)}^{\rm min}\hat{\boxtimes} 0+X(0\hat{\boxtimes}S).$$
By \cite[Theorem 3.4]{DGM}, the operator $S_{sh}$ is symmetric and regular with domain consisting of elements $\xi\in (H^1_0[0,\infty)\hat{\boxtimes}\mathpzc{E})\cap (L^2[0,\infty)\hat{\boxtimes}\Dom(S))\subseteq \Psi_{[0,\infty)}(\mathpzc{E})$ satisfying the condition $\xi(0)=0$ and 
$$\int_0^\infty\left(\|\xi'(t)\|_{\mathpzc{E}}^2+X(t)^2\|S\xi(t)\|_\mathpzc{E}^2\right)\mathrm{d}t<\infty.$$

To finish the proof, it suffices to show that $p_WS_{sh}p_W$ is regular on $p_W\mathpzc{E}$ with adjoint being the closure of $p_WS_{sh}^*p_W$. Indeed, if this is the case then
\begin{enumerate}
\item by a similar argument as in  \cite[Theorem 3.4]{DGM}, $p_WSp_W$ has compact domain inclusion because $S$ has $C^\infty_c(W)\rtimes G$-locally compact domain inclusion and since $p_W\in C^\infty_c(W)\rtimes G$ so does also $S_{sh}$;
\item for any function $\varphi\in C^\infty_b(0,\infty)$ with $\varphi=0$ near $0$, $\varphi\Dom((p_WS_{sh}p_W)^*)\subseteq \Dom(p_WS_{sh}p_W))$ since $\varphi\Dom(p_WS_{sh}^*p_W)\subseteq \Dom(p_WS_{sh}p_W))$ and $\Dom(p_WS_{sh}^*p_W)$ is a core;
\end{enumerate}
and so $(p_W\mathpzc{E},p_WS_{sh}p_W)$ is half-closed for the left action of the algebra of functions $\varphi\in C^\infty_b(0,\infty)$ with $\varphi=0$ near $0$. 

To show that $p_WS_{sh}p_W$ is regular with adjoint $p_WS_{sh}^*p_W$, it suffices to show that the operator $p_WS_{sh}p_W+(1-p_W)S_{sh}(1-p_W)$ is regular with adjoint being the closure of $p_WS_{sh}^*p_W+(1-p_W)S_{sh}^*(1-p_W)$. We note that 
\begin{align*}
p_WS_{sh}p_W+(1-p_W)S_{sh}&(1-p_W)-S_{sh}=-p_WS_{sh}(1-p_W)-(1-p_W)S_{sh}p_W=\\
=&-(2p_W-1)[S_{sh},p_W]=-(2p_W-1)X(0\hat{\boxtimes}c_\mathcal{E}(\mathrm{d}p_W))=:V,
\end{align*}
using that $S$ is the descent of a Dirac operator with principal symbol being Clifford multiplication $c_{\mathcal{E}}$. By \cite[Theorem 1.3]{DGM}, the operator $S_{sh}+V=p_WS_{sh}p_W+(1-p_W)S_{sh}(1-p_W)$ is regular and symmetric with adjoint being the closure of $S_{sh}^*+V=p_WS_{sh}^*p_W+(1-p_W)S_{sh}^*(1-p_W)$.
\end{proof}

\begin{cor}
Let $M$ be a compact Riemannian manifold with uniformly positive scalar curvature, and $\mathcal{E}_B\to M$ is a Hermitian flat $B$-bundle (of bounded geometry). Denote the closure of the twisted spin-Dirac operator $\slashed{D}_{M,\mathcal{E}}$ by $D_{M,\mathcal{E}}$. Then $(L^2(M,S_M\otimes \mathcal{E}_B),D_{M,\mathcal{E}})$ is a nullbordant $(\C,B)$-cycle which is nullbordant via the bordism 
$$\mathrm{Sh}(L^2(M,S_M\otimes \mathcal{E}_B),D=0+D_{M,\mathcal{E}}),$$
from Theorem \ref{weakdegnullbordthm} (with $D_0=0$ and $S=D_{M,\mathcal{E}}$).
\end{cor}

\begin{proof}
The corollary follows directly from Theorem \ref{somepsosslresult} by taking $G$ to be the trivial group.
\end{proof}

\subsection{Topological- and Lipschitz manifolds}
\label{bordimadinklm}

We can now study the signature operator on Lipschitz manifold from Subsection \ref{topolmfsex} further. This operator has been studied from the perspective of $KK$-theory in \cite{hilsum83,hilsum89,PSsignInd,zenobi}, and we here rephrase these results in terms of the $KK$-bordism groups. The advantage of this approach, that will be utilize later for secondary invariants, is that we all equivalences can be made explicit and geometric. 

If $W$ is a Lipschitz manifold with boundary, and we have a Lipschitz Riemannian metric $g_W$, on $W$, we say that $g_W$ is of product type if there is a collar neighborhood of $\partial W$ isometrically bi-Lipschitz homeomorphic to $[0,\epsilon]\times \partial W$ for some $\epsilon>0$. Similarly, we define the notion of product type for connections and Dirac operators.

\begin{lemma}
\label{technicallammaforlipschitzcyc}
Let $W$ be an oriented Lipschitz manifold with boundary which is either compact or equipped with a complete metric which is of product type at the boundary. Assume that $\mathcal{E}_B\to W$ is a Hermitian $B$-bundle with a Hermitian Lipschitz connection $\nabla_\mathcal{E}$ of product type at the boundary. Write 
$$D_{{\rm sign},\mathcal{E}}:=D_{\rm sign}\otimes_{\nabla_\mathcal{E}}1,$$
and $D_{{\rm sign},\mathcal{E}}^{\partial W}:=D_{\rm sign}^{\partial W}\otimes_{\nabla_\mathcal{E}}1$ for the associated boundary operator. Assume that $X$ is a metric space and that $f:W\to X$ is a proper uniformly Lipschitz continuous map. Then the $(\Lip_0(X),B)$-cycle 
$$(f|_{\partial W})^*(L^2(\partial W,\wedge^*\otimes \mathcal{E}_B),D_{{\rm sign},\mathcal{E}}^{\partial W}),$$
is $KK$-nullbordant via a bordism with interior chain coinciding with $f^*(L^2(W,\wedge^*\otimes \mathcal{E}_B),D_{{\rm sign},\mathcal{E}}^{W})$ if $\dim(W)$ is odd.
\end{lemma}

We omit the proof of Lemma \ref{technicallammaforlipschitzcyc}, as it goes along the same lines as that of Lemma \ref{generiddakladlalald} (using Proposition \ref{technicalpropforlipschitz}). We note that Lemma \ref{technicallammaforlipschitzcyc} only prescribes the precise null bordism $(f|_{\partial W})^*(L^2(\partial W,\wedge^*\otimes \mathcal{E}_B),D_{{\rm sign},\mathcal{E}}^{\partial W})\sim 0$ when $\dim(W)$ is odd. If $\dim(W)$ is even, then we only have that
$$2(f|_{\partial W})^*(L^2(\partial W,\wedge^*\otimes \mathcal{E}_B),D_{{\rm sign},\mathcal{E}}^{\partial W})=\partial f^*(L^2(W,\wedge^*\otimes \mathcal{E}_B),D_{{\rm sign},\mathcal{E}}^{W})$$
This issue can be addressed in the same way as in the proof of \cite[Theorem 2]{roswein}.

\begin{prop}
Consider a metric space $X$. Let $M$ be an oriented Lipschitz manifold which is either compact or admits a complete metric, that $\mathcal{E}_B\to W$ is a $B$-bundle and $f:M\to X$ a proper uniformly Lipschitz continuous mapping. Then there is a unique class 
$$[(L^2(M,\wedge^*\otimes \mathcal{E}_B),D_{{\rm sign},\mathcal{E}})]\in \Omega_*(\Lip_0(X),B),$$ 
independent of choice of complete Riemannian Lipschitz metric on $M$, Hermitian structure on $\mathcal{E}_B$ and Hermitian Lipschitz connection $\nabla_\mathcal{E}$. 
\end{prop}

The proof goes as in Subsection \ref{diraconcstarbundleexbordism} using Lemma \ref{technicallammaforlipschitzcyc} and is omitted. The next result follows from Theorem \ref{hilsumskandforbla} due to Higson-Skandalis \cite{hilsumskand}. It forms an analogue of the results of Subsection \ref{pscdiracexII}, and proves that a homotopy equivalence of oriented Lipschitz manifolds gives rise to an essentially canonical null-$KK$-bordism. This result can be compared to the work on secondary invariants in \cite{PSsignInd,wxy,zenobi}.

\begin{theorem}
\label{hilsumskandforblaass}
Let $M$ and $N$ be two compact oriented Lipschitz manifolds of dimension $n$ and set $Z:=M\dot{\cup}-N$. Assume that $\mathcal{E}_B\to N$ is a flat Hermitian $B$-bundle, and write $\mathcal{F}_B:=f^*\mathcal{E}_B\dot{\cup}\mathcal{E}_B\to Z$ and pick a flat Hermitian Lipschitz connection $\nabla_\mathcal{F}$. 

If $f:M\to N$ is a homotopy equivalence, then there is a canonical bordism of $(\C,B)$-cycles
$$(L^2(M,\wedge^*\otimes f^*\mathcal{E}_B),D_{{\rm sign},f^*\mathcal{E}})\sim_{\rm bor} (L^2(M,\wedge^*\otimes \mathcal{E}_B),D_{{\rm sign},\mathcal{E}}),$$
constructed from bordism
$$(L^2(Z,\wedge^*\otimes \mathcal{F}_B),D_{{\rm sign},\mathcal{F}})\sim_{\rm bor} (L^2(Z,\wedge^*\otimes \mathcal{F}_B),D_{{\rm sign},\mathcal{F}}+T_F),$$
where $T_f\in \End_B^*(L^2(Z,\wedge^*Z\otimes \mathcal{F}_B))$ is the operator from Theorem \ref{hilsumskandforbla} and the bordism is that defined from bounded perturbations as in Corollary \ref{locbddperturcor}, and the $KK$-nullbordism 
$$ (L^2(Z,\wedge^*\otimes \mathcal{F}_B),D_{{\rm sign},\mathcal{F}}+T_F)\sim_{\rm bor} 0,$$
explicitly constructed from Theorem \ref{weakdegnullbordthm} using that the operator $D_{{\rm sign},\mathcal{F}}+T_f$ is invertible. 
\end{theorem}

\part{Geometry and structure in $KK$-bordisms}
\label{part:geocons}

\section{Geometric constructions with $KK$-bordisms}
\label{sec:prelbord}

$KK$-bordisms have many similarities with manifolds with boundary and classical bordisms. We will make use of some constructions for $KK$-bordisms whose counterpart for manifolds with boundary are geometric in their nature. For instance, the straightening-the-angle technique (see for example \cite[Appendix A]{Rav} and \cite{stong}) that smooths out low-dimensional singularities allows one to define products of manifolds with boundaries. This procedure is important in showing that the double of a manifold is nullbordant (and in extension that the set of bordism classes of manifolds forms a group). We provide analogues of these techniques in the theory of $KK$-bordisms and in the later sections \ref{sec:propoered} and \ref{highihereinda} we apply said techniques to study structural properties of $KK$-bordism groups and higher Atiyah-Patodi-Singer index theory, respectively. A number of these constructions were first considered and/or are inspired by work of Hilsum \cite{hilsumcmodbun, hilsumlip, hilsumfol, hilsumbordism}.

\subsection{Straightening the angle} 
\label{strAngSubSec} 
The classical straightening-the-angle is a technique we will need a ``$KK$-version" of in some special cases. In classical bordism theory, it takes the following form. The reader can find its proof in for instance \cite[Appendix A]{Rav}.

\begin{theorem}
Suppose that $W$ is a topological manifold of dimension $n$ and $M \subset W$ is a closed submanifold, which is smooth, with empty boundary, and of dimension $n-2$. In addition, suppose that
\begin{enumerate}
\item $W\setminus M$ has a differentiable structure;
\item there exists a neighbourhood U of M in W and a homeomorphism
$$ \varphi: U \rightarrow M \times \R^+ \times \R^{+} $$ that identifies $M$ with $M \times \{0\}\times \{0\}$ and is a diffeomorphism on $U \setminus M$.
\end{enumerate}
Then there is a differentiable structure on $W$ which extends that on $W \setminus M$ and makes W a smooth manifold.
\end{theorem}

\begin{figure}[hbt!]
   \centering
  \includegraphics[height=7.35cm]{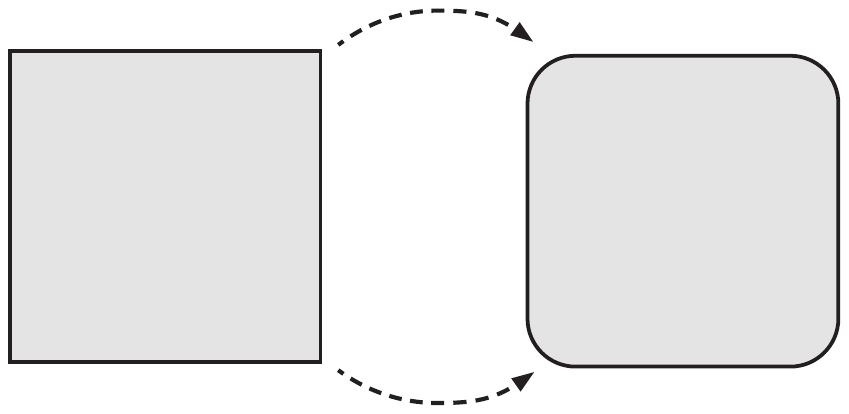}
  \caption{Straightening the angles on the unit square.}
  \label{straighteningangle3}
\end{figure}

\begin{ex}
\label{doublingmfd}
Let $W$ be a smooth manifold with boundary. Then we can straighten the angle on $W \times [0,1]$ to obtain a smooth manifold $\hat{W}$ which is homeomorphic to $W \times [0,1]$ by a map smooth outside of $\partial W\times \{0,1\}$. In particular, the boundary of $\hat{W}$ is homeomorphic to the double $2W$ of $W$ and is obtained as the union
$$W\times \{0\} \cup_{\partial W} \partial W \times [0,1] \cup_{\partial W} -W \times \{1\}. $$
For example, if $W=[0,1]$, then $W\times [0,1]$ is the unit cube, $\hat{W}$ is up to diffeomorphism the closed unit disk and $\partial \hat{W}=S^1$ is realized as  the union of four subintervals of $S^1$ that $\partial\left(W\times [0,1]\right)$ consists of. 
\end{ex}

\begin{ex}
\label{gluingalongjointpiece}
Let $W_1$, $W_2$ and $W_3$ be smooth manifolds with the same boundary $M$. Moreover, suppose that 
$$ W_1 \cup_{M} W_2 = \partial Z_{12} \hbox{ and } W_2 \cup_{M} W_3 = \partial Z_{23} $$
for some smooth manifolds with boundary $Z_{12}$ and $Z_{23}$. Then we can straighten the angle on $Z_{13}:=Z_{12}\cup_{W_2} Z_{23}$ to obtain a smooth manifold with boundary. It holds that
$$\partial Z_{13}=W_1 \cup_{M} W_3$$
is a boundary. In other words, an explicit bounding manifold can be obtained from $Z_{12}$ and $Z_{23}$. It is at first not smooth, but can be ``straightened". This explicitly realized transitivity property will be relevant in the upcoming work \cite{monographtwo} building a relative homology theory where relative cycles are formed by the relations of the theory. In our case the relations come from $KK$-bordisms and we will need similar properties.
\end{ex}

\subsection{The double of a bordism}

In the same way that we can double a manifold (see Example \ref{doublingmfd}) and obtain a boundary, one can do the same for a $KK$-bordism. We will now double a $KK$-bordism and realize it as the boundary of the product with the interval $[0,1]$ after ``straightening the angle". Recall the suspension of a chain by the unit interval from Definition \ref{suspbyinte} and its properties given in Proposition \ref{suspandprop}.

\begin{figure}[hbt!]
   \centering
  \includegraphics[height=5.35cm]{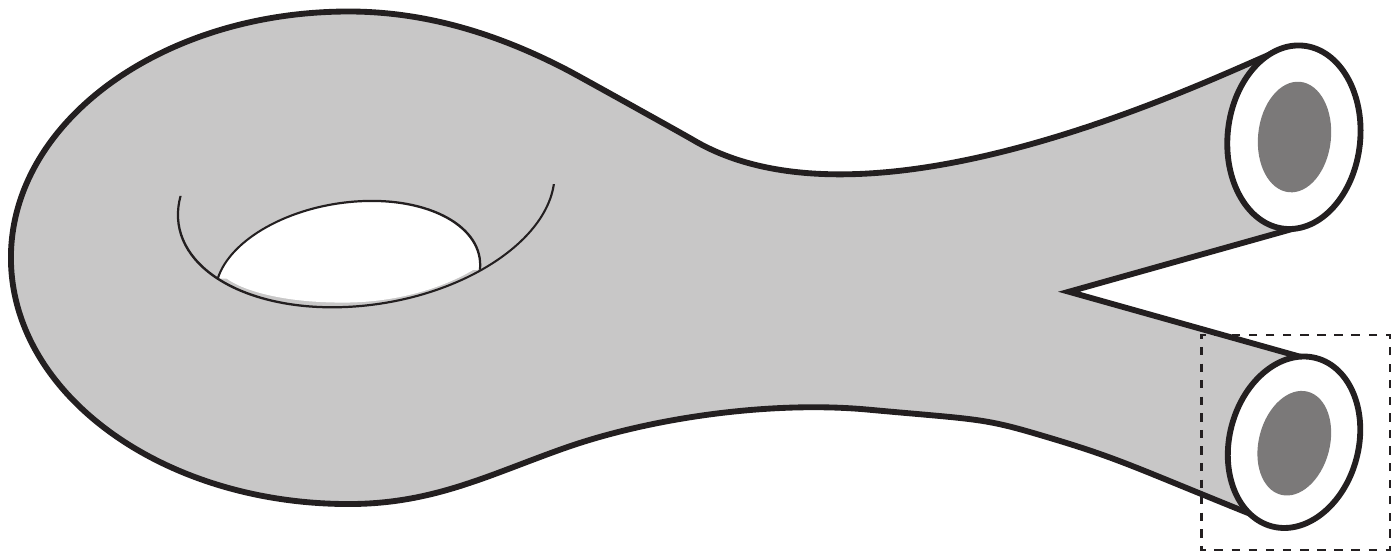}
  \caption{Picture of $W\times [0,1]$, for $W$ the grey surface with boundary, representing $\Psi(\mathpzc{N},T)$ as appearing in Lemma \ref{restsups}.}
  \label{straighteningangle1}
\end{figure}

The suspension operation does not map a bordism to a bordism. There are two obvious issues. Firstly, the suspension of a boundary is not a cycle as is seen classically in Figure \ref{straighteningangle1}. This issue is due to the fact that the boundary of a suspension changes, and the $KK$-theoretic suspension should change the boundary accordingly. Secondly, the suspension operation destroys the structure of the boundary data $\Theta$ that geometrically corresponds to a collar neighborhood. The last issue is of a more technical nature and will be fixed by ``straightening the angle" at a $KK$-theoretical level.

Define the piecewise smooth domain $\Omega_0:=[1,2]\times [0,1]\subseteq \R^2$. We will fix a choice of a piecewise smooth compact convex domain $\Omega\subseteq [1,2]\times [0,1]\subseteq \R^2$ satisfying 
\begin{enumerate}
\item[$\Omega1$] $[1,4/3]\times [0,1]\subseteq \Omega$;
\item[$\Omega2$] $[5/3,2]\times [0,1]\cap \Omega=\emptyset$; and, 
\item[$\Omega3$] $(-\infty,1]\times [0,1]\cup\Omega$ is a domain with smooth boundary. 
\end{enumerate}

\begin{figure}
   \centering
  \includegraphics[height=5.35cm]{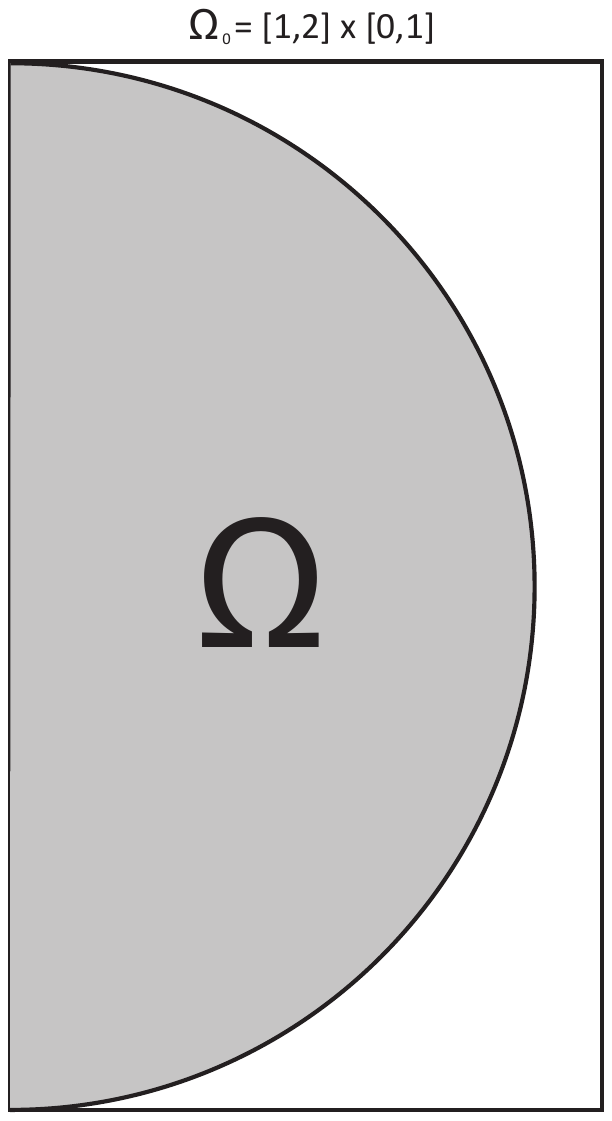}
  \caption{The domain $\Omega$.}
\label{section42omega}
\end{figure}

A picture of a choice of $\Omega$ can be found in Figure \ref{section42omega}. The way that we will use $\Omega$ can be seen in Figure \ref{straighteningangle2}. If we have a manifold with boundary $W$, we will instead of straightening the angle on $W\times [0,1]$ as in Example \ref{doublingmfd}), glue on the piece $\partial W\times \Omega$ onto $W\times [0,1]$ along $\partial W\times \{1\}\times [0,1]$ when identifying a collar neighborhood of $W$ with $\partial W\times [0,1]$. The manifold 
$$\Sigma_\Omega W:=(W\times [0,1])\cup_{W\times \{1\}\times [0,1]}\partial W\times \Omega,$$ 
already has a smooth boundary by condition $\Omega3$ and its own collar neighborhood structure.

We let $S_\Omega=\Omega\times \C^2\to \Omega$ denote the complex spinor bundle. It is graded by $1\oplus (-1)$. We will use a metric on $\Omega$ which is of product type near the boundary. We consider the differential expression $\slashed{D}_\Omega$ giving the Dirac operator on $\Omega$. We write $D_\Omega$ for the minimal extension of $\slashed{D}_\Omega$, in other words, the closure of $\slashed{D}_\Omega$ with the domain $C^\infty_c(\Omega^\circ,S_\Omega)$. The operator $D_\Omega$ is symmetric and $D_\Omega^*$ is the differential expression $\slashed{D}_\Omega$ equipped with its maximal domain in $L^2(\Omega^\circ,S_\Omega)$ consisting of those $L^2$-elements $u$ such that $\slashed{D}_\Omega u\in L^2(\Omega^\circ,S_\Omega)$ in a distributional sense.

The goal now is to use $\Omega$ to construct a restricted suspension of symmetric chains. We take a symmetric $(\mathcal{A},B)$-chain with boundary 
$$\mathfrak{X}=(\mathfrak{X}^\circ,\Theta,\mathfrak{X}^\partial)=((\mathpzc{N},T),(\theta,p),(\mathpzc{E},D)).$$
We define the $B$-module $\Sigma_\Omega\mathpzc{N}\subseteq L^2[0,1]\hat{\boxtimes}\mathpzc{N}\oplus L^2([0,1]^2\cup \Omega,S_\Omega)\hat{\boxtimes}\mathpzc{E}$ from the pullback diagram
$$\begin{CD}
\Sigma_\Omega\mathpzc{N} @>>> L^2[0,1]\hat{\boxtimes}\mathpzc{N} \\
@VVV @V\id_{L^2[0,1]}\hat{\boxtimes}\theta \circ pVV \\
L^2([0,1]^2\cup \Omega,S_\Omega)\hat{\boxtimes}\mathpzc{E}@>>> L^2([0,1]^2,S_\Omega)\hat{\boxtimes}\mathpzc{E}.
\end{CD}$$
Here the bottom horizontal map is the restriction mapping. We think of $\Sigma_\Omega\mathpzc{N}$ as a non-orthogonal sum
$$L^2([0,1]^2\cup \Omega,S_\Omega)\hat{\boxtimes}\mathpzc{E}+L^2[0,1]\hat{\boxtimes}\mathpzc{N}.$$

The $B$-module $\Sigma_\Omega\mathpzc{N}$ has a left action $b_\Omega:C^\infty([0,1]^2\cup\Omega,\mathcal{A})\to\End^*_B(\Sigma_\Omega\mathpzc{N})$ defined from the actions 
$$C^\infty([0,1]^2,\mathcal{A})=C^\infty([0,1],C^\infty([0,1],\mathcal{A}))\to\End^*_B(L^2[0,1]\hat{\boxtimes}\mathpzc{N}),$$ 
and $C^\infty([0,1]^2\cup\Omega,\mathcal{A})\to\End^*_B(L^2([0,1]^2\cup \Omega,S_\Omega)\hat{\boxtimes}\mathpzc{E})$.

We define the operator $\Sigma_\Omega T$ densely on $\Sigma_\Omega\mathpzc{N}$ as follows. Take a $\chi=\chi(x,y)\in C^\infty([0,1]^2\cup\Omega)$ only depending on $x$ and satisfying that $\chi=1$ for $x$ near $0$ and $\chi=0$ for $x>2/3$. For 
$$x_1+x_2\in b_\Omega(\chi)\Dom(\Psi(T))+b_\Omega(1-\chi)\Dom(D_\Omega\hat{\boxtimes} D),$$
define 
$$\Sigma_\Omega T(x_1+x_2):=\Psi(T)x_1+(D_\Omega\hat{\boxtimes} D)(x_2).$$
It follows from Condition b) in Definition \ref{HilBorDef} that $\Sigma_\Omega T$ is well defined on the dense submodule $b_\Omega(\chi)\Dom(\Psi(T))+b_\Omega(1-\chi)\Dom(D_\Omega\hat{\boxtimes} D)\subseteq \Sigma_\Omega\mathpzc{N}$.

\begin{figure}
   \centering
  \includegraphics[height=5.35cm]{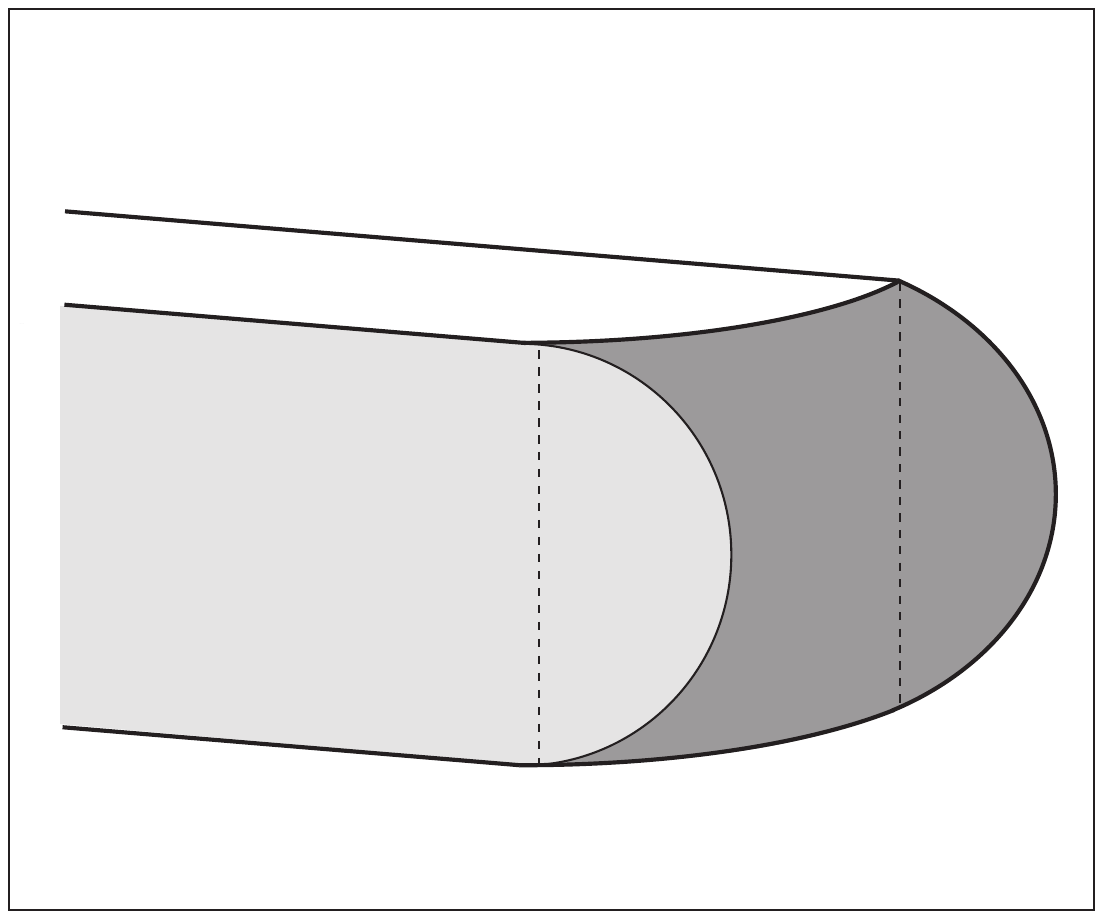}
  \caption{How straightening the angle on a bordism, resulting in $(\Sigma_\Omega \mathpzc{N},\Sigma_\Omega T)$ as in Lemma \ref{restsups}, would appear locally. The cross section is $\Omega$ from Figure \ref{section42omega}.}
  \label{straighteningangle2}
\end{figure}

Define the sets
$$\Gamma_0:=\partial([0,1]^2\cup\Omega)\quad\mbox{and}\quad \Gamma_1:=\Gamma_0\setminus \{0\}\times [0,1].$$

\begin{lemma}
\label{restsups}
Let $\mathfrak{X}=((\mathpzc{N},T),(\theta,p),(\mathpzc{E},D))$ be a symmetric $(\mathcal{A},B)$-chain with boundary. The pair $(\Sigma_\Omega \mathpzc{N},\Sigma_\Omega T)$ is a symmetric $(C^\infty([0,1]^2\cup\Omega,\mathcal{A}),B)$-chain. Moreover, $(\Sigma_\Omega \mathpzc{N},\Sigma_\Omega T)$ is half-closed as a $(C^\infty_c([0,1]^2\cup\Omega\setminus \Gamma_0,\mathcal{A}),B)$-chain, and if $\mathfrak{X}$ is a bordism then also as a $(C^\infty_c([0,1]^2\cup\Omega\setminus \Gamma_1,\mathcal{A}),B)$-chain.
\end{lemma}

\begin{proof}
The operator $ \Sigma_\Omega T$ is readily verified to be symmetric on a suitable core so it is therefore symmetric. An elementary computation shows that $\Sigma_\Omega T$ has bounded commutators with $C^\infty_c([0,1]^2\cup\Omega\setminus \Gamma_0,\mathcal{A})$. It follows from Remark \ref{interohc} and Proposition \ref{suspandprop} that the inclusions 
$$\Dom(\Psi(T))\hookrightarrow L^2[0,1]\hat{\boxtimes}\mathpzc{N}\quad\mbox{and}\quad \Dom(D_\Omega\hat{\boxtimes}D)\hookrightarrow L^2([0,1]^2\cup \Omega,S_\Omega),$$
are locally compact for the $C^\infty_c([0,1]^2\cup\Omega\setminus \Gamma_0,\mathcal{A})$-action. Therefore, if $\Sigma_\Omega T$ was regular, it would follow that
 $$b_\Omega(\phi)\left(1+(\Sigma_\Omega T)^*\Sigma_\Omega T\right)^{-1}\in \mathbb{K}_B(\Sigma_\Omega \mathpzc{N}),\quad\forall\phi\in C^\infty_c([0,1]^2\cup\Omega\setminus \Gamma_0,\mathcal{A}).$$ 
The same argument works for $\phi\in C^\infty_c([0,1]^2\cup\Omega\setminus \Gamma_1,\mathcal{A})$ when $\mathfrak{X}$ is a bordism.

That leaves us to show that $\Sigma_\Omega T$ is regular and 
\begin{equation}
\label{inclusioofdomn}
C^\infty_c([0,1]^2\cup\Omega\setminus \Gamma_1,\mathcal{A})\Dom(\Sigma_\Omega T)^*\subseteq \Dom(\Sigma_\Omega T).
\end{equation}
To do so we will use the symmetric chain
$$(\mathpzc{N}_\infty,T_\infty):=\mathfrak{X}\#_{(\mathpzc{E},D)}((L^2([0, \infty)\hat{\boxtimes}\mathpzc{E}, \partial_x^{\min}\hat{\boxtimes}D), (\id_{L^2([0, 1]\hat{\boxtimes}\mathpzc{E}}, \chi_{[0,1]}),(\mathpzc{E},D))$$
The symmetric chain with boundary 
$$((L^2([0, \infty)\hat{\boxtimes}\mathpzc{E}, \partial_x^{\min}\hat{\boxtimes}D), (\id_{L^2([0, 1]\hat{\boxtimes}\mathpzc{E}}, \chi_{[0,1]}),(\mathpzc{E},D))$$ 
was considered also in Example \ref{halflineex}. Since we are gluing along a joint boundary, $T_\infty$ is a self-adjoint and regular extension of $T$ to the glued on infinite cylinder (see \cite[Proof of Theorem 2.20]{DGM}). Define the operator 
$$S:=\partial_x\hat{\boxtimes}T_\infty,$$
on $L^2(\R)\hat{\boxtimes}\mathpzc{N}_\infty$. This is a self-adjoint and regular operator by \cite[Theorem 1.3]{DGM}. We identify $\Sigma_\Omega\mathpzc{N}$ with a complemented submodule of $L^2(\R)\hat{\boxtimes}\mathpzc{N}_\infty$ using that $[0,1]^2\cup \Omega\subseteq [0,\infty)\times \R$. Let $p_\Omega\in \End_B^*(L^2(\R)\hat{\boxtimes}\mathpzc{N}_\infty)$ denote the projection onto $\Sigma_\Omega\mathpzc{N}$. 

By construction, we can identify 
$$\Sigma_\Omega T=\overline{S|_{C^\infty_c([0,1]^2\cup \Omega\setminus \Gamma_1)\Dom(S)}}.$$
We claim that $(\Sigma_\Omega T)^*$ is the closure of the operator $Q$ defined by $\Dom(Q):=p_\Omega\Dom(S)$ and $Q(p_\Omega x):=p_\Omega Sx$. The operator $Q$ is well defined since $p_\Omega x=0$ implies that $p_\Omega Sx=0$ for $x\in \Dom(S)$. Indeed, if $x=p_\Omega x_0\in \Dom(Q)$, for an $x_0\in \Dom(S)$, and $y\in C^\infty_c([0,1]^2\cup \Omega\setminus \Gamma_1)\Dom(S)\subseteq \Dom(\Sigma_\Omega T)$ we can use that $S$ is self-adjoint to show that 
$$\langle x,(\Sigma_\Omega T)y\rangle=\langle x,Sy\rangle=\langle x_0,Sy\rangle=\langle Sx_0,y\rangle=\langle p_\Omega Sx_0,y\rangle=\langle Qx,y\rangle.$$
This shows that $Q\subseteq (\Sigma_\Omega T)^*$. To prove the reverse inclusion, we take $\xi=\xi_1+\xi_2\in \Dom((\Sigma_\Omega T)^*)\subseteq L^2[0,1]\hat{\boxtimes}\mathpzc{N}+L^2([0,1]^2\cup \Omega,S_\Omega)\hat{\boxtimes}\mathpzc{E}$. Using \cite[Theorem 1.18]{DGM}, we can assume that $\xi_1=0$ and that $\xi_2=0$ near $[0,1]^2\subseteq [0,1]^2\cup \Omega$. Using classical techniques of mollifiers in the $\Omega$-direction (see for instance \cite[Theorem 6.7]{BaerBall}), we can assume that $\xi_2\in C^\infty_c(\Omega\setminus \{1\}\times[0,1],\Dom(D))$ and $\xi_2$ will therefore take the form $\xi_2\in p_\Omega\Dom(S)$. The inclusion \eqref{inclusioofdomn} follows. Moreover, the argument above holds in any localization and the local/global principle \cite[Theorem 1.18]{Pierrot}, \cite[Theorem 4.2]{leschkaad2} shows that $\Sigma_\Omega T$ is regular.
\end{proof}

Let $U$ denote a collar neighborhood of $\Gamma_1$ and define $p_U\in \End_B^*(\Sigma_\Omega\mathpzc{N})$ as multiplication by the characteristic function of $U$ using the representation $b_\Omega$. Using that $U$ is a collar neighborhood, we can find a unitary isomorphism 
$$\theta_U:p_U\Sigma_\Omega\mathpzc{N}\to \Psi(\mathpzc{N}\oplus\Psi(\mathpzc{E})\oplus \mathpzc{N}),$$
and since we assume the metric is of product type, $\Sigma_\Omega(T)$ acts on $p_U\Sigma_\Omega\mathpzc{N}$ as of product type as in Definition \ref{HilBorDef}.b. We can therefore make the following definition using Lemma \ref{restsups}.

\begin{define}
\label{straightofdoubl}
Let $\mathfrak{X}=(\mathfrak{X}^\circ,\Theta,\mathfrak{X}^\partial)=((\mathpzc{N},T),(\theta,p),(\mathpzc{E},D))$ be a symmetric chain with boundary. We define the restricted suspension as the symmetric chain with boundary
$$\Sigma_\Omega\mathfrak{X}:=(\Sigma_\Omega\mathfrak{X}^\circ,\Sigma_\Omega\Theta,\mathfrak{X}\#_{\partial \mathfrak{X}}\Psi(\partial \mathfrak{X})\#_{\partial \mathfrak{X}}-\mathfrak{X}),$$
where $\Sigma_\Omega\mathfrak{X}^\circ:=(\Sigma_\Omega\mathpzc{N},\Sigma_\Omega T)$ and $\Sigma_\Omega\Theta=(\theta_U,p_U)$.
\end{define}

The following theorem holds in the case of a bordism using Lemma \ref{restsups}.

\begin{theorem}
\label{doublingup}
Let $\mathfrak{X}$ be an $(\mathcal{A},B)$-bordism. The restricted suspension is a well defined $(\mathcal{A},B)$-bordism satisfying that
$$\partial  \Sigma_\Omega\mathfrak{X}=\mathfrak{X}\#_{\partial \mathfrak{X}}\Psi(\partial \mathfrak{X})\#_{\partial \mathfrak{X}}-\mathfrak{X}.$$
\end{theorem}

\begin{cor}
Let $\mathfrak{X}$ be an $(\mathcal{A},B)$-bordism. The doubled cycle 
$$2\mathfrak{X}:=\mathfrak{X}\#_{\partial \mathfrak{X}}-\mathfrak{X},$$
is null-bordant.
\end{cor}

\begin{remark}
\label{doublingcycle}
Suppose that $\mathfrak{X}=(\mathfrak{X}^\circ,0,0)$ is a bordism. By \cite[Proposition 2.10]{DGM}, $\mathfrak{X}^\circ=(\mathpzc{E},D)$ is a cycle. It follows from the construction, that 
$$\Sigma_\Omega\mathfrak{X}=(\Psi(\mathfrak{X}^\circ),\theta,\mathfrak{X}^\circ\oplus -\mathfrak{X}^\circ),$$
for suitable boundary data, e.g. $\Theta=(\theta,\chi_{[0,1/4]\cup[3/4,1]})$ where $\theta$ is an isomorphism $L^2([0,1/4]\cup[3/4,1])\hat{\boxtimes}\mathpzc{E}\cong L^2[0,1]\hat{\boxtimes}(\mathpzc{E}\oplus -\mathpzc{E})$. 
\end{remark}

\subsection{Cutting and pasting}

It is also possible to cut-and-paste $KK$-bordisms almost as in Example \ref{gluingalongjointpiece}, up to canonical null-bordant cycles. A suitable rephrasing is as follows: when $\mathfrak{X}_1$, $\mathfrak{X}_2$ and $\mathfrak{X}_3$ are $KK$-bordisms with the same boundary $\mathfrak{Y}$, and $\mathfrak{X}_1\#_{\mathfrak{Y}}\mathfrak{X}_2$ and $\mathfrak{X}_2\#_{\mathfrak{Y}}\mathfrak{X}_3$ are nullbordant then via a canonical bordism, the cycle $\mathfrak{X}_1\#_{\mathfrak{Y}}\mathfrak{X}_3$ is also nullbordant.  

The next lemma is formulated using the ideas appearing in Bunke's proof of his relative index theorem \cite{bunkerelative}. The reader can find similar constructions in \cite[Section 10]{hilsumbordism}.

\begin{lemma}
\label{cliffsymlem}
Suppose that $\mathfrak{X}_1$, $\mathfrak{X}_2$ and $\mathfrak{X}_3$ are $KK$-bordisms with the same boundary $\mathfrak{Y}$. Then the cycle 
$$\mathfrak{X}_1\#_{\mathfrak{Y}}\mathfrak{X}_2+(-\mathfrak{X}_2)\#_{\mathfrak{Y}}\mathfrak{X}_3+\mathfrak{X}_1\#_{\mathfrak{Y}}\mathfrak{X}_3-\mathfrak{X}_2\#_{\mathfrak{Y}}\mathfrak{X}_2,$$
admits a Clifford symmetry (as defined in Definition \ref{cliffsymdef}).
\end{lemma}

\begin{proof}
Let $\mathpzc{E}_{ij}$ denote the module appearing in $\mathfrak{X}_i\#_{\mathfrak{Y}}\mathfrak{X}_j$ and define representations $b_{ij}:C([0,1],A)\to \End_B^*(\mathpzc{E}_{ij})$ by the action on the joint cylinder over which the gluing takes place. Take a $\chi_1,\chi_2\in C^\infty([0,1],[0,1])$ such that $\chi_1^2+\chi_2^2=1$ and $\chi_1\in C^\infty_c((0,1])$ while $\chi_2\in C^\infty_c([0,1))$. We can identify $b_{ij}(\chi_1)\mathpzc{E}_{ij}=b_{kj}(\chi_1)\mathpzc{E}_{kj}$ and $b_{ij}(\chi_2)\mathpzc{E}_{ij}=b_{ik}(\chi_2)\mathpzc{E}_{ik}$ for any $i,j,k$. We define $a:\mathpzc{E}_{12}\to \mathpzc{E}_{13}$ by the composition of $b(\chi_1)$ with this identification. The operators $c:\mathpzc{E}_{23}\to \mathpzc{E}_{13}$, $d:\mathpzc{E}_{23}\to \mathpzc{E}_{22}$, and $e:\mathpzc{E}_{12}\to \mathpzc{E}_{22}$ are defined similarly. Define the symmetry 
$$J:=i\gamma
\begin{pmatrix}
0&0& -a^*&-e^*\\
0&0&-c^*& d^*\\
a&c&0&0\\
e& -d&0&0\end{pmatrix}\in \End_B^*(\mathpzc{E}_{12}\oplus \mathpzc{E}_{23}\oplus \mathpzc{E}_{13}\oplus \mathpzc{E}_{22}).$$
A lengthier computation using the local structure of the operators coming from Definition \ref{HilBorDef}.b, shows that $J$ is a Clifford symmetry for the prescribed cycle (cf. Definition \ref{cliffsymdef}).
\end{proof}

\begin{theorem}
\label{cangluing}
Suppose that $\mathfrak{X}_1$, $\mathfrak{X}_2$ and $\mathfrak{X}_3$ are $KK$-bordisms with the same boundary $\mathfrak{Y}$ and that there are $KK$-bordisms $\mathfrak{X}_{12}$ and $\mathfrak{X}_{23}$ with 
$$\partial \mathfrak{X}_{12}=\mathfrak{X}_1\#_{\mathfrak{Y}}\mathfrak{X}_2\quad \mbox{and}\quad \partial \mathfrak{X}_{23}=\mathfrak{X}_2\#_{\mathfrak{Y}}\mathfrak{X}_3,$$
then there is a canonically associated bordism $\mathfrak{X}_{13}$ with $\partial \mathfrak{X}_{13}=\mathfrak{X}_1\#_{\mathfrak{Y}}\mathfrak{X}_3$.
\end{theorem}

\begin{proof}
It follows from Proposition \ref{cliffsymlemprop}, Lemma \ref{cliffsymlem} and transitivity of the bordism relation that there is a canonical bordism
$$\mathfrak{X}_1\#_{\mathfrak{Y}}\mathfrak{X}_3\sim \mathfrak{X}_1\#_{\mathfrak{Y}}\mathfrak{X}_2+\mathfrak{X}_2\#_{\mathfrak{Y}}\mathfrak{X}_3-\mathfrak{X}_2\#_{\mathfrak{Y}}\mathfrak{X}_2.$$
The theorem follows from that $\Omega_*(\mathcal{A},B)$ is a group.
\end{proof}

\subsection{Gluing on an infinite cylinder}
\label{subsecgluinginfinite}

A $KK$-bordism should be thought of as a noncommutative analogue of a manifold with boundary whose geometric structures near the boundary are of product type. As such, we should be able to glue on an infinite cylinder onto the boundary component. This idea is found already in the work of Hilsum \cite{hilsumcmodbun}. Recall the suspension of a chain by the half interval $[0,\infty)$ from Definition \ref{suspbyintehalf} and its properties given in Proposition \ref{suspandprophalf}.

\begin{prop}
If $\mathfrak{Y}=(\mathpzc{E},D)$ is a closed $(\mathcal{A},B)$-cycle, then $\Psi_{(0,\infty)}(\mathfrak{Y})$
is a well defined half-closed $(C^\infty_c((0,\infty),\mathcal{A}),B)$-chain and the data $\Theta=(\id, \chi_{[0,1]}\hat{\boxtimes}1)$ makes $\Psi_{(0,\infty)}(\mathfrak{Y})$ into a symmetric chain with boundary.
\end{prop}

We remark that $\Psi_{(0,\infty)}(\mathfrak{Y})$ is never closed unless $\mathfrak{Y}=0$. We omit the proof as it follows directly from Proposition \ref{suspandprophalf}.

Assume that $\mathfrak{X}=(\mathfrak{X}^\circ,\Theta,\mathfrak{X}^\partial)$ is a symmetric chain with boundary for $(\mathcal{A},B)$. We define a half-closed $(C^\infty_c((0,\infty),\mathcal{A}),B)$-chain $\mathfrak{X}_{(0,\infty)}$ by setting 
$$\mathfrak{X}_{(0,\infty)}^\circ:=\mathfrak{X}\#_{\mathfrak{X}^\partial}\Psi_{(0,\infty)}(\mathfrak{X}^\partial).$$
We say that $\mathfrak{X}_{(0,\infty)}^\circ$ is $\mathfrak{X}$ with an infinite cylinder glued on. For notational purposes, we write $\mathfrak{X}_{(0,\infty)}=(\mathpzc{N}_{(0,\infty)},T_{(0,\infty)})$. Note that as Hilbert $C^*$-modules, $\mathpzc{N}_{(0,\infty)}=\mathpzc{N}\oplus L^2[1,\infty)\hat{\boxtimes}\mathpzc{M}$ for $\mathfrak{X}^\partial=(\mathpzc{M},S)$.

\begin{prop}
\label{gluinginfcyleprop}
Let $\mathfrak{X}$ be a bordism. The chain $\mathfrak{X}_{(0,\infty)}^\circ=(\mathpzc{N}_{(0,\infty)},T_{(0,\infty)})$ defines a half-closed $(C^\infty_c([0,\infty),\mathcal{A}),B)$-chain. This half-closed chain satisfies the following:
\begin{enumerate}
\item $T_{(0,\infty)}$ is self-adjoint.
\item Equipped with the boundary data $\Theta=(\id, \chi_{[0,1]}\hat{\boxtimes}1)$, we have that $\mathfrak{X}^{\rm s.a.}:=(\mathfrak{X}^\circ_{(0,\infty)},\Theta, \mathfrak{X}^\partial)$ is a bordism for $(\mathcal{A},B)$ whose interior operator is self-adjoint and 
$$\partial \mathfrak{X}=\partial \mathfrak{X}^{\rm s.a.}.$$
\end{enumerate}
\end{prop}

\begin{proof}
It follows directly from the proof of \cite[Theorem 2.20]{DGM} that  $\mathfrak{X}_{(0,\infty)}$ is a half-closed $(C^\infty_c([0,\infty),\mathcal{A}),B)$-chain with $T_{(0,\infty)}$ is self-adjoint. It is therefore immediate from the definition of $KK$-bordism that $\mathfrak{X}^{\rm s.a.}$ is a bordism.
\end{proof}

\begin{lemma}
\label{invertinginfsusp}
Let $\mathpzc{E}$ be a (graded) $B$-Hilbert $C^*$-module and $D:\mathpzc{E}\dashrightarrow \mathpzc{E}$ be a self-adjoint regular operator (which is odd if $\mathpzc{E}$ is graded). Then the operator 
$$\Psi_\infty(D):=i\frac{\partial}{\partial x}\hat{\boxtimes}D:L^2(\R)\hat{\boxtimes}\mathpzc{E}\dashrightarrow L^2(\R)\hat{\boxtimes}\mathpzc{E},$$
is self-adjoint and regular. If $D$ is invertible, then $\Psi_\infty(D)$ is invertible. 
\end{lemma}

\begin{proof}
It follows from \cite{DGM} that $\Psi_\infty(D)$ is regular and self-adjoint, as it is the exterior product of two regular and self-adjoint operators. The reader should note that in particular, $\Psi_\infty(D)$ is closed. 

Assume now that $D$ is invertible. We shall explicitly invert $\Psi_\infty(D)$ in the case that $\mathpzc{E}$ is not graded, and in which case 
$$\Psi_\infty(D)=\begin{pmatrix}0& \frac{\partial}{\partial x}+D\\ -\frac{\partial}{\partial x}+D&0\end{pmatrix}.$$
The case that $\mathpzc{E}$ is graded (and $D$ is odd) is proven analogously. It clearly suffices to prove that $\frac{\partial}{\partial x}+D:L^2(\R)\boxtimes\mathpzc{E}\dashrightarrow L^2(\R)\boxtimes\mathpzc{E}$ is invertible. Let $\mathcal{F}:L^2(\R)\boxtimes\mathpzc{E}\to L^2(\R)\boxtimes\mathpzc{E}$ denote the unitary Fourier transform in the first leg. For $f\in L^2(\R,\mathpzc{E})=L^2(\R)\boxtimes\mathpzc{E}$, we formally write 
\begin{equation}
\label{foruicc}
\mathcal{F}\left(\frac{\partial}{\partial x}+D\right)^{-1}\mathcal{F}^*f(\xi)=(i\xi+D)^{-1}f(\xi).
\end{equation}
Since $D$ is invertible, it has a spectral gap at $0$. Combining this fact with self-adjointness of $D$ it follows that the operator in \eqref{foruicc} is a bounded operator on $L^2(\R)\boxtimes\mathpzc{E}$. This shows that $\Psi_\infty(D)$ is invertible. For completeness, we note that an elementary Fourier transform computation shows
$$\left(\frac{\partial}{\partial x}+D\right)^{-1}f(x)=\int_{-\infty}^x P\mathrm{e}^{-(x-y)D}f(y)\mathrm{d}y+\int_x^{\infty}(P- 1)\mathrm{e}^{-(y-x)D}f(y)\mathrm{d}y$$
where $P:=\chi_{[0,\infty)}(D)\in \End_B^*(\mathpzc{E})$ is well defined since $D$ is invertible.
\end{proof}

\begin{lemma}
\label{infcylfred}
Let $\mathfrak{X}$ be a bordism with $\mathfrak{X}^\circ=(\mathpzc{N},T)$ and $\mathfrak{X}^\partial=(\mathpzc{M},S)$. If $S$ is invertible, the operator $T_{(0,\infty)}$ constructed above is $B$-Fredholm.
\end{lemma}

We remark that the resolvent of the self-adjoint operator $T_{(0,\infty)}$ is compact if and only if $\mathfrak{X}^\partial=0$ in which case $T_{(0,\infty)}=T=T^*$.

\begin{proof}
The operator $T_{(0,\infty)}$ acts on $\mathpzc{N}_{(0,\infty)}$ which in turn carries an action of $C_b[0,\infty)$. Take $\chi_1\in C^\infty_c[0,\infty),[0,1])$ satisfying $\chi_1(t)=1$ near $t=0$ and $\chi_1(t)=0$ for $t>1/2$. Set $\chi_2:=1-\chi_1\in C^\infty_b[0,\infty)$. We can view $\Psi_{(0,\infty)}(\mathpzc{M})$ as a summand in $\Psi_{\infty}(\mathpzc{M})$, and by Lemma \ref{invertinginfsusp}, the operator
$$r:=b(\sqrt{\chi_1})T(1+T^*T)^{-1}b(\sqrt{\chi_1})+b(\sqrt{\chi_2})\Psi_\infty(S)^{-1}b(\sqrt{\chi_2})\in \End_B^*(\mathpzc{N}_{(0,\infty)},$$
is well defined. 
We compute that 
\begin{align*}
T_{(0,\infty)}r=&b(\sqrt{\chi_1}')T(1+T^*T)^{-1}b(\sqrt{\chi_1})+b(\sqrt{\chi_2}')\Psi_\infty(S)^{-1}b(\sqrt{\chi_2})+\\
&+b(\sqrt{\chi_1})T^*T(1+T^*T)^{-1}b(\sqrt{\chi_1})+b(\chi_2)=\\
&=b(\sqrt{\chi_1}')T(1+T^*T)^{-1}b(\sqrt{\chi_1})+b(\sqrt{\chi_2}')\Psi_\infty(S)^{-1}b(\sqrt{\chi_2})-\\
&-b(\sqrt{\chi_1})(1+T^*T)^{-1}b(\sqrt{\chi_1})+1
\end{align*}
Since $\chi_1$ and $\chi_2'$ are compactly supported, we conclude that $T_{(0,\infty)}r-1$ is compact.

In the same way see that $rT_{(0,\infty)}-1$ is compact as it suffices to prove that $T_{(0,\infty)}r^*-1$ is compact which follows from the computation above upon noting that 
\begin{align*}
r^*:&=b(\sqrt{\chi_1})(1+T^*T)^{-1}T^*b(\sqrt{\chi_1})+b(\sqrt{\chi_2})\Psi_\infty(S)^{-1}b(\sqrt{\chi_2})=\\
&=b(\sqrt{\chi_1})T^*(1+TT^*)^{-1}b(\sqrt{\chi_1})+b(\sqrt{\chi_2})\Psi_\infty(S)^{-1}b(\sqrt{\chi_2}).
\end{align*}

In conclusion, $T_{(0,\infty)}$ has an inverse modulo $B$-compacts in $r$. We conclude that $T_{(0,\infty)}$ is $B$-Fredholm.

\end{proof}

For the purpose of constructions, it is sometimes convenient to work with a different type of $KK$-bordisms, $KK$-bordisms with infinite cylindrical ends. This notion should be viewed in the same way as one attaches infinite cylinders to the boundary of a manifold with boundary; this procedure does not alter the homeomorphism type of the interior but can rather be viewed as a change of metric to a complete metric on the interior. We formalize the notion of a bordism with infinite cylindrical ends in the following definition.

\begin{define} 
\label{HilBorDefInf} 
A symmetric chain with cylindrical ends is a collection $\mathfrak{X}=(\mathfrak{Z},\Theta,\mathfrak{Y})$ as in Definition \ref{HilBorDef} par the following conditions:
\begin{enumerate}
\item The operator $T$ is self-adjoint
\item The boundary data $\Theta=(\theta,p)$ satisfies that
\begin{itemize}
\item $p$ is a projection in $\End_B^*(\mathpzc{N})$ that commutes with the $A$-action ($p$ is even in the graded situation);
\item $\theta: p\mathpzc{N} \rightarrow L^2[0,\infty)\hat{\boxtimes} \mathpzc{E}$ is an isomorphism.
\end{itemize}
\end{enumerate}
These objects are compatible in the following sense described using the representation 
$$b:C_b([0,\infty),A)\to \End_B^*(\mathpzc{N}), \quad b(\phi):=\phi(0)(1-p)+p\theta^{-1} \phi\theta.$$
We require the following conditions to hold:
\begin{enumerate}
\item[a)] For $\phi\in C^\infty_b(0,\infty)$ with $\phi=0$ near $0$,  
$$b(\phi)\Dom T^*\subseteq \Dom T\quad\mbox{and}\quad T^*b(\phi)=Tb(\phi)\quad\mbox{on}\quad \Dom T^*.$$
\item[b)] For $\phi\in C^\infty_b(0,\infty)$ with $\phi=0$ near $0$ 
$$\phi \Dom \Psi_{(0,\infty)}(D)=\theta b(\phi)\Dom T \quad \mbox{and} \quad T=\theta^{-1}\Psi_{(0,\infty)}(D)\theta\quad \mbox{on}\quad b(\phi)\Dom T.$$
\item[c)] For $\phi_1,\phi_2\in C^\infty_b[0,\infty)$ satisfying $\phi_1\phi_2=0$,  
$$b(\phi_1) T b(\phi_2)=0.$$
\end{enumerate}

\begin{remark}
We note that the representation $b$ extends to the space of all bounded Borel functions on $[0,\infty)$ using that $\theta: p\mathpzc{N} \rightarrow L^2[0,\infty)\hat{\boxtimes} \mathpzc{E}$ is an isomorphism.
\end{remark}

A symmetric chain with cylindrical ends $\mathfrak{X}$ is said to be a bordism with cylindrical ends if additionally $\mathfrak{Z}$ is a half-closed $(C^\infty_c([0,\infty),\mathcal{A}),B)$-chain under $b$. In this case we say that the $(\mathcal{A}, B)$-cycle $\partial \mathfrak{X}:=\mathfrak{Y}$ is the boundary of $\mathfrak{X}$ and that the half-closed $(C^\infty_c([0,\infty),\mathcal{A}),B)$-chain $\mathfrak{X}^\circ:=\mathfrak{Z}$ is the interior of $\mathfrak{X}$. Isomorphisms of $KK$-bordisms with infinite cylindrical ends are defined as in Definition \ref{isoofsymwbound}.
\end{define}

\begin{prop}
\label{trunactiopiad}
Let $\mathfrak{X}=(\mathfrak{Z},\Theta,\mathfrak{Y})$ be an $(\mathcal{A},B)$-bordism with cylindrical ends, with $\mathfrak{Z}=(\mathpzc{N},T)$, $\Theta=(\theta,p)$ and $\mathfrak{Y}=(\mathpzc{E},D)$. Define the $B$-Hilbert $C^*$-module 
$$\mathpzc{N}|_{[0,1]}:=\overline{b(C_0[0,1))\mathpzc{N}}=b(C[0,1])\mathpzc{N},$$
and the densely defined operator $T|_{[0,1]}$ on $\mathpzc{N}|_{[0,1]}$ as the closure of $T$ restricted to $b(C^\infty_c[0,1))\Dom(T)$. Write $\mathfrak{Z}|_{[0,1]}:=(\mathpzc{N}|_{[0,1]},T|_{[0,1]})$. Then it holds that the collection
$$\mathfrak{X}|_{[0,1]}=(\mathfrak{Z}|_{[0,1]},\Theta,\mathfrak{Y}),$$
is an $(\mathcal{A},B)$-bordism. We call $\mathfrak{X}|_{[0,1]}$ the truncation of $\mathfrak{X}$ to a collar neighborhood.
\end{prop}

\begin{proof}
To prove the proposition it suffices to prove that $T|_{[0,1]}$ is a regular operator and that $\phi\Dom (T|_{[0,1]})^*\subseteq \Dom T|_{[0,1]}$ for $\phi\in C^\infty_c[0,1)$. Indeed, $T|_{[0,1]}$ is clearly symmetric and all the compatibility conditions of Definition \ref{HilBorDef}. A standard argument with mollifiers shows that the space 
$$\mathfrak{D}:=\{\xi\in \mathpzc{N}|_{[0,1]}: \exists \eta\in \Dom(T), b(\chi_{[0,1]})\eta=\xi\},$$
is a core for $(T|_{[0,1]})^*$ and that for such a $\xi$, $(T|_{[0,1]})^*\xi=b(\chi_{[0,1]})T\eta$. By using the local-global principle and the fact that $\mathfrak{D}$ localizes to a core in all representation of $B$, we deduce that $T|_{[0,1]}$ is regular. Moreover, it is clear that $\phi\mathfrak{D}\subseteq \Dom T|_{[0,1]}$ for $\phi\in C^\infty_c[0,1)$, so $\phi\Dom (T|_{[0,1]})^*\subseteq \Dom T|_{[0,1]}$ for $\phi\in C^\infty_c[0,1)$.
\end{proof}

To compare the notion of bordism from Definition \ref{HilBorDef} to that from Definition \ref{HilBorDefInf}, we need to introduce a weaker notion of isomorphism of $KK$-bordisms than that appearing in Definition \ref{isoofsymwbound}. 

\begin{define}
\label{isoofsymwboundinf}
A weak isomorphism $\pmb{\alpha}:\mathfrak{X}_1\xrightarrow{\sim}^w \mathfrak{X}_2$ of two symmetric $(\mathcal{A},B)$-chains with boundary $\mathfrak{X}_i=(\mathfrak{X}^\circ_i,\Theta_i,\mathfrak{X}^\partial_i)$, $i=1,2$,  is a collection $\pmb{\alpha}=(\alpha^\circ,\alpha^\partial)$ where $\alpha^\partial:\mathfrak{X}^\partial_1\to \mathfrak{X}^\partial_2$ is an isomorphism (see Definition \ref{UnbKKcycDef}) and $\alpha^\circ:\mathpzc{N}_1\to \mathpzc{N}_2$ is a unitary isomorphism such that 
\begin{itemize}
\item It holds that $\alpha^\circ p_1=p_2\alpha^\circ$ and the following diagram commutes:
$$\begin{CD}
p_1\mathpzc{N}_1 @>\alpha^\circ|>\cong > p_2\mathpzc{N} \\
@V\theta_1V\cong V @V\cong V\theta_2 V \\
L^2[0,1]\hat{\boxtimes}\mathpzc{E}_1 @>\id_{L^2[0,1]}\boxtimes\alpha^\partial>\cong > L^2[0,1]\hat{\boxtimes}\mathpzc{E}_2.
\end{CD}$$
\item The map $\alpha^\circ$ restricts to a bijection
$$C^\infty_c[0,1)\Dom(T_1)\to C^\infty_c[0,1)\Dom(T_2),$$
and on this space 
$$T_2\alpha^\circ=\alpha^\circ T_1.$$
\end{itemize}
Here $\mathfrak{X}^\circ_i=(\mathpzc{N}_i,T_i)$, $\Theta_i=(\theta_i,p_i)$ and $\mathfrak{X}^\partial_i=(\mathpzc{E}_i,D_i)$.
\end{define}

The reader should interpret the notion of weak isomorphism of $KK$-bordisms identifying $KK$-bordisms that are isomorphic ``up to a choice of boundary condition''. For all intents and purposes, weakly isomorphic $KK$-bordisms are indistinguishable because the ``boundary conditions'' are auxiliary. It is clear from the construction that weak isomorphisms define an equivalence relation. The following result follows from the constructions in  Proposition \ref{trunactiopiad}.

\begin{theorem}
\label{infcylvsord}
The process of attaching an infinite cylinder 
$$\mathfrak{X}=(\mathfrak{X}^\circ,\Theta,\mathfrak{X}^\partial)\mapsto (\mathfrak{X}^\circ_{(0,\infty)},\Theta_{(0,\infty)},\mathfrak{X}^\partial),$$
for the boundary data $\Theta_\infty=(\id, \chi_{[0,\infty)}\hat{\boxtimes}1)$, cf. Proposition \ref{gluinginfcyleprop}, defines a bijection from the set of weak isomorphism classes of $(\mathcal{A},B)$-bordisms to the set of isomorphism classes of $(\mathcal{A},B)$-bordisms with infinite cylindrical ends. The inverse to this bijection maps a bordism with cylindrical ends to its truncation to a collar neighborhood as in Proposition \ref{trunactiopiad}.
\end{theorem}

\begin{remark}
Let us make a slightly technical remark that comes in handy in several constructions below. Proposition \ref{trunactiopiad} shows that if $\mathfrak{X}=(\mathfrak{X}^\circ,\Theta,\mathfrak{X}^\partial)$ is an $(\mathcal{A},B)$-bordism with 
$$\phi \Dom \Psi(D)=\theta b(\phi)\Dom T \quad \mbox{and} \quad T=\theta^{-1}\Psi(D)\theta\quad \mbox{on}\quad b(\phi)\Dom T,$$
for all $\phi\in C^\infty_c(0,1]$, then truncating back from having attached infinite cylindrical ends on $\mathfrak{X}$ reproduces $\mathfrak{X}$. In particular, for constructions of $KK$-bordisms it often suffices to produce a $KK$-bordism with infinite cylindrical ends.
\end{remark}

\subsection{Order reduction of $KK$-bordisms}
\label{subsec:orderred}

In Subsection \ref{orderredusubsec} (see page \pageref{orderredusubsec}) we studied order reduction of cycles, and saw that it did not change the bordism class. Building on Theorem \ref{interpolationthm} (see page \pageref{interpolationthm}) we shall extend the procedure of order reduction to bordisms in order to conclude that a bordism of cycles induces a bordism of the order reduced cycles. Before stating the main result of this subsection, we shall need some preliminary constructions. 

Take an $\alpha\in (0,1]$. Consider an $(\mathcal{A},B)$-bordism $\mathfrak{X}=((\mathpzc{N},T), (\theta,p),(\mathpzc{E},D))$. Pick a gluing function $\chi\in C^\infty_c(0,1]$ (i.e. $\chi'\in C^\infty_c(0,1)$ is positive and $\chi(1)=1$). We shall from gluing and order reduction (as in Subsection \ref{orderreducsubsec}) define a collection
$$\mathfrak{X}_\alpha:=((\mathpzc{N},T_\alpha), (\id, \chi_{[0,1]}),(\mathpzc{E},D_\alpha)),$$
that we will say is the order reduction of $\mathfrak{X}$, and in the next theorem we prove that this is indeed a bordism that extends to a complex interpolation space of $\mathcal{A}$. The construction is an instance of gluing, but of two quite different objects and will as such neither be canonical nor immediately well defined. We shall suppress the boundary data in our construction, as it is clear from context and notationally will just blur the construction. 

Recall the construction of gluing on an infinite cylinder $\mathfrak{X}_{(0,\infty)}=(\mathpzc{N}_{(0,\infty)},T_{(0,\infty)})$ to a bordism from Proposition \ref{gluinginfcyleprop} and order reduction of chains with self-adjoint operators from Subsection \ref{orderreducsubsec}.
We will start by defining the operator $T_\alpha$ which is densely defined on $\mathpzc{N}$, and obtained as 
the closure
\begin{equation}
\label{talpadada}
T_\alpha:=\overline{\sqrt{1-\chi}T_{(0,\infty),\alpha}\sqrt{1-\chi}+\sqrt{\chi}\Psi(D_\alpha)\sqrt{\chi}}.
\end{equation}
acting on the domain $(1-\chi)\Dom(T_{(0,\infty),\alpha})+\chi \Dom(\Psi(D_\alpha))$. We also define $T_{\alpha,\infty}$ as the closure 
\begin{equation}
\label{talpadinada}
T_{\alpha,\infty}:=\overline{\sqrt{1-\chi}T_{(0,\infty),\alpha}\sqrt{1-\chi}+\sqrt{\chi}\Psi_{\infty}(D_\alpha)\sqrt{\chi}},
\end{equation}
as an operator on $\mathpzc{N}_{(0,\infty)}$ by extending $\chi$ to a constant function on $[1,\infty)$. At this stage, we know little of both operators $T_\alpha$ (from Equation \eqref{talpadada}) and $T_{\alpha,\infty}$ (from Equation \eqref{talpadinada}) but after the fact we shall see that $T_\alpha$ fits into a bordism and that $T_{\alpha,\infty}=(T_\alpha)_{(0,\infty)}$. 

\begin{define}
Let $\mathcal{A}$ be a $*$-algebra, $B$ a $C^*$-algebra, and $(\mathpzc{N}_B,T)$ and $(\mathpzc{N}_B,T')$ two half-closed $(\mathcal{A},B)$-chains. Assume that $\mathcal{A}_0\subseteq \mathcal{M}(A)$ is a $*$-subalgebra such that $\mathcal{A}_0\mathcal{A}=\mathcal{A}\mathcal{A}_0=\mathcal{A}$. 

We say that $(\mathpzc{N}_B,T)$ and $(\mathpzc{N}_B,T')$ are locally the same with respect to $\mathcal{A}_0$ if 
\begin{itemize}
\item For all $a\in \mathcal{A}_0$, 
$$a\Dom(T)=a\Dom(T'), \quad \mbox{with}\quad aT=aT' \quad\mbox{and}\quad Ta=T'a\quad \mbox{on $\Dom(T)$.}$$
\item For all $a\in \mathcal{A}_0$, 
\begin{align*}
a\Dom(T^*T)&=a\Dom(T'^*T'), \quad \mbox{with}\\ 
aT^*T&=aT'^*T' \quad\mbox{and}\quad T^*Ta=T'^*T'a\quad \mbox{on $\Dom(T^*T)$.}
\end{align*}
\item For all $a\in \mathcal{A}_0$, 
$$[T,a]\Dom(T)\subseteq \Dom(T), \quad\mbox{and}\quad[T',a]\Dom(T')\subseteq \Dom(T').$$
\end{itemize}
\end{define}

\begin{lemma}
\label{locallythesame}
Let $\mathcal{A}$ be a $*$-algebra dense in the $C^*$-algebra $A$, $B$ a $C^*$-algebra, and $(\mathpzc{N}_B,T)$ and $(\mathpzc{N}_B,T')$ two half-closed $(\mathcal{A},B)$-chains that are locally the same for a $*$-subalgebra $\mathcal{A}_0\subseteq \mathcal{M}(A)$ such that $\mathcal{A}_0\mathcal{A}=\mathcal{A}\mathcal{A}_0=\mathcal{A}$. Then any $a\in \mathcal{A}$ and $\alpha\in (0,1]$, the operator
$$aT(1 + T^*T)^{\frac{\alpha-1}{2}}-aT'(1+T'^*T')^{\frac{\alpha-1}{2}},$$
is bounded and the operator 
$$aT(1 + T^*T)^{-\frac{1}{2}}-aT'(1+T'^*T')^{-\frac{1}{2}},$$
is compact.
\end{lemma}

\begin{proof}
The case $\alpha=1$ is trivial. We first prove the boundedness claim for $\alpha\in (0,1)$. It is clearly no restriction to assume that $\mathcal{A}=\mathcal{A}_0$ and that 
$$aT(1 + T^*T)^{\frac{\alpha-1}{2}}-aT'(1+T'^*T')^{\frac{\alpha-1}{2}},$$
is bounded for $a\in \mathcal{A}_0$. Consider the operator valued holomorphic function in $\mathrm{Re}(\lambda)>-1$
$$b(\lambda):=a(1+\lambda+T^*T)^{-1}-a(1+\lambda+T'^*T')^{-1}.$$
Since $(\mathpzc{N}_B,T)$ and $(\mathpzc{N}_B,T')$ are locally the same for  $\mathcal{A}_0$, we have the mapping property $b(\lambda):\mathpzc{E}\to \Dom(T^*T)\cap \Dom(T'^*T')$. Using that $(\mathpzc{N}_B,T)$ and $(\mathpzc{N}_B,T')$ are locally the same we compute that 
\begin{align*}
(1+&\lambda+T^*T)b(\lambda)=\\
=&(1+\lambda+T^*T)a(1+\lambda+T^*T)^{-1}-(1+\lambda+T'^*T')a(1+\lambda+T'^*T')^{-1}=\\
=&\left([T^*,a]T+T^*[T,a]\right)(1+\lambda+T^*T)^{-1}-\left([T'^*,a]T'+T'^*[T',a]\right)(1+\lambda+T'^*T')^{-1}.
\end{align*}
As a consequence, the operator $b(\lambda):\mathpzc{E}\to \Dom(T)\cap \Dom(T')$ is bounded with norm bound 
$$\|b(\lambda)\|_{\mathpzc{E}\to \Dom(T)\cap \Dom(T')}=O(\lambda^{-1}),\quad\mbox{ as $\mathrm{Re}(\lambda)\to \infty$ in sectors around $[0,\infty)$.}$$
Therefore the operator 
$$b_\alpha:=a(1 + T^*T)^{\frac{\alpha-1}{2}}-a(1+T'^*T')^{\frac{\alpha-1}{2}}=-\frac{1}{\pi}\int_0^\infty \lambda^{\frac{\alpha-1}{2}} b(\lambda)\mathrm{d} \lambda,$$
is bounded $b_\alpha:\mathpzc{E}\to \Dom(T)\cap \Dom(T')$. We conclude that
$$aT(1 + T^*T)^{\frac{\alpha-1}{2}}-aT'(1+T'^*T')^{\frac{\alpha-1}{2}}=Tb_\alpha+[a,T](1 + T^*T)^{\frac{\alpha-1}{2}}-[a,T'](1+T'^*T')^{\frac{\alpha-1}{2}},$$
is bounded.

For the case $\alpha=0$, we can not reduce to $\mathcal{A}_0=\mathcal{A}$. But the same argument as above implies that for $a=ba_0$ for $a_0\in \mathcal{A}_0$, the operator 
$$b_0:=a(1 + T^*T)^{-\frac{1}{2}}-a(1+T'^*T')^{-\frac{1}{2}}=-\frac{1}{\pi}\int_0^\infty \lambda^{-\frac{1}{2}} b(\lambda)\mathrm{d} \lambda,$$
is compact $b_0:\mathpzc{E}\to \Dom(T)\cap \Dom(T')$. The argument then proceeds as in the $\alpha>0$ case to show that $aT(1 + T^*T)^{-\frac{1}{2}}-aT'(1+T'^*T')^{-\frac{1}{2}}$ is compact.
\end{proof}

\begin{lemma}
\label{comparingpowers}
In the setting above at the beginning of the subsection (see page \pageref{subsec:orderred}), it holds for any $\phi_1,\phi_2\in C^\infty_b(0,\infty)$ of which at least one vanish near $0$, that 
$$\phi_1T_{(0,\infty),\alpha}\phi_2-\phi_1\Psi_\infty(D)_{\alpha}\phi_2,$$
is a bounded operator on $\mathpzc{N}_{(0,\infty)}$.
\end{lemma}

\begin{proof}
We note that for any $\phi_1,\phi_2\in C^\infty_b(0,\infty)$ of which at least one vanish near $0$, we have that 
$$
\phi_1T_{(0,\infty)}\phi_2=\phi_1\Psi_\infty(D)\phi_2.
$$
In particular, $T_{(0,\infty)}$ and $\Psi_\infty(D)$ are locally the same with respect to $\mathcal{A}_0=C^\infty_c(0,1]$. The result now follows from Lemma \ref{locallythesame}. 
\end{proof}

\begin{lemma}
\label{comparingdifforder}
Let $D$ be a self-adjoint and regular operator on a $B$-Hilbert $C^*$-module $\mathpzc{E}$ and $\chi\in C^\infty_b(\R)$ a positive increasing function with $\chi'\in C^\infty_c(\R)$. Then the operator 
$$\overline{\sqrt{1-\chi}\Psi_\infty(D)_{\alpha}\sqrt{1-\chi}+\sqrt{\chi}\Psi_{\infty}(D_\alpha)\sqrt{\chi}},$$
is self-adjoint and regular on $L^2(\R)\hat{\boxtimes}\mathpzc{E}$.
\end{lemma}

\begin{proof}
We write 
\begin{align*}
\sqrt{1-\chi}\Psi_\infty(D)_{\alpha}&\sqrt{1-\chi}+\sqrt{\chi}\Psi_{\infty}(D_\alpha)\sqrt{\chi}=\\
=&\sqrt{1-\chi}\Psi_{\infty}(D_\alpha)\sqrt{1-\chi}+\sqrt{\chi}\Psi_{\infty}(D_\alpha)\sqrt{\chi}+\\
&+\sqrt{1-\chi}(\Psi_\infty(D)_{\alpha}-\Psi_{\infty}(D_\alpha))\sqrt{1-\chi}=\\
=&\Psi_{\infty}(D_\alpha)+V_1+V_2,
\end{align*}
where $V_1=[\sqrt{1-\chi},\Psi_\infty(D_{\alpha})]\sqrt{1-\chi}+[\sqrt{\chi},\Psi_\infty(D_{\alpha})]\sqrt{\chi}$ and 
$$V_2=\sqrt{1-\chi}(\Psi_\infty(D)_{\alpha}-\Psi_{\infty}(D_\alpha))\sqrt{1-\chi}.$$
Clearly $V_1$ is bounded. The inequality $\mu^2+\lambda^{2\alpha}\geq (\mu^2+\lambda^2)^\alpha$ for large enough real $(\mu,\lambda)$ shows that $V_2$ is relatively bounded by $\Psi_{\infty}(D_\alpha)$ with relative norm bound $1$. The operator $\Psi_{\infty}(D_\alpha)$ is self-adjoint and regular, and $V_1$ and $V_2$ are symmetric so the lemma now follows from Wüst's extension theorem on Hilbert $C^*$-modules (see \cite[Theorem 4.6]{leschkaad2}).
\end{proof}

\begin{lemma}
\label{orderreducingbordismslemmainfty}
The operator $T_{\alpha,\infty}$ constructed in Equation \eqref{talpadinada} is a well defined self-adjoint and regular operator on $\mathpzc{N}_{(0,\infty)}$. Moreover, it holds that 
$$T_\alpha=(T_{\alpha,\infty})|_{[0,1]},$$
for notation see Proposition \ref{trunactiopiad} and Theorem \ref{infcylvsord}, and that $T_\alpha$ is a symmetric and regular operator.
\end{lemma}

\begin{proof}
The final conclusion of the lemma follows from Proposition \ref{infcylvsord}, so it suffices to prove that $T_{\alpha,\infty}$ is self-adjoint and regular. To shorten the notation, we write $\#_\chi$ for the closure of gluing operators together along $\chi$, so that $T_{\alpha,\infty}:=T_{(0,\infty),\alpha}\#_\chi\Psi_{\infty}(D_\alpha)$. By Lemma \ref{comparingpowers}, we can write 
$$T_{\alpha,\infty}=T_{(0,\infty),\alpha}\#_{\tilde{\chi}}(\Psi_\infty(D)_{\alpha}\#_\chi\Psi_{\infty}(D_\alpha))+V,$$
for a suitable $\tilde{\chi}$ and a bounded self-adjoint operator $V$. By Lemma \ref{comparingdifforder}, the operator $\Psi_\infty(D)_{\alpha}\#_\chi\Psi_{\infty}(D_\alpha)$ is self-adjoint and regular. Therefore a cumbersome construction of approximate resolvents, as in the construction of gluings of $KK$-bordisms in \cite[Theorem 2.20]{DGM}, shows that the gluing over $\tilde{\chi}$ produces a self-adjoint and regular operator proving that $T_{\alpha,\infty}$ is self-adjoint and regular.
\end{proof}

We remark that not only does the operator $T_\alpha$ depend on the choice of $\chi$, even its domain depends on $\chi$. This can readily be seen from a simple example on the unit disc\footnote{To obtain product structure near the boundary, we can use a torpedo metric.}. On the unit disc, elements of $\Dom(T_{(0,\infty),\alpha})$ will have interior Sobolev regularity $\alpha$ and therefore the regularity in the radial direction of an element in $\Dom(T_\alpha)$ goes from radial Sobolev regularity $\alpha$ on the subset where $\chi=0$ to radial Sobolev regularity $1$ on the subset where $\chi=1$.

\begin{theorem}
\label{orderreducingbordismsthm}
Assume that $\mathcal{A}$ is a Banach $*$-algebra and that the inclusion $\mathcal{A}\hookrightarrow A$ is continuous. Let $\alpha\in (0,1]$, $\beta\in (\alpha,1]$ and let $\mathfrak{X}$ be an $(\mathcal{A},B)$-bordism. The order reduction $\mathfrak{X}_\alpha$ of $\mathfrak{X}$ is defined as the collection 
$$\mathfrak{X}_\alpha:=((\mathpzc{N},T_\alpha), (\id, \chi_{[1-\epsilon,1]}),(\mathpzc{E},D_\alpha)),$$
where $\epsilon$ is small enough so that the $\chi$ from Equation \eqref{talpadada} satisfies $\chi=1$ on a neighborhood of $[1-\epsilon,1]$. The order reduction $\mathfrak{X}_\alpha$ is a well defined $(\mathcal{A}_\beta,B)$-bordism with 
$$\partial (\mathfrak{X}_\alpha)=(\partial \mathfrak{X})_\alpha.$$
\end{theorem}

\begin{proof}
It follows from Theorem \ref{interpolationthm} (see page \pageref{interpolationthm}), Proposition \ref{trunactiopiad} and Lemma \ref{orderreducingbordismslemmainfty} that $\mathfrak{X}_\alpha$ is a symmetric chain with boundary. To prove the theorem it therefore remains to prove that $\mathfrak{X}_\alpha$ is half-closed as a $(C^\infty_c((0,1],\mathcal{A}_\alpha),B)$-chain. In other words, we need to prove that for $a\in C^\infty_c((0,1],\mathcal{A}_\alpha)$, the operator 
\begin{equation}
\label{cptdomaorder}
\Dom(T_\alpha)=(1-\chi)\Dom(T_{\infty,\alpha})+\chi \Dom(\Psi(D_\alpha))\xrightarrow{a}\mathpzc{N},
\end{equation}
is compact. Compactness of the operator in Equation \eqref{cptdomaorder} now follows from that $\Dom(T_{(0,\infty),\alpha})\xrightarrow{(1-\chi)a}\mathpzc{N}$ is compact by Theorem \ref{interpolationthm} and that $\Dom(\Psi(D_\alpha))\xrightarrow{\chi a}\mathpzc{N}$ is compact by Proposition \ref{suspandprop} and Theorem \ref{interpolationthm}.
\end{proof}

The following corollary is an immediate consequence of Theorem \ref{orderreducingbordismsthm}.

\begin{cor}
\label{injofrestcomplx}
Assume that $\mathcal{A}$ is a Banach $*$-algebra and that the inclusion $\mathcal{A}\hookrightarrow A$ is continuous. Let $\beta\in (0,1]$ and $\alpha\in (0,\beta)$. For any two $(\mathcal{A},B)$-cycles $(\mathpzc{E},D)$ and $(\mathpzc{E}',D')$, it holds that $(\mathpzc{E},D)\sim_{\rm bor}(\mathpzc{E}',D')$ as $(\mathcal{A},B)$-cycles if and only if $(\mathpzc{E},D_\alpha)\sim_{\rm bor}(\mathpzc{E}',D'_\alpha)$ as $(\mathcal{A}_\beta,B)$-cycles.

\end{cor}

\section{Properties of the $KK$-bordism group}
\label{sec:propoered}

In this section we will utilize the $KK$-bordisms considered in Sections \ref{exampleofboridmsmssec} and \ref{sec:prelbord} to deduce structural results about the $KK$-bordism groups. We conclude invariance under functional calculus, Morita invariance under finitely generated projective modules as well as under complex interpolation. We continue by discussing Hilsum's work \cite{hilsumbordism} on the bounded transform and a variety of technical simplifications.

\subsection{Invariance properties}
\label{invariancesubsec}

We start by summarizing some immediate functoriality properties of the $KK$-bordism group in a theorem.

\begin{theorem}

\begin{enumerate}
\item Let $\varphi:\mathcal{A}_1\to \mathcal{A}_2$ be a $*$-homomorphism of $*$-algebras which is continuous in the $C^*$-topologies. For any $C^*$-algebra $B$, the mapping 
$$\varphi^*:\Omega_*(\mathcal{A}_2,B)\to  \Omega_*(\mathcal{A}_1,B), \quad (\mathpzc{E},D)\mapsto (\varphi\mathpzc{E},D),$$
is a well defined graded group homomorphism. If $\varphi$ is an inclusion, the map $\varphi^*$ is surjective if any $(\mathcal{A}_1,B)$-cycles is bordant to a cycle extending to $(\mathcal{A}_2,B)$ and the map $\varphi^*$ is injective if any nullbordant $(\mathcal{A}_1,B)$-cycle extending to $(\mathcal{A}_2,B)$ is nullbordant as an $(\mathcal{A}_2,B)$-cycle.
\item Let $\varphi:B_1\to B_2$ be a $*$-homomorphism of $C^*$-algebras. For any $*$-algebra $\mathcal{A}$, the mapping 
$$\varphi_*:\Omega_*(\mathcal{A},B_1)\to  \Omega_*(\mathcal{A},B_2), \quad (\mathpzc{E},D)\mapsto (\mathpzc{E}\otimes_\varphi B_2,D\otimes 1_{B_2}),$$
is a well defined graded group homomorphism. 
\end{enumerate}

\end{theorem}

\begin{theorem}
\label{cpxinerpbofodthm}
Let $B$ be a $C^*$-algebra. Assume that $\mathcal{A}$ has a Banach $*$-algebra norm making the inclusion $\mathcal{A}\subseteq A$ continuous. Let $\mathcal{A}_\alpha:=[A,\mathcal{A}]_\alpha$ denote the complex interpolation space in $\alpha\in (0,1]$. It holds that the restriction 
$$\Omega_*(\mathcal{A}_\alpha,B)\to \Omega_*(\mathcal{A},B),$$
along the inclusion $\mathcal{A}\hookrightarrow \mathcal{A}_\alpha$, is an isomorphism for all $\alpha\in (0,1]$. In other words, the $KK$-bordism group is invariant under complex interpolation along the inclusion $\mathcal{A}\subseteq A$. 

More generally, for any $*$-subalgebra $\mathcal{A}'\subseteq A$ with $\mathcal{A}\subseteq \mathcal{A}'\subseteq \mathcal{A}_\alpha$, for some $\alpha\in (0,1]$, we have that the restriction along $\mathcal{A}\hookrightarrow\mathcal{A}'$ induces an isomorphism 
$$\Omega_*(\mathcal{A}',B)\to \Omega_*(\mathcal{A},B),$$
\end{theorem}

\begin{proof}
We will prove that $\Omega_*(\mathcal{A}',B)\to \Omega_*(\mathcal{A},B)$ is an isomorphism. Surjectivity of $\Omega_*(\mathcal{A}_\alpha,B)\to \Omega_*(\mathcal{A},B)$ follows from Corollary \ref{surofrestcomplx} which in turn implies surjectivity of $\Omega_*(\mathcal{A}',B)\to \Omega_*(\mathcal{A},B)$. To show that the map is injective, assume that $(\mathpzc{E},D)$ is a cycle for $(\mathcal{A}',B)$ which is nullbordant as an $(\mathcal{A},B)$-cycle. By Theorem \ref{interpolationthm}, for $\beta\in (0,\alpha)$ we have that $(\mathpzc{E},D_\beta)$ is an $(\mathcal{A}_\alpha,B)$-cycle which is nullbordant as an $(\mathcal{A}_\alpha,B)$-cycle by Corollary \ref{injofrestcomplx}. By Theorem \ref{bordismforodrdthm} we conclude that 
$$(\mathpzc{E},D)\sim_{\rm bor}(\mathpzc{E},D_\beta)\sim_{\rm bor} 0,$$ 
as $(\mathcal{A}_\alpha,B)$-cycles. In particular, $(\mathpzc{E},D)\sim_{\rm bor} 0$ as $(\mathcal{A}',B)$-cycles.
\end{proof}

\begin{theorem}
\label{caalpainandthm}
Let $B$ be a $C^*$-algebra and $\mathcal{A}$ is a $*$-algebra with $C^*$-closure $A$. For $\alpha>1$, let $\overline{\mathcal{A}}^{(\alpha)}$ denote its $C^\alpha$-functional calculus closure. It holds that the restriction 
$$\Omega_*(\overline{\mathcal{A}}^{(\alpha)},B)\to \Omega_*(\mathcal{A},B),$$
along the inclusion $\mathcal{A}\hookrightarrow \overline{\mathcal{A}}^{(\alpha)}$, is an isomorphism for all $\alpha>1$. In other words, the $KK$-bordism group is invariant under $C^\alpha$-functional calculus. 

More generally, for any $*$-subalgebra $\mathcal{A}'\subseteq A$ with $\mathcal{A}\subseteq \mathcal{A}'\subseteq \overline{\mathcal{A}}^{(\alpha)}$, for some $\alpha>1$, we have that the restriction along $\mathcal{A}\hookrightarrow\mathcal{A}'$ induces an isomorphism 
$$\Omega_*(\mathcal{A}',B)\to \Omega_*(\mathcal{A},B),$$
\end{theorem}

\begin{proof}
By Theorem \ref{funccalccompthm} (see page \pageref{funccalccompthm}), the class of $(\mathcal{A},B)$-cycles/bordisms is the same as the class of $(\overline{\mathcal{A}}^{(\alpha)},B)$-cycles/bordisms and the theorem follows. 
\end{proof}

By the same argument, the following theorem holds.

\begin{theorem}
Let $B$ be a $C^*$-algebra and $\mathcal{A}$ is a $*$-algebra with $C^*$-closure $A$. Let $\overline{\mathcal{A}}^{\mathcal{O}}$ denote its holomorphic functional calculus closure. It holds that the restriction 
$$\Omega_*(\overline{\mathcal{A}}^{\mathcal{O}},B)\to \Omega_*(\mathcal{A},B),$$
along the inclusion $\mathcal{A}\hookrightarrow \overline{\mathcal{A}}^{\mathcal{O}}$, is an isomorphism. In other words, the $KK$-bordism group is invariant under holomorphic functional calculus. 
\end{theorem}

\begin{theorem}
Let $\mathcal{A}$ be a $*$-algebra. 
\begin{itemize}
\item A Morita equivalence $B_1\sim_M B_2$ of $C^*$-algebras induces an isomorphism 
$$\Omega_*(\mathcal{A},B_1)\cong \Omega_*(\mathcal{A},B_2).$$
\item If $\mathcal{A}$ is unital and $\mathcal{E}$ is a full finitely generated projective Hermitian $\mathcal{A}$-module, the Morita equivalence $\mathcal{A}\sim_M \End_\mathcal{A}(\mathcal{E})$ induces an isomorphism 
$$\Omega_*(\mathcal{A},B)\cong \Omega_*(\End_\mathcal{A}(\mathcal{E}),B),$$
for any $C^*$-algebra $B$.
\end{itemize}
\end{theorem}

\begin{proof}
If ${}_{B_1}\mathpzc{M}_{B_2}$ is a Morita equivalence $B_1\to B_2$, the mapping 
$$\Omega_*(\mathcal{A},B_1)\cong \Omega_*(\mathcal{A},B_2), \quad (\mathpzc{E},D)\mapsto (\mathpzc{E}\otimes_{B_1}\mathpzc{M}_{B_2},D\otimes 1_{\mathcal{M}}),$$
is a well defined group homomorphism. It is readily verified that its inverse is given by 
$$\Omega_*(\mathcal{A},B_2)\cong \Omega_*(\mathcal{A},B_1), \quad (\mathpzc{E},D)\mapsto (\mathpzc{E}\otimes_{B_2}\mathpzc{M}_{B_1}^*,D\otimes 1_{\mathcal{M}^*}),$$

To prove (fgp) Morita invariance in the first leg, we write $\mathcal{E}=p\mathcal{A}^N$ for some Hermitian projection $p\in M_N(\mathcal{A})$. Consider the mapping 
$$\Omega_*(\mathcal{A},B)\cong \Omega_*(\End_\mathcal{A}(\mathcal{E}),B), \quad (\mathpzc{E},D)\mapsto (\mathcal{E}\otimes_\mathcal{A}\mathpzc{E},p(1_{\C^N}\otimes D)p).$$
By \cite[Subsection 1.4]{DGM}, this is a well defined mapping. The inverse to this mapping is defined analogously but using that $\mathcal{E}^*$ defines a Morita equivalence $ \End_\mathcal{A}(\mathcal{E})\sim_M\mathcal{A}$. We remark that at the level of cycles, the composition of these two morphisms produces a bounded perturbation which does not affect bordism classes by Corollary \ref{locbddperturcor} (see page \pageref{locbddperturcor}).
\end{proof}

\subsection{The bounded transform}
\label{subsecbddtrans}

An important tool in the study of the $KK$-bordism groups is the bounded transform, which defines a mapping to the ordinary $KK$-groups. Indeed, the $KK$-bordism groups find much of its relevance through its relation to the $KK$-groups implemented by the bounded transform. We shall later in this monograph see that for large classes of $*$-algebras $\mathcal{A}$, the bounded transform is an isomorphism, but in some cases it is neither injective nor surjective. A geometric analogue for the reader to keep in their mind is the mapping $\Omega^{{\rm Spin}^c}_*(X)\to K_*(X)$ mapping a spin$^c$-bordism class to its fundamental class; this analogue is only partial as the $KK$-bordisms encode all of the Baum-Douglas relations as discussed in Remark \ref{discusdinbd} below.

\begin{define}
Let $\mathcal{A}\subseteq A$ be a dense $*$-subalgebra of the $C^*$-algebra $A$ and let $B$ be a $C^*$-algebra. For an unbounded $(\mathcal{A},B)$-cycle $(\mathpzc{E},D)$, define the bounded transform by 
$$\beta(\mathpzc{E},D):=(\mathpzc{E},D(1+D^2)^{-1/2}).$$
\end{define}

\begin{remark}
It is well known that for an unbounded $(\mathcal{A},B)$-cycle $(\mathpzc{E},D)$, the bounded transform $\beta(\mathpzc{E},D)$ is an $(A,B)$-Kasparov cycle. This was proven by Baaj-Julg \cite{baajjulg}. In fact, Hilsum  \cite{hilsumbordism} proved that for a half-closed $(\mathcal{A},B)$-chain $(\mathpzc{E},D)$, the bounded transform $\beta(\mathpzc{E},D):=(\mathpzc{E},D(1+D^*D)^{-1/2})$ is a well defined $(A,B)$-Kasparov cycle.
\end{remark}

\begin{theorem}[Hilsum's theorem on bordism invariance, \cite{hilsumbordism}]
\label{alknlkanda}
Let $\mathcal{A}\subseteq A$ be a dense $*$-subalgebra of the $C^*$-algebra $A$ and let $B$ be a $C^*$-algebra. If an unbounded $(\mathcal{A},B)$-cycle $(\mathpzc{E},D)$ is null bordant, then 
$$[\beta(\mathpzc{E},D)]=0 \quad\mbox{in $KK_*(A,B)$}.$$
In particular, the bounded transform induces a well defined homomorphism
$$\beta:\Omega_*(\mathcal{A},B)\to KK_*(A,B).$$
\end{theorem}

For the purposes of utilizing $KK$-bordism in constructing relative $KK$-bordism groups, let us briefly sketch the main ideas in the proof of this theorem. The proof is based on Hilsum's proof that $\beta(\mathpzc{N},T):=(\mathpzc{N},T(1+T^*T)^{-1/2})$ is a well defined Kasparov cycle for any half-closed chain $(\mathpzc{N},T)$. In particular, if $(\mathpzc{N},T)$ is the interior chain in a bordism with boundary $(\mathpzc{E},D)$, then $\beta(\mathpzc{N},T)$ is a $(C_0((0,1]A),B)$-Kasparov cycle which with some work one can show restricts to a $(C_0((0,1)A),B)$-Kasparov cycle homotopic to $\beta\circ \Psi(\mathpzc{E},D)$. The class of the Kasparov cycle $\beta\circ \Psi(\mathpzc{E},D)$ is in turn given by the image of $\beta(\mathpzc{E},D)$ under the suspension mapping $KK_*(A,B)\to KK_{*+1}(C_0((0,1)A),B)$. As such, $\beta(\mathpzc{E},D)$ is in the image of $0=KK_{*+1}(C_0((0,1]A),B)\to KK_*(A,B)$ and must be $0$. Unraveling the homotopies in this proof is fundamental in the context of relative $KK$-groups and more details will be provided in \cite{monographtwo}.

\begin{define}
Let $\mathcal{A}$ be a dense $*$-subalgebra of a $C^*$-algebra $A$. We say that $\mathcal{A}$ is $\Omega_*$-admissible if the bounded transform 
$$\beta:\Omega_*(\mathcal{A},B)\to KK_*(A,B),$$
is an isomorphism for all $C^*$-algebras $B$.
\end{define}

\begin{remark}
In a previous paper by the authors \cite{DGM}, some further properties of the bounded transform were proven:
\begin{itemize}
\item The $*$-algebra $\mathcal{A}=\C$ is $\Omega_*$-admissible, i.e. $\beta:\Omega_*(\C,B)\to KK_*(\C,B)$ is an isomorphism for all $C^*$-algebras $B$.
\item If $M$ is a closed manifold and $C^\infty(M)\subseteq \mathcal{A}\subseteq \mathrm{Lip}(M)$, then $\beta:\Omega_*(\mathcal{A},B)\to KK_*(C(M),B)$ is a split surjection.
\end{itemize}
We also note that the lifting result of Baaj-Julg \cite{baajjulg} implies that if $\mathcal{A}$ is a countably generated $*$-algebra, the bounded transform $\beta:\Omega_*(\mathcal{A},B)\to KK_*(A,B)$ is a surjection. From this surjectivity statement, and the invariance results of Subsection \ref{invariancesubsec} we see that there is an abundance of examples where surjectivity of the bounded transform holds. In Part \ref{partoniso} below the reader can find ample amounts of examples of $\Omega_*$-admissible $*$-algebras. We shall below in Section \ref{bddtransformisosec} see that the lifting method of Baaj-Julg \cite{baajjulg} can be use to show that the bounded transform is an isomorphism if $\mathcal{A}$ is a countably generated $*$-algebra. The next proposition will be used to prove that.
\end{remark}

\begin{prop}
Let $(\mathcal{E},D)$ be an $(\mathcal{A},B)$-cycle, and set 
$$\tilde{\mathpzc{E}}:=\mathpzc{E}\oplus \mathpzc{E}\quad\mbox{and}\quad \tilde{D}:=
\begin{pmatrix}
D&(1+D^2)^{-1/2}\\
(1+D^2)^{-1/2}& -D
\end{pmatrix},$$
with the left $A$-action defined by $a.(\xi_1\oplus \xi_2):=a\xi_1\oplus 0$, as in Corollary \ref{invertibleldodcor} (see page \pageref{invertibleldodcor}) and $\hat{D}:=\tilde{D}|\tilde{D}|^{-1/2}$. Then the following holds 
\begin{enumerate}
\item The cycle $(\tilde{\mathpzc{E}},\hat{D})$ is Lipschitz and invertible, and there is a bordism $(\mathpzc{E},D)\sim_{\rm bor} (\tilde{\mathpzc{E}},\hat{D})$.
\item There exists an operator homotopy $\beta(\tilde{\mathpzc{E}},\hat{D})\sim_{\rm op} \beta(\tilde{\mathpzc{E}},D\oplus 0)$.
\end{enumerate}
In particular, since $\beta(\tilde{\mathpzc{E}},D\oplus 0)=\beta(\mathpzc{E},D)+({}_0\mathpzc{E},0)$ for the degenerate Kasparov cycle $({}_0\mathpzc{E},0)$, the bounded transform of $(\mathpzc{E},D)$ is up to a degenerate Kasparov cycle operator homotopic to the bounded transform of an explicit, Lipschitz, and invertible cycle which is bordant to $(\mathpzc{E},D)$.
\end{prop}

\begin{proof}
There is a bordism $(\mathpzc{E},D)\sim_{\rm bor} (\tilde{\mathpzc{E}},\tilde{D})$ by Corollary \ref{invertibleldodcor}. There is a bordism $(\tilde{\mathpzc{E}},\tilde{D})\sim_{\rm bor} (\tilde{\mathpzc{E}},\hat{D})$ by Theorem \ref{bordismforodrdthm}. Therefore we have a bordism $(\mathpzc{E},D)\sim_{\rm bor} (\tilde{\mathpzc{E}},\hat{D})$ proving item 1). To prove item 2), we note that a short computation gives that $a(\hat{D}(1+\hat{D}^2)^{-1/2}-D(1+D^2)^{-1/2}\oplus 0)$ is compact for all $a\in A$. 
\end{proof}

In relation to the notion of weakly degenerate cycles from Subsection \ref{subsecweakdegff}, we have the following result. Recall that a Kasparov cycle $(\mathpzc{E},F)$ is said to be degenerate if for all $a\in A$, 
$$[F,a]=a(F^2-1)=a(F^*-F)=0.$$

\begin{prop}
Suppose that $(\mathpzc{E},D)$ is an $(\mathcal{A},B)$-cycle with $D$ having closed range and whose bounded transform $\beta (\mathpzc{E},D)$ satisfies that its operator homotopic cycle $(\mathpzc{E},\mathrm{sign}(\beta(D)))$ is degenerate, then the bounded perturbation $(\mathpzc{E},D+i\gamma\mathrm{sign}(\beta(D)))$ of $(\mathpzc{E},D)$ is weakly degenerate. In particular, if sign of the bounded transform of an $(\mathcal{A},B)$-cycle is degenerate, the cycle is weakly degenerate.
\end{prop}

\begin{proof}
The result is immediate from Corollary \ref{specedmopmad}. Alternatively, the proposition follows from Proposition \ref{cliffsymlemprop} and the fact that $i\gamma\mathrm{sign}(\beta(D))$ is a Clifford symmetry.
\end{proof}

We note that several of the results in this monograph holds also in the equivariant setting. Recall the terminology from the Subsections \ref{subsecgequicalalad} and \ref{subsecgequicalaladbordism}. For a $G$-equivariant cycle $(\mathpzc{E},D)$, we again write 
$$\beta(\mathpzc{E},D):=(\mathpzc{E},D(1+D^2)^{-1/2}),$$
which is considered as a $G$-equivariant Kasparov cycle. We have the following result by the same proof as in the non-equivariant case.

\begin{theorem}
Let $G$ be a second countable locally compact group, $\mathcal{A}\subseteq A$ be a $G$-invariant dense $*$-subalgebra of the $G-C^*$-algebra $A$ and let $B$ be a $G-C^*$-algebra. If a $G$-equvariant unbounded $(\mathcal{A},B)$-cycle $(\mathpzc{E},D)$ is $G$-equivariantly null bordant, then 
$$[\beta(\mathpzc{E},D)]=0 \quad\mbox{in $KK_*^G(A,B)$}.$$
In particular, the bounded transform induces a well defined homomorphism
$$\beta:\Omega_*^G(\mathcal{A},B)\to KK_*^G(A,B).$$
\end{theorem}

\section{Higher Atiyah-Patodi-Singer theory for $KK$-bordisms}
\label{highihereinda}

As is the case for a manifold with boundary, one can equip the operator appearing in a $KK$-bordism with Fredholm boundary conditions (of APS-type) and obtain an index associated with a $KK$-bordism. For manifolds, these boundary conditions were first introduced by Atiyah-Patodi-Singer and our approach follows the subsequent developments for Dirac operators on manifolds with boundary twisted by a $C^*$-bundle. The idea of studying APS-realizations of $KK$-bordisms was introduced and developed by Hilsum \cite{hilsumcmodbun}. The purpose of this section is to set Hilsum's ideas of APS-realizations of $KK$-bordisms in the foundations developed in the preceding parts of this monograph.

The main technical issue concerns the existence of the boundary conditions, but under a suitable technical assumption that we call \emph{very full} (see Definition \ref{veryfodldleed}), the theory can be developed just as for manifolds with boundary and a $C^*$-bundle. To overcome this issue we follow the methodology of \cite{LP}. The main difference to \cite{LP} is that there a considerable effort was needed to prove that manifolds twisted by a $C^*$-bundle are very full, that we in the case of $KK$-bordisms deem to be a property forced only by assumption. We restrict to the case of unital cycles (see Definition \ref{UnbKKcycDef}); this restriction is not that important for some of the preliminary results but will make a difference for the index theoretical results.

Throughout this section we let $B$ denote a unital $C^*$-algebra and consider $KK$-bordisms and cycles for $(\C,B)$. 

\subsection{Spectral sections of cycles}
\label{subsec:specsec}

Let $\mathfrak{Y}=(\mathpzc{E},D)$ be a cycle for $(\C,B)$. 

\begin{define}
Let $\mathfrak{Y}=(\mathpzc{E},D)$ be a cycle for $(\C,B)$. 
\begin{itemize}
\item A spectral cut for $\mathfrak{Y}$ is an operator of the form $\chi(D)\in \End^*_B(\mathpzc{E})$ where $\chi\in C^\infty(\R,[0,1])$ satisfies that $\chi(t)=0$ for $t\ll 0$ and $\chi(t)=1$ for $t\gg 0$.
\item A spectral section for an odd cycle $\mathfrak{Y}$ is a projection $P\in \End_B^*(\mathpzc{E})$ such that $\chi_1(D)\mathpzc{E}\subseteq P\mathpzc{E}\subseteq \chi_2(D)\mathpzc{E}$ for two spectral cuts $\chi_1(D)$ and $\chi_2(D)$. If $\mathfrak{Y}$ is an even cycle, a spectral section is a projection $P\in \End_B^*(\mathpzc{E})$ satisfying the condition for odd cycles and additionally $P-1/2$ is an odd operator.
\end{itemize}
\end{define}

The main technical problem in this subsection is to find a sufficient condition for the existence of spectral sections. The case of a general cycle is quite similar to the situation of a Dirac operator on a closed manifold twisted by a $C^*$-bundle, see Subsection \ref{diraconcstarbundleex}, and in this situation the existence of spectral sections is determined by the index $\ind_B(\mathfrak{Y})\in K_*(B)$. For closed manifolds, these results were proven by Wu \cite{Wu} and a gap in Wu's proof was later filled by Leichtnam-Piazza \cite{LP}. 

We use the notation $\mathcal{Q}_B(\mathpzc{E}):=\End^*_B(\mathpzc{E})/\K_B(\mathpzc{E})$ for the corona algebra of $\mathpzc{E}$ and $q_\mathpzc{E}:\End^*_B(\mathpzc{E})\to \mathcal{Q}_B(\mathpzc{E})$ for the quotient mapping. 

\begin{define}
\label{veryfodldleed}
Let $\mathfrak{Y}=(\mathpzc{E},D)$ be a unital cycle for $(\C,B)$. We say that $\mathfrak{Y}$ is very full if there exists an isometry $u\in \mathcal{Q}_B(\mathpzc{E})$ such that 
$$uu^*\leq \frac{1}{2}\left(1-q_\mathpzc{E}\left(D(1+D^2)^{-1/2}\right)\right).$$
\end{define}

The reader should note that $\frac{1}{2}\left(1-q_\mathpzc{E}\left(D(1+D^2)^{-1/2}\right)\right)$ is a projection in the corona algebra $\mathcal{Q}_B(\mathpzc{E})$. By definition, $\mathfrak{Y}=(\mathpzc{E},D)$ is very full if and only if the projection $\frac{1}{2}\left(1-q_\mathpzc{E}\left(D(1+D^2)^{-1/2}\right)\right)$ is very full (see \cite[Definition 1]{LP}). Here we use that $(\mathpzc{E},D)$ is a unital cycle to ensure that $\frac{1}{2}\left(1-q_\mathpzc{E}\left(D(1+D^2)^{-1/2}\right)\right)$ is a projection. A cycle defined from a Dirac operator on a closed manifold twisted by a full $C^*$-bundle is very full by \cite[Theorem 2]{LP}

\begin{theorem}
\label{existsppse}
Let $\mathfrak{Y}=(\mathpzc{E},D)$ be a  very full cycle for $(\C,B)$. Then $\mathfrak{Y}$ admits a spectral section if and only if $\ind_B(\mathfrak{Y})=0$ in $K_*(B)$.
\end{theorem}

The proof of this theorem follows the proof of \cite[Theorem 3]{LP}. Let us discuss some of the details. Firstly, we note that if $\mathfrak{Y}$ admits a spectral section, then $\ind_B(\mathfrak{Y})=0$ in $K_*(B)$. Here is a short proof of this fact. By definition, $-\ind_B(\mathfrak{Y})\in K_*(B)$ is the image of 
\begin{equation}
\label{fudnalcls}
\left[\frac{1}{2}-q_\mathpzc{E}\left(D(1+D^2)^{-1/2}\right)\right]\in K_{*+1}(\mathcal{Q}_B(\mathpzc{E})),
\end{equation}
under the composition $K_{*+1}(\mathcal{Q}_B(\mathpzc{E}))\to K_*(\K_B(\mathpzc{E}))\to K_*(B)$. Since $P-(1/2-D(1+D^2)^{-1/2})\in \K_B(\mathpzc{E})$, $[P]\in K_{*+1}(\End_B^*(\mathpzc{E}))$ is a preimage of the class \eqref{fudnalcls} under $(q_\mathpzc{E})_*:K_{*+1}(\End_B^*(\mathpzc{E}))\to K_{*+1}(\mathcal{Q}_B(\mathpzc{E}))$. But the composition 
$$K_{*+1}(\End_B^*(\mathpzc{E}))\to K_{*+1}(\mathcal{Q}_B(\mathpzc{E}))\to K_*(\K_B(\mathpzc{E})),$$ 
is zero, so $\ind_B(\mathfrak{Y})=0$.

Secondly, we note that if $\mathfrak{Y}=(\mathpzc{E},D)$ is very full, there is a projection $P_0\in \End_B^*(\mathpzc{E})$ with $P_0-(1/2-D(1+D^2)^{-1/2})\in \K_B(\mathpzc{E})$ by \cite[Theorem 1]{LP} (a result attributed to Meyer). The spectral section is then constructed from $P_0$ using functional calculus arguments as in \cite[Page 364-366]{LP} (a method attributed to Wu \cite{Wu}).

\begin{ex}
\label{prototypveryfull}
For a unital $C^*$-algebra $B$, we define
$$\mathfrak{Y}_{S^1,B}:=\left(\begin{matrix}L^2(S^1,B)\\\oplus\\ L^2(S^1,B)\end{matrix}, \begin{pmatrix}
(i\partial_x+\frac{1}{2})\otimes 1_B&0\\0&  -(i\partial_x+\frac{1}{2})\otimes 1_B\end{pmatrix}\right).$$
Morally, this is the Dirac operator on $S^1$ twisted by $B$. We can consider $\mathfrak{Y}_{S^1}$ as an odd $(\C,B)$-cycle or as an even $(\C,B)$-cycle for the grading 
$$\gamma:=\begin{pmatrix} 0& 1\\ 1&0\end{pmatrix}.$$
Denote the operator in $\mathfrak{Y}_{S^1,B}$ by $D_{S^1,B}$. Since $D_{S^1,B}$ is invertible, it admits the spectral section $\chi_{[0,\infty)}(D_{S^1,B})$. We note that $P_\pm:=\chi_{[0,\infty)}(\pm D_{S^1,B})$ satisfy that $P_{\pm} L^2(S^1,B)\oplus L^2(S^1,B)\cong L^2(S^1,B)$ and is a free infinitely generated $B$-module. 

The cycle $\mathfrak{Y}_{S^1,B}$ is explicitly nullbordant via filling in $S^1$ to a disk. We denote this bordism by $\mathfrak{X}_{\mathbb{D},B}$.
\end{ex}

Before proceeding with our study of spectral sections, let us stop and prove that it is always possible to perturb a cycle by a very full cycle satisfying an additional hypothesis, to ensure very fullness. 

\begin{prop}
\label{veryfullperturb}
Let $\mathfrak{Y}$ and $\mathfrak{Y}'$ be unital cycles for $(\C,B)$. If $\mathfrak{Y}'=(\mathpzc{E}',D')$ satisfies that $D'$ is invertible and the range of $\chi_{(-\infty,0]}( D')$ is a free infinitely generated $B$-module, then $\mathfrak{Y}+\mathfrak{Y}'$ is very full. In particular, for any cycle $\mathfrak{Y}$ the cycle $\mathfrak{Y}\oplus \mathfrak{Y}_{S^1,B}$ is very full and bordant to $\mathfrak{Y}$.
\end{prop}

We remark that a cycle with invertible operator always admits a spectral section, the positive spectral projection.

\begin{proof}
If $D'$ is invertible, we can define the projection $P:=\chi_{(-\infty,0]}( D')\in \End_B^*(\mathpzc{E}')$. Note that 
$$\frac{1}{2}\left(1-q_{\mathpzc{E}'}\left(D'(1+D'^2)^{-1/2}\right)\right)=q_{\mathpzc{E}'}(P).$$
By Kasparov's stabilization theorem and the assumption that $P\mathpzc{E}'$ is a free infinitely generated $B$-module, there exists a unitary isomorphism $v:P\mathpzc{E}'\to \mathpzc{E}\oplus \mathpzc{E}'$. Let $u_0$ denote the isometry 
$$\mathpzc{E}\oplus \mathpzc{E}'\xrightarrow{v^*} P\mathpzc{E}'\hookrightarrow \mathpzc{E}\oplus \mathpzc{E}'.$$
We have that $u_0$ satisfies that $u_0u_0^*=0\oplus P$. We define the isometry $u:=q_{\mathpzc{E}\oplus \mathpzc{E}'}(u_0)$ which satisfies that 
\begin{align*}
uu^*&=0\oplus q_{\mathpzc{E}'}(P)=0\oplus \frac{1}{2}\left(1-q_{\mathpzc{E}'}\left(D'(1+D'^2)^{-1/2}\right)\right)\leq \\
&\leq \frac{1}{2}\left(1-q_{\mathpzc{E}}\left(D(1+D^2)^{-1/2}\right)\right)\oplus \frac{1}{2}\left(1-q_{\mathpzc{E}'}\left(D'(1+D'^2)^{-1/2}\right)\right)=\\
&=\frac{1}{2}\left(1-q_{\mathpzc{E}\oplus \mathpzc{E}'}\left((D\oplus D')(1+(D\oplus D')^2)^{-1/2}\right)\right).
\end{align*}
\end{proof}

The following result shows that if there exists a spectral section, there is an abundance of spectral sections. The proof goes as in \cite[Section 2.5]{LP}, for more details, see also \cite{Wu}.

\begin{lemma}
\label{biggerspecseclem}
Let $\mathfrak{Y}=(\mathpzc{E},D)$ be a cycle admitting a spectral section. 
\begin{enumerate}
\item Let $\chi_1(D)$ be a spectral cut. Then there is a spectral cut $\chi_2(D)$ such that $\chi_1\chi_2=\chi_1$ and a spectral section $P$ such that $\chi_1(D)\mathpzc{E}\subseteq P\mathpzc{E}\subseteq \chi_2(D)\mathpzc{E}$.
\item Let $P_1$ and $P_2$ be spectral sections. There there exists spectral sections $Q_1$ and $Q_2$ such that for $j=1,2$
$$P_jQ_1=Q_1P_j=P_j\quad\mbox{and}\quad P_jQ_2=Q_2P_j=Q_2.$$
\end{enumerate}
\end{lemma}

Spectral sections in fact always arises as positive spectral projections of certain perturbations of $D$. 

\begin{define}
Let $\mathfrak{Y}=(\mathpzc{E},D)$ be a cycle. A trivializing operator for $\mathfrak{Y}$ is a self-adjoint operator $A\in \End_B^*(\mathpzc{E})$, which is odd if $\mathfrak{Y}$ is even, such that 
\begin{itemize}
\item for some $R>0$, $\psi(D)A=A\psi(D)=0$ for all bounded $\psi\in C^\infty(\R)$ with $\mathrm{supp}(\psi)\cap[-R,R]=\emptyset$, 
\item $D+A$ is invertible.
\end{itemize}
\end{define}

The reader should note that if $A$ is a trivializing operator for $\mathfrak{Y}$, then 
$$A\mathpzc{E}\subseteq \cap_{k\in \N}\Dom(D^k).$$
We also note that for any self-adjoint $A\in \End_B^*(\mathpzc{E})$ (which is odd if $\mathfrak{Y}$ is even), we have $\ind_B(D)=\ind_B(D+A)$ so $\ind_B(\mathfrak{Y})$ obstructs the existence of a trivializing operator. Finally, we note that if $D+A$ is invertible then there is an $\epsilon>0$ such that $\mathrm{Spec}(D+A)\cap (-\epsilon,\epsilon)=\emptyset$ and $\psi(D+A):\mathpzc{E}\dashrightarrow \mathpzc{E}$ is a well defined regular operator for any continuous function $\psi:\R\setminus (-\epsilon,\epsilon)\to \C$.

The following result is proven just as in \cite[Proposition 2.10]{LPGAFA}. 

\begin{prop}
\label{trivopprprop}
Let $\mathfrak{Y}=(\mathpzc{E},D)$ be a cycle admitting a spectral section $P$. Then there exists a trivializing operator $A\in \End_B^*(\mathpzc{E})$, such that 
$$P=\frac{1}{2}+\frac{D+A}{2|D+A|}=\chi_{[0,\infty)}(D+A).$$
In particular, if $\mathfrak{Y}$ is very full, $\mathfrak{Y}$ admits a trivializing operator if and only if $\ind_B(\mathfrak{Y})=0$ in $K_*(B)$.
\end{prop}

\begin{cor}
\label{existenceofveryfoidleled}
Assume that $(\mathpzc{E},D)$ is a unital very full $(\mathcal{A},B)$-cycle. Then the following are equivalent:
\begin{enumerate}
\item There exists a bounded selfadjoint operator $A\in \End_B^*(\mathpzc{E})$, which can be taken odd if  $(\mathpzc{E},D)$ is even, such that $D+A$ is invertible.
\item $\ind_B(D)=0$
\item $(\mathpzc{E},D)$ is null-bordant.
\end{enumerate}
\end{cor}

\begin{proof}
It follows from \cite[Theorem 3.7]{DGM}  that (2) and (3) are equivalent. Clearly, (1) implies (2). That (2) implies (1) follows from Theorem \ref{existsppse} and Proposition \ref{trivopprprop}.
\end{proof}

\subsection{Atiyah-Patodi-Singer theory for $KK$-bordisms}
\label{subsec:apsforborddad}

Let $\mathfrak{X}$ be a bordism for $(\C,B)$. Recall that the boundary $\mathfrak{X}^\partial=(\mathpzc{M},S)$ of $\mathfrak{X}$ is a cycle for $(\C,B)$. 

\begin{define}
Let $\mathfrak{X}$ be a bordism for $(\C,B)$ with boundary $\mathfrak{X}^\partial$. 
\begin{itemize}
\item We say that $\mathfrak{X}$ admits a spectral section if $\mathfrak{X}^\partial$ admits a spectral section. In this case, a spectral section for $\mathfrak{X}$ will refer to a spectral section for $\mathfrak{X}^\partial$ and a trivializing operator for $\mathfrak{X}$ will refer to a trivializing operator for $\mathfrak{X}^\partial$.
\item We say that $\mathfrak{X}$ is very full if $\mathfrak{X}^\partial$ is very full.
\end{itemize}
\end{define}

It follows from bordism invariance of the index that $\ind_B(\mathfrak{X}^\partial)=0$. By Corollary \ref{existenceofveryfoidleled}, a very full $(\C,B)$-bordism $\mathfrak{X}$ admits a spectral section.

We shall now give two index theoretically equivalent ways of constructing a $B$-Fredholm operator from a bordism $\mathfrak{X}$ and a choice of trivializing operator (assuming it exists). The two approaches are to glue on an infinite cylinder (as in Subsection \ref{subsecgluinginfinite}) or to equip the interior operator in the bordism with APS-boundary conditions. Fix a bordism $\mathfrak{X}$ and a trivializing operator $A$.

First up is constructing the $B$-Fredholm operator from gluing on an infinite cylinder onto $\mathfrak{X}$. We let $T_{(0,\infty)}$ denote the operator constructed in Subsection \ref{subsecgluinginfinite}. For a trivializing operator $A$ of $\mathfrak{X}$ and a function $\chi\in C^\infty([0,\infty),[0,1])$ with $\chi(t)=0$ near $t=0$ and $1-\chi\in C^\infty_c[0,\infty)$, we let $A_\infty:=\chi\hat{\boxtimes}A\in \End_B^*(\mathpzc{N}_{(0,\infty)})$. Define 
$$T_\infty(A):=T_{(0,\infty)}+A_\infty.$$

\begin{prop}
Let $\mathfrak{X}$ be a bordism with interior chain $\mathfrak{X}^\circ=(\mathpzc{N},T)$ and assume that $A$ is a trivializing operator for $\mathfrak{X}$. The operator $T_\infty(A):\mathpzc{N}_{(0,\infty)}\dashrightarrow\mathpzc{N}_{(0,\infty)}$ constructed in the preceding paragraph is $B$-Fredholm.
\end{prop}

\begin{proof}
Pick an $R>0$ such that $\chi(t)=1$ for $t>R-1$. Write $\mathfrak{X}^\partial=(\mathpzc{M},S)$ The symmetric chain $(L^2[0,R]\hat{\boxtimes}\mathpzc{M},\partial^{\rm min}\hat{\boxtimes}S+\chi|_{[0,R]}\hat{\boxtimes}A)$ fits into a bordism with boundary $\mathfrak{X}^\partial+(-(\mathpzc{M},S+A))$. Upon gluing on the bordism constructed in the previous sentence onto $\mathfrak{X}$, we can assume that $S$ is invertible, $A=0$, and the proposition follows from Lemma \ref{infcylfred}.
\end{proof}

Second up is constructing the $B$-Fredholm operator from equipping the interior chain of $\mathfrak{X}$ with suitable boundary conditions. We introduce the notation 
$$\mathpzc{N}^{H^1}\subseteq (1-p)\mathpzc{N}\oplus H^1[0,1]\hat{\boxtimes} \mathpzc{M},$$
for the submodule consisting of elements $f$ for which $b(\chi)f\in \Dom(T)$ and $(1-b(\chi))f\in L^2[0,1]\hat{\boxtimes}\Dom(D)$ for all $\chi\in C^\infty_c(0,1]$ with $\chi(t)=1$ near $t=1$. We can consider $\mathpzc{N}^{H^1}$ to be a $B$-Hilbert $C^*$-module by declaring the embedding 
$$\mathpzc{N}^{H^1}\xrightarrow{b(\chi)\oplus (1-b(\chi))}\Dom(T)\oplus \left(H^1[0,1]\hat{\boxtimes} \mathpzc{M}\cap  L^2[0,1]\hat{\boxtimes}\Dom(D)\right),$$
to be an isometry for some $\chi\in C^\infty_c(0,1]$ with $\chi(t)=1$ near $t=1$. If $\mathfrak{X}$ is even, the grading of $\mathpzc{N}$ induces a grading on $\mathpzc{N}^{H^1}$. The following result is immediate from the definition and the fact that $b(\chi)(1+T^*T)^{-1}$ is compact for $\chi\in C^\infty_c(0,1]$.

\begin{prop}
\label{compachone}
The inclusion $\mathpzc{N}^{H^1}\hookrightarrow \mathpzc{N}$ is adjointable and compact.
\end{prop}

Since the trace mapping $H^1[0,1]\to \C$, $f\mapsto f(0)$ is continuous there is a continuous adjointable trace mapping 
$$\gamma_0:\mathpzc{N}^{H^1}\to \C\hat{\boxtimes}\mathpzc{M}, \quad f\mapsto [(1-b(\chi))f](0),$$
for some $\chi\in C^\infty_c(0,1]$ with $\chi(t)=1$ near $t=1$. Here we consider $\C$ as a trivially graded Hilbert space, so $\C\hat{\boxtimes}\mathpzc{M}=\mathpzc{M}$ with trivial grading if $\mathfrak{N}$ is odd and $\C\hat{\boxtimes}\mathpzc{M}=\mathpzc{M}\oplus \mathpzc{M}$ graded by $1\oplus (-1)$ if $\mathfrak{X}$ is even. The map $\gamma_0$ is even if $\mathfrak{X}$ is even.

For a trivializing operator $A$, we introduce the notation 
\begin{align}
\label{adefeq}
P_A:&=\chi_{[0,\infty)}(0_\C\hat{\boxtimes}(S+A))=\\
\nonumber
&=\begin{cases}
\chi_{[0,\infty)}(S+A), \quad\mbox{if $\mathfrak{X}$ is even,}\\
\chi_{[0,\infty)}(S+A)\oplus \chi_{(-\infty,0]}(S+A), \quad\mbox{if $\mathfrak{X}$ is odd.}\end{cases}
\end{align}

For a bordism $\mathfrak{X}$ and a trivializing operator $A$, we define an operator $T_{\rm APS}(A)$ as
$$\Dom(T_{\rm APS}(A)):=\{f\in \mathpzc{N}^{H^1}: \; P_A(\gamma_0f)=0\},$$
and for $f\in \Dom(T_{\rm APS}(A))$ we set
$$T_{\rm APS}(A)f:=T(b(\chi)f)+(\partial\hat{\boxtimes}S)((1-b(\chi))f),$$
for a $\chi\in C^\infty_c(0,1]$ with $\chi(t)=1$ near $t=0$. A short computation using the axioms of a bordism shows that $T_{\rm APS}(A)$ does not depend on the choice of $\chi$. 

\begin{prop}
\label{regularapspsps}
The densely defined operator $T_{\rm APS}(A):\mathpzc{N}\dashrightarrow \mathpzc{N}$ is regular, self-adjoint, $B$-Fredholm and has $B$-compact resolvent.
\end{prop}

Before proving this proposition, we will need a lemma.

\begin{lemma}
Let $\mathfrak{Y}=(\mathpzc{E},D)$ be a cycle and $A$ a trivializing operator for $\mathfrak{Y}$. Define $D_A:=\partial\hat{\boxtimes}D$ on $L^2[0,\infty)\hat{\boxtimes}\mathpzc{E}$ with the domain
$$\Dom(D_A):=\{f\in H^1[0,\infty)\hat{\boxtimes}\mathpzc{E}\cap L^2[0,\infty)\hat{\boxtimes}\Dom(D): P_Af(0)=0\}.$$
Then $D_A$ is self-adjoint, regular and invertible.
\end{lemma}

\begin{proof}
It follows from \cite[Proof of Proposition 2.1]{WahlProductAPS} that $D_A$ is invertible (see the formula for the operator $D_{Z_l}(A)^{-1}$ in \cite[Proof of Proposition 2.1]{WahlProductAPS}). By the construction of \cite[Proof of Proposition 2.1]{WahlProductAPS}, $D_A^{-1}$ is self-adjoint and adjointable, so it follows that $(i+D_A)^{-1}=D_A^{-1}(iD_A^{-1}+1)^{-1}$ exists so $D_A$ is self-adjoint and regular.
\end{proof}

\begin{proof}[Proof of Proposition \ref{regularapspsps}]
It suffices to prove that $T_{\rm APS}(A)$ is closed, regular and selfadjoint. Indeed, if this is the case, $T_{\rm APS}(A)$ has $B$-compact resolvent (and is therefore $B$-Fredholm) by Proposition \ref{compachone}. The operator $T_{\rm APS}(A)$ is closed since $T$ and $D_A$ is. Indeed, if $(f_j)_{j\in \N}\subseteq \Dom(T_{\rm APS}(A))$ is Cauchy in the graph norm then for a cutoff function $\chi$, $(b(\chi)f_j)_{j\in \N}\subseteq \Dom(T)$ and $((1-b(\chi))f_j)_{j\in \N}\subseteq \Dom(D_A)$ are Cauchy so closedness of $T$ and $D_A$ implies that the limit of $(f_j)_{j\in \N}$ belongs to $\Dom(T_{\rm APS}(A))$.

To finish the proof, we show that $T_{\rm APS}(A)$ is self-adjoint and regular by constructing an approximate resolvent. Take a $\chi_1\in C^\infty_c([0,\infty),[0,1])$ satisfying $\chi_1(t)=1$ near $t=0$ and $\chi_1(t)=0$ for $t>1/2$. Set $\chi_2:=1-\chi_1\in C^\infty_b[0,\infty)$. For $\lambda\in \R\setminus \{0\}$, define 
$$r_\lambda:=b(\sqrt{\chi_1})(i\lambda+T_{(0,\infty)})^{-1}b(\sqrt{\chi_1})+b(\sqrt{\chi_2})(i\lambda+D_A)^{-1}b(\sqrt{\chi_2})\in \End_B^*(\mathpzc{N}),$$
is well defined. Since $(i\lambda+D_A)^{-1}$ satisfies the APS-boundary condition, it follows that $r_\lambda$ maps $\mathpzc{N}$ into $\Dom(T_{\rm APS}(A)$.
We compute that 
\begin{align*}
(i\lambda+T_{\rm APS}(A))r_\lambda=&b(\sqrt{\chi_1}')(i\lambda+T_{(0,\infty)})^{-1}b(\sqrt{\chi_1})+\\
&+b(\sqrt{\chi_2}')(i\lambda+D_A)^{-1}b(\sqrt{\chi_2})+b(\chi_1)+b(\chi_2)=\\
=&b(\sqrt{\chi_1}')(i\lambda+T_{(0,\infty)})^{-1}b(\sqrt{\chi_1})+\\
&+b(\sqrt{\chi_2}')(i\lambda+D_A)^{-1}b(\sqrt{\chi_2})+1
\end{align*}
Since $\chi_1$ and $\chi_2'$ are compactly supported, we conclude that $(i\lambda+T_{\rm APS}(A))r_\lambda-1$ is compact and extends to an adjointable operator on $\mathpzc{N}$. The computation also shows that $\|(i\lambda+T_{\rm APS}(A))r_\lambda-1\|_{\End_B^*(\mathpzc{N})}=O(\lambda^{-1})$ as $\lambda\to \infty$. Since $r_\lambda^*=r_{-\lambda}$ and $T_{\rm APS}(A)$ is closed, one readily verifies that $r_\lambda(i\lambda+T_{\rm APS}(A))-1$ also extends to a compact operator with $\|(i\lambda+T_{\rm APS}(A))r_\lambda-1\|_{\End_B^*(\mathpzc{N})}=O(\lambda^{-1})$ as $\lambda\to \infty$. This shows that $T_{\rm APS}(A)$ is self-adjoint and regular. The reader should note that this also gives a direct proof that $T_{\rm APS}(A)$ is $B$-Fredholm.
\end{proof}

\begin{remark}
We note that if $\mathfrak{X}$ is even, then 
$$T_{\rm APS}(A):=\begin{pmatrix}0&T_{\rm APS}(A)^+\\T_{\rm APS}(A)^-& 0\end{pmatrix},$$
where 
$$\Dom(T_{\rm APS}(A)^\pm):=\{f\in \mathpzc{N}^{H^1\mp}: \; \chi_{[0,\infty)}(\pm(S+A))(\gamma_0f)=0\}.$$
\end{remark}

\begin{theorem}
\label{apsthemams}
Let $\mathfrak{X}$ be a bordism with a trivializing operator $A$. Then the following equality holds in the group $K_*(B)$:  
$$\ind_B(T_{\rm APS}(A))=\ind_B(T_\infty(A)).$$
\end{theorem}

The proof of this theorem goes verbatim to \cite[Proof of Proposition 2.1]{WahlProductAPS}.

We shall need to deal with $KK$-bordisms that are not necessarily very full. This will be done by means of a very full modification, constructed in the spirit of Proposition \ref{veryfullperturb}. Recall the construction of $\mathfrak{X}_{\mathbb{D},B}$ and $\mathfrak{Y}_{S^1,B}$ from Example \ref{prototypveryfull} (see page \pageref{prototypveryfull}).

\begin{define}
\label{verefubeubd}
Let $\mathfrak{X}$ be a $(\C,B)$-bordism. Define its very full modification as 
$$\mathfrak{X}_{\rm vf}:=\mathfrak{X}+\mathfrak{X}_{\mathbb{D},B}.$$
\end{define}

\begin{theorem}
\label{veryfullthemaadn}
Let $\mathfrak{X}$ be a $(\C,B)$-bordism. Its very full modification $\mathfrak{X}_{\rm vf}$ satisfies the following:
\begin{enumerate}
\item $\partial \mathfrak{X}_{\rm vf}=\partial \mathfrak{X}+\mathfrak{Y}_{S^1,B}$ and $\mathfrak{X}_{\rm vf}$ is very full.
\item If $\mathfrak{X}$ is very full with trivializing operator $A$, then $A\oplus 0$ is a trivializing operator for $\mathfrak{X}_{\rm vf}$ and 
$$\ind_B((T\oplus T_{\mathbb{D}})_{\rm APS}(A\oplus 0))=\ind_B(T_{\rm APS}(A)),$$
where $T$ and $T_{\mathbb{D}}$ are the interior operator for $\mathfrak{X}$ and $\mathfrak{X}_{\mathbb{D},B}$, respectively.
\end{enumerate}
\end{theorem}

\begin{proof}
The first item follows from Proposition \ref{veryfullperturb}. For the second item, we first note that the boundary operator of $\mathfrak{X}_{\mathbb{D},B}$ is invertible, so $0$ is a trivializing operator for $\mathfrak{X}_{\mathbb{D},B}$. Therefore, $A\oplus 0$ is a trivializing operator for $\mathfrak{X}_{\rm vf}$ and the second item is equivalent to 
$$\ind_B((T_{\mathbb{D}})_{\rm APS}(0))=0.$$
This follows from functoriality ensuring that $\ind_B((T_{\mathbb{D}})_{\rm APS}(0))=\ind_\C((T_{\mathbb{D}})_{\rm APS}(0))\cdot [1_B]=0$ since the APS-index of the disc is zero. 
\end{proof}

The first item of Theorem \ref{veryfullthemaadn}, together with Theorem \ref{existsppse} (see page \pageref{existsppse}) and Proposition \ref{trivopprprop} (see page \pageref{trivopprprop}), shows that the very full modification of a bordism always admits a trivializing operator.

The second item of Theorem \ref{veryfullthemaadn} shows that the APS-index of a bordism is independent of very full modification. The reader should take this statement with a grain of salt, since the APS-index still depends on trivializing operators and there are more trivializing operators for the very full modification. However, it allows us to make the following definition.

\begin{define}
Let $\mathfrak{X}$ be a $(\C,B)$-bordism with trivializing operator $A$, we define 
$$\ind_{\rm APS}(\mathfrak{X},A):=\ind_B(T_{\rm APS}(A))\equiv \ind_B(T_{\infty}(A)).$$
More generally, if $\mathfrak{X}$ is a $(\C,B)$-bordism and $A$ is a trivializing operator for the very full modification, we set 
$$\ind_{\rm APS}(\mathfrak{X},A):=\ind_{\rm APS}(\mathfrak{X}_{\rm vf},A).$$
\end{define}

\begin{ex}[APS-theory for proper actions]
Let us study some further details for the chains studied in Example \ref{diraconcstarbundleex} and \ref{groupactionsex}. Consider a (second countable) locally compact group $G$ acting properly and cocompactly by smooth isometries on a Riemannian manifold with boundary $W$. Consider a unital $G-C^*$-algebra $B$ and a $G$-equivariant Clifford $B$-bundle $\mathcal{E}_B\to \overline{W}$. Pick a complete Dirac operator $\slashed{D}_\mathcal{E}$. By the arguments of Example \ref{groupactionsexbord}, $(L^2(W,\mathcal{E}_B), D_{\mathcal{E},{\rm min}})$ fits into a $G$-equivariant $(\C,B)$-bordism. Applying assembly, we arrive at a $(\C,B\rtimes G)$-bordism that we by an abuse of notation write as $\mu(L^2(W,\mathcal{E}_B), D_{\mathcal{E},{\rm min}})$. The constructions above therefore give meaning to APS-index theory for $G$-equivariant Clifford $B$-bundle $\mathcal{E}_B\to \overline{W}$ over a proper cocompact manifold with boundary $W$, and for a trivializing operator $A$ we arrive at a higher APS-index 
$$\ind_{\rm APS}(\mu(L^2(W,\mathcal{E}_B), D_{\mathcal{E},{\rm min}}),A)\in K_*(B\rtimes G).$$
APS-index theory for proper cocompact actions have been studied in \cite{guohochs,hochswang} for invertible boundary operators. 

We remark that it is an open problem to show that in this context, the nullbordant cycle $\partial \mu(L^2(W,\mathcal{E}_B), D_{\mathcal{E},{\rm min}})$ is very full. The direct method of proof from \cite{LP}, see in particular \cite[page 358]{LP}, shows that $\partial \mu(L^2(W,\mathcal{E}_B), D_{\mathcal{E},{\rm min}})$ is very full as soon as there exists a collection $(e_k)_{k\in \N}\subseteq \mu(L^2(\partial W,\mathcal{E}_B^\partial))\otimes_{B\rtimes G}\mathcal{M}(B\rtimes G)$ with $\langle e_k,e_j\rangle=\delta_{jk}1$ and $\chi(\mu(D_{\mathcal{E}}^\partial))e_k=e_k$ for a spectral cut.
\end{ex}

\begin{ex}[APS-theory for Lipschitz manifolds with boundary]

We continue to study Lipschitz manifolds as in Subsection \ref{topolmfsex} and \ref{bordimadinklm}. Consider a unital $C^*$-algebra $B$ and a Hermitian $B$-bundle $\mathcal{E}_B\to W$ with connection over a compact Riemannian Lipschitz manifold with boundary $W$. Form the twisted signature operator $D_{{\rm sign},\mathcal{E}}$ on $L^2(W;\wedge^*\otimes \mathcal{E}_B)$ as in Subsection \ref{topolmfsex}.The constructions above therefore gives meaning to APS-index theory for twisted signature operators on Lipschitz manifolds with boundary $W$, and for a trivializing operator $A$ we arrive at a higher APS-index 
$$\ind_{\rm APS}((L^2(W;\wedge^*\otimes \mathcal{E}_B), D_{{\rm sign},\mathcal{E}}),A)\in K_*(B).$$
APS-index theory for Lipschitz manifolds have been studied in \cite{hilsumlip} for the scalar case $B=\C$ where there is no need for trivializing operators.

\end{ex}

\subsection{Gluing, relative indices and spectral flows}
\label{subsec:glue}

We saw how to define the APS-index of a bordism equipped with a trivializing operator in the last subsection. It is at this stage unclear how the APS-index depends on the trivializing operator. The reader can rest easy, its dependence on the trivializing operator will behave in the same way as it does on a manifold. We shall make this analogy mathematically specific in this subsection by proving gluing formulas and spectral flow formulas for the APS-index. The results of this subsection are well known (see e.g. \cite{LP}) in the case of manifolds with boundary, and most of the proofs extend verbatim from that case.

\begin{theorem}[Gluing formula for the APS-index]
\label{gluingapsthememe}
Let $\mathfrak{X}_1$ and $\mathfrak{X}_2$ be $(\C,B)$-bordisms with 
$$\partial \mathfrak{X}_1=\mathfrak{Y}_1+(-\mathfrak{Y}_2)\quad\mbox{and}\quad \partial \mathfrak{X}_2=\mathfrak{Y}_2+(-\mathfrak{Y}_3)$$ 
for very full cycles $\mathfrak{Y}_1$, $\mathfrak{Y}_2$ and $\mathfrak{Y}_3$. For any trivializing operators $A_1$, $A_2$, and $A_3$ of  $\mathfrak{Y}_1$, $\mathfrak{Y}_2$ and $\mathfrak{Y}_3$, respectively, it holds that 
\begin{align}
\label{gluingformeq}
\ind_{\rm APS}(\mathfrak{X}_1\#_{\mathfrak{Y}_2}\mathfrak{X}_2,A_1\oplus -A_3)=\ind_{\rm APS}&(\mathfrak{X}_1,A_1\oplus -A_2)+\\
\nonumber &+\ind_{\rm APS}(\mathfrak{X}_2,A_2\oplus -A_3).
\end{align}
More generally, if $\mathfrak{Y}_1$, $\mathfrak{Y}_2$ and $\mathfrak{Y}_3$ are not necessarily very full and $A_1$, $A_2$, and $A_3$ are trivializing operators for the very full modifications of $\mathfrak{Y}_1$, $\mathfrak{Y}_2$ and $\mathfrak{Y}_3$ the equality \eqref{gluingformeq} still holds. 
\end{theorem}

\begin{proof}
The theorem follows using Bunke's proof of his relative index theorem \cite{bunkerelative}. The details are analogous to  Lemma \ref{cliffsymlem} (see page \pageref{cliffsymlem}) using the operators $(T_1\# T_2)_\infty(A_1\oplus -A_3)$, $(T_1)_\infty(A_1\oplus -A_2)$ and $(T_2)_\infty(A_2\oplus -A_3)$.
\end{proof}

To study the dependence of the APS-index on the trivializing operator, we will introduce the spectral flow of spectral sections. This notion is well studied in the case of manifolds \cite{LP} and for $KK$-bordisms it was studied in \cite{hilsumcmodbun}. Recall from Lemma \ref{biggerspecseclem} (see page \pageref{biggerspecseclem}) that if $P$ and $P'$ are two spectral sections of a $(\C,B)$-cycle $\mathfrak{Y}=(\mathpzc{E},D)$, there is a spectral section $Q$ commuting with $P$ and $P'$ and $QP=P$ and $QP'=P'$. In particular, $Q-P$ and $Q-P'$ are $B$-compact projections on $\mathpzc{E}$.

\begin{define}
Let $P$ and $P'$ be two spectral sections of an odd $(\C,B)$-cycle $\mathfrak{Y}=(\mathpzc{E},D)$. For a choice of $Q$ as in the preceding paragraph, we define 
$$[P-P']:=[Q-P']-[Q-P]\in K_0(B).$$
\end{define}

A short $K$-theoretical argument shows that $[P-P']$ does not depend on the choice of $Q$. In the even case, we need a slightly more convoluted definition constructed following \cite[Definition 8]{LP}. 

If $P$ is a spectral section of an even $(\C,B)$-cycle $\mathfrak{Y}=(\mathpzc{E},D)$, where the grading of $\mathpzc{E}$ will be denoted by $\gamma$, we define the projection valued function $P_0:[0,1]\to \End_B^*(\mathpzc{E})$ by 
$$P_0(t):=\frac{1}{2}\left(\cos(\pi t)\gamma+1\right)+\sin(\pi t) (2P-1).$$
If we decompose in terms of the grading $\mathpzc{E}=\mathpzc{E}_+\oplus \mathpzc{E}_-$, we can write 
$$P=\frac{1}{2}
\begin{pmatrix}
1& u_P\\
u_P^*& 1
\end{pmatrix},$$
for a unitary $u:\mathpzc{E}_-\to \mathpzc{E}_+$. In this decomposition, 
$$P_0(t)=\frac{1}{2}
\begin{pmatrix}
\cos(\pi t)+1& u_P\sin(\pi t)\\
u_P^*\sin(\pi t)& -\cos(\pi t)+1
\end{pmatrix}.$$
In fact, if $P=\chi_{[0,\infty)}(D+A)$ for a trivializing operator $A$, we have that 
$$P_0(t)=\chi_{[0,\infty)}\left(\gamma|D+A|\cos(\pi t)+(D+A)\sin(\pi t)\right).$$
It follows that $P_0$ is a spectral section for the odd $(\C,B\otimes C_0(0,1))$-cycle $(\mathpzc{E}\otimes C_0(0,1),D_0)$ where 
$$D_0(t)=\frac{\gamma|D+A|\cos(\pi t)+(D+A)\sin(\pi t)}{t(1-t)}.$$

The paragraph above and Lemma \ref{biggerspecseclem} (see page \pageref{biggerspecseclem}) allow us to conclude that if $P$ and $P'$ are two spectral sections for an even $(\C,B)$-cycle $\mathfrak{Y}=(\mathpzc{E},D)$, then $P_0-P_0'\in C_0((0,1),\mathbb{K}_B(\mathpzc{E}))$ and that there exists a spectral section $Q_0$ for $(\mathpzc{E}\otimes C_0(0,1),D_0)$ commuting with $P_0$ and $P_0'$ with $Q_0-P_0,Q_0-P_0'\in C_0((0,1),\mathbb{K}_B(\mathpzc{E}))$ being projections. 

\begin{define}
Let $P$ and $P'$ be two spectral sections of an even $(\C,B)$-cycle $\mathfrak{Y}=(\mathpzc{E},D)$. For a choice of $Q_0$ as in the preceding paragraph, we define 
$$[P-P']:=[Q_0-P_0']-[Q_0-P_0]\in K_1(B).$$
\end{define}

A short $K$-theoretical argument shows that $[P-P']$ does not depend on the choice of $Q_0$. Following the terminology of \cite{LP}, we make the following definition.

\begin{define}
Let $\mathfrak{Y}=(\mathpzc{E},D)$ be a $(\C,B)$-cycle of parity $*$ admitting a trivializing operator. If $A$ and $A'$ are trivializing operators of $\mathfrak{Y}$, we define 
$$\mathrm{sf}(\mathfrak{Y};A,A'):=[\chi_{[0,\infty)}(D+A)-\chi_{[0,\infty)}(D+A')]\in K_{*+1}(B).$$
\end{define}

\begin{theorem}[APS-index of a cylinder]
\label{apscylinderthm}
Let $\mathfrak{Y}$ be a $(\C,B)$-cycle admitting a trivializing operator. For two trivializing operators $A$ and $A'$ of $\mathfrak{Y}$, it holds that 
$$\ind_{\rm APS}(\Psi(\mathfrak{Y}),-A\oplus A' )=\mathrm{sf}(\mathfrak{Y};A,A'),$$
where we have used the shorthand notation $\Psi(\mathfrak{Y})$ for the cylinder bordism 
$$(\Psi(\mathfrak{Y}),(\id,\chi_{[0,1/4],[3/4,1]}),-\mathfrak{Y}\oplus \mathfrak{Y}).$$
\end{theorem}

\begin{proof}
This result is found in \cite[Section 3, Theorem 6 and Theorem 7]{LP} and the proof extends to the setting of $KK$-bordisms. For clarity, we will give a direct proof in the case that $\mathfrak{Y}$ is odd. 

First, we make some reductions. Write $\mathfrak{Y}=(\mathpzc{E},D)$. Recall the notation for $P_A$ from Equation \eqref{adefeq} (see page \pageref{adefeq}). After replacing $D$ with $D+A'$, we can assume that $A'=0$ and that $D$ is invertible and commutes with $P_{A'}$. By Lemma \ref{biggerspecseclem} (see page \pageref{biggerspecseclem}), in combination with Proposition \ref{trivopprprop} (see page \pageref{trivopprprop}), and Theorem \ref{gluingapsthememe} (see page \pageref{gluingapsthememe}) we can restrict our attention to the case that $A$ satisfies that $P_A\geq P_{A'}$ and that $P_{A'}$ commutes with $P_{A}$. 

We need to compute the index of $T_{\rm APS}(-A\oplus A')^+$ which we will write $T^+$ for notational simplicity. The operator $T^+$ takes the form
$$T^+=\frac{\partial}{\partial t}+D,$$
with the domain given by 
$$\Dom(T^+)=\{\xi\in H^1([0,1],\mathpzc{E})\cap L^2([0,1],\Dom(D)):(1-P_A)\xi(0)=P_{A'}\xi(1)=0\}.$$ 
For any $\xi\in \Dom(T^+)$, it holds that 
\begin{align*}
\xi(0)&=P_A\xi(0)\in P_A\mathpzc{E}\subseteq \Dom(\mathrm{e}^{-D})\quad\mbox{and}\\
\xi(1)&=(1-P_{A'})\xi(1)\in (1-P_{A'})\mathpzc{E}\subseteq \Dom(\mathrm{e}^{D}).
\end{align*}
If $\xi\in \ker(T^+)$, then solving $T^+\xi=0$ as a differential equation gives us that 
$$\xi(t)=\mathrm{e}^{-tD}\xi(0)=\mathrm{e}^{(1-t)D}\xi(1).$$ 
As such, for $\xi\in \ker(T^+)$, we have that $\xi(0)=\mathrm{e}^{D}\xi(1)$. Using that $D$ commutes with $P_{A'}$ we conclude that, $\xi\in \ker(T^+)$ if and only if $\xi(t)=\mathrm{e}^{-tD}\xi(0)$ for a 
$$\xi(0)\in P_A\mathpzc{E}\cap P_{A'}\mathpzc{E}= (P_A-P_{A'})\mathpzc{E},$$ 
because $P_A$ commutes with $P_{A'}$. It follows that restriction to $t=0$ defines an isomorphism of $B$-modules
$$\ker(T^+)\cong (P_A-P_{A'})\mathpzc{E}.$$
We conclude that $\ker(T^+)$ is a finitely generated and projective module with 
$$[\ker(T^+)]=[P_A-P_{A'}]\in K_0(B).$$ 
By a similar argument, $\ker(T^-)=\ker((T^+)^*)=0$. Since $T^+$ is a priori known to be Fredholm, 
$$\ind(T^+)=[P_A-P_{A'}].$$
\end{proof}

\begin{theorem}[Spectral flow formula for the APS-index]
\label{sffofmdmd}
Let $\mathfrak{X}$ be a $(\C,B)$-bordism admitting a trivializing operator. For two trivializing operators $A$ and $A'$ of $\mathfrak{X}$, it holds that 
$$\ind_{\rm APS}(\mathfrak{X},A)-\ind_{\rm APS}(\mathfrak{X},A')=\mathrm{sf}(\partial \mathfrak{X};A,A').$$
\end{theorem}

\begin{proof}
Again using Bunke's proof of his relative index theorem \cite{bunkerelative} (for the operators $T\# (-T)$, $T_\infty(A)$, $-T_\infty(-A')$, and $\Psi(D)_\infty(A\oplus -A')$) we have that 
$$\ind_{\rm APS}(\mathfrak{X},A)-\ind_{\rm APS}(\mathfrak{X},A')=\ind_B(\Psi(D)_\infty(A\oplus -A')).$$ 
We are here using that $\mathfrak{X}\#-\mathfrak{X}$ is a nullbordant closed cycle whose index therefore vanishes. By Theorem \ref{apsthemams}, the index of $\Psi(D)_\infty(A\oplus -A')$ is the APS-index of a cylinder as in Theorem \ref{apscylinderthm}. 
\end{proof}

Let us introduce yet another description of the APS-index of a bordism, it will be useful for relative constructions. Recall the notion of  the transgression bordism from Definition \ref{transgres} and the Shubin bordism from Theorem \ref{weakdegnullbordthm}. 

\begin{define}
\label{defsinadjonaojn}
Assume that $\mathcal{Y}=(\mathpzc{E},D)$ is a $(\C,B)$-cycle with trivializing operator $A$. Then, by definition, $(\mathpzc{E},D+A)$ is weakly degenerate (for the weakly degenerate decomposition $D_0=0$ and $S=D+A$) and we define the transgressed Shubin bordism 
$$\mathrm{Sh}(\mathpzc{E},D,A):=\mathcal{T}(\mathpzc{E},D,A)\#_{(\mathpzc{E},D+A)}\mathrm{Sh}(\mathpzc{E},D+A).$$
\end{define}

\begin{lemma}
\label{sinadjonaojn}
Assume that $\mathcal{Y}$ is a $(\C,B)$-cycle with trivializing operator $A$. Then the transgressed Shubin bordism $\mathrm{Sh}(\mathcal{Y},A)$ is a well defined $(\C,B)$-bordism with boundary $\mathcal{Y}$. Moreover, we have that 
$$\ind_{\rm APS}(\mathrm{Sh}(\mathcal{Y},A),A)=0.$$
\end{lemma}

\begin{proof}
The collection $\mathrm{Sh}(\mathcal{Y},A)$ is a bordism since it is a gluing of $KK$-bordisms. That $\ind_{\rm APS}(\mathrm{Sh}(\mathcal{Y},A),A)=0$ follows from the fact that the associated operator has vanishing kernel. Indeed, the kernel is found from solving an explicit operator valued ODE on $\R$ and any homogeneous solution $f$ thereof must fulfil 
$$f(t_0)\in \mathrm{im} \chi_{(0,\infty)}(D+A)\cap \mathrm{im}\chi_{(-\infty,0)}(D+A)=0,$$
for some $t_0\gg0$. Therefore $f=0$. 
\end{proof}

From Lemma \ref{sinadjonaojn}, and the gluing formula for APS-indices (Theorem \ref{gluingapsthememe}) we deduce the following theorem.

\begin{theorem}[APS-index from a Shubin suspension]
\label{shubinada}
Let $\mathfrak{X}$ be a $(\C,B)$-bordism with a trivializing operator $A$. Then $\mathfrak{X}\#_{\partial \mathfrak{X}} \mathrm{Sh}(\partial \mathfrak{X},A)$ is a $(\C,B)$-cycle with index computed from 
$$\ind_B(\mathfrak{X}\#_{\partial \mathfrak{X}} \mathrm{Sh}(\partial \mathfrak{X},A))=\ind_{\rm APS}(\mathfrak{X},A).$$
In other words, the APS-index of $(\mathfrak{X},A)$ can be computed either from equipping the interior cycle with APS boundary conditions $T_{\rm APS}(A)$, or attaching an infinite cylinder and perturbing by $A$ arriving at $T_\infty(A)$ or finally by attaching an infinite cone (by means of a linearly growing perturbation defined from $A$).
\end{theorem}

The last statement follows from the first and Theorem \ref{apsthemams}.

\subsection{Cayley- and Higson transforms}
\label{cayleandihigsoin}

The main usage of APS index theory in this monograph will be when constructing secondary invariants for various relative constructions using $KK$-bordisms. We will now discuss a construction relating to correcting for the dependence on the choice of trivializing operator. The two constructions, the Cayley transform of odd cycles and the Higson transform of even cycles, can be seen as implementing an isomorphism $\Omega_*(\C,B)\to K_*(B)$. We consider only $\mathcal{A}=\C$ and $B$ unital in this subsection. We tacitly assume that all cycles are unital, in which case a $(\C,B)$-cycle is a pair $(\mathpzc{E},D)$ where $D$ is self-adjoint, regular with compact resolvent (and odd if $\mathpzc{E}$ is graded).

\begin{define} 
\label{casandlknad}
Let $D$ be a self-adjoint regular operator on a Hilbert $C^*$-module $\mathpzc{E}$. We define its Cayley transform as 
$$c(D):=(iD+1)(iD-1)^{-1}=1-2i(D+i)^{-1}.$$
If $\mathpzc{E}$ is graded by a grading $\gamma$ making $D$ odd, we define the Higson transform of $D$ as
$$h(D):=(1+D^2)^{-1}\gamma-D(1+D^2)^{-1}+\frac{1-\gamma}{2}.$$
\end{define} 

If $\mathpzc{E}$ is graded by a grading $\gamma$, we set $Q_-:=\frac{1-\gamma}{2}$ and $Q_+=1-Q_-$. We define 
$$\tilde{\mathbb{K}}_B(\mathpzc{E}):=\widetilde{\mathbb{K}_B(\mathpzc{E})}+\widetilde{\mathbb{K}_B(\mathpzc{E})}\gamma=\widetilde{\mathbb{K}_B(\mathpzc{E})}Q_++\widetilde{\mathbb{K}_B(\mathpzc{E})}Q_-,$$ 
which is a split extension of $\C\ell_1$ by $\mathbb{K}_B(\mathpzc{E})$. Here $\widetilde{\mathbb{K}_B(\mathpzc{E})}$ denotes the unitalization of $\mathbb{K}_B(\mathpzc{E})$.

\begin{prop}
\label{candhareproj}
Let $D$ be a self-adjoint regular operator on a Hilbert $C^*$-module $\mathpzc{E}$. Then $c(D)$ is a unitary. If $\mathpzc{E}$ is graded and $D$ is odd, then $h(D)$ is a projection. Moreover, 
\begin{itemize}
\item If $(\mathpzc{E},D)$ is an odd $(\C,B)$-cycle then $c(D)\in \widetilde{\mathbb{K}_B(\mathpzc{E})}$. 
\item If $(\mathpzc{E},D)$ is an even $(\C,B)$-cycle then $h(D)-Q_-\in \mathbb{K}_B(\mathpzc{E})$, in particular $h(D)\in \tilde{\mathbb{K}}_B(\mathpzc{E})$
\end{itemize}
\end{prop}

\begin{proof}
It is well known that the Cayley transform of a self-adjoint regular operator is unitary, it follows from \cite[Lemma 9.8 and Theorem 10.5]{lancesbook}. To prove that $h(D)$ is a projection, we follow the construction in \cite[page 191]{Hig}. The operator $U:=(1+D^2)^{-1/2}+D(1+D^2)^{-1/2}\gamma$ is a unitary and by construction $h(D)=UQ_-U^*$ so it is a projection.

If $(\mathpzc{E},D)$ is a $(\C,B)$-cycle, then $1-c(D)=2i(D+i)^{-1}\in \mathbb{K}_B(\mathpzc{E})$ since $D$ has compact resolvent. If $(\mathpzc{E},D)$ is an even $(\C,B)$-cycle, then $h(D)-Q_-=(1+D^2)^{-1}\gamma-D(1+D^2)^{-1}\in \mathbb{K}_B(\mathpzc{E})$ since $D$ has compact resolvent. 
\end{proof}

\begin{prop}
\label{limitsofcandh}
Let $D$ be an invertible self-adjoint regular operator on a Hilbert $C^*$-module (graded) $\mathpzc{E}$. Then 
$$c(t^{-1}D)-1=O(t), \quad\mbox{in norm as $t\to 0$}.$$
If $D$ is odd for a grading on $\mathpzc{E}$, then 
$$h(t^{-1}D)-Q_-=O(t), \quad\mbox{in norm as $t\to 0$}.$$
\end{prop}

\begin{proof}
The first statement follows from the computation 
$$c(t^{-1}D)=1-2it(D+it)^{-1},$$
and the fact that $\sup_{t\in \R} \|(D+it)^{-1}\|<\infty$ if $D$ is invertible. The second statement follows by the similar computation 
$$h(t^{-1}D):=t^2(t^2+D^2)^{-1}\gamma-tD(t^2+D^2)^{-1}+Q_-.$$
\end{proof}

We shall fix a function $\chi\in C^\infty((0,1],\mathbb{R})$ such that $\chi\geq 1$ and 
$$
\chi(t)=\begin{cases} 
t^{-1}, \; &\mbox{near $t=0$}\\ 
1, \; &\mbox{near $t=1$}
\end{cases}.
$$
The set of such functions $\chi$ is path connected because it forms a convex subset of $C^\infty(0,1]$. The following proposition follows from Proposition \ref{candhareproj} and \ref{limitsofcandh}.

\begin{prop}
\label{pathrslsdsdsdl}
Assume that $(\mathpzc{E},D)$ is a cycle for $(\C,B)$ and that $A:[0,1]\to \mathrm{End}_B^*(\mathpzc{E})$ is a self-adjoint norm-continuous path such that $A$ is constant near $t=0$ and $A(0)$ is a trivializing operator for $(\mathpzc{E},D)$. Then 
$$t\mapsto c(\chi(t)(D+A(t))),$$
defines an element of $\widetilde{C_0((0,1],\mathbb{K}_B(\mathpzc{E}))}$. 

If additionally $\mathpzc{E}$ has a grading in which $D$ and $A$ are odd, then 
$$t\mapsto h(\chi(t)(D+A(t))),$$
defines an element of $Q_-+C_0((0,1],\mathbb{K}_B(\mathpzc{E}))$.
\end{prop}

\begin{theorem}
Let $B$ be a unital $C^*$-algebra. The mappings 
\begin{align*}
(\mathpzc{E},D)&\mapsto [c(D)]\in K_1(B), \; *=1,\\
(\mathpzc{E},D)&\mapsto [h(D)-Q_-]\in K_0(B), \; *=0,
\end{align*}
defines an isomorphism $\Omega_*(\C,B)\cong K_*(B)$.
\end{theorem}

\begin{proof}
By \cite[Theorem 3.7]{DGM} (see also Theorem \ref{bddisothm} below), the bounded transform defines an isomorphism $\Omega_*(\C,B)\cong K_*(B)$, so the proof is readily reduced to showing that the indicated mapping is well defined (i.e. respects bordism).

Assume that $(\mathpzc{E},D)$ is nullbordant. It follows from Theorem \ref{veryfullthemaadn}  and Proposition \ref{pathrslsdsdsdl} that we can replace $(\mathpzc{E},D)$ by its very full modification without altering the image in $K$-theory. If $(\mathpzc{E},D)$ is very full and nullbordant, it admits a trivializing operator by Corollary \ref{existenceofveryfoidleled}. By homotopy invariance of $K$-theory we can assume that $D$ is invertible, and again by homotopy invariance of $K$-theory the theorem can be concluded from Proposition \ref{pathrslsdsdsdl}.
\end{proof}

\begin{lemma}
\label{sfandcayhgidid}
Assume that $(\mathpzc{E},D)$ is a cycle for $(\C,B)$ and that $A:[0,1]\to \mathrm{End}_B^*(\mathpzc{E})$ is a self-adjoint norm-continuous path such that $A$ is constant near $t=0$ and $t=1$ and $A(0)$ and $A(1)$ are trivializing operators for $(\mathpzc{E},D)$. Take a function $\hat{\chi}\in C^\infty((0,1],\mathbb{R})$ such that $\chi\geq 1$ with $\hat{\chi}(t)=\chi(t)$ near $t=0$ and $\hat{\chi}(t)=\chi(1-t)$ near $t=1$. Then 
$$t\mapsto c(\hat{\chi}(t)(D+A(t))),$$
defines an element of $\widetilde{C_0((0,1),\mathbb{K}_B(\mathpzc{E}))}$ and it holds that 
$$[c(\hat{\chi}(\cdot)(D+A(\cdot)))]=[P_{A(0)}-P_{A(1)}],$$
under the Bott isomorphism $K_1(SB)\cong K_0(B)$.

If additionally $\mathpzc{E}$ has a grading in which $D$ and $A$ are odd, then 
$$t\mapsto h(\hat{\chi}(t)(D+A(t))),$$
defines an element of $Q_-+C_0((0,1),\mathbb{K}_B(\mathpzc{E}))$ and it holds that 
$$[h(\hat{\chi}(\cdot)(D+A(\cdot)))]=[P_{A(0)}-P_{A(1)}],$$
under the Bott isomorphism $K_0(SB)\cong K_1(B)$.

\end{lemma}

\begin{proof}
We only consider the odd case. The even case is proven analogously. Our argument follows the proof of \cite[Lemma 4.8]{DGrelI}. Consider the choice function $f(x):=\frac{1}{\pi}\tan^{-1}(x)+\frac{1}{2}$. The function $f\in C^\infty(\field{R})$ satisfies $f(+\infty)=1$, $f(-\infty)=0$ and $c(x)=\mathrm{e}^{2\pi if(x)}$. Write $\check{D}$ for the resolvent continuous path of self-adjoint regular operators $(0,1)\ni t\mapsto \hat{\chi}(t)(D+A(t))$ and note that $\check{D}(t)$ is invertible for $t$ near $0$ or $1$.  We recall the notation 
$$P_A:=\chi_{[0,\infty)}(D+A)\equiv\lim_{t\to 0} f(\check{D})(t)\quad\mbox{and}\quad P_{A'}:=\chi_{[0,\infty)}(D+A')\equiv \lim_{t\to 1}f(\check{D})(t).$$ 
By Lemma \ref{biggerspecseclem} we can pick a spectral section $Q_0$ for $D_E^\partial$ such that $Q_0-P_A$ and $Q_0-P_{A'}$ are projections in $\mathbbm{K}_{B}(\mathpzc{E})$. Consider the continuous path
$$p(t)=\begin{cases}
(1-2t)P_A+2tQ_0, \;& t\in [0,\frac{1}{2}),\\
(2t-1)P_{A'}+2(1-t)Q_0, \;& t\in [\frac{1}{2},1].
\end{cases}$$
In fact, $p$ is a constant projection modulo $C[0,1]\otimes \mathbbm{K}_{B}(\mathpzc{E})$ satisfying $p(0)=P_A$ and $p(1)=P_{A'}$ in the end-points. Since $f(\check{D})-p\in C_0(0,1)\otimes \mathbbm{K}_{B}(\mathpzc{E})$ we conclude the identity
$$[c(\check{D})]=[\mathrm{e}^{2\pi if(\check{D})}]=[\mathrm{e}^{2\pi ip}],$$
as classes in $K_1(C_0(0,1)\otimes B)$. 

Consider the function $q:[0,1]\to  \mathbbm{K}_{B}(\mathpzc{E})$ defined by 
$$q(t)=\begin{cases}
2t(Q_0-P_A), \;& t\in [0,\frac{1}{2}),\\
(2t-1)(P_{A'}-Q_0), \;& t\in [\frac{1}{2},1].
\end{cases}$$
The path $q$ is not continuous, but $\mathrm{e}^{2\pi i q}$ is. Since 
$$p(t)-q(t)=\begin{cases}
P_A, \;& t\in [0,\frac{1}{2}),\\
Q_0, \;& t\in [\frac{1}{2},1].
\end{cases}$$ 
is a projection that commutes with $p(t)$ and $q(t)$ for each $t$, we have the identity $\mathrm{e}^{2\pi i q(t)}=\mathrm{e}^{2\pi i p(t)}$ for each $t$. Therefore we have the equality $[c(\check{D})]=[\mathrm{e}^{2\pi ip}]=[\mathrm{e}^{2\pi iq}]$.

On the other hand, by construction we have the identity 
$$[P_{A'}-P_A]\equiv [Q_0-P_A]-[Q_0-P_{A'}].$$
Under the Bott isomorphism, we can identify
$$[P_{A'}-P_A]=[\mathrm{e}^{2\pi i t(Q_0-P_A)}]-[\mathrm{e}^{2\pi i t(Q_0-P_{A'})}]=[\mathrm{e}^{2\pi i t(Q_0-P_A)}\#\mathrm{e}^{2\pi i t(P_{A'}-Q_0)}].$$
We note that $\mathrm{e}^{2\pi i t(Q_0-P_A)}\#\mathrm{e}^{2\pi i t(P_{A'}-Q_0)}=\mathrm{e}^{2\pi i q}$ which proves the desired equality $[c(\hat{\chi}(\cdot)(D+A(\cdot)))]=[P_{A(0)}-P_{A(1)}]$.
\end{proof}

We shall fix another function $\tilde{\chi}\in C^\infty((0,1],\mathbb{R})$ such that 
$$\tilde{\chi}(t)=\begin{cases} t^{-1}, \; &\mbox{near $t=0$}\\ 0, \; &\mbox{near $t=1$}\end{cases}.$$
The set of such $\tilde{\chi}$:s is path connected because it forms a convex subset of $C^\infty(0,1]$. For a cycle $\mathfrak{Y}=(\mathpzc{E},D)$ for $(\C,B)$ and a trivializing operator $A$, we define the following families of operators 
\begin{equation}
\label{pathesospsdp}
\hat{A}_t:=\tilde{\chi}(t)A\quad\mbox{and}\quad \hat{D}_{t}:=\chi(t)D,\quad \mbox{for}\;\; t\in (0,1].
\end{equation}
If $\mathfrak{Y}$ is odd, we write $c(\mathfrak{Y},A)\in \widetilde{C_0((0,1],\mathbb{K}_B(\mathpzc{E}))}$ for the mapping $t\mapsto c(\hat{D}_{t}+\hat{A}_t))$ (it is well defined by Proposition \ref{pathrslsdsdsdl}). If $\mathfrak{Y}$ is even, we write $h(\mathfrak{Y},A)\in Q_-+C_0((0,1],\mathbb{K}_B(\mathpzc{E}))$ for the mapping $t\mapsto h(\hat{D}_{t}+\hat{A}_t)$ (it is well defined by Proposition \ref{pathrslsdsdsdl}).

The following result is immediate from the spectral flow formula (see Theorem \ref{sffofmdmd} on page \pageref{sffofmdmd}) and Lemma \ref{sfandcayhgidid}. For a continuous path $f:[0,1]\to X$ into a topological space $X$, we write $f^{\rm op}$ for $t\mapsto f(1-t)$, and for two paths $f_1,f_2:[0,1]\to X$ with $f_1(1)=f_2(0)$ we define the continuous path
$$f_1\#f_2(t):=
\begin{cases}
f_1(2t), \; &t\in [0,1/2],\\
f_2(2t-1), \; &t\in (1/2,1].
\end{cases}$$

\begin{cor}
Let $\mathfrak{X}$ be a $(\C,B)$-bordism of degree $*$ with two trivializing operators $A$ and $A'$. Then under the Bott isomorphism it holds that
$$\ind_{\rm APS}(\mathfrak{X},A)-\ind_{\rm APS}(\mathfrak{X},A')=
\begin{cases}
\left[c(\partial \mathfrak{X},A)\#c(\partial \mathfrak{X},A')^{\rm op}\right], \; &*=1,\\
{}\\
[h(\partial \mathfrak{X},A)\#h(\partial \mathfrak{X},A')^{\rm op}-Q_-], \; &*=0.
\end{cases}$$
\end{cor}

\part{Isomorphism results for $KK$-bordism groups}
\label{partoniso}

In this part of the monograph we study the qualitative properties of the bounded transform
$$\beta:\Omega_*(\mathcal{A},B)\to KK_*(A,B).$$
Recall the bounded transform discussed in Subsection \ref{subsecbddtrans}. The bounded transform is well defined by work of Hilsum \cite{hilsumbordism}. We are interested in finding a generically applicable sufficient condition for $\beta$ to be an isomorphism. For countably generated $\mathcal{A}$, surjectivity of $\beta$ follows from Baaj-Julg's lifting theorem \cite{baajjulg}. Injectivity of $\beta$ is essentially a surjectivity statement at the level of equivalences of cycles, and Baaj-Julg's lifting argument was applied in this context independently by Kaad \cite{Kaadunbdd} and Mesland-van Dungen \cite{meslandvandung} to homotopy classes of unbounded Kasparov cycles. We follow Kaad's argument \cite{Kaadunbdd} to lift operator homotopies to $KK$-bordisms proving $\beta$ to be an isomorphism for countably generated $\mathcal{A}$. Using the technology of Part \ref{part:prelim} and \ref{part:geocons}, we arrive at a plethora of geometric examples when $\beta$ is an isomorphism. To complement this picture, we show that when $\mathcal{A}$ is an amenable Banach algebra, then injectivity and surjectivity of $\beta$ fail under additional mild assumptions on $A$.

\section{The bounded transform for countably generated $\mathcal{A}$}
\label{bddtransformisosec}

The goal of this section is to prove the following theorem.

\begin{theorem}
\label{bddisothm}
Let $\mathcal{A}$ be a countably generated $*$-algebra with $C^*$-closure $A$. For any $C^*$-algebra $B$, the bounded transform 
$$\beta:\Omega_*(\mathcal{A},B)\to KK_*(A,B),$$
is an isomorphism.
\end{theorem}

With minor modifications, our proof of Theorem \ref{bddisothm} extends to show that the equivariant bounded transform
$$\beta:\Omega_*^G(\mathcal{A},B)\to KK_*^G(A,B),$$
is an isomorphism as soon as $\mathcal{A}$ is countably generated and $G$ is either compact or discrete and acting trivially on $B$. We prove the theorem by a sequence of lemmas.

\begin{lemma}
\label{lemmaa}
Let $A$ and $B$ be $C^*$-algebras, $\mathpzc{E}$ a $B$-Hilbert $C^*$-module and $F\in C([0,1],\End^*_B(\mathpzc{E}))$ an operator homotopy of $(A,B)$-Kasparov modules on $\mathpzc{E}$. Assume that $F^2=1$ and $F=F^*$. Then there exists a norm-smooth operator homotopy of $(A,B)$-Kasparov modules $\tilde{F}\in C^\infty([0,1],\End^*_B(\mathpzc{E}))$ constant near the endpoints and satisfying
$$F(0)=\tilde{F}(0)\quad\mbox{and}\quad F(1)=\tilde{F}(1),$$
and $\tilde{F}^2=1$ and $\tilde{F}=\tilde{F}^*$
\end{lemma}

\begin{proof}
We can assume that $F$ is constant near the endpoints after possibly reparametrizing the unit interval. By mollifying $F$, we can for any $\epsilon>0$ construct a self-adjoint $F_\epsilon \in C^\infty([0,1],\End^*_B(\mathpzc{E}))$ such that $[F_\epsilon,a]\in C^\infty([0,1],\mathbb{K}_B(\mathpzc{E}))$, $\|F_\epsilon^2-1\|<\epsilon$ and finally $F(0)=F_\epsilon(0)$ and $F(1)=F_\epsilon(1)$. For $\epsilon<1/4$, $F_\epsilon$ is invertible in $C^\infty([0,1],\End^*_B(\mathpzc{E}))$ and we can take 
$$\tilde{F}:=F_\epsilon|F_\epsilon|^{-1}.$$
\end{proof}

\begin{lemma}
\label{lemmab}
Let $A$ and $B$ be $C^*$-algebras, and $\mathcal{A}\subseteq A$ a countably generated, dense $*$-subalgebra. The bounded transform
$$\beta:\Omega_*(\mathcal{A},B)\to KK_*(A,B),$$
is injective if and only if whenever $(\mathpzc{E},D)$ and $(\mathpzc{E},D')$ are $(\mathcal{A},B)$-cycles such that 
\begin{enumerate}
\item $(\mathpzc{E},D)$ and $(\mathpzc{E},D')$ are Lipschitz and both $D$ and $D'$ are invertible;
\item there is a self-adjoint $\tilde{F}\in C^\infty([0,1],\End^*_B(\mathpzc{E}))$ such that $\tilde{F}^2=1$ and 
$$\tilde{F}(0)=D|D|^{-1}\quad\mbox{and}\quad \tilde{F}(1)=D'|D'|^{-1},$$
\end{enumerate}
then $(\mathpzc{E},D)\sim_{\rm bor} (\mathpzc{E},D')$.
\end{lemma}

\begin{proof}
We only prove the if-direction, the other direction is immediate. Since $\mathcal{A}\subseteq A$ is a countably generated, dense $*$-subalgebra $\beta$ is surjective and the lemma consists of showing that the equivalence relation on the set of $(\mathcal{A},B)$-cycles induced from $KK_*(A,B)$ and $\beta$ coincides with the bordism relation under the stated assumption.

Assume that  $(\mathpzc{E}_0,D_0)$ and $(\mathpzc{E}_0',D'_0)$ are $(\mathcal{A},B)$-cycles such that $[\beta(\mathpzc{E}_0,D_0)]=[\beta(\mathpzc{E}_0',D'_0)]$ in $KK_*(A,B)$. The equivalence between $\beta(\mathpzc{E}_0,D_0)$ and $\beta(\mathpzc{E}_0',D'_0)$ in $KK_*(A,B)$ comes from a finite sequence of degenerate equivalences and operator homotopies. We therefore need to prove that degenerate equivalences and operator homotopies in $KK_*(A,B)$ lift to $KK$-bordisms of cycles in $\Omega_*(\mathcal{A},B)$. By Theorem \ref{specdecompnulbord}, degenerate equivalence lifts to $KK$-bordisms. So it suffices to show $(\mathpzc{E},D)$ and $(\mathpzc{E},D')$ are $(\mathcal{A},B)$-cycles admitting an operator homotopy $\beta(D)\sim_{\rm oh}\beta(D')$ then $(\mathpzc{E},D)\sim_{\rm bor} (\mathpzc{E},D')$. By Corollary \ref{ononaod}, we can assume that $(\mathpzc{E},D)$ and $(\mathpzc{E},D')$ are Lipschitz and both $D$ and $D'$ are invertible. Let $F_0$ denote the operator homotopy $\beta(D)\sim_{\rm oh}\beta(D')$. We can assume that $F_0$ is self-adjoint. By performing a standard doubling trick (as in Corollary \ref{invertibleldodcor}), we can assume that $F_0$ is self-adjoint and invertible, so $F:=F_0|F_0|^{-1}$ satisfies the assumptions of Lemma \ref{lemmaa}. Therefore, the required $\tilde{F}\in C^\infty([0,1],\End^*_B(\mathpzc{E}))$ exists and $(\mathpzc{E},D)\sim_{\rm bor} (\mathpzc{E},D')$. In conclusion, $[\beta(\mathpzc{E}_0,D_0)]=[\beta(\mathpzc{E}_0',D'_0)]$ in $KK_*(A,B)$ implies $(\mathpzc{E},D)\sim_{\rm bor} (\mathpzc{E},D')$ and so $\beta$ is injective.
\end{proof}

\begin{lemma}
\label{lemmae}
Let $A$ and $B$ be $C^*$-algebras, and $\mathcal{A}\subseteq A$ a countably generated, dense $*$-subalgebra. Consider two $(\mathcal{A},B)$-cycles $(\mathpzc{E},D)$ and $(\mathpzc{E},D')$ such that 
\begin{enumerate}
\item Both cycles $(\mathpzc{E},D)$ and $(\mathpzc{E},D')$ satisfy that $D$ and $D'$ are invertible.
\item There is a self-adjoint $\tilde{F}\in C^\infty([0,1],\End^*_B(\mathpzc{E}))$ such that $\tilde{F}^2=1$, $[\tilde{F}(t),a]\in \mathbb{K}_B(\mathpzc{E})$ for any $a\in \mathcal{A}$, and 
$$\tilde{F}(0)=D|D|^{-1}\quad\mbox{and}\quad \tilde{F}(0)=D'|D'|^{-1}.$$
\end{enumerate}
Then there exists a positive $\Delta\in \mathbb{K}_B(\mathpzc{E})$ with dense range such that 
\begin{itemize}
\item For any $t\in [0,1]$, $\tilde{F}(t)\Delta=\Delta \tilde{F}(t)$.
\item For any $a\in \mathcal{A}$ and $t\in [0,1]$, the densely defined operators 
$$[a,\Delta^{-1}], \quad\mbox{and}\quad [a,\tilde{F}(t)]\Delta^{-1},$$
are norm bounded and their continuous extensions depend continuously on $t\in [0,1]$. 
\end{itemize}

\end{lemma}

\begin{proof}
Consider the countably generated $*$-algebra $\tilde{\mathcal{A}}\subseteq \End_{B[0,1]}^*(\mathpzc{E}[0,1])$ generated by $\mathcal{A}$ and $\tilde{F}$. By taking a quasicentral approximate unit from the subalgebra $\mathbb{K}_B(\mathpzc{E})\subseteq \mathbb{K}_{B[0,1]}^*(\mathpzc{E}[0,1])$, Baaj-Julg's lifting argument \cite{baajjulg} produces a $\Delta_0\in \mathbb{K}_B(\mathpzc{E})$ such that $[b,\Delta_0^{-1}]$ and $[b,\tilde{F}]\Delta_0^{-1}$ are norm bounded and their continuous extensions belong to $\End_{B[0,1]}^*(\mathpzc{E}[0,1])$, for any $b\in \tilde{\mathcal{A}}$. We can now take $\Delta=(C+\tilde{F}\Delta_0^{-1}+\Delta_0^{-1}\tilde{F})^{-1}$ for a sufficiently large $C>0$.
\end{proof}

\begin{theorem}
\label{thmf}
Let $A$ and $B$ be $C^*$-algebras, and $\mathcal{A}\subseteq A$ a countably generated, dense $*$-subalgebra. Whenever $(\mathpzc{E},D)$ and $(\mathpzc{E},D')$ are $(\mathcal{A},B)$-cycles such that 
\begin{enumerate}
\item $(\mathpzc{E},D)$ and $(\mathpzc{E},D')$ are Lipschitz and both $D$ and $D'$ are invertible;
\item there is a self-adjoint $\tilde{F}\in C^\infty([0,1],\End^*_B(\mathpzc{E}))$ such that $\tilde{F}^2=1$ and 
$$\tilde{F}(0)=D|D|^{-1}\quad\mbox{and}\quad \tilde{F}(0)=D'|D'|^{-1},$$
\end{enumerate}
then $(\mathpzc{E},D)\sim_{\rm bor} (\mathpzc{E},D')$. In particular, whenever $(\mathpzc{E},D_1)$ and $(\mathpzc{E},D_2)$ are two $(\mathcal{A},B)$-cycles admitting an operator homotopy $(\mathpzc{E},\beta(D_1))\sim_{\rm oh} (\mathpzc{E},\beta(D_2))$, there is a $KK$-bordism $(\mathpzc{E},D_1)\sim_{\rm bor} (\mathpzc{E},D_2)$.
\end{theorem}

\begin{proof}
By the argument in Lemma \ref{lemmab}, we can assume that the assumptions of Lemma \ref{lemmae} hold and that $\tilde{F}$ and $\Delta$ exists. We also note that invertiblity and Lipschitz regularity, assumed in Lemma \ref{lemmae}, ensures that $[\tilde{F}(0),a]\mathpzc{E}\subseteq \mathrm{Dom}(D)$ and $[\tilde{F}(1),a]\mathpzc{E}\subseteq \mathrm{Dom}(D')$ for any $a\in \mathcal{A}$. Consider the family $\hat{D}(t):=\tilde{F}(t)\Delta^{-1}$, $t\in [0,1]$, with constant domain $\Delta\mathpzc{E}$. By construction, $\tilde{F}(0)$ satisfies the assumptions of Lemma \ref{prop62kaad} and implies that $(\mathpzc{E},D)\sim_{\rm bor}(\mathpzc{E},\hat{D}(0))$. Similarly, $(\mathpzc{E},D')\sim_{\rm bor}(\mathpzc{E},\hat{D}(1))$. It remains to show that $(\mathpzc{E},\hat{D}(0))\sim_{\rm bor}(\mathpzc{E},\hat{D}(1))$. This is immediate from Proposition \ref{homotopywithstrongcond} using that $\tilde{F}$ is norm smooth and commutes with $\Delta$. The final conclusion of the theorem follows from arguing as in Lemma \ref{lemmab}.
\end{proof}

The proof of Theorem \ref{thmf} closely follows Kaad's argument \cite{Kaadunbdd}. Theorem \ref{bddisothm} now follows from combining Lemma \ref{lemmab} with Theorem \ref{thmf}. 

\section{Homotopy relations and universal cycles}
\label{sec:universal}

We can now relate the $KK$-bordism group to the homotopy constructions of Kaad \cite{Kaadunbdd} and Mesland-van Dungen \cite{meslandvandung}. We do again emphasize how the $KK$-bordism group and the homotopy groups are constructed with different aims. The $KK$-bordism group comes with an abundance of geometric examples but also with additional technicalities. The homotopy groups have a conceptual elegance but can prove unwieldy in geometric examples, for instance in realizing geometric bordisms. Mesland-van Dungen \cite{meslandvandung} considers a universal homotopy group $\overline{UKK}_*(A,B)$ which by definition is independent of the dense subalgebra, so we first formulate the result in terms of Kaad's unbounded $KK$-group $UK^{\rm top}(\mathcal{A},B)$. The group  $UK^{\rm top}(\mathcal{A},B)$ is defined as the set of equivalence classes of $(\mathcal{A},B)$-cycles under the equivalence relation generated by degenerate equivalence and unbounded operator homotopy, for details see \cite[Section 2]{Kaadunbdd}. For clarity, let us also mention that there are other approaches to refining $KK$-theory that are nicely overviewed in \cite{cumero}.

\begin{theorem}
Let $\mathcal{A}$ be a countably generated $*$-algebra with $C^*$-closure $A$. For any $C^*$-algebra $B$, we have a commuting diagram
$$\xymatrix{
\Omega_*(\mathcal{A}, B) \ar[rdd]^{\beta}   \ar[rr] & & UK^{\rm top}_*(\mathcal{A},B) \ar[ldd]_{\mathcal{F}}  
 \\ \\
&KK_*(A,B)&
}$$
where all the maps are isomorphisms, the horizontal map is induced from the identity map on cycles and $\mathcal{F}$ denotes the bounded transform on homotopy classes.
\end{theorem}

The theorem is immediate from Theorem \ref{alknlkanda} and \ref{bddisothm}, and \cite[Theorem 7.1]{Kaadunbdd}.

A new approach to unbounded $KK$-theory was found \cite{meslandvandung} that does not require fixing a dense subalgebra $\mathcal{A}\subseteq A$. In \cite{meslandvandung}, universal unbounded $KK$-cycles for a pair of $C^*$-algebras $(A,B)$ were introduced and the corresponding set of homotopy classes form an abelian group $\overline{UKK}_*(A,B)$ naturally isomorphic to $KK_*(A,B)$ via the bounded transform. We will in future work \cite{monographtwo} make use of the universal picture of unbounded $KK$-theory below to prove independence of dense subalgebra $\mathcal{A}\subseteq A$ in certain relative constructions.

\begin{define}
Let $A$ and $B$ be two $C^*$-algebras. 
\begin{itemize}
\item A universal $A-B$-chain is a triple $(\mathpzc{E},D,i)$ where $i:A\to \pmb{A}$ is an inclusion of $C^*$-algebras and $(\mathpzc{E},D)$ is an $(\pmb{\mathcal{A}},B)$-chain for some dense $\pmb{\mathcal{A}}\subseteq \pmb{A}$. 
\item A universal $A-B$-chain is said to be symmetric, half-closed or closed if $(\mathpzc{E},D)$ is so. A closed universal chain is called a universal cycle.
\item A universal $A-B$-bordism consists of $(\mathfrak{X},i)$ where $i:A\to \pmb{A}$ is an inclusion of $C^*$-algebras and $\mathfrak{X}=(\mathfrak{X}^\circ,\Theta,\mathfrak{X}^\partial)$ is an $(\pmb{\mathcal{A}},B)$-bordism for some dense $\pmb{\mathcal{A}}\subseteq \pmb{A}$. In this case, the universal $A-B$-cycle $\mathfrak{X}^\partial$ is said to be nullbordant which is written as 
$$\mathfrak{X}^\partial\sim_{\rm un-bor} 0.$$
\item An isomorphism of two universal $A-B$-chains $(\mathpzc{E}_1,D_1,i_1)$ and $(\mathpzc{E}_2,D_2,i_2)$ is a unitary isomorphism $U:\mathpzc{E}_1\xrightarrow{\sim} \mathpzc{E}_2$ such that $D_1=U^*D_2U$ and such that the induced $A$-actions on $\mathpzc{E}_1$ and $\mathpzc{E}_2$ coincide under $U$.
\end{itemize}
We write $Z^{\rm un}_*(A,B)$ for the set of isomorphism classes of universal $A-B$-cycles.
\end{define}

\begin{remark}
\label{remarnknad}
We note that a universal $A-B$-cycle is a minor variation of the unbounded $A-B$-cycles in \cite{meslandvandung}, the latter consists of triples $(\pi,\mathpzc{E},D)$ where $\pi: A\to \overline{\mathrm{Lip}_0(D)}$. We include the inclusion $i$ in the definition to simplify the notation in limiting arguments. 
\end{remark}

It is clear that $Z^{\rm un}_*(A,B)$ forms an abelian semigroup under direct sum 
$$(\mathpzc{E}_1,D_1,i_1)+(\mathpzc{E}_2,D_2,i_2):=(\mathpzc{E}_1\oplus \mathpzc{E}_2,D_1\oplus D_2,i_1\oplus i_2),$$
where $i_1\oplus i_2:A\to \pmb{A}_1\oplus \pmb{A}_2$ denotes $a\mapsto i_1(a)\oplus i_2(a)$. If we have an equality of universal cycles
$$(\mathpzc{E}_1,D_1,i_1)+(-(\mathpzc{E}_2,D_2,i_2))=\mathfrak{X}^\partial,$$
for a universal $A-B$-bordism $(\mathfrak{X},i)$ we write $(\mathpzc{E}_1,D_1,i_1)\sim_{\rm un-bor}(\mathpzc{E}_2,D_2,i_2)$ and say that the two cycles are universally bordant.

\begin{prop}
Bordism of universal $A-B$-cycles $\sim_{\rm un-bor}$ defines an equivalence relation on $Z^{\rm un}_*(A,B)$. The set $\Omega_*^{\rm un}(A,B)$ of bordism classes forms a $\Z/2$-graded abelian group under direct sum of cycles that depends covariantly on $B$ and contravariantly on $A$.
\end{prop}

\begin{proof}
The discussion in Section \ref{subsecglying} applies verbatim showing $\sim_{\rm un-bor}$ is an equivalence relation on $Z^{\rm un}_*(A,B)$. It can be proven as in \cite[Theorem 2.24]{DGM} that $\Omega_*^{\rm un}(A,B)$ forms a $\Z/2$-graded abelian group under direct sum of cycles; indeed $(\mathpzc{E},D,i)+(-(\mathpzc{E},D,i))$ is nullbordant. It is immediate from the definition that $\Omega_*^{\rm un}(A,B)$ depends covariantly on $B$ and contravariantly on $A$.
\end{proof}

\begin{theorem}
\label{lngjdjkddkj}
Let $A$ and $B$ be two $C^*$-algebras. The bounded transform 
\begin{align*}
\beta_{\rm un}:\Omega_*^{\rm un}(A,B)&\to KK_*(A,B), \\ 
(\mathpzc{E},D,i)&\mapsto i^*\beta(\mathpzc{E},D)=(i^*\mathpzc{E},D(1+D^2)^{-1/2}),
\end{align*}
is a well defined map of $\Z/2$-graded abelian groups.
\end{theorem}

Theorem \ref{lngjdjkddkj} can be proven directly at the level of cycles using the idea underlying Theorem \ref{alknlkanda}. Alternatively, it is a consequence of the next theorem in combination with Theorem \ref{alknlkanda}. We write $\mathfrak{Inc}(A)$ for the set of unitary equivalence classes of pairs $(i,\pmb{\mathcal{A}})$ where $i$ is an inclusion of $C^*$-algebras $i:A\to \pmb{A}$ and $\pmb{\mathcal{A}}\subseteq \pmb{A}$ is a dense $*$-subalgebra. The set $\mathfrak{Inc}(A)$ becomes a directed set by declaring $(i_1,\pmb{\mathcal{A}}_1)\geq (i_2,\pmb{\mathcal{A}}_2)$ if there is an inclusion $j:\pmb{A}_1\to \pmb{A}_2$ such that $i_2=j\circ i_1$. If $A$ is separable, we write $\mathfrak{SInc}(A)\subseteq \mathfrak{Inc}(A)$ for the directed subset of pairs $(i,\pmb{\mathcal{A}})$ where $\pmb{\mathcal{A}}$ is countably generated. We conclude the following description of the universal $KK$-bordism group.

\begin{prop}
\label{dirneirnedin}
Let $A$ and $B$ be two $C^*$-algebras. The map 
$$\Omega_*^{\rm un}(A,B)\to \varinjlim_{\mathfrak{Inc}(A)} \Omega_*(\pmb{\mathcal{A}},B),$$
induced from $(\mathpzc{E},D,i)\mapsto (\mathpzc{E},D)$, is a well defined isomorphism. If $A$ is separable, the similarly defined map defines an isomorphism 
$$\Omega_*^{\rm un}(A,B)\cong \varinjlim_{\mathfrak{SInc}(A)} \Omega_*(\pmb{\mathcal{A}},B).$$
\end{prop}

\begin{theorem}
\label{bladalnl}
Let $A$ and $B$ be two $C^*$-algebras. If $A$ is separable, the bounded transform $\beta_{\rm un}:\Omega_*^{\rm un}(A,B)\to KK_*(A,B)$ is an isomorphism.
\end{theorem}

\begin{proof}
We have a commuting diagram 
\begin{align}
\label{comomondianrin}
\xymatrix{
\Omega_*^{\rm un}(A, B) \ar[dd]^{\beta_{\rm un}}   \ar[rr] & & \varinjlim_{\mathfrak{SInc}(A)} \Omega_*(\pmb{\mathcal{A}},B)
\ar[dd]_{\beta_{\rm un}}  
 \\ \\
KK_*(A,B)\ar[rr]&&\varinjlim_{\mathfrak{SInc}(A)} KK_*(\pmb{A},B)
}
\end{align}
The top horizontal arrow is an isomorphism by Proposition \ref{dirneirnedin} and the right vertical map is an isomorphism by Theorem \ref{bddisothm}. Write $SInc(A)$ for the directed set of unitary isomorphism classes of inclusions of $A$ into separable $C^*$-algebras. The identity on $A$ is a final object in $SInc(A)$. We then have the isomorphisms 
$$\varinjlim_{\mathfrak{SInc}(A)} KK_*(\pmb{A},B)\cong\varinjlim_{SInc(A)} KK_*(\pmb{A},B)\cong KK_*(A,B).$$
This isomorphism coincides with the bottom arrow in \eqref{comomondianrin}. We conclude that $\beta_{\rm un}:\Omega_*^{\rm un}(A,B)\to KK_*(A,B)$ is an isomorphism.
\end{proof}

Theorem \ref{bladalnl} and \cite[Theorem B]{meslandvandung} imply the next theorem.

\begin{theorem}
Let $A$ be a separable $C^*$-algebra.  For any $C^*$-algebra $B$, we have a commuting diagram
$$\xymatrix{
\Omega_*^{\rm un}(A, B) \ar[rdd]^{\beta}   \ar[rr] & & \overline{UKK}_*(A,B) \ar[ldd]_{\beta_{\rm un}}  
 \\ \\
&KK_*(A,B)&
}$$
where all the maps are isomorphisms and the vertical map is induced from the identification of cycles discussed in Remark \ref{remarnknad}. 
\end{theorem}

\section{Examples when the bounded transform is an isomorphism}
\label{examplesofisooss}

While Theorem \ref{bddisothm} is only stated for countably generated $*$-algebras, we can readily generate more examples using the rigidity properties of Part \ref{part:prelim} and \ref{part:geocons}. We summarize this procedure in the following theorem. 

\begin{theorem}
\label{bddisothmmoregen}
Let $\mathcal{A}$ be a $*$-algebra with $C^*$-closure $A$. Assume that there exists a countably generated subalgebra $\mathcal{A}_0\subseteq \mathcal{A}$, a number $\alpha_0>1$ such that $\mathcal{A}\subseteq \mathcal{A}^{(\alpha_0)}$ or for some $\alpha_1\in (0,1]$ there is a Banach $*$-algebra $\mathcal{A}_1\subseteq \mathcal{A}^{(\alpha_0)}$ such that $\mathcal{A}\subseteq [\mathcal{A}_1,A]_{\alpha_1}$. Then $\mathcal{A}$ is $\Omega_*$-admissible, i.e. for any $C^*$-algebra $B$, the bounded transform 
$$\beta:\Omega_*(\mathcal{A},B)\to KK_*(A,B),$$
is an isomorphism.
\end{theorem}

Theorem \ref{bddisothmmoregen} follows directly from combining Theorem \ref{cpxinerpbofodthm} and \ref{caalpainandthm} with Theorem \ref{bddisothm}. While the assumptions of Theorem \ref{bddisothmmoregen} might seem technical at first, we shall below see numerous examples where the conditions naturally appear and are easy to verify.

\subsection{Manifolds and closed subsets of the Hilbert cube}
\label{ainaodna}

In this section we will give examples of commutative $*$-algebras for which the bounded transform is an isomorphism. We will make use of the Hilbert cube 
$$\pmb{\Box}:=[-1,1]^\N.$$
By Tychonoff's theorem, the Hilbert cube is compact. We can consider $\pmb{\Box}$ as a separable metric space by embedding $\pmb{\Box}$ into the unit ball of $\ell^2(\N,\R)$ as the subset $\prod_{n\in \N}[-1/(n+1),1/(n+1)]$ of sequences $(x_n)_{n\in \N}$ such that $|x_n|\leq 1/(n+1)$. Any compact, separable, metric space is homeomorphic to a closed subset of $\pmb{\Box}$. For the sake of examples, we also take a Cantor set $\pmb{C}$.

We define $C^\infty(\pmb{\Box}\times \pmb{C})$ as the space of functions that in the Cantor direction is locally constant and in the Hilbert cube direction only depend on finitely many variables and do so in a smooth way. We note that $C^\infty(\pmb{\Box}\times \pmb{C})$ is the direct limit of $C^\infty([-1,1]^N\times \pmb{C})$, embedded via pullback along the projections $\pmb{\Box}\to [-1,1]^N$. If $X\subseteq X_0$ is an open subset of a closed subset $X_0\subseteq \pmb{\Box}\times \pmb{C}$, we define 
$$C^\infty_c(X):=\{f|_X: f\in C^\infty(\pmb{\Box}\times \pmb{C})\}\cap C_c(X).$$
We conclude from Stone-Weierstrass' theorem that $C^\infty(X_0)\subseteq C(X_0)$ is dense, so in particular we conclude the following. 

\begin{prop}
If $X\subseteq X_0$ is an open subset of a closed subset $X_0\subseteq \pmb{\Box}\times \pmb{C}$, we have a dense inclusion 
$$C^\infty_c(X)\subseteq C_0(X).$$
\end{prop}

\begin{theorem}
\label{closedinrnthm}
Let $X\subseteq X_0$ be an open subset of a closed subset $X_0\subseteq \pmb{\Box}\times \pmb{C}$. For any $C^*$-algebra $B$, the bounded transform 
$$\beta:\Omega_*(C^\infty_c(X),B)\to KK_*(C_0(X),B),$$
is an isomorphism.
\end{theorem}

\begin{proof}
Let $X_k$ denote the coordinate functions on $\pmb{\Box}$. Enumerate the cylinder functions $c_l$ on $\pmb{C}$. We can pick a countable partition of unity $(\chi_j)_{j\in \N}$ of $X$ from $C^\infty_c(X)$ and define $\mathcal{A}_0$ as the $*$-algebra generated by $\{\chi_j,c_l\chi_j, X_k\chi_j: j,k,l\in \N\}$. By Stone-Weierstrass' theorem, $\mathcal{A}_0\subseteq C_0(X)$ is dense. It follows from Theorem \ref{bddisothm} that 
$$\beta:\Omega_*(\mathcal{A}_0,B)\to KK_*(C_0(X),B),$$
is an isomorphism for all $B$. A short computation shows that $\mathcal{A}_0\subseteq C^\infty_c(X)\subseteq \overline{\mathcal{A}_0}^{(\infty)}$ and the theorem now follows from Theorem \ref{caalpainandthm} (see page \pageref{caalpainandthm}, compare to Theorem \ref{bddisothmmoregen}).
\end{proof}

\begin{cor}
\label{closedinrcor}
Let $X$ be an open subset of a compact, separable and metrizable space (e.g. a locally finite CW-complex or a paracompact manifold with boundary). For $C^\infty_c(X)\subseteq C_0(X)$ defined via an embedding into $\pmb{\Box}\times \pmb{C}$ as above, and any $C^*$-algebra $B$ the bounded transform 
$$\beta:\Omega_*(C^\infty_c(X),B)\to KK_*(C_0(X),B),$$
is an isomorphism.
\end{cor}

\begin{remark}
\label{discusdinbd}
By combining Corollary \ref{closedinrcor} with \cite[Theorem 4.11]{DGM} we see that the Baum-Douglas relations \cite{BD,BDbor} can be replaced with Hilsum's analytic $KK$-bordism relation. The latter analytic relation is implemented on geometric Baum-Douglas cycles via the map $\gamma_0$ introduced in  \cite[Section 4]{DGM} mapping geometric cycles to analytic cycles.
\end{remark}

\subsection{Poincaré duality of manifolds} 

Let us discuss how Corollary \ref{closedinrcor} can be used to study Poincaré duality explicitly at the unbounded level. The $KK$-bordism groups $\Omega_*(C^\infty(M;\C\ell(M)),B)$ inherits several structures from the $KK$-groups that allow us to qualitatively describe unbounded $KK$-cycles on manifolds up to $KK$-bordism. Consider a compact manifold $M$ with an auxiliary choice of Riemannian metric and write  $\C\ell(M)\to M$ for its Clifford algebra bundle. Write $\Delta_M:=\{(x,x): x\in M\}\subseteq M\times M$ for the diagonal, and choose an identification $\varphi:N^*\Delta_M\hookrightarrow M\times M$ of the conormal bundle to the diagonal with a tubular neighborhood of the diagonal. The Thom class of the bundle $p_N:N^*\Delta_M\to M$ defines a class $\mathrm{Th}(N^*\Delta_M)\in KK_0(\C,C_0(N^*\Delta_M,p_N^*\C\ell(M)))$ and we can define a class $\Theta\in KK_0(\C,C(M)\otimes C(M,\C\ell(M)))$ as the image of  $\mathrm{Th}(N^*\Delta_M)$ under the wrong way map 
\begin{align*}
KK_0(\C,C_0(N^*\Delta_M,p_N^*\C\ell(M)))\to &KK_0(\C,C(M\times M,p_2^*\C\ell(M)))=\\
&=KK_0(\C,C(M)\otimes C(M,\C\ell(M))),
\end{align*}
induced by $\varphi$. The class $\Theta$ can be represented in $K$-theory as $[\theta_M]-[1]$ for a smooth projection $\theta_M$. The projection $\theta_M$ can explicitly be described using that $\mathrm{Th}(N^*\Delta_M)$ is the class of the unbounded cycle $(C_0(N^*\Delta_M,p_N^*\wedge^*T^*M),c)$ where $c$ is  pointwise Clifford multiplication on $N^*\Delta_M$; indeed $\theta_M$ is the extension of the Higson transform $h_f\in 1+C_c^\infty(N^*\Delta_M,\mathrm{End}(p_N^*\wedge^*T^*M))$ of $c$ defined as
$$h_f(c):=(1-f(c)^2)\gamma-f(c)\sqrt{1-f(c)^2}+\frac{1-\gamma}{2},$$
for an odd function $f\in C^\infty(\R,[-1,1])$ with $f^2-1\in C^\infty_c(\R)$ and $f(x)\geq 0$ for $x>0$, and where $\gamma$ denotes the grading. Compare to Definition \ref{casandlknad} above and \cite[Equation (1.5)]{Hig}.

We write $[\slashed{D}_{\C\ell}]\in KK_0(C(M,\C\ell(M)),\C)$ for the class of the Euler-de Rham operator; in \cite{Kas2} this is called the Dirac element of $M$ and is denoted by $[d_M]$. It is represented by the unbounded cycle $(L^2(M;\wedge^*T^*M),\slashed{D}_{\C\ell})$ where $\slashed{D}_{\C\ell}=\mathrm{d}+\mathrm{d}^*$. The cap product with $[\slashed{D}_{\C\ell}]$ on $KK$ lifts to a mapping 
$$\cap[\slashed{D}_{\C\ell}]:K_*(C^\infty(M,B))\to\Omega_*(C^\infty(M;\C\ell(M)),B),$$
simply by a mapping of cycles defined by 
$$p\mapsto  (pL^2(M;\wedge^*T^*M\otimes B^N),p(\slashed{D}_{\C\ell}\otimes 1_{B^N})p),$$ 
for a projection $p\in M_N(C^\infty(M,B))$. We shall tacitly use a connection $\nabla$ for the cotangent bundle of $M$.

\begin{cor}
\label{poincaredual}
Let $M$ be a compact manifold with Clifford algebra bundle $\C\ell(M)\to M$. Then for any $\sigma$-unital $C^*$-algebra $B$, the mapping 
\begin{align*}
\rho:\Omega_*(C^\infty(M;\C\ell(M)),B)&\to \Omega_*(\C,C(M,B)),\\
\rho(\mathpzc{E},D):=&(\theta_M\cdot (C(M\times M;p_2^*\wedge^*T^*M)\otimes_{C(M)} \mathpzc{E}),\theta_M\cdot (1\otimes_\nabla D)\cdot \theta_M)-\\
&-(C(M\times M;p_2^*\wedge^*T^*M)\otimes_{C(M)} \mathpzc{E},1\otimes_\nabla D),
\end{align*}
is an isomorphism. Its inverse is given by the composition 
$$\Omega_*(\C,C(M,B))\xrightarrow{\ind}K_*(C(M,B))=K_*(C^\infty(M,B))\xrightarrow{\cap[\slashed{D}_{\C\ell}]}\Omega_*(C^\infty(M;\C\ell(M)),B).$$
In particular, the mapping 
\begin{align*}
\Omega_*(C^\infty(M;\C\ell(M)),B)\to &K_*(C(M,B))),\\
(\mathpzc{E},D)\mapsto &\ind_{C(M,B)}(\rho(\mathpzc{E},D)),
\end{align*}
is an isomorphism of abelian groups and any cycle $(\mathpzc{E},D)$ for $\Omega_*(C^\infty(M;\C\ell(M)),B)$ is bordant to one of the form $(L^2(M;\mathcal{E}_B), D_\mathcal{E})$ where $\mathcal{E}_B\to M$ is a $B$-Clifford bundle and $D_\mathcal{E}$ a Dirac operator thereon.
\end{cor}

\begin{proof}
The classes $\Theta\in KK_0(\C,C(M)\otimes C(M,\C\ell(M)))$ and $m^*[\slashed{D}_{\C\ell}]\in KK_0(C(M)\otimes C(M,\C\ell(M)),\C)$, for $m:C(M)\otimes C(M,\C\ell(M))\to C(M,\C\ell(M))$ denoting multiplication, define  ordinary Poincaré duality of $KK$-groups by \cite[Section 4]{Kas2}. The theorem now follows from a diagram chase with the bounded transform using Corollary \ref{closedinrcor}. 
\end{proof}

\subsection{Lipschitz functions and topological manifolds}

The goal of the subsection is to prove the following theorem.

\begin{theorem}
\label{lipmfdthm}
Let $X$ be a compact Lipschitz manifold. For any $C^*$-algebra $B$, the bounded transform 
$$\beta:\Omega_*(\mathrm{Lip}(X),B)\to KK_*(C(X),B),$$
is an isomorphism.
\end{theorem}

The key step in the proof of the theorem is made using the next lemma.

\begin{lemma}
\label{lipmfdtlemmaveofoe}
Let $X$ be a compact Lipschitz manifold with a Lipschitz embedding $i:X\to \R^N$. Take $k>2$. Then $C^k(i(X)):=\mathrm{im}(C^k_c(\R^N)\to C(i(X)))$ satisfies that 
\begin{enumerate}
\item $C^k(i(X))=\overline{C^\infty(i(X))}^{(k)}$
\item $C^k(i(X))\subseteq \mathrm{Lip}(X)\subseteq C^k(i(X))_\alpha$ for any $\alpha\in (0,1/k)$.
\end{enumerate}
In the last statement, we identified $\mathrm{Lip}(X)$ with $\mathrm{Lip}(i(X))$ under the embedding $i$.
\end{lemma}

\begin{proof}
It is readily verified that $\overline{C^\infty_c(\R^N)}^{(k)}=C^k_c(\R^N)$, and since restrictions commute with functional calculus we can conclude that $C^k(i(X))=\overline{C^\infty(i(X))}^{(k)}$. By the same token, $C^k_c(\R^N)\subseteq \mathrm{Lip}(\R^N)\cap C_c(\R^N)\subseteq C^k_c(\R^N)_\alpha$ for any $\alpha\in (0,1/k)$, and complex interpolation commutes with restriction so $C^k(i(X))\subseteq \mathrm{Lip}(X)\subseteq C^k(i(X))_\alpha$ follows for any $\alpha\in (0,1/k)$.
\end{proof}

\begin{proof}[Proof of Theorem \ref{lipmfdthm}]
If $X$ is a compact Lipschitz manifold, there exists a Lipschitz embedding into $\R^N$ as a closed subset. By Lemma \ref{lipmfdtlemmaveofoe} (item 2) and Theorem \ref{cpxinerpbofodthm} (see page \pageref{cpxinerpbofodthm}), it suffices to prove that 
$$\beta:\Omega_*(C^k(i(X)),B)\to KK_*(C(X),B),$$
is an isomorphism for all $B$ and some $k>2$. By Lemma \ref{lipmfdtlemmaveofoe} (item 1) and Theorem \ref{caalpainandthm} (see page \pageref{caalpainandthm}), it suffices to prove that  that 
$$\beta:\Omega_*(C^\infty(i(X)),B)\to KK_*(C(X),B),$$
is an isomorphism for all $B$. The last mapping is an isomorphism (for all $B$) by Theorem \ref{closedinrnthm} (see page \pageref{closedinrnthm}).
\end{proof}

\subsection{Étale groupoids}

In this subsection we consider étale groupoids over a totally disconnected space or a manifold with boundary, more generally over an open subset $X\subseteq X_0$ of a closed subset $X_0\subseteq \pmb{\Box}\times \pmb{C}$ of the Hilbert cube times a Cantor set. Let $\mathcal{G}$ denote such an étale groupoid. We write $C^\infty_c(\mathcal{G})$ for the compactly supported functions that under local bisection pull back to functions from $C^\infty(X)$ (see Subsection \ref{ainaodna}). The $*$-algebra $C^\infty_c(\mathcal{G})$ can be considered a dense $*$-subalgebra of the full groupoid $C^*$-algebra, or the reduced groupoid $C^*$-algebra. For our purposes it is not important which $C^*$-completion we pick, but when writing $C^\infty_c(\mathcal{G})$ we are implicitly choosing the $C^*$-completion and this choice remains implicit in the construction of $\Omega_*(C^\infty_c(\mathcal{G}),B)$ (and in general $\Omega_*(C^\infty_c(\mathcal{G}),B)$ depends on this choice).

\begin{theorem}
\label{etalegroupoidthm}
Let $\mathcal{G}$ be an étale groupoid over an open subset $X\subseteq X_0$ of a closed subset $X_0\subseteq \pmb{\Box}\times \pmb{C}$. For any $C^*$-algebra $B$, the bounded transform 
$$\beta:\Omega_*(C^\infty_c(\mathcal{G}),B)\to KK_*(C^*(\mathcal{G}),B),$$
is an isomorphism.
\end{theorem}

\begin{proof}
Pick a countable partition of unity $(g_j)_{j\in \N}\subseteq C^\infty_c(\mathcal{G})$ such that the domain and source mapping $\mathrm{supp}(g_j)\to X$ are homeomorphisms onto their ranges. Define $\mathcal{A}$ to be the countably generated $*$-algebra generated by $C^\infty_c(X)$ and $(g_j)_{j\in \N}$. By Stone-Weierstrass there is a linear embedding $\mathcal{A}\to C_c(\mathcal{G})$ with dense range, so $\mathcal{A}\subseteq C^*(\mathcal{G})$ is dense. 

Any $a\in C^\infty_c(\mathcal{G})$ can be written as a finite sum $a=\sum a_j g_j$ for $a_j\in C^\infty_c(X)$. Since $(g_j)_{j\in \N}\subseteq \mathcal{A}$ (by construction) and $C^\infty_c(X)\subseteq \overline{\mathcal{A}}^{(\infty)}$, we have that $\mathcal{A}\subseteq C^\infty_c(\mathcal{G})\subseteq \overline{\mathcal{A}}^{(\infty)}$. The theorem now follows from Theorem \ref{bddisothmmoregen} (see also Theorem \ref{caalpainandthm} on page \pageref{caalpainandthm} and Theorem \ref{bddisothm} on page \pageref{bddisothm}).

\end{proof}

\begin{cor}
\label{talegroupoidcor}
Consider an open subset $X\subseteq X_0$ of a closed subset $X_0\subseteq \pmb{\Box}\times \pmb{C}$. Let $\mathcal{G}$ be of one of the following forms:
\begin{itemize}
\item $\mathcal{G}:=\Gamma$ is a discrete countable group, and so $C^\infty_c(\mathcal{G})=\C[\Gamma]$ is the group algebra and $C^*(\mathcal{G})=C^*(\Gamma)$. 
\item $\mathcal{G}:=X\rtimes \Gamma$, for a discrete countable group $\Gamma$ acting on $X$, and so $C^\infty_c(\mathcal{G})=C^\infty_c(X)\rtimes_{\rm alg} \Gamma$ and $C^*(\mathcal{G})=C_0(X)\rtimes \Gamma$.
\item $\mathcal{G}$ is the Deaconu groupoid associated with a surjective local homeomorphism $g:X\to X$, and so $C^\infty_c(\mathcal{G})=C^\infty_c(X)\rtimes_{\rm alg} \N$ and $C^*(\mathcal{G})=C_0(X)\rtimes_g \N=O_{E_g}$ is a Cuntz-Pimsner algebra.
\item $\mathcal{G}$ is the groupoid defining an AF-algebra $A$, i.e. $\mathcal{G}$ is a groupoid over the totally disconnected spectrum of a MASA in $A$, and $A=C^*(\mathcal{G})$.
\end{itemize}
Then for any $C^*$-algebra $B$, the bounded transform 
$$\beta:\Omega_*(C^\infty_c(\mathcal{G}),B)\to KK_*(C^*(\mathcal{G}),B),$$
is an isomorphism.
\end{cor}

\section{Exactness and homotopy invariance in $KK$-bordism groups}
\label{secexact}

Two important tools to compute $KK$-groups relies on their exactness and homotopy invariance. In light of Theorem \ref{bddisothmmoregen}, there is an abundance of $\Omega_*$-admissible $*$-algebras, so we can transfer results for $KK$-groups to $KK$-bordism groups. We note that the proofs of the results in this section are more or less trivial, as they piggyback on corresponding results in $KK$-theory, but we include them for completeness and hope that the results can be appreciated as having rather deep implications in Hilsum's theory of $KK$-bordisms (see for instance Remark \ref{exicsion} below).

\begin{theorem}
\label{exact2nd}
Let $\mathcal{A}$ be an $\Omega_*$-admissible $*$-algebra and 
$$0\to J\xrightarrow{j}B\xrightarrow{q} D\to 0,$$
a semisplit short exact sequence of $C^*$-algebras. Then there is a natural map
$$\partial:\Omega_*(\mathcal{A},D)\to \Omega_{*+1}(\mathcal{A},J),$$
that fits into a six term exact sequence
\[
\begin{CD}
\Omega_0(\mathcal{A},J) @>j_*>>\Omega_0(\mathcal{A},B) @>q_*>> \Omega_0(\mathcal{A},D)\\
@A\partial AA  @. @ VV\partial V\\
\Omega_1(\mathcal{A},D) @<q_*<< \Omega_1(\mathcal{A},B) @<j_*<< \Omega_1(\mathcal{A},J) \\
\end{CD} \]
\end{theorem}

By the $\Omega_*$-admissibility assumption on $\mathcal{A}$, the bounded transforms are isomorphisms and Theorem \ref{exact2nd} follows from exactness of $KK$-groups for semisplit short exact sequences, see for instance \cite[Chapter 19.5]{Bla}. To state the corresponding result for exactness in the contravariant leg, we make the following definition. 

\begin{define}
A sequence of $*$-homomorphisms
$$0\to \mathcal{I}\xrightarrow{i} \mathcal{A}\xrightarrow{p} \mathcal{C}\to 0,$$
will be called an $\Omega_*$SES if 
\begin{itemize}
\item $\mathcal{I}\subseteq I$, $\mathcal{A}\subseteq A$ and $\mathcal{C}\subseteq C$ are $\Omega_*$-admissible dense $*$-subalgebras of $C^*$-algebras;
\item the maps $i$ and $p$ are continuous in the $C^*$-norm;
\item the sequence of maps 
$$0\to I\xrightarrow{i}A\xrightarrow{p}C\to 0,$$
is a semisplit short exact sequence of $C^*$-algebras.
\end{itemize}
\end{define}

We note that if we have an $\Omega_*$SES $0\to \mathcal{I}\xrightarrow{i} \mathcal{A}\xrightarrow{p} \mathcal{C}\to 0$, we are in its definition implicitly identifying the $*$-homomorphisms with their extensions by continuity to the ambient $C^*$-algebras and view the $\Omega_*$SES via its embedding into the short exact sequence of $C^*$-algebras: 
\[
\begin{CD}
0  @>>> \mathcal{I} @>i>>\mathcal{A} @>p>>  \mathcal{C}@>>>0 \\
@. @V\subseteq VV  @V\subseteq VV @ V\subseteq VV@.\\
0  @>>>  I @>i>>A @>p >> C@>>>0 \\
\end{CD} \]
The reader should be aware that $p:\mathcal{A}\to \mathcal{C}$ need not be surjective and $i(\mathcal{I})$ is potentially a proper subalgebra of $\ker(p:\mathcal{A}\to \mathcal{C})=\mathcal{A}\cap \ker(p:A\to C)$, which need not even be an ideal in $\mathcal{A}$.

\begin{theorem}
\label{exact1st}
Assume that $B$ is a $C^*$-algebra and that 
$$0\to \mathcal{I}\xrightarrow{i} \mathcal{A}\xrightarrow{p} \mathcal{C}\to 0,$$
is an $\Omega_*$SES. Then there is a natural map
$$\partial:\Omega_*(\mathcal{I},B)\to \Omega_{*-1}(\mathcal{C},B),$$
that fits into a six term exact sequence
\[
\begin{CD}
\Omega_0(\mathcal{C},B) @>p^*>>\Omega_0(\mathcal{A},B) @>i^*>> \Omega_0(\mathcal{I},B)\\
@A\partial AA  @. @ VV\partial V\\
\Omega_1(\mathcal{I},B) @<i^*<< \Omega_1(\mathcal{A},B) @<p^*<< \Omega_1(\mathcal{C},B) \\
\end{CD} \]
\end{theorem}

Again, Theorem \ref{exact1st} follows from the $\Omega_*$-admissibility assumption on $\mathcal{A}$ and exactness of $KK$-groups for semisplit short exact sequences, see for instance \cite[Chapter 19.5]{Bla}. 

\begin{remark}
\label{exicsion}
We note that the authors are unaware of any direct proof of Theorem \ref{exact1st}, even of the weaker statement that 
$$\Omega_*(\mathcal{C},B) \xrightarrow{p^*}\Omega_*(\mathcal{A},B) \xrightarrow{i^*} \Omega_*(\mathcal{I},B)$$
is exact for a short exact sequence $0\to \mathcal{I}\xrightarrow{i} \mathcal{A}\xrightarrow{p} \mathcal{C}\to 0$ giving rise to an $\Omega_*$SES. 

In fact, the exactness statement at $\Omega_*(\mathcal{A},B) $ in Theorem \ref{exact1st} is equivalent to the following rather deep statement. Suppose that $(\mathpzc{E},D)$ is an $(\mathcal{A},B)$-cycle with $i^*(\mathpzc{E},D)=\partial \mathfrak{X}$ for an $(\mathcal{I},B)$-bordism $\mathfrak{X}$, then there is an $(\mathcal{A},B)$-cycle $(\mathpzc{E}_0,D_0)$ with $I$ acting trivially on $\mathpzc{E}_0$ and an $(\mathcal{A},B)$-bordism $\tilde{\mathfrak{X}}$ with 
$$\partial\tilde{\mathfrak{X}}=-(\mathpzc{E},D)+(\mathpzc{E}_0,D_0).$$
\end{remark}

\begin{theorem}
Let $\mathcal{A}_1$ and $\mathcal{A}_2$ be $\Omega_*$-admissible $*$-algebras and $B$ a $C^*$-algebra. Assume that 
$$\pi:A_1\to C([0,1],A_2),$$
is a $*$-homomorphism such that $\pi_t:=\mathrm{ev}_t\circ \pi$ satisfies 
$$\pi_0(\mathcal{A}_1)\subseteq \mathcal{A}_2\quad\mbox{and}\quad \pi_1(\mathcal{A}_1)\subseteq \mathcal{A}_2.$$
Then 
$$\pi_0^*=\pi_1^*:\Omega_*(\mathcal{A}_2,B)\to \Omega_*(\mathcal{A}_1,B).$$
\end{theorem}

By a similar argument as above, the theorem follows from the $\Omega_*$-admissibility assumption on $\mathcal{A}_1$ and $\mathcal{A}_2$ and homotopy invariance of $KK$-groups, see for instance \cite[Chapter 17.9]{Bla}.

\section{Examples when the bounded transform is not an isomorphism}
\label{bddtransformnotisosec}

We will prove that if $\A$ is an amenable Banach $*$-algebra that includes densely and continuously in a $C^{*}$-algebra $A$ then $\Omega_{*}(\A,\mathbb{C})$ is "small" in a precise sense. This yields examples of dense subalgebras for which $\Omega_{*}(\A,\mathbb{C})\ncong KK_{*}(A,\mathbb{C})$.
The same argument applies to Kaad's group $UK^{\rm top}(\A,\mathbb{C})$ from \cite{Kaadunbdd}. 

For a Banach bimodule $X$ over a Banach algebra $\A$ we denote by $HH^{1}(\A,X)$ the first Hochschild cohomology group of $\A$ with coefficients in $X$. We recall that a bimodule derivation $\delta:\A\to X$ is a continuous linear map 
\[\delta:\A\to X,\quad  \forall a,b\in\A \quad \delta(ab)=a\delta(b)+\delta(a)b.\]
 The space of all such derivations is denoted $Z^{1}(\A,X)$. A bimodule derivation $\delta$ is called \emph{inner} if there exists $x\in X$ such that $\delta(a)=\delta_{x}(a):=ax-xa$. The space of inner derivations is denoted $B^{1}(\A,X)$ and is clearly a subspace of $Z^{1}(\A,X)$. The first \emph{Hochschild cohomology group of $\A$ with coefficients in $X$} then admits the description
\[HH^{1}(\A,X):=\frac{Z^{1}(\mathcal{A},X)}{B^{1}(\mathcal{A},X)}.\]
For a Banach bimodule $X$, the \emph{dual module} is the continuous dual space $X^{*}$ provided with the bimodule structure
\[a\cdot \varphi \cdot b (x):=\varphi(bxa),\quad a,b\in\A,\quad x\in X.\]
We call Banach bimodule $Y$ a \emph{dual Banach bimodule} if $Y=X^{*}$ for some Banach bimodule $X$.

\begin{define} 
A Banach algebra $\A$ is \emph{amenable} if $HH^{1}(\A,X^*)=0$ for all Banach bimodules $X$. We say that $\A$ is $\mathbb{B}(H)$-amenable if $HH^{1}(\A,\mathbb{B}(H))=0$.
\end{define}

Note that since $\mathbb{B}(H)=\mathcal{L}^1(H)^*$, $\mathbb{B}(H)$ is a dual Banach bimodule. Therefore amenability implies $\mathbb{B}(H)$-amenability.

\begin{lemma}
\label{commpert}
Let $\A$ a $\mathbb{B}(H)$-amenable Banach algebra and $(\A,H,D)$ an unbounded $(\A,\mathbb{C})$-cycle. Then there exists a self-adjoint $R\in \mathbb{B}(H)$ such that $[D-R,a]=0$ for all $a\in\A$. If $H$ is graded and the operator $D$ is odd, then $R$ can be chosen to be odd.
\end{lemma}

\begin{proof}
It follows from Proposition \ref{autoodkkd} that $\delta_{D}(a):=[D,a]$ defines a bimodule derivation $\delta_{D}:\A\to \mathbb{B}(H)$.
Since $\A$ is $\mathbb{B}(H)$-amenable, there exists $R\in\mathbb{B}(H)$ such that $\delta_{D}(a)=\delta_{R}(a)=[R,a]$. Since $\delta_D=\delta_R$ is a $*$-derivation, we can assume that $R$ is self-adjoint after replacing it with $(R+R^*)/2$. In case $(\A,H,D)$ is graded, since $[D,a]$ is an odd operator, it follows that $[R,a]$ is an odd operator and thus the even part of $R$ commutes with $\A$. We can thus without loss of generality replace $R$ by its odd part.
The result now follows.
\end{proof}

\begin{prop}
\label{cyclereduction}
Let $\A$ be an $\mathbb{B}(H)$-amenable Banach algebra, $(\A,H,D)$ an unbounded $(\A,\mathbb{C})$-cycle and $R$ be as in Lemma \ref{commpert}. Then the $\A$ representation on $H$ restricts to $\ker (D-R)$ and there it acts by compact operators.   We have
\[[(\A,H,D)]=[(\A,\ker(D-R),0)],\]
in $\Omega^{*}(\A,\mathbb{C})$.
\end{prop}

\begin{proof}
Since $R$ is bounded $(\A,H,D)$ is bordant to $(\A,H,D-R)$. Upon replacing $D-R$ with the bordant cycle defined from $D-R+\mathrm{sign}(D-R)$ we can assume that $D-R$ has closed range. Denote by $\iota:\ker (D-R)\to H$ the inclusion. The cycle
\begin{equation}
\label{degeneratemodification}
\left(\A, H\oplus -\ker(D-R),\begin{pmatrix}D-R & \iota\\ \iota^{*} & 0\end{pmatrix}\right),
\end{equation}
is bordant to  $(\A,H,D-R)\oplus -(\A,\ker(D-R),0)$ as it is a bounded perturbation. However, the cycle \eqref{degeneratemodification} is degenerate, as its operator commutes with $\A$ and is invertible while $D-R$ has closed range. Hence it is null-bordant and
\[[(\A,H,D-R)]-[ (\A,\ker(D-R),0)]=0,\]
in $\Omega^{*}(\A,\mathbb{C})$ as desired.
\end{proof}

\begin{theorem}
\label{amenable}
Let $\mathcal{A}$ be a $\mathbb{B}(H)$-amenable Banach algebra that includes densely and continuously in $A$. Then $\Omega_{1}(\A,\mathbb{C})=0$ and $\Omega_{0}(\A,\mathbb{C})$ is generated by $KK$-bordism classes of $*$-homomorphisms $\pi:A\to \mathbb{K}(H)$. Moreover, $\Omega_{0}(\A,\mathbb{C})\neq 0$ when there exists a $*$-homomorphism $\pi:A\to \mathbb{K}(H)$.
\end{theorem}

\begin{proof}
By Proposition \ref{cyclereduction} any element in $\Omega_{*}(\A,\mathbb{C})$ can be represented by a cycle $(A,H,0)$ (in particular $A$ acts by compact operators). Therefore
$\Omega_{0}(\A,\mathbb{C})$ consists of bordism classes of compact $A$ representations and $\Omega_{1}(\A,\mathbb{C})=0$.
\end{proof}

\begin{remark}
Theorem \ref{amenable} is an instance of the principle that if $\mathcal{A}\subseteq A$ is ``too big'', then there are too few $(\mathcal{A},B)$-cycles to exhaust $KK_*(A,B)$. We note that this stands in sharp contrast to the situation for $K$-theory, where the opposite phenomenon occurs: if $\mathcal{A}\subseteq A$ is ``too small'' then both injectivity and surjectivity of $K_*(\mathcal{A})\to K_*(A)$ can fail.
\end{remark}

We can apply Theorem \ref{amenable} to concrete $\mathbb{B}(H)$-amenable Banach algebras so as to conclude that $\Omega_{*}(\A,\mathbb{C})\ncong KK_{*}(A,\mathbb{C})$. A first class of examples arises by taking $\A=A$ a $C^{*}$-algebra. Recall that a $C^{*}$-algebra $A$ has the \emph{similarity property} if every homomorphism $A\to\mathbb{B}(H)$ is similar to a $*$-homomorphism, and $A$ has the \emph{derivation property} if every derivation $\delta:A\to \mathbb{B}(H)$ is inner. By a theorem of Kirchberg \cite{kirchequi} these two properties are equivalent.

\begin{cor} 
If $A$ is a nuclear $C^{*}$-algebra or a $C^*$-algebra with no tracial states, and $K^1(A)\neq 0$ there is no isomorphism $\Omega_{*}(A,\mathbb{C})\cong KK_{*}(A,\mathbb{C})$.
\end{cor}

\begin{proof}
All nuclear $C^{*}$-algebras and all traceless $C^{*}$-algebras have the similarity property \cite{bunce81,christensen81} and hence the derivation property. Thus it follows from Theorem \ref{amenable} that $\Omega_1(A,\mathbb{C})=0$. If $K^1(A)=KK_1(A,\C)\neq 0$, there is no isomorphism $\Omega_{*}(A,\mathbb{C})\cong KK_{*}(A,\mathbb{C})$.
\end{proof}

Clearly, the counterexamples arising from taking $\A$ to be a $C^{*}$-algebra are somewhat artificial. A more natural class of counterexamples arises from convolution algebras of locally compact groups. Recall that for a locally compact group $G$, $L^1(G)$ is amenable if and only if $G$ is amenable by a theorem of Johnson \cite{johnsonmem}.

\begin{cor} 
\label{lonenontosdd}
If $G$ is an amenable locally compact group for which $K^1(C^*(G))\neq 0$, there is no isomorphism $\Omega_*(L^{1}(G),\mathbb{C})\ncong KK_*(C^{*}(G),\mathbb{C})$.
\end{cor}

\begin{proof}
By Johnson's theorem \cite{johnsonmem}, $L^{1}(G)$ is $\mathbb{B}(H)$-amenable. Thus for amenable groups we can apply Theorem \ref{amenable} to $L^{1}(G)$ to deduce $\Omega_1(L^1(G),\C)=0$. 
\end{proof}

We note that $G=\Z$ satisfies all the assumptions of Corollary \ref{lonenontosdd}. The Fourier transform realizes $\ell^1(\Z)$ as a $*$-subalgebra of $C(S^1)$ -- the Wiener algebra. Let $\Z[S^1]$ denote the free group generated on $S^1$ viewed as a discrete set. It is readily deduced from Theorem \ref{amenable} that there is a surjection 
$$\Z[S^1]\to \Omega_*(\ell^1(\Z),\C).$$
The mapping sends a point $t$ to the even spectral triple $(\C_t,0)$ where $\C_t:=\C$ as an evenly graded Hilbert space with a left action of $\ell^1(\Z)$ given by the representation of $\Z$ with character $t$. We note that the class of $(\C_t,0)$ in $\Omega_*(\ell^1(\Z)\cap C^1(S^1),\C)$ is independent of $t$. Indeed, for $s,t\in S^1$ said bordism stems from Theorem \ref{homotopyofrep} using an interval $[s_0,t_0]\subseteq [0,1)$ for $\mathrm{e}^{2\pi is_0}=s$ and  $\mathrm{e}^{2\pi it_0}=t$. We can even say that $\Omega_*(\ell^1(\Z)\cap C^1(S^1),\C)\cong \Z$ for $*=0,1$ using Theorem \ref{cpxinerpbofodthm} and Corollary \ref{closedinrcor}. The situation for $\ell^1(\Z)$ is drastically different.

\begin{theorem}
\label{bordismforwiener}
The mapping defined in the preceding paragraph is an isomorphism 
$$\Omega_*(\ell^1(\Z),\C) \cong
\begin{cases}
\Z[S^1], \; &*=0,\\
0, \; &*=1.\end{cases}$$
In particular, $\Omega_*(\ell^1(\Z),\C)$ is an uncountably generated abelian group. Moreover, the bounded transform 
$$\beta:\Omega_*(\ell^1(\Z),\C)\to KK_*(C(S^1),\C),$$
fails to be injective for $*=0$ and fails to be surjective for $*=1$.
\end{theorem}

\begin{proof}
For $k\in \Z$ and $t\in S^1$, we let $\C_t^k$ denote the Hilbert space $\C^{|k|}$ graded as an even subspace if $k>0$ and as an odd subspace if $k<0$. We use the convention that $\C^0=0$. The graded Hilbert space $\C_t^k$ carries an action of $\ell^1(\Z)$ defined as the diagonal action by the character $t$. It suffices to show that for $t_1,\ldots, t_N\in S^1$ and $k_1,\ldots, k_N\in \Z$ are such that $\bigoplus_{j=1}^N (\C_{t_j}^{k_j},0)\sim_{\rm bor}0$, then $k_1=\cdots=k_N=0$. Set $\mathpzc{E}:=\bigoplus_{j=1}^N \C_{t_j}^{k_j}$ and $D:=0$. We need to show that $\mathpzc{E}=0$ if $(\mathpzc{E},D)\sim_{\rm bor}0$.

Assume that $\mathfrak{X}$ is a bordism with boundary $(\mathpzc{E},D)$. By the same argument as in the proof of Theorem \ref{amenable}, the interior chain $(\mathpzc{N},T)$ of $\mathfrak{X}$ must satisfy that $\mathpzc{N}$ is finite-dimensional since $C_0((0,1],S^1)$ acts by compacts which also holds true on $p\mathcal{N}\cong \Psi(\mathpzc{E})$. We can from this argument also conclude that $p=0$. As such, $\partial \mathfrak{X}=0$ and it follows that $\mathpzc{E}=0$.

The claim concerning failure of injectivity for $*=0$ and surjectivity for $*=1$ follows from that $KK_*(C(S^1),\C)\cong \Z\oplus \Z[1]$.
\end{proof}

\begin{remark}
For $s,t\in S^1$ with $s\neq t$, consider the cycle $(\C_t,0)\oplus (-\C_s,0)$ which is both a bounded and an unbounded cycle. Theorem \ref{bordismforwiener} shows that this cycle is not nullbordant as an $(\ell^1(\Z),\C)$-cycle. This stands in contrast to the homotopy approach, this cycle is null homotopic both in the bounded and the unbounded sense.
\end{remark}

\section{Examples when $\Omega_*(-,\C)$ fails to satisfy excision}
\label{sec:failexc}

A natural question is to ask whether the $KK$-bordism groups can be computed from excision. This is the case for Kasparov's $KK$-groups: exact sequences arise from semisplit extensions as shown in for instance \cite[Chapter 19.5]{Bla}, see also \cite{HRbook}, and plays an important role in applications \cite{HigsonChar,HigsonCount}. In this section we will construct an example where excision in the contravariant leg fails. Of course, if we are in the class of $*$-algebras studied in Section \ref{bddtransformisosec} where the bounded transform is an isomorphism, excision holds because it holds in $KK$-theory. Our construction relies on the counterexamples to the bounded transform being an isomorphism from Section \ref{bddtransformnotisosec} in combination with the computational virtues of Section \ref{bddtransformisosec}.

Let $\mathcal{A}$ denote the Wiener algebra on $S^1$. This algebra consists of functions $a\in C(S^1)$ whose Fourier expansion $a=\sum_{k\in \Z}a_kz^k$ satisfies that $\sum_k |a_k|<\infty$. In other words, $\mathcal{A}$ can be defined from the property that it is isomorphic to $\ell^1(\Z)$ under the Fourier transform. We note that $\mathcal{A}\hookrightarrow C(S^1)$ is closed under holomorphic functional calculus and induces an isomorphism on $K$-theory. Nevertheless, amenability of $\Z$ and Theorem \ref{amenable}  shows that $\Omega_1(\mathcal{A},\C)=0$ despite $KK_1(C(S^1),\C)\cong \Z$. 

Define the *-subalgebra $\mathcal{B}\subseteq \mathcal{T}$ of the Toeplitz algebra to consist of Toeplitz operators with symbol in $\mathcal{A}$. By construction, there is a short exact sequence of $*$-algebras
\begin{equation}
\label{toepwiener}
0\to \K(H^2(S^1))\to \mathcal{B}\to \mathcal{A}\to 0.
\end{equation}
Alternatively, we could set $\mathcal{B}:=P\mathcal{A}P+ \K(H^2(S^1))$ where $P$ denotes the Szegö projection. 

\begin{theorem}
\label{failureofexicsionthm}
There are no maps $\Omega_0(\K,\C)\to\Omega_1(\mathcal{A},\C)$ and $\Omega_1(\K,\C)\to \Omega_0(\mathcal{A},\C)$ for which the  six term sequence 
$$\begin{CD}
\Omega_0(\mathcal{A},\C)@>>>\Omega_0(\mathcal{B},\C) @>>>\Omega_0(\K,\C) \\
@AAA  @. @VVV \\
\Omega_1(\K,\C)@<<<\Omega_1(\mathcal{B},\C)@<<< \Omega_1(\mathcal{A},\C)
\end{CD},$$
induced from Equation \eqref{toepwiener} is exact. 
\end{theorem}

\begin{remark}
By Theorem \ref{amenable}, there is only one map $\Omega_0(\K,\C)\to\Omega_1(\mathcal{A},\C)$ and only one map $\Omega_1(\K,\C)\to \Omega_0(\mathcal{A},\C)$ because $\Omega_1(\mathcal{A},\C)=0$ and $\Omega_1(\K,\C)=0$.
\end{remark}

\begin{proof}[Proof of Theorem \ref{failureofexicsionthm}]
Theorem \ref{amenable} (see page \pageref{amenable}) shows that $\Omega_1(\mathcal{A},\C)=0$. Moreover, $\K$ is a nuclear $C^*$-algebra, and therefore amenable, so Theorem \ref{amenable} shows that $\Omega_0(\K,\C)=\Z$ and $\Omega_1(\K,\C)=0$. Moreover, $\Omega_0(\mathcal{A},\C)\neq 0$ by Theorem \ref{amenable}. Let $j:\K\to \mathcal{B}$ denote the inclusion. We thus need to show that the following sequence is not exact:
$$\begin{CD}
\Omega_0(\mathcal{A},\C)@>>>\Omega_0(\mathcal{B},\C) @>j^*>>\Z \\
@AAA  @. @VVV \\
0@<<<\Omega_1(\mathcal{B},\C)@<<< 0
\end{CD}.$$
We shall prove that exactness fails because $j^*=0$. 

Let $\mathcal{A}_0$ denote the polynomial algebra on $S^1$. In other words, $\mathcal{A}_0$ can be defined from the property that it is isomorphic to $c_c(\Z)$ under the Fourier transform. We let $M_\infty(\C)$ denote the infinite matrices (with only finitely many non-zero entries) which we identify with a dense $*$-subalgebra of $\K(H^2(S^1))$ via the Fourier basis of $H^2(S^1)$. We define $\mathcal{B}_0:= P\mathcal{A}_0P+ M_\infty(\C)$, or in other words $\mathcal{B}_0$ is defined from a commuting diagram with exact rows
$$\begin{CD}
0@>>>M_\infty(\C)@>>>\mathcal{B}_0@>>> \mathcal{A}_0@>>> 0\\
@.@VVV  @VVV @VVV @.\\
0@>>>\K(H^2(S^1))@>>>\mathcal{B}@>>> \mathcal{A}@>>> 0
\end{CD}.$$
The algebras $\mathcal{A}_0$, $\mathcal{B}_0$ and $M_\infty(\C)$ are all countably generated. 

By Theorem \ref{bddisothm} and excision in $KK$-theory, we have a six term exact sequence
$$\begin{CD}
\Omega_0(\mathcal{A}_0,\C)@>>>\Omega_0(\mathcal{B}_0,\C) @>j_0^*>>\Omega_0(M_\infty(\C),\C) \\
@AAA  @. @VVV \\
\Omega_1(M_\infty(\C),\C)@<<<\Omega_1(\mathcal{B}_0,\C)@<<< \Omega_1(\mathcal{A}_0,\C)
\end{CD},$$
where $j:M_\infty(\C)\to \mathcal{B}_0$ denotes the inclusion. The Toeplitz algebra $\mathcal{T}$, which is the $C^*$-closure of $\mathcal{B}_0$, is $KK$-equivalent to $\C$ via the composition of the symbol mapping $\mathcal{T}\to C(S^1)$ composed with a point evaluation. As such, $KK_0(C(S^1),\C)\to KK_0(\mathcal{T},\C)$ is surjective. Theorem \ref{bddisothm} (see page \pageref{bddisothm}) implies that $\Omega_0(\mathcal{A}_0,\C)\to \Omega_0(\mathcal{B}_0,\C)$ is surjective and so $j_0^*=0$.

Note that the pull back mapping $\Omega_0(\K,\C)\to \Omega_0(M_\infty(\C),\C)$ is an isomorphism by Theorem \ref{bddisothm}. By naturality, we have a commuting diagram 
$$\begin{CD}
\Omega_0(\mathcal{B},\C)@>j^*>>\Omega_0(\K,\C) \\
@VVV  @V\cong VV \\
\Omega_0(\mathcal{B}_0,\C)@>0>> \Omega_0(M_\infty(\C),\C)
\end{CD}$$
We conclude that $j^*=0$ since it factors over the zero mapping.

\end{proof}

\end{document}